\documentclass[reqno]{amsproc}
\usepackage{mathtools}
\mathtoolsset{showonlyrefs=true}

\usepackage{amsmath,amsfonts,amssymb,amsthm}
\usepackage[abbrev,lite,nobysame]{amsrefs}
\usepackage{mathrsfs}
\usepackage{graphics,graphicx}
\usepackage[usenames,dvipsnames]{color}
\usepackage{times}
\usepackage{bbm}
\usepackage[margin=1in]{geometry}
\usepackage[colorlinks=true, pdfstartview=FitV, linkcolor=blue, citecolor=blue, urlcolor=blue]{hyperref}
\usepackage{subcaption}
\usepackage{enumitem}
\makeatletter

\renewcommand\subsection{\@startsection{subsection}{2}%
\normalparindent{.5\linespacing\@plus.7\linespacing}{-.5em}
{\normalfont\bfseries}}

\renewcommand\subsubsection{\@startsection{subsubsection}{3}%
\normalparindent{.5\linespacing\@plus.7\linespacing}{-.5em}
{\normalfont\bfseries}}

\newcommand{\diampar}[1]{\vspace{.5em}\noindent $\diamond$ \normalfont {\itshape #1.}}

\newcommand{\noindpar}[1]{\vspace{.5em}\noindent \normalfont {\itshape #1.}}

\def\@tocline#1#2#3#4#5#6#7{\relax
  \ifnum #1>\c@tocdepth % then omit
  \else
    \par \addpenalty\@secpenalty\addvspace{#2}%
    \begingroup \hyphenpenalty\@M
    \@ifempty{#4}{%
      \@tempdima\csname r@tocindent\number#1\endcsname\relax
    }{%
      \@tempdima#4\relax
    }%
    \parindent\z@ \leftskip#3\relax \advance\leftskip\@tempdima\relax
    \rightskip\@pnumwidth plus4em \parfillskip-\@pnumwidth
    #5\leavevmode\hskip-\@tempdima
      \ifcase #1
       \or\or \hskip 1em \or \hskip 2em \else \hskip 3em \fi%
      #6\nobreak\relax
    \dotfill\hbox to\@pnumwidth{\@tocpagenum{#7}}\par
    \nobreak
    \endgroup
  \fi}
\makeatother

\newtheorem{theorem}{Theorem}[section]
\newtheorem{proposition}[theorem]{Proposition}
\newtheorem{lemma}[theorem]{Lemma}
\newtheorem{corollary}[theorem]{Corollary}

\theoremstyle{definition}
\newtheorem{definition}[theorem]{Definition}

\newtheorem{remark}[theorem]{Remark}

\numberwithin{equation}{section}

\newcommand{\norm}[1]{\left\lVert #1 \right\rVert}

\newcommand{\eps}{{\varepsilon}}
\newcommand\dd{{\rm d}}

\newcommand{\Id}{{\mathbb{I}}}

\newcommand{\Lie}{{\text{Lie}}}

\newcommand{\supp}{{\text{supp}}}

\newcommand{\Lip}{\mathrm{Lip}}
\newcommand{\Leb}{\mathrm{Leb}}
\newcommand{\rank}{\mathrm{rank}}

\newcommand{\cE}{{\mathcal{E}}}

\newcommand{\cV}{{\mathcal{V}}}

\newcommand\cR{{\mathcal R}}

\newcommand{\bB}{{\mathbf{B}}}

\newcommand{\N}{{\mathbb N}}
\newcommand{\EE}{{\mathbb E}}
\newcommand{\R}{{\mathbb R}}
\newcommand{\Z}{{\mathbb Z}}
\newcommand{\Q}{{\mathbb Q}}
\newcommand{\D}{{\Delta}}
\newcommand{\T}{{\mathbb T}}
\newcommand{\cN}{{\mathcal N}}
\newcommand{\rM}{{\mathrm M}}
\newcommand{\cS}{{\mathcal S}}
\newcommand{\cU}{{\mathcal U}}

\newcommand{\Sph}{{{\mathbb S}^2}}

\renewcommand{\d}{\mathrm{d}}

\newcommand{\ep}{\varepsilon}
\newcommand\e{{\rm e}}

\newcommand{\bv}{{\boldsymbol{v}}}

\renewcommand{\a}{{\boldsymbol{a}}}
\renewcommand{\b}{{\boldsymbol{b}}}
\renewcommand{\c}{{\boldsymbol{c}}}
\newcommand{\bfe}{{\mathbf{e}}}

\newcommand{\x}{{\mathbf{x}}}
\newcommand{\z}{{\mathbf{z}}}
\newcommand{\sfD}{{\mathsf{D}}}
\newcommand{\sfx}{{\mathsf{x}}}
\newcommand{\sfy}{{\mathsf{y}}}
\newcommand{\y}{{\mathbf{y}}}
\newcommand{\sfz}{{\mathsf{z}}}

\renewcommand{\l}{\left}

\renewcommand{\r}{\right}

\newcommand{\NN}{\mathbb{N}}

\newcommand\TT {{\mathbb T}}
\newcommand\RR {{\mathbb R}}
\newcommand\PP {{\mathbb P}}

\newenvironment{indentblock}[1][1]
  {%
    \begin{list}{}%
      {%
        \setlength{\leftmargin}{#1\parindent}%
        \setlength{\rightmargin}{0pt}%
      }%
    \item\relax
  }
  {%
    \end{list}
  }
  
\renewcommand{\parallel}{\mathrel{/\mkern-5mu/}}
\makeatletter
\newcommand{\notparallel}{%
  \mathrel{\mathpalette\not@parallel\relax}%
}
\newcommand{\not@parallel}[2]{%
  \ooalign{\reflectbox{$\m@th#1\smallsetminus$}\cr\hfil$\m@th#1\parallel$\cr}%
}
\makeatother

\begin{document}

\title[Exponential mixing and enhanced dissipation]{Exponential mixing and enhanced dissipation\\
on the unit sphere with Rossby-Haurwitz flows}

\author[A. Del Zotto]{Augusto Del Zotto}

\author[M. Nualart]{Marc Nualart}

\begin{abstract}
    We exhibit a family of smooth incompressible velocity fields on the two-dimensional unit sphere such that the time evolution of any mean-free initial data passively advected by any of them is mixed exponentially fast.
    In the presence of molecular diffusivity, we show that the solution to the associated advection-diffusion equation experiences enhanced dissipation with optimal decay rates.
    Each member of this family is an alternating combination of two Rossby-Haurwitz flows with random amplitudes and constitutes a spherical analogue to the sine shear-alternating example of Pierrehumbert.
\end{abstract}

\maketitle
\tableofcontents

\section{Introduction}\label{sec:intro}

The evolution of a passive scalar $\rho:[0,\infty)\times\rM \rightarrow\R$ on a compact boundary-less surface $\rM$, advected by a smooth divergence-free velocity field $u\in \mathfrak{X}(\rM)$, and undergoing molecular diffusion with diffusivity coefficient $\kappa\geq 0$, is governed by
\begin{equation}\label{eq:IntroAdvDiff}
\begin{cases}
    \partial_t \rho + u\cdot\nabla\rho = \kappa\D_\rM\rho & \x\in \rM, \quad t >0, \\
    \rho(0,\x) = \rho_0(\x), & \x\in\rM.
\end{cases}
\end{equation}
Here, $\Delta_\rM$ denotes the Laplace-Beltrami operator on $\rM$.
In this manuscript, we are interested in describing the long-time dynamics of arbitrary mean-free initial data $\rho_0$ in terms of the incompressible vector field $u$. 

The transport nature of \eqref{eq:IntroAdvDiff} for $\kappa=0$ ensures that all $L^p$ norms of the solution are preserved, for all $p\in[1,\infty]$, where no property of $u$ is used other than its incompressibility.
However, we do observe in everyday phenomena that some stirring velocity fields are more effective than others in mixing initial passive scalars, one of the most instructive examples being the mixing of milk in the morning's cup of coffee.
This stretching and folding of $\rho$ at time $t$ can be measured in terms of homogeneous negative Sobolev norms,
\begin{align}\label{eq:mixingnegSolbolev}
    \Vert \rho \Vert_{\dot{H}^{-s}}^2 = \sum_{k>0}|\lambda_k|^{-2s}| \rho_k(t)|^2, 
\end{align}
where $(\lambda_k)_{k\geq 0}$ are the eigenvalues of $(-\Delta_{\rM})$ and $\rho_k$ denote the Fourier coefficients of $\rho$ in the basis given by the eigenfunctions of $(-\Delta_{\rM})$.
By duality and due to the conservation of $L^2$ norm, any decay of \eqref{eq:mixingnegSolbolev} yields the transfer of $L^2$ mass of $\rho(t)$ to higher and higher frequencies, a testament of mixing of the passive scalar.
In this work the exponential mixing is understood as the exponential decay in time of the negative Sobolev norm \eqref{eq:mixingnegSolbolev}.

In parallel, when $\kappa>0$, combining the classical energy equality
\begin{align}\label{eq:energyeq}
    \frac{1}{2}\frac{\d}{\d t}\int_\rM |\rho(t,\x)|^2 \d \x = -\kappa\int_{\rM} |\nabla \rho(t,\x)|^2 \d \x
\end{align}
together with the Poincaré inequality shows that
\begin{align}\label{eq:introheatdiff}
    \Vert \rho(t, \cdot) - \overline{\rho} \Vert_{L^2} \leq e^{-\lambda_1 \kappa t}\Vert \rho_0 - \overline{\rho} \Vert_{L^2},
\end{align}
namely the solution to \eqref{eq:IntroAdvDiff} converges exponentially fast to the average $\overline{\rho} = \int_\rM \rho_0(\x) \d \x$ of its initial datum\footnote{We remark that $\overline{\rho}$ is a stationary solution of \eqref{eq:IntroAdvDiff}.}, with exponent $-\lambda_1\kappa$ where $\lambda_1>0$ is the first non-zero eigenvalue of $-\Delta_{\mathrm{M}}$.

However, this simple bound does not take into account any possible mixing of the solution $\rho(t)$: the creation of large gradients of $\rho$ due to the stretching and folding produced by $u$ should transfer most of the energy of $\rho$ to large frequencies, where the diffusive operator $\kappa\Delta_\mathrm{M}$ is most effective and significantly improves the rate of convergence of the solution to its average.
To quantify this improvement, we say that $u$ enhances dissipation if there exists $\mathsf{d}:[0,1)\rightarrow[0,\infty)$ with
\begin{align}
    \lim_{\kappa\rightarrow 0}\frac{\kappa}{\mathsf{d}(\kappa)}=0
\end{align}
and a constant $C>0$ such that the unique solution $\rho(t,x)$ to \eqref{eq:IntroAdvDiff} with initial data $\rho(0,\x)=\rho_0(\x)$ satisfies
\begin{align}
    \Vert \rho(t, \cdot) - \overline{\rho} \Vert_{L^2} \leq Ce^{-\mathsf{d}(\kappa) t}\Vert \rho_0 - \overline{\rho} \Vert_{L^2},
\end{align}
for all $t>0$.

In this work we consider $\mathrm{M}=\Sph$, the two-dimensional unit sphere, and we exhibit an incompressible velocity field on $\Sph$ that is a uniform-in-$\kappa$ exponential mixer and enhances dissipation.
To the best of our knowledge, this is the first explicit construction of such a velocity field on the two-dimensional unit sphere. 
Inspired by the example of Pierrehumbert \cites{pierrehumbert1994tracer} on $\T^2$, we define a vector field on $\Sph$ that consists of alternating zonal Rossby-Haurwitz waves of degree 2 with randomized strengths.
These zonal flows are the spherical analogue of the sine shearing profile on $\T^2$ (see Figure~\ref{fig:sphere_vfield} and Figure~\ref{fig:sphere_triptych}), and their chaotic transport properties were already investigated by Pierrehumbert in the $\beta$-plane approximation \cites{pierrehumbert1991chaotic, pierrehumbert1994tracer}, see also \cites{joseph1996chaotic}.

We show that for almost all noise realizations, such a time-dependent vector field exponentially mixes and enhances dissipation of any passive tracer $\rho$ solving \eqref{eq:IntroAdvDiff} for small enough diffusivities $\kappa\geq 0$.
See Figures~\ref{fig:mixing_snaps} and \ref{fig:enhanced_dissip_snaps} for numerical simulations that illustrate the two phenomena analysed in this paper.

\subsection{Lagrangian dynamics and construction of vector field}
The solution to the advection-diffusion equation \eqref{eq:IntroAdvDiff} is studied in terms of the Lagrangian dynamics generated by the convective velocity field $u\in \mathfrak{X}(\Sph)$ and the Laplace-Beltrami operator $\Delta_\Sph$.
Let $\phi_t^\kappa(\x)$ denote the flow map associated to the infinitesimal generator of $u\cdot \nabla - \kappa\Delta_\Sph$ on $\Sph$, which satisfies the stochastic differential equation
\begin{align}\label{eq:introSDEStrat}
    \begin{cases}
        \d \phi_t^{\kappa}(\x) = u(t,\phi_t^{\kappa}(\x))\d t + \sqrt{2\kappa}\Pi_{\phi_t^{\kappa}(\x)}\circ\d \bB_t, & \\
        \phi_0^{\kappa}(\x) = \x\in \Sph,
    \end{cases} 
\end{align}
in the Stratonovich sense, where $\d \bB_t= ( \d B_t^1,  \d B_t^2,  \d B_t^3)$ denotes a standard Brownian motion on $\R^3$ and $\Pi_\x = I - \x \otimes \x$ stands for the usual orthogonal projection from $\R^3$ to $T_\x\Sph$, for all $\x\in\Sph$.
The associated It\^{o} formulation reads
\begin{align}\label{eq:introSDEITO}
    \begin{cases}
        \d \phi_t^{\kappa}(\x) = v_\kappa(t,\phi_t^{\kappa}(\x))\d t + \sqrt{2\kappa}\Pi_{\phi_t^{\kappa}(\x)}\d \bB_t, & \\
        \phi_0^{\kappa}(\x) = \x\in \Sph,
    \end{cases} 
\end{align}
where $v_\kappa(t,\x) = u(t,\x) - 2\kappa \x$ encodes the It\^{o} correction.
Then, for any $\rho_0:\Sph \rightarrow \R$, the solution $\rho(t,\x)$ to the advection-diffusion equation \eqref{eq:IntroAdvDiff} is given in terms of the Feynman-Kac formula by
\begin{align}
    \rho(t,\x) := \EE_\bB \left[ (\phi_t^{\kappa})_\# \rho_0(\x) \right],
\end{align}
where $(\phi_t^{\kappa})_\#g(\x) = g\circ(\phi_t^\kappa)^{-1}(\x)$, for all measurable functions {$g$}. 

Let $\a\in \Sph$ denote a unit direction axis on $\Sph$ and $\omega\in \R$. We define $f_\omega^{\a}(\x)$, for all $\x\in \Sph$, as the time-$1$ flow map of \eqref{eq:introSDEITO} for $\kappa=0$ given by the vector field
\begin{equation}\label{eq:u_a}
    u_\omega^\a(\x) := \omega u^\a(\x), \quad u^{\a}(\x):=(\a\cdot \x)(\a\times \x).  
\end{equation}

\begin{figure}[!h]
    \centering
    \includegraphics[width=0.25\linewidth]{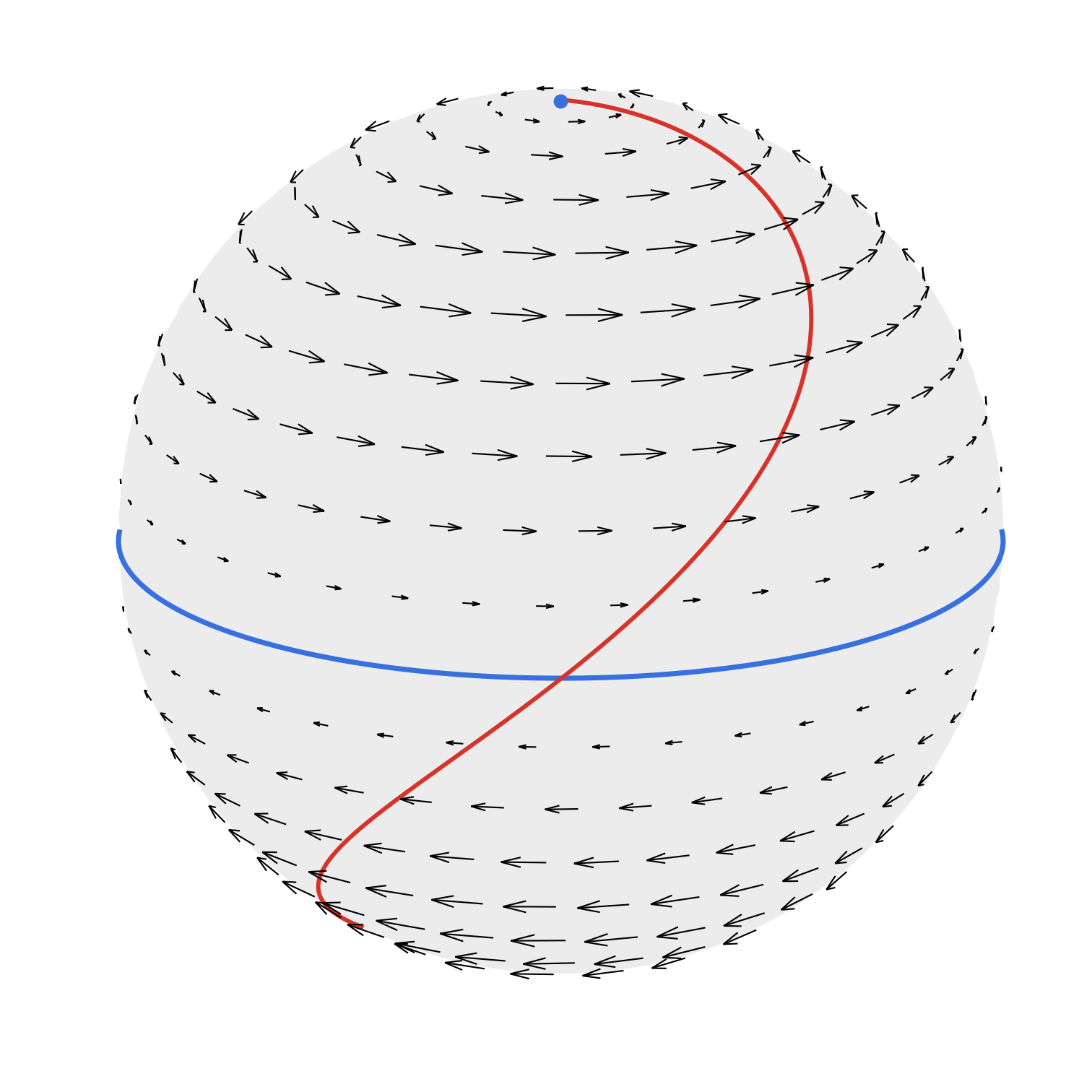}
    \includegraphics[angle=90,width=0.25\linewidth]{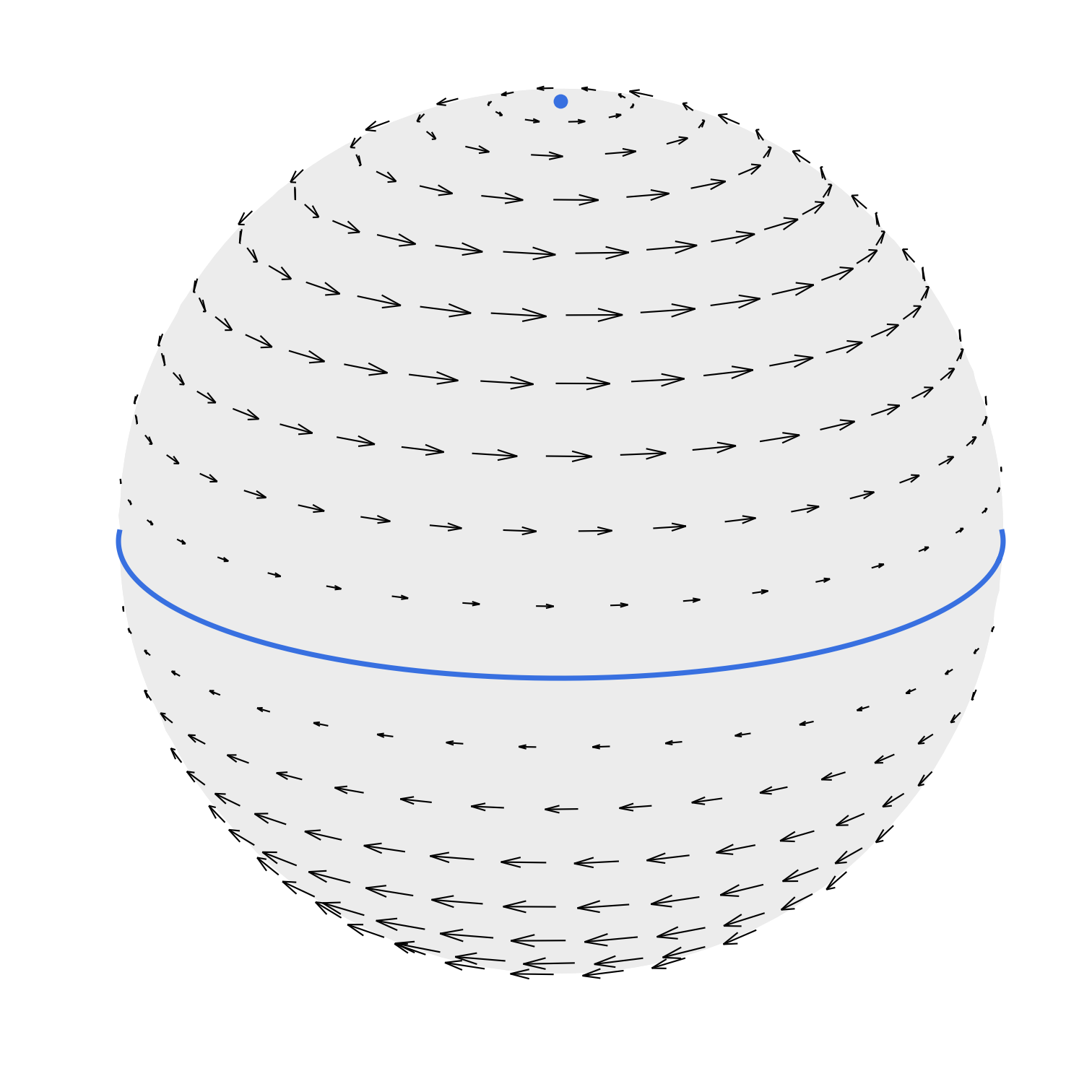}
    \caption{ \small Visual representation of $u^{e_3}$ (on the left) and $u^{e_2}$ (on the right). Black arrows encode direction and strength of vector fields, blue continuous line and dots (equator and poles) indicate zero-level sets, and the red oblique line represents the velocity profile (omitted on the right).}
    \label{fig:sphere_vfield}
\end{figure}

The vector field $u^\a$ is, up to a constant, the velocity field associated with the Rossby-Haurwitz wave $Y_2^0(\a\cdot \x)$, where $Y_2^0$ denotes the zonal spherical harmonic of degree 2.
As such, $u^{e_3}$ is the most direct analogue on the sphere of the sine shear flow $(\sin(y),0)$ on $\T^2$.
The velocity field $u_\omega^\a$ is the infinitesimal generator on $\Sph$ of the rotation of angle $\omega(\a\cdot\x)$ around the $\a$ axis, see Figure \ref{fig:sphere_triptych}.

\begin{figure}[!h]
    \centering
    \includegraphics[width=.7\linewidth]{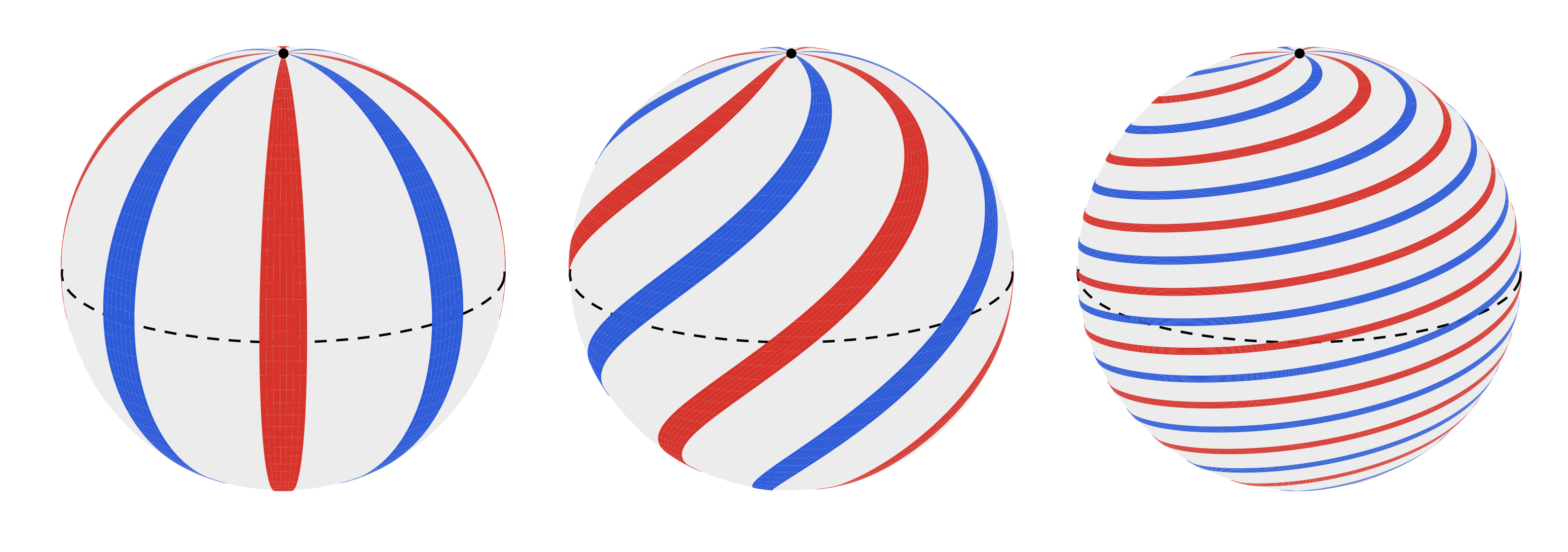}
    \caption{ \small Visual representation of the shearing effect of the flow map $f^{e_3}_\omega$ on red and blue strips for amplitudes $\omega=0,1,6$ respectively.}
    \label{fig:sphere_triptych}
\end{figure}

To obtain exponential mixing, in this manuscript we consider time-dependent velocity fields that consist of alternating between two distinct autonomous fields with randomly chosen amplitudes.
More precisely, for a fixed sufficiently large real number $N>1$, we consider two independent and identically distributed random variables $\omega_1,\omega_2\sim \text{Unif}([-N,N])$. 
The probability space where the random variables take values is defined to be
\begin{align}
    \Omega_0 = [-N,N]^2, \quad \mathcal{F}_0 = \mathcal{B}([-N,N]^2), \quad \PP_0= \text{Unif}([-N,N]^2).
\end{align}
Then, for two fixed non-parallel rotation axes $\a,\b\in \Sph$, we consider the composition flow map
\begin{align}
    f_\omega(\x) := f_{\omega_2}^{\a}\circ f_{\omega_1}^{\b}(\x)
\end{align}
for $\omega=(\omega_1,\omega_2)\in [-N,N]^2$.
At the $j$-th iteration, for $\underline{\omega}_j =(\underline{\omega}_{j,1}, \underline{\omega}_{j,2})\in \Omega_0$ we obtain the composition 
\begin{equation}
f_{\underline{\omega}_j}(\x) = f_{{\underline{\omega}_{j,2}}}^{\a}\circ f_{{\underline{\omega}_{j,1}}}^{\b}\circ f_{\underline{\omega}_{j-1}}(\x).
\end{equation} 
The corresponding velocity field $u$ depends on the noise path $\underline{\omega}=(\underline{\omega}_1, \underline{\omega}_2, ...)\in \underline{\Omega} = \Omega_0^\N$ and is given by
\begin{align}\label{eq:DefAlternating_u}
    u(t,\x,\underline{\omega};\a,\b) := \sum_{i\geq 0}u^\b_{\underline{\omega}_{i,1}}(\x) \mathbf{1}_{[2i,2i+1)}(t) + \sum_{i\geq 0}u^\a_{\underline{\omega}_{i,2}}(\x) \mathbf{1}_{[2i+1,2i+2)}(t).
\end{align}
In particular, $u(t,\x,\underline{\omega};\a,\b)$ is a time-$2$ vector field alternating between the two Rossby-Haurwitz velocity fields $u^\a$ and $u^\b$ with random amplitudes $\underline{\omega}$ and the solution $\phi_t:=\phi_t^0$ to \eqref{eq:introSDEITO} for $\kappa=0$ is such that $\phi_t(\x) = f_{\underline{\omega}}^n(\x)$ for all $t=2n$ with $n\in \N$. 

\begin{figure}[!h]
    \centering
    \includegraphics[width=1\linewidth]{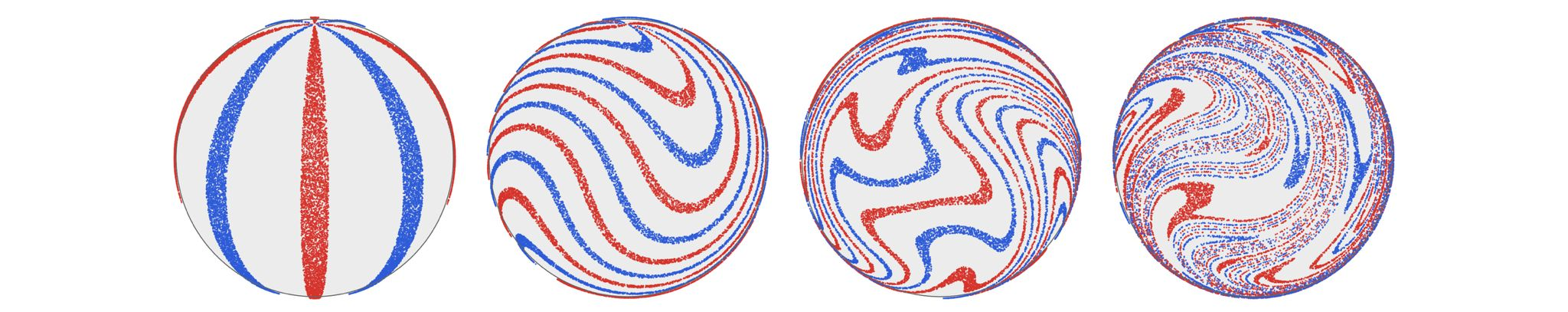}
    \caption{ \small Inviscid evolution for the alternating Rossby-Haurwitz flow on $\mathbb S^2$ with $\a=e_2$, $\b=e_3$, $\kappa=0$, $N=4$, and $40000$ particles. Snapshots are at times $t=0,2,4,8$ from left to right. The sampled phase pairs by period are $(\omega_1,\omega_2)=(+3.292,-1.490)$, $(+1.204,+2.502)$, $(-3.164,+0.079)$, and $(-0.361,-3.023)$.}
    \label{fig:mixing_snaps}
\end{figure}

When $\kappa>0$, to define the process $\phi_t^\kappa$ we introduce a probability space $(\Omega_\bB, \mathcal{F}_\bB, \PP_\bB)$ and we let $\bB_t=(B_t^1,B_t^2,B_t^3)$ be an $\R^3$-valued Brownian motion on this space.
The probability space for the process $\phi_t^\kappa$ solving \eqref{eq:introSDEITO} with vector field $u(t,\cdot,\underline{\omega};\a,\b)$ is taken to be $(\Omega, \mathcal{F}, \PP)$, where $\Omega$ denotes the product space $\Omega=\Omega_\bB\times \underline{\Omega}$ with $\sigma$-algebra $\mathcal{F} = \mathcal{F}_\bB \otimes\mathcal{F}_0^\N$ and probability measure $\PP := \PP_\bB\otimes \PP_0^\N$.
We shall treat random variables on each coordinate space $\Omega_0^\N$, $\Omega_\bB$ as random variables on the product space $\Omega$, and we may treat {$\underline{\omega}$} and $\bB_t$ as independent processes on $\Omega$.
Just as for the deterministic setting, for $\kappa>0$ and for $\a\in \Sph$ and $\omega\in \R$, we define $f^{\a,\kappa}_\omega(\x)$ as the time-1 flow map of \eqref{eq:introSDEITO} given by the vector field $u^\a_\omega$.
We then consider the stochastic composition
\begin{align}
    f_\omega^{\kappa}(\x) := f_{\omega_2}^{\a,\kappa}\circ f_{\omega_1}^{\b,\kappa}(\x)
\end{align}
for $\omega=(\omega_1,\omega_2)\in [-N,N]^2$ and $\a,\b\in \Sph$. 
Similarly, at the $j$-th iteration, for $\underline{\omega}_j=(\underline{\omega}_{j,1}, \underline{\omega}_{j,2})\in \Omega_0$ we have 
\begin{equation}
f_{\underline{\omega}_j}^{\kappa}(\x) = f_{{\underline{\omega}_{j,2}}}^{\a,\kappa}\circ f_{{\underline{\omega}_{j,1}}}^{\b,\kappa}\circ f_{\underline{\omega}_{j-1}}^\kappa(\x).
\end{equation}

\begin{figure}[!h]
    \centering
    \includegraphics[width=1\linewidth]{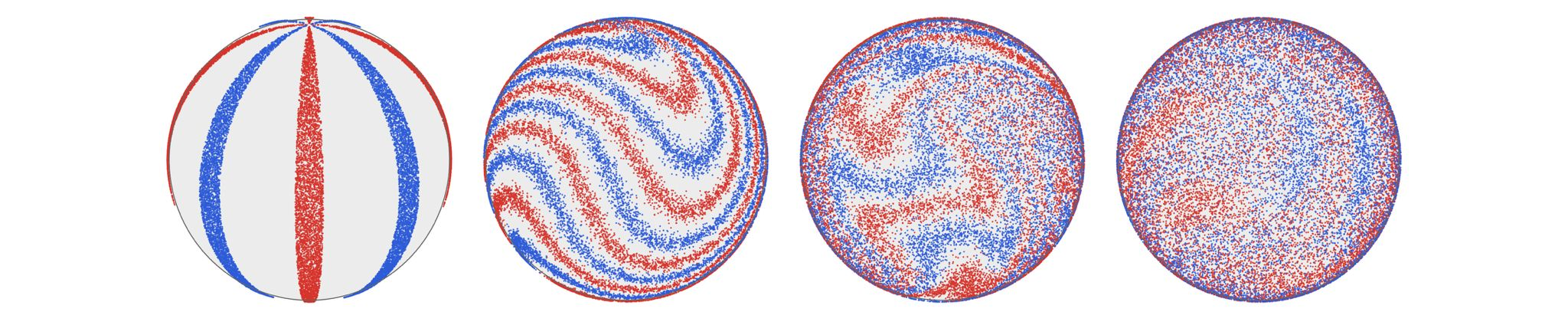}
    \caption{ \small Viscous evolution for the alternating Rossby-Haurwitz flow on $\mathbb S^2$ with $\a=e_2$, $\b=e_3$, $\kappa=0.0005$, $N=4$, and $40000$ particles. Snapshots are at times $t=0,2,4,8$ from left to right. The same sampled phase pairs as in the inviscid figure are used: $(\omega_1,\omega_2)=(+3.292,-1.490)$, $(+1.204,+2.502)$, $(-3.164,+0.079)$, and $(-0.361,-3.023)$.}
    \label{fig:enhanced_dissip_snaps}
\end{figure}

\subsection{Main results}

The primary result of this work is an almost-sure exponential decay of correlations for the stochastic Lagrangian dynamics generated by the amplitude-randomized alternating Rossby-Haurwitz flows. 
More precisely, for every pair of non-parallel unit axes $\a,\b\in\Sph$, we show that the flow associated with $u(t,\x,\underline{\omega};\a,\b)$ exhibits exponential decay of correlations between arbitrary regular mean-free observables. 
The estimate is uniform for all sufficiently small diffusivities $\kappa>0$, and the random prefactor has moments bounded independently of $\kappa$. 
The precise statement is as follows.

\begin{theorem}\label{thm:mainexpdecaycorr}
Let $\a,\b\in \Sph$ be two non-parallel unit axes. Let $\phi_t^\kappa$ be the flow defined by the It\^{o} SDE \eqref{eq:introSDEITO}, with $\kappa\geq 0$ and $u(t,\cdot,\underline{\omega};\a,\b)$ the Pierrehumbert amplitude-randomized alternating vector field given by \eqref{eq:DefAlternating_u}. 
For any $q, s > 0$, there exists $\kappa_0>0$, a random constant $\mathrm{C}_{\underline{\omega},\kappa}\geq 1$ independent of $\bB$ and a deterministic $\kappa$-independent $\lambda_s > 0$ such that for all mean-free functions $\varphi, \psi \in H^s(\Sph)$, we have almost surely 
\begin{equation}\label{eq:mainthm_decaycorr}
\left| \int_{\Sph} \varphi(\x) \psi( \phi_t^\kappa(\x)) \dd \x \right|
\leq \mathrm{C}_{\underline{\omega},\kappa} e^{-\lambda_s t} \|\varphi\|_{H^s} \|\psi\|_{H^s},
\end{equation}
for all $t >0$ and all $\kappa\in [0,\kappa_0)$. 
Moreover,  there exists a $\kappa$-independent constant $\overline{\mathrm{C}}_q>0$ such that $\mathbb{E}_\omega|\mathrm{C}_{\underline{\omega},\kappa}|^q \leq \overline{\mathrm{C}}_q$. 
\end{theorem}

The first consequence of Theorem~\ref{thm:mainexpdecaycorr} is the corresponding quantitative mixing estimate for solutions of the passive scalar equation, which follows by the standard duality between decay of correlations and decay of homogeneous negative Sobolev norms.

\begin{corollary}\label{cor:unifexpmixing}
Let $\a,\b\in \Sph$ be two non-parallel unit axes. Let $\rho_0\in H^s(\Sph)$ be mean-free and let $\rho(t,x)$ be the solution of \eqref{eq:IntroAdvDiff}.
Then, there exists $\kappa_0>0$ such that for all $\kappa\in[0,\kappa_0)$ and all $s>0$ there holds
    \begin{align}\label{eq:cor_exp_mixing}
        \Vert \rho(t,\cdot)  \Vert_{\dot{H}^{-s}} \leq \mathrm{C}_{\underline{\omega},\kappa}e^{-\lambda_s t}\Vert \rho_0 \Vert_{\dot{H}^s}
    \end{align}
    almost surely for all $t\geq 0$.
\end{corollary}

The second consequence concerns the viscous problem. 
When $\kappa>0$, the uniform-in-$\kappa$ mixing estimate of Corollary~\ref{cor:unifexpmixing} yields an enhanced dissipation estimate for the advection-diffusion equation.
The mechanism is recalled in Remark~\ref{rem:main_results_comments} below.

\begin{corollary}\label{cor:optimalenhanceddiss}
   Let $\a,\b\in \Sph$ be two non-parallel unit axes. Let $\rho_0\in H^1(\Sph)$ be mean-free and $\rho(t,\x)$ solves \eqref{eq:IntroAdvDiff}. 
   Then, there exists $\lambda>0$ and $\kappa_0>0$ such that for all $\kappa\in(0,\kappa_0)$ there exists a random constant $\mathrm{C}_{\underline{\omega},\kappa}'\geq 1$ independent of $\bB$ and an almost surely finite random constant $\mu(\kappa)>0$ such that $\EE_{\omega}[|\mathrm{C}_{\underline{\omega},\kappa}'|^2]\leq \mathrm{C}'_2$ for some $\mathrm{C}_2'\geq 1$ uniformly in $\kappa$ and  there holds
    \begin{align}
        \Vert \rho(t,\cdot)  \Vert_{L^2} \leq \min \left\lbrace \frac{\mathrm{C}'_{\underline{\omega},\kappa}}{\sqrt{\kappa}}e^{-\lambda t}, e^{-\mu(\kappa)t} \right\rbrace \Vert \rho_0 \Vert_{L^2}
    \end{align}
    almost surely for all $t\geq 0$. Moreover, 
    \begin{align}
        \lim_{\kappa\rightarrow 0} \mu(\kappa) \log\left(   \frac{1}{\kappa} \right) \in \left(\frac{2\lambda}{1+M}, 2\lambda \right)
    \end{align}
    with probability at least $1-e^{2}\kappa_0^M {\mathrm{C}}'_2$, for all $M \geq 0$.
\end{corollary}

Theorem~\ref{thm:mainexpdecaycorr} and Corollaries~\ref{cor:unifexpmixing} and \ref{cor:optimalenhanceddiss} give, to the best of our knowledge, the first example of a smooth universal exponentially mixing and enhanced-dissipating velocity field on a two-dimensional compact surface without boundary different from $\T^2$. 
The qualification universal refers to the fact that the same velocity field mixes every mean-free initial datum in the stated class.

\begin{remark}\label{rem:main_results_comments}
We record a few comments clarifying the role of the random amplitudes and the scope of the main statements for the passive scalar equation.
\begin{enumerate}[label=(\roman*)]
    \item The estimate in Theorem~\ref{thm:mainexpdecaycorr} is pointwise with respect to the random amplitudes. The random prefactor $\mathrm C_{\underline\omega,\kappa}$ may depend on the realization of the amplitude sequence and on $\kappa$, but it is independent of the Brownian motion $\bB$, and its moments with respect to $\underline{\omega}$ are bounded uniformly as $\kappa\to 0$. 
    The exponential rate $\lambda_s$ is deterministic and independent of $\kappa$. The proof follows the random dynamical systems strategy of \cites{BCZG2023} in the inviscid case and its uniform-in-diffusivity extension in \cites{CIS2024}. In Section~\ref{sec:asbtractRDS}, we recall the abstract framework and explain how it applies to the present spherical model.

    \item As in \cites{BCZG2023}, the piecewise-in-time vector field $u(t,\cdot,\underline{\omega};\a,\b)$ can be made smooth in time by a suitable time reparametrization, for instance by using the construction of \cite{yao2017mixing}*{page 1914}. This smoothing does not make the flow time-periodic: the use of independent random amplitudes at each step is essential to the random dynamical systems framework used below. 

    \item For the inviscid dynamics underlying Theorem~\ref{thm:mainexpdecaycorr}, multiplying $u^\a$ and $u^\b$ by the random amplitudes is equivalent to alternating between the two deterministic Rossby-Haurwitz vector fields for random signed times (with the convention that negative times correspond to flowing with $-u^\a$ or $-u^\b$ forwards in time). Thus the randomness enters through the signed duration of each elementary shearing motion.

    \item Theorem~\ref{thm:mainexpdecaycorr} holds for every pair of non-parallel axes $\a,\b\in\Sph$. In particular, the two families of streamlines generated by the elementary flows $u^\a$ and $u^\b$ need not be orthogonal. Whether deterministically alternating between two non-parallel \textit{and non-orthogonal} smooth shearing flows generically produces exponential mixing remains an outstanding open problem.

    \item The mixing estimate in Corollary~\ref{cor:unifexpmixing} has the optimal exponential time decay among velocity fields with uniformly bounded Lipschitz norm, see \cites{crippa2008estimates,iyer2014lower,seis2013maximal}. Such optimal decay rates cannot be achieved by smooth two-dimensional autonomous divergence-free velocity fields, since they mix only at algebraic rates, see \cites{brue2024enhanced}.
    
    \item Corollary~\ref{cor:optimalenhanceddiss} follows from the uniform-in-$\kappa$ stochastic decay of correlations, rather than from a deterministic inviscid mixing estimate alone. More precisely, Proposition~\ref{prop:unifdecaycorr} gives the pointwise estimate \eqref{eq:stochdecaycorr} with a random prefactor whose moments are bounded independently of $\kappa$. This is the key ingredient used, following \cites{CIS2024}, to obtain the enhanced dissipation estimate for the advection-diffusion equation. The underlying intuition is that mixing transfers information/energy to high frequencies where the diffusive term $\kappa\Delta_{\Sph}$ is most effective.

    \item The enhanced dissipation estimate in Corollary~\ref{cor:optimalenhanceddiss} has the optimal logarithmic scale for uniformly-in-time Lipschitz velocity fields; see \cites{seis2023bounds}. More precisely, the global-in-time enhanced dissipation rate is of order $|\log\kappa|^{-1}$ as $\kappa\to 0$, while the first term in the estimate of Corollary~\ref{cor:optimalenhanceddiss} shows that, after a logarithmic transient time, the decay rate becomes independent of $\kappa$, see \cites{navarro2025exponential} for more details.
\end{enumerate}
\end{remark}

\subsection{Context and contributions}\label{subsec:context}

The study of mixing properties of divergence-free vector fields on solutions to the advection-diffusion equation \eqref{eq:IntroAdvDiff} has seen a huge development over the last two decades. Quantifying mixing in terms of negative Sobolev norms (also called functional mixing) as in \eqref{eq:mixingnegSolbolev} was first introduced in \cites{mathew2005multiscale}, and later revisited in \cites{lin2011optimal}. While there are other, more geometrical, notions of mixing, see \cites{thiffeault2012using}, that are equivalent to the decay in time of homogeneous negative Sobolev norms in special cases, in this manuscript we consider functional mixing. 

It is by now a classical result that for sufficiently regular vector fields mixing cannot be faster than exponential, see \cites{crippa2008estimates}, and also \cites{iyer2014lower, seis2013maximal}. This exponential decay rate was first proved to be sharp in \cites{yao2017mixing, alberti2019exponential}, where exponentially mixing velocity fields $u$ in $\T^2$ specifically tailored to the initial data $\rho_0$ were exhibited. The question of universal mixers (a velocity field that mixes all mean-free initial data) was subsequently first resolved in \cites{elgindi2019universal}, see also \cites{elgindi2025optimal} where deterministic alternating Lipschitz vector fields on $\T^2$ are shown to mix exponentially fast. Introducing stochasticity in the system may enhance the chaotic dynamics generated by $u$. Indeed, universal exponential mixers can also be obtained from almost every solution to the forced stochastic Navier-Stokes equations, see \cites{bedrossian2022almost, cooperman2026exponential}.
It has also been recently shown in \cites{navarro2025exponential} that a cellular flow in $\T^2$ with a randomly moving centre is a further instance of an exponentially mixing vector field.

The examples above differ in several aspects: while some are tailored to a prescribed initial datum, some are not smooth in space, some are not uniformly bounded in time, and others rely on stochastic forcing. To address these caveats, in \cites{BCZG2023} the authors established the existence of space-time smooth and uniformly bounded universal exponential mixers on $\T^2$ by developing a robust random dynamical systems framework.
It was shown there that the Pierrehumbert model \cites{pierrehumbert1994tracer}, which consists of alternating periodic sine shear flows with randomized phases, is an almost-surely exponential mixer.
This approach was later applied in \cites{coti2026three} to show that the ABC flows with randomized phases and amplitudes in $\T^3$ are exponentially mixing. It was also used in \cites{C2023} to show that the Pierrehumbert example in $\T^2$ with randomized time-intervals is exponentially mixing as well. 

The present work reproduces this Pierrehumbert-type mechanism on the sphere. 
The elementary building blocks are zonal Rossby-Haurwitz vector fields. They are natural spherical analogues of shear flows that arise in large-scale atmospheric dynamics, including stratospheric flows (see, for instance, \cites{constantin2022stratospheric, constantin2025onset, dowling1995dynamics}), and, in this sense, the amplitude-randomized alternating vector field $u(t,\x,\underline{\omega};\a,\b)$ defined in \eqref{eq:DefAlternating_u} should be viewed as the spherical counterpart of the random-time alternating sine-shear models studied in \cites{pierrehumbert1994tracer,BCZG2023,C2023}. 

In the presence of molecular diffusivity $\kappa>0$, enhanced dissipation for time-independent velocity fields $u$ takes place if $u\cdot\nabla$ has no eigenfunction in $H^1$, see \cites{constantin2008diffusion, zlatovs2010diffusion, kiselev2008relaxation}. We refer the interested reader to \cites{coti2024mixing, coti2020relation, coti2020stable} and references therein for the relation between mixing and enhanced dissipation. While a heuristic argument suggests that exponentially mixing flows $u$ should produce a rate $\mathsf{d}(\kappa)\geq C|\log(\kappa)|^{-1}$, it has only been rigorously shown in \cites{coti2020relation, feng2019dissipation} that exponentially mixing flows $u$ produce at least a rate $\mathsf{d}(\kappa)\geq C|\log(\kappa)|^{-2}$. The optimal rate $\mathsf{d}(\kappa)\geq C|\log(\kappa)|^{-1}$ is recovered if one further assumes that $u$ is Lipschitz in space uniformly in time, see \cites{seis2023bounds}. The flows constructed in \cites{bedrossian2021almost, coti2026three, elgindi2025optimal ,CIS2024} are uniform-in-$\kappa$ exponentially mixing and obtain the optimal enhanced dissipation rate $O(|\log \kappa|^{-1})$ in view of \cites{seis2023bounds}.  Another instance of stochastic vector fields that mix exponentially fast uniformly in diffusivity and enhance dissipation is provided by the Kraichnan model, see \cites{gess2025stabilization, luo2024elementary, coti2024gaussian}.

We now comment on the specific difficulties that arise in the spherical setting and are not present in the previous works \cites{pierrehumbert1994tracer,BCZG2023,C2023}. 
From a technical point of view, proving exponential mixing and enhanced dissipation on $\Sph$ requires new arguments tailored both to the geometry of the sphere and to the specific structure of the Rossby-Haurwitz vector fields $u^\a$ and $u^\b$. 
Indeed, the passage from $\T^2$ to $\Sph$ is not merely geometric: the presence and nature of stagnation points, global symmetries, and the non-parallelizability of the tangent bundle changes the structure of the one-point, projective, and two-point dynamics. 
Accordingly, the proof cannot be obtained by directly adapting the torus arguments. 
The main new technical points are the following:
\begin{enumerate}[label=(\roman*)]
    \item The existence of a set of fixed points $F$ on $\Sph$ for almost every realization of $u(t,\cdot,\underline{\omega};\a,\b)$. The one-point process is thus not uniformly geometrically ergodic as there is no unique stationary measure on $\Sph$. Hence, a Lyapunov-Foster drift condition must be satisfied for a Lyapunov function specifically designed to capture the dynamics of $u(t,\cdot, \underline{\omega};\a,\b)$ near the fixed points. These dynamics are, in some cases, not uniformly hyperbolic, and they must be stable enough to be persistent in the presence of diffusivity $\kappa>0$. 
    \item The existence of other invariant sets for the two-point process other than the diagonal due to symmetries of the sphere. Just as for the one-point process, these sets have to be determined to construct a Lyapunov-Foster drift function for the two-point process.
    \item The open small set condition for non-orthogonal axes. The open small set condition for the one-point, two-point and projective processes is ensured once these processes are shown to be submersions for specific choices of noise parameters, see Lemma \ref{lemma:blackboxsmallset} below. In previous constructions \cites{BCZG2023, coti2026three, C2023} the authors exhibited specific sequences of noise parameters for which the relevant processes were submersions. In the spherical setting we consider, and especially when the axes $\a$ and $\b$ are non-orthogonal, it is remarkably more difficult to find specific noise parameters with the submersion property due to the complexity of the processes arising from $u(t,\cdot,\underline{\omega};\a,\b)$.
    Instead, we propose a distinct approach based on the Lie Algebra Rank Condition. In short, we derive a framework that ensures the submersion property of the relevant processes if the associated vector fields and their nested commutators have full rank in the associated tangent spaces. Such framework is fundamental when proving the positivity of the top Lyapunov exponent. See Section \ref{sec:LARC} for more details.
    \item The presence of diffusivity breaks the Lagrangian deterministic dynamics. To show the uniform-in-$\kappa$ decay of correlations in Theorem \ref{thm:mainexpdecaycorr} by means of a quantitative Harris theorem, we need to devise a drift function for the stochastic two-point process that satisfies a Lyapunov-Foster condition with $\kappa$-independent constants. Such function must take into account the dynamics of the stochastic one-point process near the fixed points, where there is a competition between the Brownian noise and the inviscid dynamics. We put forward a careful analysis near the fixed points and distinguish several regimes according to whether the processes are driven by either the deterministic dynamics or by the Brownian motion. We refer the reader to Section~\ref{sec:driftstochTPP} for more details.
\end{enumerate}

\subsection*{Organization of the paper}
We begin in Section \ref{sec:asbtractRDS} by introducing the one-point, projective, and two-point Markov transition kernels that are the primary object of study in this manuscript. We also provide there the main theoretical background on Random Dynamical Systems for establishing geometric ergodicity of the three Markov chains and the positivity of the top Lyapunov exponent by means of an abstract Harris Theorem and Furstenberg's criterion. These abstract results require the respective Markov chains to satisfy specific conditions closely related to the chaotic behaviour of the associated dynamical systems. These conditions may be verified by means of intermediate results that are presented in Section \ref{sec:asbtractRDS} as well.

These intermediate results need a precise understanding of the differential geometric structures of the two-dimensional sphere, which is carried out in Section \ref{sec:DiffGeom}. In Section \ref{sec:LARC} we set up a novel scheme to verify the existence of an open small set (one of the conditions for the abstract Harris Theorem) and the positivity of the top Lyapunov exponent. Next, Sections \ref{sec:OPP}, \ref{sec:PJP} and \ref{sec:TPP} show that the one-point, projective, and two-point processes satisfy the conditions of the abstract Harris Theorem and are hence geometrically ergodic. For the geometric ergodicity of the two-point process, the positivity of the top Lyapunov exponent is also necessary, and it is achieved in Section \ref{sec:toplyapexp}. 

The two-point process with non-zero Brownian noise is treated in Section \ref{sec:stochTPP}, where we set the main stochastic stability estimates to prove the uniform-in-$\kappa$ decay of correlations. Among other things, the scheme needs the stochastic two-point process to admit a Lyapunov-Foster drift function, which is established in Sections \ref{sec:driftstochOPP} and \ref{sec:driftstochTPP}. We finish with Section \ref{sec:unifkappa}, where we verify that the stochastic process $f_\omega^\kappa$ satisfies the required conditions to be uniformly geometrically ergodic.

\section{Random Dynamical Systems and a Harris Theorem}\label{sec:asbtractRDS}
In this section we introduce the random dynamical system framework that we use to prove Theorem \ref{thm:mainexpdecaycorr}. First, we observe that $u^\a$ and $u^\b$ share, at least, two fixed points (see Proposition \ref{prop:fixed_points} below), which constitute an invariant set for the flow map $f_\omega$. Let $F$ denote the set of such fixed points of the random-amplitude Pierrehumbert velocity field $u(t,\cdot,\underline{\omega};\a,\b)$, and consider the phase space $\mathsf{X}:=\Sph\setminus F$. 
We define on $\mathsf{X}$ the three main Markov transition kernels that describe the dynamics for the inviscid ($\kappa=0$) flow map $f_\omega$ arising from  $u(t,\cdot,\underline{\omega};\a,\b)$ on $\Sph$.

\begin{definition}\label{def:OPPPJPTPP}
Let $\a,\b\in\Sph$ be non-parallel axes and let $f_\omega(\x):=f^{\a}_{\omega_2}\circ f^{\b}_{\omega_1}(\x)$ be the associated time-1 flow map, for $\omega=(\omega_1,\omega_2)\sim \text{Unif}([-N,N]^2)$, for some $N>1$ fixed.
\begin{enumerate}
    \item The one-point chain $P$ in $\mathsf{X}$ associated to the one-point process $f_\omega$ is given by 
\begin{equation}\label{eq:defOPP}
    P(\x,U):=\PP_0(f_\omega(\x)\in U) 
\end{equation}
for any Borel set $U$ in $\mathsf{X}$.
\item The projective chain associated to the projective process $\hat f_\omega(\x,v)$ on the unit tangent bundle 
\begin{equation}
    \mathsf{Y}:=\lbrace  (\x,v): \x\in \mathsf{X}, v\in T_\x\Sph, |v|=1 \rbrace  
\end{equation}
is defined via $\hat f_\omega(\x,v):=(f_\omega(\x),\frac{D_\x f_\omega(\x)v}{| D_\x f_\omega(\x)v |})$ as 
\begin{equation}\label{eq:defPJP}
    \hat P((\x,v),\hat{U})=\PP_0(\hat f_\omega(\x,v)\in \hat{U})
\end{equation}
    for any Borel set $\hat{U}$ in $\mathsf{Y}$.
\item The two-point chain $P^{(2)}$ associated to the two-point process $f^{(2)}_\omega(\x,\y)=(f_\omega(\x),f_\omega(\y))$ in the phase space $\mathsf{X}^{(2)}:=\mathsf{X}\times \mathsf{X}\setminus \cS$ is given by
\begin{equation}\label{eq:defTPP}
    P^{(2)}((\x_1,\x_2),U^{(2)})=\PP_0(f_\omega^{(2)}(\x_1,\x_2)\in U^{(2)})
\end{equation}
    for any Borel set $U^{(2)}$ in $\mathsf{X}^{(2)}$. Here, $\cS$ denotes the set of symmetries on the product space $\mathsf{X}\times \mathsf{X}$ for which $f^{(2)}_\omega$ is invariant, see Proposition \ref{prop:symmetries}, and includes the diagonal $\D = \lbrace (\x,\y)\in \mathsf{X} \times \mathsf{X} : \x = \y \rbrace$.
\end{enumerate}
\end{definition}

The analysis leading to Theorem \ref{thm:mainexpdecaycorr} relies on showing ergodic properties for the Markov chains introduced above. This section is devoted to presenting a primer of the abstract general setting and results on Markov chains and random dynamical systems developed in \cites{BCZG2023,CIS2024} that are applicable to the one-point, projective and two-point chains.

\subsection{Markov transition kernels and Harris Theorem}\label{sec:kernel andHarris}
Let $\mathrm{X}$ be a complete metric space (not necessarily compact), let $\mathcal{B}(\mathrm{X})$ denote the set of Borel measurable sets on $\mathrm{X}$ and let $\mathcal{P}(\mathrm{X})$ be the set of Borel probability measures on $\mathrm{X}$. We say that $\mathrm{P}$ is a Markov transition kernel on $\mathrm{X}$ if for each $\mathrm{x}\in \mathrm{X}$ we have that ${\mathrm{P}}(\mathrm{x},\cdot)$ is a Borel probability on $\mathrm{X}$. Iterates of the Markov transition kernel ${\mathrm{P}}$ are defined through Chapman-Kolmogorov by
\begin{align}
    {\mathrm{P}}^{n+1}(\mathrm{x},A)=\int_\mathrm{X} {\mathrm{P}}^n(\mathrm{z},A) {\mathrm{P}}(\mathrm{x}, \d \mathrm{z}),
\end{align}
for all $n\in \NN$ and all Borel measurable sets $A\in \mathcal{B}(\mathrm{X})$. The transition kernel ${\mathrm{P}}$ acts on bounded Borel measurable functions $\varphi:\mathrm{X}\rightarrow\R$ via
\begin{align}
    \mathrm{P}\varphi(\mathrm{x}) = \int_\mathrm{X} \varphi(\mathrm{z}) {\mathrm{P}}(\mathrm{x}, \d \mathrm{z}),
\end{align}
and if $\mathrm{P}g$ is a continuous and bounded function on $\mathrm{X}$ for all continuous and bounded functions $g$, we say that ${\mathrm{P}}$ has the \textit{Feller} property. Finally, ${\mathrm{P}}$ also acts on probability measures $\nu\in \mathcal{M}(\mathrm{X})$ on $\mathrm{X}$ through
\begin{align}
    {\mathrm{P}}\nu(A) = \int_\mathrm{X} {\mathrm{P}}(\mathrm{x},A) \d \nu(\mathrm{x}),
\end{align}
for any Borel measurable $A\in \mathcal{B}(\mathrm{X})$.
Hence, we say that a Borel probability measure $\pi\in \mathcal{M}(\mathrm{X})$ is \textit{stationary} for ${\mathrm{P}}$ if it is a fixed point of ${\mathrm{P}}$, that is,
\begin{align}
    \pi(A) = \int_\mathrm{X} {\mathrm{P}}(\mathrm{x},A) \d \pi(\mathrm{x}),
\end{align}
for all $A\in \mathcal{B}(\mathrm{X})$. A stationary measure $\pi$ is called \textit{ergodic} if all $({\mathrm{P}},\pi)$ invariant sets have $\pi$-measure either zero or one.
In this setting, it is a standard fact that if $\pi$ is the unique stationary measure for a transition kernel ${\mathrm{P}}$ then it is ergodic, see \cites{kifer1986ergodic}.

Given any function $\mathrm{V}:\mathrm{X}\rightarrow[1,\infty)$, we say that $\varphi\in L^\infty_\mathrm{V}(\mathrm{X})$ if $\varphi$ is measurable and 
\begin{align}
    \Vert \varphi \Vert_{L^\infty_\mathrm{V}(\mathrm{X})} := \sup_{\mathrm{x}\in \mathrm{X}}\frac{|\varphi(\mathrm{x})|}{\mathrm{V}(\mathrm{x})} < \infty.
\end{align}
In this paper, we are interested not only in the existence of an ergodic stationary measure for a Markov transition kernel ${\mathrm{P}}$ but also in deriving a geometric rate of convergence in $L^\infty_\mathrm{V}(\mathrm{X})$ towards it. More precisely, we say that ${\mathrm{P}}$ is $\mathrm{V}$-\textit{uniformly geometrically ergodic} if ${\mathrm{P}}$ admits a unique stationary measure $\pi$ and there exists $C>0$ and $\gamma\in(0,1)$ such that
\begin{align}
    \left| {\mathrm{P}}^n\varphi (\mathrm{x}) - \int_\mathrm{X} \varphi(\mathrm{x}) \d\pi(\mathrm{x}) \right| \leq C \mathrm{V}(\mathrm{x}) \Vert \varphi \Vert_{L^\infty_\mathrm{V}(\mathrm{X})}\gamma^n,
\end{align}
for all $\mathrm{x}\in \mathrm{X}$, all $\varphi\in L^\infty_\mathrm{V}$ and all $n\in \N$. The main tool to show that ${\mathrm{P}}$ is $\mathrm{V}$-{uniformly geometrically ergodic} is Harris Theorem.

\begin{theorem}[Harris Theorem]\label{thm:abstractHarris}
Let ${\mathrm{P}}$ be a Markov transition kernel with the Feller property. Assume that
\begin{enumerate}[label=(\arabic*)]
    \item ${\mathrm{P}}$ admits an open small set: there exists a set $A\subset \mathrm{X}$, a positive measure $\nu$ on $\mathrm{X}$ and a natural number $n=n(A)\in\NN$ such that, for all $\x\in A$, we have 
    \begin{equation}
        {\mathrm{P}}^n(\x,B)\geq \nu(B),
    \end{equation}
    for all measurable subsets $B\subset \mathrm{X}$. Such set $A$ is said to be small.
    \item ${\mathrm{P}}$ is topologically irreducible: for every point $\x\in \mathrm{X}$ and every open non-empty set $A\subset \mathrm{X}$, there exists a natural number $n=n(\x,A)\in\NN$ such that 
    \begin{equation}
        {\mathrm{P}}^n(\x,A)>0.
    \end{equation}
    \item ${\mathrm{P}}$ is strongly aperiodic: there exists some $\x\in \mathrm{X}$ such that for all open neighbourhoods $A$ of $\x$, we have 
    \begin{equation}
        {\mathrm{P}}(\x,A)>0.
    \end{equation}
    \item There exists a function $\mathrm{V}:\mathrm{X}\rightarrow [1,\infty)$ that meets the Lyapunov-Foster drift condition: there exist some constants $\gamma\in(0,1)$, $\beta>0$, and a compact set $K\subset \mathrm{X}$ such that
    \begin{equation}
        \mathrm{P}\mathrm{V}\leq \gamma \mathrm{V}+\beta \boldsymbol{1}_K.
    \end{equation}
\end{enumerate}
Then, ${\mathrm{P}}$ is $\mathrm{V}$-uniformly geometrically ergodic.
\end{theorem}
The proof of this version of Harris Theorem is found in \cites{BCZG2023}. See also \cites{hairer2011yet, meyn2012markov} for more general versions of the theorem. 

\subsection{Random dynamical systems and sufficient conditions for Harris Theorem}
Most of the effort in this paper is dedicated to showing that the one-point, two-point and projective processes are Markov transition kernels that satisfy all the hypotheses of Harris theorem. We do so following \cites{BCZG2023}, where the authors state a set of conditions for a random dynamical system with absolutely continuous noise that ensure its associated Markov transition kernel admit an open small set, is topologically irreducible, and is strongly aperiodic.

Let $(\Omega_0, \mathcal{F}_0, \mathbb{P}_0)$ be a fixed probability space. A continuous Random Dynamical System (RDS) with independent increments on $\mathrm{X}$ is an assignment to each $\omega\in \Omega_0$ of a continuous mapping $\mathrm{f}_\omega:\mathrm{X}\rightarrow\mathrm{X}$  such that the set $\lbrace   \omega\in \Omega_0 : \mathrm{f}_\omega(\mathrm{x}) \in A \rbrace  $ is $\mathcal{F}_0$-measurable, for all $A\in \mathcal{B}(\mathrm{X})$ and all $\mathrm{x}\in \mathrm{X}$. Setting $(\Omega, \mathcal{F}, \mathbb{P}) = (\Omega_0, \mathcal{F}_0, \mathbb{P}_0)^\N$ and writing $\underline{\omega}=(\omega_1,\omega_2,...)$ for $\underline{\omega}\in \Omega$, we next consider the random composition of functions
\begin{align}
    \mathrm{f}_{\underline{\omega}}^n = \mathrm{f}_{\omega_n}\circ \dots \circ \mathrm{f}_{\omega_1}
\end{align}
for $n\in \NN$ and with the convention that $\mathrm{f}_{\underline{\omega}}^0$ is the identity map in $\mathrm{X}$. Let $\theta:\Omega \rightarrow\Omega$ denote the leftward shift on $\Omega$, that is, $\theta\underline{\omega} = (\omega_2,\omega_3,...)$ for $\underline{\omega}=(\omega_1,\omega_2,...)$. Note that $\theta$ is a measure-preserving transformation on $(\Omega, \mathcal{F}, \mathbb{P})$ and that $\mathrm{f}_{\underline{\omega}}^n$ satisfies the cocycle property
\begin{align}
    \mathrm{f}_{\underline{\omega}}^{n+m} = \mathrm{f}_{\theta^m\underline{\omega}}^n \circ \mathrm{f}^m_{\underline{\omega}},
\end{align}
for all $m,n\geq 0$. In this framework, continuous RDS define Markov chains with transition kernels
\begin{align}
    {\mathrm{P}}(\mathrm{x}, A) := \mathbb{P}_0 (\mathrm{f}_\omega(\mathrm{x}) \in A),
\end{align}
and continuity of $\mathrm{f}_\omega$ directly yields the Feller property for the transition kernel ${\mathrm{P}}$. Further specializing to the Markov chains derived from the one-point, two-point and projective processes, we assume the following:
\begin{itemize}
    \item[(A1)] $\Omega_0$ is a smooth complete Riemannian manifold, the law $\mathbb{P}_0$ on $\Omega_0$ admits a density $\varrho_0$ with respect to the Lebesgue measure $\d \omega$ on $\Omega_0$ and the mapping $\Omega_0\times \mathrm{X}\rightarrow \mathrm{X}$ given by $(\omega,\mathrm{x})\mapsto \mathrm{f}_\omega(\mathrm{x})$ is $C^2$.
    \item[(A2)] There exists a constant $C_0>0$ such that $\mathbb{P}_0$-a.s.,
    \begin{align}
        \Vert (D_\x \mathrm{f}_\omega)^{-1} \Vert + \Vert D_\x \mathrm{f}_\omega \Vert + \Vert \mathrm{f}_\omega \Vert_{C^2} \leq C_0.
    \end{align}
    \item[(A3)] The RDS $\mathrm{f}_\omega$ preserves, $\mathbb{P}_0$-a.s., the Lebesgue measure on $\mathrm{X}$.
\end{itemize}
 Given finitely many coordinates $\omega_i$ of the random sequence $\underline{\omega} = \lbrace   \omega_i \rbrace_{i\geq 1}$, we adopt the notation $\underline{\omega}^n = (\omega_1,...,\omega_n)\in \Omega_0^n$, the Cartesian product of $n$ copies of $\Omega_0$. Similarly, we abuse notation to denote $\varrho_0$ the density of the product law $\mathbb{P}_0^n$ on $\Omega_0^n$ given by
\begin{align}
    \varrho_0(\underline{\omega}^n) = \varrho_0(\omega_1)\dots \varrho_0(\omega_n).
\end{align}
We begin by giving a simple condition for a Markov kernel ${\mathrm{P}}$ to be strongly aperiodic in the sense of Theorem \ref{thm:abstractHarris}. Its proof is given in \cite{BCZG2023}*{Lemma 3.2}.
\begin{lemma}\label{lemma:blackboxaperiodic}
    Assume there exists $\omega_*\in \text{Supp}(\mathbb{P}_0)$ and $\mathrm{x}_*\in \mathrm{X}$ such that $\mathrm{f}_{\omega_*}(\mathrm{x}_*)=\mathrm{x}_*$. Then, ${\mathrm{P}}$ is strongly aperiodic.
\end{lemma}
Next, to derive a sufficient condition for the existence of open small sets, fix $\mathrm{x}\in \mathrm{X}$ and define $\Psi_\mathrm{x}:\Omega_0^n \rightarrow \mathrm{X}$ by
\begin{align}\label{eq:defPsi}
    \Psi_\mathrm{x}(\underline{\omega}^n) := \mathrm{f}_{\omega_n}\circ \dots \circ \mathrm{f}_{\omega_1}(\mathrm{x}), \quad \underline{\omega}^n = (\omega_1, ..., \omega_n).
\end{align}
The following Lemma is drawn from \cite{BCZG2023}*{Proposition 3.1}.
\begin{lemma}\label{lemma:blackboxsmallset}
    Suppose the RDS $\mathrm{f}_\omega$ satisfies (A1) and there exists $n\geq 1$ and $(\underline{\omega}_*^n,\mathrm{x}_*)\in \text{Supp}(\mathbb{P}_0^n) \times \text{Supp}(\pi)$ such that
    \begin{enumerate}
        \item There exist some $c>0$ and $\ep>0$ such that $\varrho_0(\underline{\omega}^n)\geq c>0$ if $|\underline{\omega}^n - \underline{\omega}_*^n|< \ep$.
        \item The map $\Psi_{\mathrm{x}_*}$ is a submersion at $\underline{\omega}^n = \underline{\omega}_n^*$.
\end{enumerate}
Then, ${\mathrm{P}}$ admits an open, ${\mathrm{P}}^n$-small set with corresponding measure $\nu_n \ll \Leb_\mathrm{X}$.
\end{lemma}

Lemma 10 in \cites{C2023} gives a sufficient condition for ${\mathrm{P}}$ to be topologically irreducible. It reads:
\begin{lemma}\label{lemma:blackboxirreducibility}
Suppose the RDS $\mathrm{f}_\omega$ satisfies (A1). Assume that for all $\mathrm{x}\in\mathrm{X}$ and all open sets $\mathrm{U}\subset\mathrm{X}$ there exists $n\in \N$ and ${\underline{\omega}}_*^n\in \text{Supp}(\mathbb{P}_0^n)\subset\Omega_0^n$ such that 
\begin{enumerate}
    \item There exist some $c>0$ and $\ep>0$ such that $\varrho_0(\underline{\omega}^n)\geq c>0$ if $|\underline{\omega}^n - \underline{\omega}_*^n|< \ep$.
    \item $\Psi_\mathrm{x}({\underline{\omega}}_*^n)\in \mathrm{U}$.
\end{enumerate}
Then, ${\mathrm{P}}$ is topologically irreducible.
\end{lemma}

If $\mathrm{X}$ is a compact phase space, irreducibility, small sets and aperiodicity are sufficient conditions for exponential ergodicity of the Markov chain. However, the one-point, projective and two-point processes given in Definition~\ref{def:OPPPJPTPP} are defined in non-compact phase spaces (because of the fixed point of the dynamics in the one-point and projective processes and because of the diagonal in the two-point process). Hence, we must find a Lyapunov-drift function for each of these processes. While a Lyapunov-drift function for the one-point process can be constructed after understanding the behaviour of the flow maps $f_\omega^\a$ and $f_\omega^\b$ near the fixed points of $f_{\underline{\omega}}$, a Lyapunov-drift function for the two-point process must address, among others, the dynamics near the diagonal $\Delta$. This is achieved by introducing the \textit{top Lyapunov exponent}, which is defined by
\begin{align}
    \lambda_1 := \lim_{n\rightarrow\infty}\frac{1}{n} \log |D_{\mathrm{x}}\mathrm{f}_{\underline{\omega}^n}|
\end{align}
and measures the maximum rate of separation of two arbitrarily close points by the RDS $\mathrm{f}_\omega$. If the Markov transition kernel ${\mathrm{P}}$ has a stationary ergodic measure and the map $\mathrm{f}_\omega$ is almost surely bounded in $C^1$, then the Multiplicative Ergodic Theorem, see \cite{kifer1986ergodic}*{Part II}, ensures that $\lambda_1$ is well-defined and is almost surely a deterministic constant. In \cite{BCZG2023}*{Proposition 4.5} it is shown that the positivity of the top Lyapunov exponent is enough to construct a Lyapunov-drift function for the two-point process:

\begin{proposition}\label{prop:blackboxdriftTPP}
    Assume that the one-point and projective chains are geometrically ergodic. If the top Lyapunov exponent is positive, then there exists a function $\mathrm{V}^{(2)}\in L^1(\mathrm{X}^2\setminus\Delta)$ of the form
    \begin{align}
        \mathrm{V}^{(2)}(\mathrm{x}, \mathrm{y}) = d_X(\mathrm{x},\mathrm{y})^{-\xi}\psi_\xi(\mathrm{x}, \hat{\mathrm{w}}(\mathrm{x}, \mathrm{y})), \quad \hat{\mathrm{w}}(\mathrm{x}, \mathrm{y}) = \frac{\mathrm{exp}_\mathrm{x}^{-1}(\mathrm{y})}{|\mathrm{exp}_\mathrm{x}^{-1}(\mathrm{y})|}
    \end{align}
    for some $\xi\in(0,1)$ and some continuous and strictly positive $\psi_\xi$ such that $\mathrm{V}^{(2)}\geq 1$, and there exists $\gamma < 1$ and $s_*>0$ with
    \begin{align}
        {\mathrm{P}}^{(2)}\mathrm{V}^{(2)}(\mathrm{x},\mathrm{y}) < \gamma \mathrm{V}^{(2)}(\mathrm{x},\mathrm{y}),
    \end{align}
    for all $(\mathrm{x},\mathrm{y})\in \Delta(s_*)$, where $\Delta(s) = \lbrace   (\mathrm{x},\mathrm{y})\in \mathrm{X}^2 :  0 < d_X(\mathrm{x}, \mathrm{y}) < s \rbrace  $.
\end{proposition}
It then remains to check whether the top Lyapunov exponent is positive. For this, let $d$ be the dimension of $\mathrm{X}$ and consider the fiber bundle $SL(\mathrm{X})$ over $\mathrm{X}$ with fiber
\begin{align}
    SL_\mathrm{p}(\mathrm{X}) = \lbrace   A: \R^d\rightarrow T_{\mathrm{p}}\mathrm{X} \text{ linear map such that det}(A) = 1 \rbrace  
\end{align}
understood as a principal bundle over $\mathrm{X}$ with structure group $SL_d(\R)$. For a fixed $\mathrm{x}\in \mathrm{X}$ and a determinant 1 isomorphism $E:\R^d \rightarrow T_{\mathrm{x}}\mathrm{X}$, the fiber $SL_\mathrm{p}(\mathrm{X})$ over $\mathrm{p}\in \mathrm{X}$ can now be viewed as the space of determinant 1 mappings $A:T_\mathrm{x}\mathrm{X}\rightarrow T_{\mathrm{p}}\mathrm{X}$ and we define $\widehat\Psi_\mathrm{x}:\Omega_0^n\rightarrow SL(\mathrm{X})$ given by
\begin{align}
    \widehat\Psi_\mathrm{x}(\underline{\omega}^n) := \left( \Psi_\mathrm{x}(\underline{\omega}^n), D_\mathrm{x} \Psi_\mathrm{x}(\underline{\omega}^n) \right).
\end{align}
The following result corresponds to \cite{BCZG2023}*{Proposition 3.3} and gives conditions on $\Psi_\mathrm{x}$ and $\widehat \Psi_\mathrm{x}$ for which the top Lyapunov exponent associated to $\Psi_\mathrm{x}$ is positive.
\begin{proposition}\label{prop:blackbox_lyapexponent}
Assume that the RDS $\mathrm{f}_\omega$ on $\mathrm{X}$ is such that
\begin{enumerate}
    \item the law $\PP_0$ admits a density $\rho_0$ with respect to the Lebesgue measure $\dd\omega$ on $\Omega_0$ (smooth complete Riemannian manifold) and the mapping $(\omega,\mathrm{x})\mapsto \mathrm{f}_\omega(\mathrm{x})$ is $C^2$;
    \item for a.e. $\omega\in \Omega_0$, $\mathrm{f}_\omega$ is a local diffeomorphism and admits an ergodic stationary measure $\pi$ on $\mathrm{X}$ with the property
    \begin{equation}
        \int (\log^+|D_{\mathrm{x}}\mathrm{f}_\omega|+\log^+|(D_{\mathrm{x}}\mathrm{f}_\omega)^{-1}|)\dd \pi(\mathrm{x})\dd\PP_0(\omega)<\infty.
    \end{equation}
\end{enumerate}
Moreover, assume that the associated Markov kernel ${\mathrm{P}}$ is $\mathrm{V}$-uniformly geometrically ergodic with stationary measure $\pi$. 
If there exist $n\geq 1$ and a point $(\omega_*^{(n)},\x_*)\in \supp(\rho_0)\times \supp(\pi)$ which satisfies 
\begin{itemize}
    \item[(i)] There exists $c,\eps>0$ such that $\varrho(\underline{\omega}^{n})\geq c>0$ if $|\underline{\omega}^{n}-\underline{\omega}_*^{n}|<\eps$;
    \item[(ii)] The mapping $\Psi_{\mathrm{x}_*}$ is a submersion at $\underline{\omega}^n = \underline{\omega}_*^n$.
    \item[(iii)] Let $\mathrm{K}_{\mathrm{x},\underline{\omega}_*^n}:= \text{ker}D_{\underline{\omega}^n}\Psi_\mathrm{x}\subset T_{\underline{\omega}_*^n}\Omega_0^n$. The restriction of $D_{\underline{\omega}_*^n}\widehat\Psi_{\mathrm{x}_*}$ to $\mathrm{K}_{\mathrm{x},\underline{\omega}_*^n}$ is surjective as a linear operator onto $T_{\widehat\Psi_{\mathrm{x}_*}(\underline{\omega}_*^n)} SL(\mathrm{X})$.
\end{itemize}
Then $\lambda_1>0$.
\end{proposition}

It was observed in \cites{C2023} that while \cite{BCZG2023}*{Proposition 3.3} is stated for a compact Riemannian manifold $\mathrm{X}$, the result holds under the weaker assumption of $\mathrm{X}$ being $\sigma$-compact, which is our case here, as $\mathrm{X}$ is compact up to a finite number of points.

\subsection{Almost-sure exponential mixing}
Assume next that 
\begin{itemize}
    \item [(S)] The continuous RDS $\mathrm{f}_\omega$ admits a stationary measure $\pi$ which is almost-surely invariant, that is, $(\mathrm{f}_\omega)_*\pi = \pi$ for $\mathbb{P}_0$-almost every $\omega\in \Omega_0$.
\end{itemize}
By a duality argument, almost-sure mixing follows from decay of the correlations
\begin{align}
    \text{Cor}_n(g, h) := \left| \int_{\mathrm{X}} g(\mathrm{x}) h \circ \mathrm{f}_{\underline{\omega}}^n(\mathrm{x}) \d \pi(\mathrm{x}) \right|
\end{align}
with probability 1 as $n\rightarrow \infty$, for sufficiently smooth observables $g,h$ with $\pi$-mean zero. As $\text{Cor}_n(g, h)$ is a random variable on $\underline{\omega}$, Chebyshev's inequality shows that
\begin{align}
    \mathbb{P}_0 \left( \lbrace   \text{Cor}_n(g, h) > \ep \rbrace \right)  \leq \ep^{-2} \mathbb{E}_{\mathbb{P}_0} \left| \int_{\mathrm{X}} g(\mathrm{x}) h \circ \mathrm{f}_{\underline{\omega}}^n(\mathrm{x}) \d \pi(\mathrm{x}) \right|^2.
\end{align}
In turn, we can write
\begin{align}
    \mathbb{E}_{\mathbb{P}_0}\left| \int_{\mathrm{X}} g(\mathrm{x}) h \circ \mathrm{f}_{\underline{\omega}}^n(\mathrm{x}) \d \pi(\mathrm{x}) \right|^2 &=  \mathbb{E}_{\mathbb{P}_0} \int_{\mathrm{X}\times\mathrm{X}}g(\mathrm{x})g(\mathrm{y}) h \circ \mathrm{f}_{\underline{\omega}}^n(\mathrm{x})h \circ \mathrm{f}_{\underline{\omega}}^n(\mathrm{y}) \d \pi(\mathrm{x}) \d \pi(\mathrm{y}) \\
    &= \int_{\mathrm{X}^2} g^{(2)}(\mathrm{x},\mathrm{y}) ({\mathrm{P}}^{(2)})^n h^{(2)}(\mathrm{x}, \mathrm{y}) \d \pi^{(2)}(\mathrm{x}, \mathrm{y}),
\end{align}
where $g^{(2)}(\mathrm{x},\mathrm{y}) = g(\mathrm{x})g(\mathrm{y})$, $h^{(2)}$ is defined analogously, $\d \pi^{(2)}(\mathrm{x}, \mathrm{y}) = \d \pi(\mathrm{x}) \d \pi(\mathrm{y})$ and ${\mathrm{P}}^{(2)}$ denotes the Markov semigroup for the two-point process $(\mathrm{x}_n,\mathrm{y}_n) = (\mathrm{f}_{\underline{\omega}}^n(\mathrm{x}),\mathrm{f}_{\underline{\omega}}^n(\mathrm{y}))$. Hence, geometric ergodicity of the two-point process implies almost-surely exponential mixing for $\mathrm{f}_\omega$. This principle and the argument presented above are well known in the field of random dynamical systems, see \cites{bedrossian2022almost, dolgopyat2004sample}. However, ${\mathrm{P}}^{(2)}$ cannot be uniformly geometrically ergodic with respect to $\pi^{(2)}$ because the diagonal set
\begin{align}
    \Delta := \lbrace   (\mathrm{x}, \mathrm{x}): \mathrm{x}\in \mathrm{X} \rbrace   \subset \mathrm{X}\times \mathrm{X}
\end{align}
is almost-surely invariant for the two-point process and thus it supports at least one additional stationary measure for ${\mathrm{P}}^{(2)}$ distinct from $\pi^{(2)}$. For $\mathrm{X}^{(2)}=\mathrm{X}\times \mathrm{X}\setminus \Delta$, the above principle can be made precise if there exists a Lyapunov-drift function for ${\mathrm{P}}^{(2)}$ on $\mathrm{X}^{(2)}$, since it is now a non-compact phase space, that ensures positive recurrence away from $\Delta$. The next result can be found in \cite{BCZG2023}*{Proposition 4.6}.

\begin{proposition}\label{prop:decaycorr}
    Let $\mathrm{f}_\omega$ be a continuous RDS on a compact, orientable Riemannian manifold $\mathrm{X}$ without boundary satisfying condition $(S)$. Assume that the two-point process with kernel ${\mathrm{P}}^{(2)}$ on $\mathrm{X}^{(2)}$ is $\mathrm{V}^{(2)}$-geometrically ergodic, for some $\mathrm{V}^{(2)}:\mathrm{X}^{(2)}\rightarrow[1,\infty)$ that is integrable with respect to $\pi^{(2)}$. Then, for all $q, s>0$ there exist a function $\mathrm{C}=\mathrm{C}_{q,s}:\Omega_0\rightarrow[1,\infty)$ and a constant $\lambda = \lambda_{q,s}>0$ such that 
\begin{align}
\left| \int_{\mathrm{X}} g(\mathrm{x}) h \circ \mathrm{f}_{\underline{\omega}}^n(\mathrm{x}) \d \pi(\mathrm{x}) \right| \leq \mathrm{C}(\underline{\omega}) e^{-\lambda n}\Vert g \Vert_{H^s}\Vert h \Vert_{H^s},
\end{align}
for all $\pi$-mean free functions $g,h\in H^s=H^s(\mathrm{X})$ with $\mathbb{E}_{\mathbb{P}}[\mathrm{C}^q] < \infty$. 
\end{proposition}
An inspection of the proof shows that the result remains true if $\mathrm{X}$ admits a finite cover by smooth charts.

\subsection{Summary of sufficient conditions for exponential mixing}
Theorem \ref{thm:mainexpdecaycorr}, and also Corollary \ref{cor:unifexpmixing} after a standard duality argument, is ensured for $\kappa=0$ by Proposition \ref{prop:decaycorr} if the one-point process $f_\omega$ given by \eqref{eq:defOPP} satisfies (S) and its associated two-point process $f_\omega^{(2)}$ is $\mathrm{V}^{(2)}$-uniformly geometrically ergodic. Throughout the manuscript, we define
\begin{align}
    \pi:=\Leb_\mathsf{X}, \quad \hat\pi:=\Leb_\mathsf{Y}, \quad \pi^{(2)}:=\pi\otimes \pi. 
\end{align}
Since $u(t,\cdot,{\omega};\a,\b)$ is divergence-free, the measure $\pi$ is invariant under $f_\omega$ and thus stationary. Hence, $\hat\pi$ is stationary for $\hat{f}_\omega$ on $\mathsf{Y}$ and $\pi^{(2)}$ is stationary for $f_\omega^{(2)}$ on $\mathsf{X}^{(2)}$ as well so that (S) is satisfied for all three processes.

The geometric ergodicity of $f_\omega^{(2)}$ is, in turn, a consequence of Theorem \ref{thm:abstractHarris} if $f_\omega^{(2)}$ admits an open small set, is topologically irreducible, aperiodic and admits a Lyapunov-drift function $\mathrm{V}^{(2)}$. The existence of such $\mathrm{V}^{(2)}$ is a consequence of Proposition \ref{prop:blackboxdriftTPP}, which requires the one-point process $f_\omega$ and the projective process $\hat{f}_\omega$ to be themselves geometrically ergodic, and for the top Lyapunov exponent of $f_\omega$ to be positive. 

The geometric ergodicity of $f_\omega$ and $\hat{f}_\omega$ is also obtained by means of the Harris Theorem \ref{thm:abstractHarris} after showing these two processes are also topologically irreducible, aperiodic and admit open small sets and Lyapunov drift functions $\mathrm{V}$ and $\hat{\mathrm{V}}$. For these three processes, the existence of open small sets is obtained via Lemma \ref{lemma:blackboxsmallset}; their aperiodicity is automatic from Lemma \ref{lemma:blackboxaperiodic} as they are the identity for $\omega=0$ and they are continuous maps, while the topological irreducibility is a consequence of Lemma \ref{lemma:blackboxirreducibility}. Lastly, the positivity of the top Lyapunov exponent is deduced by the combination of Propositions \ref{prop:blackbox_lyapexponent} and \ref{prop:joint-surj}.

In Section \ref{sec:OPP}, \ref{sec:PJP} and Section \ref{sec:TPP} we show that the Markov transition kernels ${P}$, $\hat{P}$ and ${P}^{(2)}$ are topologically irreducible, admit an open small set (via Lemma \ref{lemma:blackboxsmallset}) and admit Lyapunov-drift function ${\mathrm{V}}$, $\hat{\mathrm{V}}$ and ${\mathrm{V}}^{(2)}$, respectively. In Section \ref{sec:toplyapexp} we investigate the positivity of the top Lyapunov exponent by means of Proposition \ref{prop:joint-surj} and Proposition \ref{prop:blackbox_lyapexponent}.

\subsection{Uniform-in-$\kappa$ exponential mixing and enhanced dissipation}
In the presence of molecular diffusivity the composition flow map $f_\omega^\kappa$ is no longer deterministic but instead satisfies the stochastic differential equation \eqref{eq:introSDEITO} with convective vector field $v_\kappa(t,\x) = u(t,\x, \underline{\omega};\a,\b) - 2\kappa\x$ and the deterministic RDS framework of \cites{BCZG2023} does not apply directly any more. However, the recent contribution \cites{CIS2024} shows that the exponential decay of correlations of Proposition \ref{prop:decaycorr} remains true for $f_\omega^\kappa$ with uniform-in-$\kappa$ constants if the inviscid dynamics are robust enough. To obtain a uniform-in-$\kappa$ version of Proposition \ref{prop:decaycorr}, we follow the framework derived in \cites{CIS2024} and we introduce the stochastic two-point Markov transition kernel $P^{(2)}_\kappa$ defined by
\begin{align}
    P^{(2)}_\kappa((\x,\y),U^{(2)}) = \PP(f_\omega^{(2),\kappa}(\x,\y)\in U^{(2)})
\end{align}
where $f_\omega^{(2),\kappa}(\x,\y) = (f_\omega^\kappa(\x), f_\omega^\kappa(\y))$ denotes the stochastic two-point process for $(\x,\y)\in \mathsf{X}^{(2)}$. We consider the following assumption $P^{(2)}_\kappa$.

\begin{itemize}
    \item [($H_\kappa$)] Assume there exists $\ell\in \N$, $\gamma\in(0,1)$, $\beta>0$, $R> \frac{2\beta}{1-\gamma}$, $\upsilon>0$ and a probability measure $\nu_0$ such that for all $\kappa>0$ sufficiently small there holds 
\begin{align}\label{eq:Hkappadrift}
    P^{(2), \ell}_\kappa \mathrm{V}^{(2)}_\kappa \leq \gamma \mathrm{V}^{(2)}_\kappa + \beta,
\end{align}
and
\begin{align}\label{eq:Hkappasmallset}
    \inf_{(p,q)\in \lbrace \mathrm{V}^{(2)}_\kappa\leq R \rbrace} P^{(2),\ell}_\kappa((p,q), \cdot) \geq \upsilon\nu_0(\cdot),
\end{align}
for some $\mathrm{V}^{(2)}_\kappa:\mathsf{X}^{(2)} \rightarrow [1,\infty)$.
\end{itemize}

If $(H_\kappa)$ is satisfied, a quantitative version of Harris Theorem, see \cites{CIS2024, hairer2011yet, meyn2012markov}, shows that $P^{(2)}_\kappa$ is $\mathrm{V}_\kappa^{(2)}$-uniformly geometrically ergodic with geometric constant $\lambda>0$ independent of $\kappa\in (0,\kappa_0)$ for some $\kappa_0>0$ small. As a result, almost sure uniform-in-$\kappa$ exponential decay of correlations follows.

\begin{proposition}[Lemma 3.3 in \cites{CIS2024}]\label{prop:unifdecaycorr}
    Assume that $P^{(2)}_\kappa$ is $\mathrm{V}_\kappa^{(2)}$-geometrically ergodic. Then, for all $s>0$ and all $0<q<\infty$ there exists a random variable $\mathrm{C}_\kappa\geq 1$ and a deterministic $\lambda>0$, both independent of $\bB$, such that for every pair of mean-zero test functions $g,h\in \dot{H}^s(\Sph)$ and every $n\in \N$, there holds
    \begin{align}\label{eq:stochdecaycorr}
        \left| \int_\Sph g(\x) h(f_{\underline{\omega}_n}^\kappa(\x)) \d \x \right| \leq \mathrm{C}_\kappa e^{-\lambda n} \Vert g \Vert_{H^s}\Vert h \Vert_{H^s},
    \end{align}
    almost surely. Moreover, there exists $\overline{\mathrm{C}}_q>0$ independent of $\kappa$ such that $\EE |\mathrm{C}_\kappa|^q \leq \overline{\mathrm{C}}_q$, for all $\kappa\in (0,\kappa_0)$, for some $\kappa_0>0$ small enough.
\end{proposition}
Once Proposition \ref{prop:unifdecaycorr} is established, the continuous-in-time statement of Theorem \ref{thm:mainexpdecaycorr} follows as in \cites{coti2026three}. A standard duality argument completes Corollary \ref{cor:unifexpmixing}, while Corollary \ref{cor:optimalenhanceddiss} is a consequence of \eqref{eq:stochdecaycorr} and an inspection of the proof of \cite{navarro2025exponential}*{Proposition 2.2}, see Section \ref{sec:unifkappa}.

To obtain Proposition \ref{prop:unifdecaycorr} we shall see that the two-point chain $P_\kappa^{(2)}$ satisfies $(H_\kappa)$. 
This is achieved in Proposition \ref{prop:kappasmallset} in Section \ref{sec:unifkappa} which requires, among other things, the construction of some $\widetilde{\mathrm{W}}_\kappa$ that satisfies a Lyapunov-Foster drift-condition for the stochastic two-point process $f_\omega^{(2),\kappa}$ with constants independent of $\kappa\in (0,\kappa_0)$, for some $\kappa_0>0$ small enough.
In Section \ref{sec:driftstochTPP}, we define such $\widetilde{\mathrm{W}}_\kappa$ in terms of suitable modifications $\mathrm{V}_\kappa$ and $\mathrm{W}_\kappa$ of the Lyapunov-drift functions $\mathrm{V}$ and $\mathrm{W}$ for the inviscid one-point and two-point processes respectively and we show the Lyapunov-Foster drift estimate in Proposition \ref{prop:tildeWkappadrift}.

\section{Differential geometry on the sphere}\label{sec:DiffGeom}
To understand the projective process and the positivity of the top Lyapunov exponent a precise analysis of the differentials $D_\x f_\omega^\a$ and $D_\x f_\omega^\b$ of the two flow maps and their composition is required. 
Compared to the $\TT^2$ setting considered in \cites{BCZG2023,C2023} and the $\T^3$ setting of \cites{coti2026three}, one of the main difficulties when considering flows on the sphere is that $\Sph$ is not parallelizable, that is, the tangent bundle $T\Sph$ cannot be identified with the product space $\Sph\times\RR^2$.  Consequently, the differential of the flow map cannot be interpreted as a linear endomorphism,  since in general it does not map the tangent space to a point to itself. This obstruction can be circumvented by defining a local reference frame and a map that encodes how the reference frame on $T_\x\Sph$ changes whenever $\x\in\Sph$ moves with $f_\omega^\a$ and $f_\omega^\b$.

In this section we present the main geometric tools used throughout the paper.
We equip $\Sph$ with the standard metric induced by $\RR^3$ in view of the canonical embedding $\Sph\subset \RR^3$. For all $\x\in\Sph$, the tangent space $T_\x\Sph$ is given by
\begin{align}
    T_\x\Sph = \lbrace   v\in \R^3 : \x\cdot v = 0 \rbrace  ,
\end{align}
and the orthogonal projection of a vector $v\in \R^3$ to $T_\x\Sph$ is defined by
\begin{equation}
 \Pi_\x v:=v-(v\cdot \x)\x, 
\end{equation}
Then, for a smooth vector field $w$ in $\RR^3$ such that $w|_\Sph\in T\Sph$, the Levi-Civita covariant derivative with respect to $v\in T_\x\Sph$ on $\Sph$ is given by
\begin{equation}\label{eq:covariantderivative}
    \nabla_v w(\x)=\Pi_\x(D_\x w(\x)v),\quad \x\in\Sph,
\end{equation}
where $D_\x w$ denotes the ambient space $3\times 3$ Jacobian matrix of ${w}$. The commutator of two vector fields $v_1,v_2\in T\Sph$, using smooth extensions to $\R^3$ denoted by the same symbols, is the vector field
\begin{align}
    [v_1,v_2](\x) = \nabla_{v_1}v_2(\x) - \nabla_{v_2}v_1(\x) = [v_1,v_2]_{\R^3}(\x),
\end{align}
where $[\cdot,\cdot]_{\R^3}$ denotes the commutator of vector fields on $\R^3$. The equality holds because the commutator of vector fields tangent to the sphere is itself tangent to the sphere.
 \subsection{Frames on $\Sph$}
Let $\cU\subset\Sph$ be an open set of $\Sph$. A smooth local frame $E$ on $\cU$ is an ordered pair of smooth vector fields $(e_1, e_2)$ defined on $\cU$ that are linearly independent and span the tangent bundle, namely the pair $(e_1(\x), e_2(\x))$ forms a basis of $T_\x\Sph$ for all $\x\in \cU$.  We say that a smooth frame $E$ is orthonormal if $(e_1, e_2)$ are smooth orthonormal vector fields with respect to the ambient $\R^3$ scalar product. If we identify the space of frames on $\cU$ with the space of real $3\times 2$ matrices, which we denote by $M^{3\times 2}(\R)$, a left inverse for $E(\x)$ is given\footnote{In general, if $E(\x)$ is a $3\times 2$ matrix of rank 2, there exists a (not necessarily unique) left-inverse $E(\x)^{-1}$. If the column vectors of $E(\x)$ are orthonormal, one can choose $E(\x)^{-1} = E(\x)^T$.} by $E(\x)^T\in M^{2\times3}(\R)$, i.e., $E(\x)^TE(\x) = I\in M^{2\times2}(\R)$. In particular, for any $\x\in\cU$ and any vector $v\in T_\x\Sph$, the vector $w=E(\x)^T v\in \R^2$ denotes the local coordinates of $v\in T_\x\Sph$ in the frame $E(\x)$. With this identification, we have 
\begin{align}
    T\Sph|_\cU \cong \cU\times \R^2.
\end{align}
If $E$ is a smooth orthonormal frame on $\cU$, for all $\x\in \cU$ the projection of a vector $v\in \R^3$ to $T_\x\Sph$ can now be written as
\begin{equation}\label{eq:projframe}
    \Pi_\x v =\sum_{i=1}^2(e_i(\x)\cdot v)e_i(\x)= E(\x)E(\x)^T v
\end{equation}
and we can find the local coordinates of the covariant derivative $\nabla_{v}{w}(\x)$ in the orthonormal frame $E$.
\begin{lemma}\label{lemma:framecovariantD}
Let $E(\x)$ be a local smooth orthonormal frame on $\cU\subset \Sph$. Let $\x\in \cU$ and $v_1\in T_\x\Sph$. Then,
\begin{align}
    E(\x)^T \nabla_{v_1}v_2(\x) = E(\x)^T D_\x v_2(\x) v_1(\x)
\end{align}
for all vector fields $v_2$ defined on open subsets of $\RR^3$ that contain $\x\in \Sph$ and such that $v_2|_\Sph$ is tangent to $\Sph$.
\end{lemma}
\begin{proof}
    The proof follows directly from \eqref{eq:covariantderivative} and \eqref{eq:projframe}, namely
    \begin{equation}
        E(\x)^T \nabla_{v}{w}(\x)=E(\x)^T \Pi_\x(D_\x {w}(\x) v(\x))=(E(\x)^TE(\x))E(\x)^TD_\x {w}(\x) v(\x)=E(\x)^T D_\x {w}(\x) v(\x).
    \end{equation}
\end{proof}
We further note that if $E$ is a smooth orthonormal frame on $\cU$, then for all $\x\in \cU$ and all $v\in UT_\x(\Sph)$, namely $|v|^2 := v\cdot v =1$, there exists a unique $w\in \mathbb{S}^1= \lbrace   (w_1,w_2)\in \R^2: w_1^2 + w_2^2 =1 \rbrace  $ such that $v = E(\x)w$. Such $w\in \mathbb{S}^1$ is indeed given by $w=E(\x)^Tv$ and we thus see that
\begin{align}
        UT\Sph|_\cU \cong \cU\times \mathbb{S}^1.
\end{align}
We say that a smooth frame $E=(e_1,e_2)$ on $\cU$ has unit area if $\mu_\x(e_1(\x),e_2(\x))=1$, for all $\x\in \cU$. Here,
\begin{align}
    \mu_\x(e_1(\x),e_2(\x)):= \x\cdot (e_1(\x) \times e_2(\x))
\end{align}
denotes the standard area form on $\Sph$ inherited from $\R^3$. Fixing now a smooth frame $E_0$ on $\cU$ with unit area we observe that any other unit area and smooth frame $E$ is given by $E(\x) = E_0(\x)\mathrm{g}(\x)$, for some $g(\x)\in SL(2,\R)$. With this identification, if $\mathcal{F}^{SL}(\Sph)$ denotes the principal $SL(2,\R)$-bundle of unit area frames, we have
\begin{align}
    \mathcal{F}^{SL}(\Sph)|_\cU \cong \cU\times SL(2,\R)
\end{align}
and $g(\x)\in SL(2,\R)$ is given by $E_0(\x)^TE(\x)$, if $E_0$ is an orthonormal frame. 

For all $\x\in\cU$ and $u\in T_\x\Sph$, the Levi-Civita covariant derivative of $E_0(\x)$ in the direction $u$ is given by
\begin{equation}
    (\nabla_u E_0)(\x) := (\nabla_u e_1(\x), \nabla_u e_2(\x)),
\end{equation}
and represents the rate of change of the reference frame $E_0(\x)$ in the direction $u\in T_\x\Sph$. To compute the covariant derivative, we radially extend $e_1(\x)$ and $e_2(\x)$ from $\Sph$ to $\R^3$ and still denote them by $e_1(\x)$ and $e_2(\x)$ to simplify notation. Since $\lbrace   e_1(\x), e_2(\x) \rbrace  $ is an orthonormal basis of $T_\x\Sph$ for all $\x\in \cU$, there exists a one-form $\gamma_\x$ such that
\begin{align}\label{eq:covariantvectorframe}
    \nabla_u e_1(\x) = \gamma_\x(u)e_2(\x), \quad \nabla_u e_2(\x) = -\gamma_\x(u)e_1(\x),
\end{align}
for all  $u \in T_\x\Sph$, so that for $u=u^1e_1(\x) + u^2e_2(\x)$, we have 
\begin{equation}\label{eq:covariantformframe}
    \gamma_\x(u) = u^1\gamma_\x(e_1(\x)) + u^2 \gamma_\x(e_2(\x)).
\end{equation}

\subsection{A convenient reference frame}
We now fix a local smooth orthonormal and unit-area reference frame which will be employed in Section~\ref{sec:PJP}.
Let $\mathbf{e}_3=(0,0,1)\in \Sph$ and set
\begin{align}
    \cU_3 := \lbrace   \x\in \Sph : \Pi_\x \mathbf{e}_3 \neq 0 \rbrace   = \Sph\setminus \lbrace   \pm \mathbf{e}_3 \rbrace  .
\end{align}
For all $\x\in \cU_3$, we define the orthonormal tangent vector fields
\begin{align}\label{eq:referenceframeE}
   e_1(\x) := \frac{\Pi_\x\mathbf{e}_3}{\Vert \Pi_\x\mathbf{e}_3 \Vert}\in T_\x\Sph, \quad e_2(\x):=\x\times e_1(\x)\in T_\x\Sph.
\end{align}
Then, $E_0(\x):= ( e_1(\x), e_2(\x))\in(T_\x\Sph)^2$ is a smooth orthonormal local frame over $\cU_3$. Moreover, $E_0(\x)$ has also unit area, since for all $\x\in \cU_3$ there holds
\begin{align}
    \mu_\x(e_1(\x),e_2(\x))
    =\x\cdot(e_1(\x)\times (\x\times e_1(\x)))
    =\x\cdot ((e_1(\x)\cdot e_1(\x))\x-(e_1(\x)\cdot \x)e_1(\x))
    =\x\cdot\x =1.
\end{align}
In coordinates, $E_0(\x)$ is given by
\begin{align}\label{eq:defframeA}
    e_1(\x) = \frac{1}{r}(-zx,-zy,r^2), \quad e_2(\x) = \frac{1}{r}(y,-x,0),
\end{align}
for $r=\sqrt{x^2+y^2}$ and $\x=(x,y,z)\in \cU_3$. 

\begin{lemma}\label{lemma:covariantframe}
    Let $E_0(\x) = (e_1(\x), e_2(\x))$ be given by \eqref{eq:defframeA}. For $\x_0=\left(\frac1{\sqrt2},\frac1{\sqrt2},0\right)$ we have
    \begin{align}\label{eq:framex0}
        e_1(\x_0)=(0,0,1), \quad  e_2(\x_0) =  \left(\frac1{\sqrt2},-\frac1{\sqrt2},0\right)
    \end{align}
    and there holds $(\nabla_u E_0)(\x_0) = 0$ for all $u\in T_{\x_0}\Sph$.
\end{lemma}

\begin{proof}
A simple computation shows that for all $\x=(x,y,z)\in \cU_3$,
\begin{equation}\label{eq:oneformonbase}
    \gamma_\x(e_1(\x)) = 0, \quad \gamma_\x(e_2(\x)) = \frac{z}{\sqrt{x^2 + y^2}}.
\end{equation}
Hence, at $\x_0=\left(\frac1{\sqrt2},\frac1{\sqrt2},0\right)$ we have
\begin{align}
     \gamma_{\x_0}(e_1(\x_0)) = 0, \quad \gamma_{\x_0}(e_2(\x_0)) = 0
\end{align}
and thus  $(\nabla_u E_0)(\x_0) = 0$ for all $u\in T_{\x_0}\Sph$ due to \eqref{eq:covariantvectorframe} and \eqref{eq:covariantformframe}.
\end{proof}

\section{The Lie algebra rank condition}\label{sec:LARC}

In this section we present the new approach based on the Lie algebra rank condition (LARC), introducing the necessary definitions and showing how it connects to the deterministic Harris Theorem framework developed in \cites{BCZG2023}.
More precisely, we show that global LARC plays a central role when proving topological irreducibility for the one-point process while local LARC enters the picture when discussing the existence of open small sets for the processes defined in Section \ref{sec:kernel andHarris}, and more importantly, when proving the positivity of the top Lyapunov exponent. The main results of this section are Theorem~\ref{thm:RashevskyChow} and Corollary~\ref{coroll:LARC-to-submersion} which connect the LARC  with topological irreducibility and a local submersion property.

\subsection{The Lie algebra rank condition on manifolds}

We introduce the discussion on a smooth $d$-dimensional manifold $\mathrm{M}$, since processes arising from the Rossby-Haurwitz model take place on open subsets of $\Sph$, unit tangent bundles and in product spaces. 
Nevertheless, because all statements are local, a reduction to coordinate charts reduces the discussion to the standard Euclidean definition of vector fields on open subsets of $\RR^d$.

To begin with, let $\mathrm{M}$ be a smooth connected $d$-dimensional manifold together with its Lie algebra of smooth vector fields $\mathfrak{X}(\mathrm{M})$. 
Consider a finite family  of smooth vector fields $\mathcal{V}=\{V_1,\dots,V_m\}\subset \mathfrak{X}(\mathrm{M})$.

\begin{definition}
    The Lie algebra $\Lie(\cV)$ generated by $\cV$ is the smallest Lie subalgebra of $\mathfrak{X}(\mathrm{M})$ that contains $V_1,\dots,V_m$. 
    Equivalently, $\Lie(\cV)$ is the linear span of all iterated Lie brackets, which are defined via induction as
    \begin{equation}\label{eq:iterated_lie_brackets}
    V_{(i)}:=V_i,\qquad V_{(i_1,\dots,i_k)}:=[V_{i_1},V_{(i_2,\dots,i_k)}],
    \end{equation}
    for all multi-indices $(i_1,\dots,i_k)\in\{1,\dots,m\}^k$ and all $k\in\NN$.
For a point $x\in \mathrm{M}$, the pointwise Lie algebra is the pointwise evaluation
\begin{equation}
    \Lie_x(\mathcal{V}):= \lbrace   V(x)\,:\, V\in\Lie(\mathcal{V})\rbrace  \subset T_x\mathrm{M}.
\end{equation}
\end{definition}
\noindent
We now present the main tool introduced in this section.

\begin{definition}\label{def:LARC}
    We say that the family $\mathcal{V}$ satisfies the Lie algebra rank condition (LARC) at $x\in \mathrm{M}$ if 
    \begin{equation}
    \Lie_x(\mathcal{V})=T_x\mathrm{M}.
    \end{equation}
    In addition, if the Lie algebra rank condition is satisfied for all $x\in \mathrm{M}$, we say that the family $\mathcal{V}$ is bracket-generating on $\mathrm{M}$.
\end{definition}

For each $i\in\{1,\dots,m\}$, the vector field $V_i$ generates a local smooth flow $\Phi_t^i$ that satisfies 
\begin{equation}
    \frac{\dd}{\dd t}\Phi_t^i(x) = V_i(\Phi_t^i(x)), \qquad \Phi_0^i(x)=x.
\end{equation}
For a choice of parameters $\omega=(\omega_1,\dots\omega_m)\in \RR^m$ we define the composed flow
\begin{equation}\label{eq:elementary-flow-block}
    \Phi_\omega := \Phi_{\omega_m}^m \circ \cdots \circ \Phi_{\omega_1}^1,
\end{equation}
where the order in \eqref{eq:elementary-flow-block} is fixed as follows: the flow of $V_1$ acts first, then the flow of $V_2$, and so on.  
We denote by a word of length $q$ the tuple
\begin{equation}
    \varpi=(\omega^1,\dots,\omega^q)\in(\RR^m)^q, \qquad \omega^j=(\omega^j_1,\dots,\omega^j_m),
\end{equation}
and, respectively, the flow $\Phi_\varpi$ associated with $\varpi$ and the endpoint map $\Psi_x^q$ for the point $x$ as
\begin{equation}\label{eq:flow-associated-word-and-endpoint-map}
    \Phi_\varpi(x) := \Phi_{\omega^q} \circ \cdots \circ \Phi_{\omega^1}(x), \qquad \Psi_x^q(\varpi):=\Phi_\varpi(x).
\end{equation}

\begin{remark}\label{rem:LARC-local-conventions}
We shall use the following elementary facts throughout the section.
\begin{enumerate}[label=\textup{(\roman*)}]
    \item If $(\mathcal U,\chi)$ is a chart on $\mathrm{M}$ and $\widetilde V=D\chi\circ V\circ \chi^{-1}$ denotes $V$ in local coordinates, then coordinate push-forward preserves Lie brackets. Hence the condition $\Lie_x(\cV)=T_x\mathrm M$ is equivalent to the corresponding Euclidean rank condition in $\RR^d$.

    \item If $\mathrm M$ is embedded in some $\R^n$, the rank condition is imposed in the intrinsic tangent space $T_x\mathrm M$, not in the ambient space.  For example, on $\Sph$ the brackets must span $T_x\Sph$, equivalently $\RR^2$ after choosing a local chart, not $\RR^3$.

    \item All flow constructions are local in time and the parameters are chosen sufficiently small so that the relevant compositions are well-defined. 
\end{enumerate}
\end{remark}

\subsection{Topological irreducibility}\label{ssec:top-irred-LARC}
The first consequence of the bracket-generating property is topological irreducibility. To translate the LARC condition into a checkable condition for topological irreducibility we use the Rashevsky-Chow theorem in the form stated by Agrachev and Sachkov~\cite{AScontroltheory}*{Chapter~5, Theorem~5.9}.
\begin{theorem}[Rashevsky--Chow]\label{thm:RashevskyChow}
Let $\mathrm M$ be a smooth connected manifold and let $\cV=\{V_1,\dots,V_m\}\subset\mathfrak X(\mathrm M)$ be bracket-generating on $\mathrm M$.
Then, for every $x\in \mathrm M$, the orbit generated by $\cV$ is the whole manifold:
\begin{equation}\label{eq:orbit}
    \mathcal O_x(\cV):=\left\{ \Phi_{t_n}^{i_n}\circ\cdots\circ \Phi_{t_1}^{i_1}(x): n\in\NN,\; i_j\in\{1,\dots,m\},\; t_j\in\RR\right\}=\mathrm M.
\end{equation}
\end{theorem}
Theorem~\ref{thm:RashevskyChow} gives an equivalent condition to verify the irreducibility criterion from Lemma~\ref{lemma:blackboxirreducibility}.
Indeed, any composition appearing in \eqref{eq:orbit} can be written as a composition of elementary blocks $\Phi_\omega$ by choosing vectors $\omega$ with only one non-zero coordinate.  
Consequently, Theorem~\ref{thm:RashevskyChow} implies that, whenever $\cV$ is bracket-generating on $\mathrm M$, for every $x\in\mathrm M$ and every non-empty open set $\mathcal O\subset\mathrm M$ there exists a length $q\in\NN$ and a parameter $\varpi\in(\RR^m)^q$ such that $\Psi_x^q(\varpi)\in\mathcal O$.

\subsection{Commutator loops and small sets}\label{ssec:small-sets-LARC}
In contrast to topological irreducibility, the existence of open small sets is a local non-degeneracy property.
By Lemma~\ref{lemma:blackboxsmallset}, it is enough to exhibit a point $x_*$ and a word $\varpi_*$ such that the density of the law is bounded from below in a neighbourhood of $\varpi_*$ and the endpoint map $\Psi^q_{x_*}$ is a submersion at $\varpi_*$. 
In the present setting, the density condition follows from the fact that the random amplitudes have a uniform law on a box, therefore, the substantive point is to construct a word with the required submersion property, and this is precisely what the LARC provides.

We now specialize to $\mathrm{M}=\R^d$. It is immediate to note that, whenever the vector fields $V_1(x),\dots,V_m(x)$ already span $T_x\mathrm M$, the map  $\Psi_x^q(\varpi)$ is a submersion at the trivial word $\varpi=0$. Indeed, if a single small parameter turns on the flow of $V_i$, then the first-order variation of the endpoint-map along $\varpi_i$ at $\varpi=0$ is in the direction $V_i$.
However, if they do not generate $T_x\mathrm{M}$, the remaining directions are generally spanned by Lie brackets, which appear as the first non-trivial leading order contribution of loops of flows. 
We develop this idea to understand suitably chosen compositions of flow maps and how their leading order in time expansion connects with the Lie algebra generated by $\cV$ at $x$. 

To that purpose, we introduce commutator loops of flow maps associated to the iterated Lie brackets defined in \eqref{eq:iterated_lie_brackets} by induction.
For simplicity, write $I:=(i_1,\dots,i_k)$ and $I':=(i_2,\dots,i_k)$ with all $i_j\in\lbrace 1,\dots,m\rbrace$, and define the associated loops inductively as
\begin{equation}\label{eq:loops-induction}
    \Gamma_t^{(i)}:=\Phi_t^i,\qquad \Gamma_t^I:=\Phi_{-t}^{i_1}\circ(\Gamma_t^{I'})^{-1}\circ\Phi_t^{i_1}\circ\Gamma_t^{I'}.
\end{equation}
The following lemma connects flow map loops \eqref{eq:loops-induction} with their associated iterated Lie brackets \eqref{eq:iterated_lie_brackets} showing that the non-trivial leading order of the Taylor expansion in time is exactly the iterated bracket.  

\begin{lemma}\label{lem:commutator-loop-taylor}
Let $I$ be a multi-index of length $k$ and let $K\Subset \mathrm M$ be contained in a coordinate chart.
Then there exist $t_0>0$ and a smooth map $\cE_I:(-t_0,t_0)\times K\to\RR^d$ in local coordinates such that, for $|t|<t_0$ and $x\in K$,
\begin{equation}\label{eq:Gamma-expansion}
    \Gamma_t^I(x) = x+t^k V_I(x)+t^{k+1}\cE_I(t,x),
\end{equation}
in the chosen coordinate chart.  Moreover, the expansion holds in $C^1(K)$.
\end{lemma}
The lemma follows from an induction argument on $|I|$ and is presented in Appendix \ref{app:taylor-flows} together with some preliminary standard results on flow maps in $\RR^d$.

We focus now on showing how $n$ linearly independent vector fields of the Lie Algebra generated by $\cV = \lbrace   V_1, ..., V_k \rbrace  $ for some $k\geq 2$ give rise to flow maps obtained as suitable combinations of $\Phi_t^{j}$ whose time-differentials have rank $n$. 
For $i=1,...,k$, we define $h_i:\R\rightarrow \R^k$ by $h_i(a)=a\mathbf{e}_i$, where $\mathbf{e}_i$ is the usual Cartesian basis of $\R^k$, so that we have $\Phi_{h_i(a)}=\Phi_a^{i}$.

The next proposition shows how to construct words $\varpi$ that recover information on the Lie algebra generated by $\cV$.

\begin{proposition}\label{prop:word-construction}
   Let $\cV=\{V_1,\dots,V_m\}$ be a family of smooth vector fields on a smooth $d$-dimensional manifold $\mathrm{M}$ and let $x_0\in\mathrm M$.
   Assume there exist multi-indices $I_1,\dots, I_r$ such that $V_{I_1}(x_0),\dots,V_{I_r}(x_0)$ are linearly independent in $T_{x_0}\mathrm{M}$. Then there exist $q\in \NN$, $\varepsilon_0>0$, and a word $\varpi_*\in ([-\varepsilon_0,\varepsilon_0]^m)^q\setminus \{0\}$ such that 
   \begin{equation}
       \mathrm{rank }D_\varpi\Phi_{\varpi_*}(x_0)\geq r.
   \end{equation}
\end{proposition}
\begin{proof}
    Let $d=\dim\mathrm M$ and let $\lbrace V_{I_j}:j=1,\dots,r\rbrace$ be the set of vector fields generated by repeated nested commutators that are linearly independent at $x=x_0$. We work locally around $x_0\in\mathrm{M}$ and we give a recursive procedure that determines the words $\varpi_{I_j}$ associated to each $V_{I_j}$. We already know that $\varpi_i=h_i(a)$ is the word associated to $V_i$ by means of $\Phi_{\varpi_i}=\Phi_{a}^{i}$. Next, let $V_{i_1},V_{i_2}\in \cV$ and note that the flow map
    \begin{align}
        \Gamma_a^{[V_{i_1},V_{i_2}]} = \Phi_{a}^{i_1}\circ \Phi_a^{i_2}\circ (\Phi_{a}^{i_1})^{-1} \circ (\Phi_a^{i_2})^{-1}
    \end{align}
    gives rise to the commutator $[V_{i_1},V_{i_2}]$ for any two distinct $i_1,i_2\in \{1,\dots,m\}$. Since $\Phi_a^{i}$ is the flow map generated by the autonomous vector field $V_i$, we have that $(\Phi_a^{i})^{-1}=\Phi_{-a}^{i}$ for all $a\in \R$. In particular, we observe that
    \begin{align}
        \Gamma_a^{[V_{i_1},V_{i_2}]}  = \Phi_{\varpi_{i_1,i_2}(a)}, \quad \varpi_{i_1,i_2}(a) := (-h_{i_2}(a),-h_{i_1}(a),h_{i_2}(a), h_{i_1}(a))\in (\R^m)^4
    \end{align}
    and
    \begin{align}
        \left( \Gamma_a^{[V_{i_1},V_{i_2}]} \right)^{-1} = \Phi_a^{i_2}\circ \Phi_a^{i_1}\circ \Phi_{-a}^{i_2}\circ \Phi_{-a}^{i_1} = \Phi_{\varpi^{-1}_{i_1,i_2}(a)}, \hspace{0.8em} \varpi^{-1}_{i_1,i_2}(a) := (h_{i_1}(a), h_{i_2}(a), -h_{i_1}(a), - h_{i_2}(a)) \in (\R^m)^4.
    \end{align}
    By induction, we are able to find words $\varpi_{I_j}$ that give rise to the commutators $V_{I_j}$ as follows: Assume that $V_{I_j}$ is a vector field generated by $\ell_j>1$ commutators, that is, $I_j=(i_1,\dots,i_{\ell_j})$ with $i_s\in \lbrace 1,\dots,m\rbrace$ for all $s=1,\dots,\ell_j$ and $V_{I_j} = [V_{i_1},[V_{i_2}, \dots, [V_{i_{\ell_j-1}},V_{i_{\ell_j}}]\dots]] = [V_{i_1}, V_{\tilde{I}_j}]$ for ${\tilde{I}_j} = (i_2,\dots,i_{\ell_j})$. Let $\varpi_{\tilde{I}_j}(a)$ and $\varpi_{\tilde{I}_j}^{-1}(a)$ be the words for $\Gamma_a^{V_{\tilde{I}_j}}$ and $\left(\Gamma_a^{V_{\tilde{I}_j}}\right)^{-1}$. Then,
    \begin{align}
        \varpi_{I_j}(a) := (\varpi^{-1}_{\tilde{I}_j}(a), -h_{i_1}(a), \varpi_{\tilde{I}_j}(a), h_{i_1}(a)), \quad \varpi_{I_j}^{-1}(a) := (-h_{i_1}(a), \varpi_{\tilde{I}_j}^{-1}(a), h_{i_1}(a), \varpi_{\tilde{I}_j}(a))
    \end{align}
    are such that 
    \begin{align}
        \Gamma_a^{V_{{I}_j}} = \Phi_{\varpi_{{I}_j}(a)}, \quad \left(\Gamma_a^{V_{{I}_j}}\right)^{-1} = \Phi_{\varpi^{-1}_{{I}_j}(a)},
    \end{align}
    respectively, for all $a\in \R$. Once we have determined the words $\varpi_{I_j}$, for $a_1,\dots,a_r\in \R$, we now define 
    \begin{align}
        h(a_1,\dots,a_r) := (\varpi_{I_1}(a_1), \dots ,\varpi_{I_r}(a_r))\in (\R^m)^q,
    \end{align}
    for some $q\in \N$ and consider the composition map $F:\R^r\times \mathrm M\rightarrow\mathrm M$ given by $F(a_1,\dots,a_r;x) := \Phi_{h(a_1,\dots,a_r)}(x)$. Then, 
    \begin{align}\label{eq:DFDphiDh}
        D_{a_1,\dots,a_r}F(a_1^*,\dots,a_r^*;x) = D_{\varpi}\Phi_{\varpi_*}(x) D_{a_1,\dots,a_r}h(a_1^*, \dots, a_r^*),
    \end{align}
    where $\varpi_* = h(a_1^*,\dots,a_r^*) = (\varpi_{I_1}(a_1^*), \dots ,\varpi_{I_r}(a_r^*))\in (\R^m)^q$. It is immediate to see that $D_{a_1,\dots,a_r}h(a_1^*, \dots, a_r^*)$ has full rank. Since the $\lbrace V_{I_j}(x_0) \rbrace_{j=1}^r$ are linearly independent, we now show that $D_{a_1,\dots,a_r}F(a_1^*,\dots,a_r^*;x_0)$ has rank $r$, for some $a_*:=(a_1^*,\dots,a_r^*)$ with $\Vert a_* \Vert$ sufficiently small. To do this, we first note that
    \begin{align}
        F(a_1,\dots,a_r;x) = \Phi_{\varpi_{I_r}(a_r)}\circ \dots \circ \Phi_{\varpi_{I_1}(a_1)}(x),
    \end{align}
    with further
    \begin{align}
        \Phi_{\varpi_{I_j}(a)}(x)=x + a^{\ell_j}V_{I_j}(x) + a^{\ell_j+1}\mathcal{E}_{I_j}(a,x)
    \end{align}
    for all $j=1,\dots,r$, and all $a\in (-\ep,\ep)$ and $x\in U$ due to Lemma~\ref{lem:commutator-loop-taylor}, where we recall that $\ell_j\geq 1$ is the length of the multi-index $I_j$, with the convention that $\ell_j=1$ whenever $V_{I_j}=V_{i_j}$, for some $i_j\in \lbrace 1,\dots,m\rbrace$. Then, 
    \begin{align}
        F(a_1,\dots,a_r;x) &= \left( Id + a_r^{\ell_r}V_{I_r}(\cdot) + a_r^{\ell_r+1}\mathcal{E}_{I_r}(a_r,\cdot) \right) \circ \dots \circ \left( Id + a_1^{\ell_1}V_{I_1}(\cdot) + a_1^{\ell_1+1}\mathcal{E}_{I_1}(a_1,\cdot) \right)(x) \\
        &= x + \sum_{j=1}^r \left( a_{j}^{\ell_j}V_{I_j}(x) + a_j^{\ell_j+1}\mathcal{E}_{I_j}(a_j,x) \right) + \sum_{1\leq j_1<j_2\leq r}a_{j_1}^{\ell_{j_1}}a_{j_2}^{\ell_{j_2}} \mathcal{E}_{j_1,j_2}(a_1,\dots,a_r;x)
    \end{align}
    for some smooth error functions $\mathcal{E}_{j_1,j_2}(a_1,\dots,a_r;x)$ such that $\Vert \partial_{j_i}\mathcal{E}_{j_1,j_2}(a_1,\dots,a_r;x)\Vert_{L^\infty}\lesssim 1$. In particular,
    \begin{align}
        \partial_{a_j}F(a_1,\dots,a_r;x) &= \ell_j a_j^{\ell_j-1}V_{I_j}(x) + a_j^{\ell_j} \widetilde{\mathcal{E}}_j(a_1,\dots,a_r;x) \\
        &\quad +\ell_j a_j^{\ell_j-1}\left( \sum_{j<\ell\leq r} a_\ell^{\ell_\ell}\mathcal{E}_{j,\ell}(a_1,\dots,a_r;x) + \sum_{1\leq\ell < j} a_\ell^{\ell_\ell}\mathcal{E}_{\ell,j}(a_1,\dots,a_r;x) \right) \\
        &= \ell_j a_j^{\ell_j-1} \left( V_{I_j}(x) + \sum_{1\leq \ell \leq r} a_\ell \mathcal{E}'_{j,\ell}(a_1,\dots,a_r;x) \right)
    \end{align}
    with $\Vert {\mathcal{E}}'_{j, \ell}(a_1,\dots,a_r;x) \Vert_{L^\infty}\lesssim 1$ for $\Vert a \Vert\leq 1$. Since $\lbrace V_{I_j}(x_0) \rbrace_{j=1}^r$ are linearly independent, we have that
    \begin{align}\label{eq:linindeplowerbound}
        \left\Vert \sum_{j=1}^r c_j V_{I_j}(x_0) \right\Vert \gtrsim \sum_{j=1}^r |c_j| > 0,
    \end{align}
    for all $(c_1,\dots,c_r)\neq 0$.
    Let $\widetilde{c_j}=\prod_{\ell\neq j}\ell_\ell a_\ell^{\ell_\ell-1}$, as $\lbrace \partial_{a_j}F(a_1,\dots,a_r;x) \rbrace_{j=1}^r$ are linearly independent if and only if $\lbrace \widetilde{c_j}\partial_{a_j}F(a_1,\dots,a_r;x) \rbrace_{j=1}^r$ are, we check the linear independence of 
    \begin{equation}
        V_{I_j}(x_0) + \sum_{1\leq \ell \leq r} a_\ell \mathcal{E}'_{j,\ell}(a_1,\dots,a_r;x_0),\qquad j=1,\dots, r.
    \end{equation}
    These are indeed linearly independent in view of \eqref{eq:linindeplowerbound} if $\sum_{j=1}^r|a_j|$ is small enough.
    Consequently, we see that $\lbrace \partial_{a_j}F(a_1,\dots,a_r;x_0) \rbrace_{j=1}^r$ are linearly independent and thus $D_{a_1,\dots,a_r}F(a_1^*,\dots,a_r^*;x_0)$ has rank $r$.
    In view of \eqref{eq:DFDphiDh}, we conclude that $D_\varpi\Phi_{\varpi_*}(x_0)$ has rank at least $r$.
\end{proof}

\begin{remark}\label{rem:negative-small-times}
The use of both positive and negative times in commutator loops is compatible with the random dynamical systems studied in this manuscript because the deterministic flow map is generated by smooth vector fields whose amplitudes are supported in a neighbourhood of the origin.
Since the parameter $\varpi_*$ in Proposition~\ref{prop:word-construction} can be chosen arbitrarily close to the origin, all the positive and negative times used in the construction belong to the admissible support of the random amplitudes after choosing $\varepsilon>0$ sufficiently small.
\end{remark}

To conclude, we state as a general fact the following corollary connecting the LARC condition with the submersion property used in Lemma~\ref{lemma:blackboxsmallset}.
However, in the subsequent sections, we tailor the analysis to the process under consideration and use directly Proposition~\ref{prop:word-construction}.
\begin{corollary}\label{coroll:LARC-to-submersion}
    Under the assumptions of Proposition~\ref{prop:word-construction}, if the LARC condition holds for $\cV$ at $x_0\in\mathrm{M}$ then there exists $q\in \NN$ and a word $\varpi_*\in (\RR^m)^q$ such that $\Psi_{x_0}^q$ is a submersion at $\varpi_*$.
\end{corollary}
\begin{proof}
If the LARC holds at $x_0$, then there exist bracket words $I_1,\dots,I_d$ such that $V_{I_1}(x_0),\dots,V_{I_d}(x_0) $ form a basis of $T_{x_0}\mathrm M$.
Applying Proposition~\ref{prop:word-construction} with $r=d=\dim\mathrm M$ gives the existence of some $\varpi_*\in (\R^m)^q$ such that $\operatorname{rank}D_\varpi\Psi_{x_0}^q(\varpi_*)=d$, which is precisely the submersion property.
\end{proof}

\section{The one-point process}\label{sec:OPP}
Let $\a,\b\in \Sph$ and $f_{\omega} = f_{\omega_2}^\a \circ f_{\omega_1}^\b$ for $\omega=(\omega_1,\omega_2)\in [-N,N]^2$ denote the one-point process associated to the vector field $u(t,\cdot,\omega;\a,\b)$. In this section we use Theorem \ref{thm:abstractHarris} to show that $f_{\omega}$ is $\mathrm{V}$-uniformly geometrically ergodic. To do so, we first describe the motion induced by the vector field $u(t,\cdot,\omega;\a,\b)$ and we record several formulas that will prove useful for the remainder of the manuscript. We show in Proposition \ref{prop:smallsetOPP} that the process is topologically irreducible and admits an open small set. Since the phase space $\mathsf{X}=\Sph\setminus F$ is not compact due to the presence of the fixed points $F$, in Proposition \ref{prop:lyapdriftopp} we construct a function $\mathrm{V}:\mathsf{X} \rightarrow[1,\infty)$ that satisfies a Lyapunov-Foster drift condition.

\subsection{Fixed points and relevant identities}
We begin by characterizing the set of fixed points $F$ for $f_\omega$.

\begin{proposition}\label{prop:fixed_points}
For non-parallel and non-orthogonal axes $\a,\b\in \Sph$, the set of fixed points for the random dynamics is  $F:=\lbrace  \pm\c\rbrace  $, where $\c:=\frac{\a\times\b}{|\a\times\b|}$.
If $\a$ and $\b$ are mutually orthogonal, we further have $F:=\lbrace  \pm\c,\pm\a,\pm\b\rbrace  $.
\end{proposition}
\begin{proof}
We look for the points where both vector fields vanish simultaneously, deducing that either $\a\cdot\x=0$ or $\x=\pm \a$ and analogously for $\b$. 
If $\a$ and $\b$ are not orthogonal to each other, the only points invariant are the intersection of the two great circles $\a\cdot\x=0$ and $\b\cdot\x=0$. 
Solving for $\x$ gives the first part of the statement.
However, if $\a\perp\b$ the four poles $\pm\a, \pm\b$ lie on the great circle of the other vector field, implying that they are fixed points for the dynamics.
\end{proof}

The next lemma describes the motion generated by the vector field $u^\a$.

\begin{lemma}
    The time-1 flow map $f_\omega^{\a}(\x_0)$ associated to $u_\omega^{\a}$ is a rigid rotation of angle $\theta(\x_0):=\omega\a\cdot\x_0$ around the axis $\a$. 
\end{lemma}

\begin{proof}
    We solve the ODE 
    \begin{equation}
        \dot \x(t)=u_\omega^{\a}(\x(t)), \quad \x(0)=\x_0\in \Sph.
    \end{equation}
    This equation has two invariants, the first one being the norm $|\x(t)|=|\x_0|$ for all $t$, and the second one the projection onto the rotation axis $\a\cdot\x(t)=\a\cdot\x_0$.
    From these and the matrix representation $\a\times \x= A\x$, we easily reduce the ODE to 
    \begin{equation}
        \dot \x(t)=\omega(\a\cdot\x) A\x(t),\quad \x(0)=\x_0,
    \end{equation}
    for which the solution is $\x(t)=\e^{t\omega(\a\cdot\x_0) A}\x_0$, which is a rotation of angle $\theta(\x_0)=\omega(\a\cdot\x_0)$ around the axis $\a$.
\end{proof}

Up to a fixed initial rotation of the coordinates, it is always possible to reduce the study of the dynamics to the computationally simpler case where $\b=(0,0,1)$ and $\a=(0,a_2,a_3)$, with $a_2\neq 0$. This simplification will be adopted from now on throughout the paper and the proofs of the results will be carried out in whichever setting is the most demanding among $\b=(0,0,1)$ and $\a=(0,a_2,a_3)$ or the orthogonal case $\b=(0,0,1)$ and $\a=(0,1,0)$. 

For $\x=(x,y,z)\in \Sph$ and $\omega\in [-N,N]$, we have
\begin{align}\label{eq:flowmapfa}
    f_\omega^\a(\x) = \begin{pmatrix}
        x\cos(\vartheta\omega) + (a_2z - a_3)y\sin(\vartheta\omega) \\
        a_3 x\sin(\vartheta\omega) + y \cos(\vartheta\omega) + a_2\vartheta (1-\cos(\vartheta\omega)) \\
        -a_2 x \sin(\vartheta\omega) + z \cos(\vartheta\omega) + a_3 \vartheta (1-\cos(\vartheta\omega))
    \end{pmatrix}, \quad \vartheta = a_2y + a_3z.
\end{align}
In the special case where $\a = (0,1,0)$ we denote $f^2_\omega= f^{\a}_\omega$ and $f^3_\omega= f^{\b}_\omega$ and we have
\begin{equation}\label{eq:flowmapfE}
     f_{\omega}^2 (\x) = \begin{pmatrix}
        x\cos \left( \omega y\right) + z\sin \left( \omega y\right) \\
        y \\
        -x\sin \left( \omega y\right) + z\cos \left( \omega y\right) \\
        \end{pmatrix}, \quad f_{\omega}^3(\x) = \begin{pmatrix}
        x\cos \left( \omega z\right) -y\sin  \left( \omega z\right)  \\
        x\sin \left( \omega z\right) +y \cos \left( \omega z\right) \\
        z \\
    \end{pmatrix}, 
\end{equation}
see Figure \ref{fig:sphere_vfield}.

\subsection{Open small set and topological irreducibility}
The first step towards Proposition \ref{prop:smallsetOPP} is proving that the family $\lbrace   u^\a(\x), u^\b(\x) \rbrace  $ is bracket-generating for almost all $\x\in \Sph$.

\begin{lemma}\label{lemma:onepointbracket}
    Let $\a, \b$ be two non-parallel unit axes and denote $\c=\a\times\b$. Let $\x_a^\pm=\pm \frac{\a-(\a\cdot\b) \b}{\sqrt{1-(\a\cdot\b)^2}}$, $\x_\b^\pm=\pm \frac{\b-(\a\cdot\b) \a}{\sqrt{1-(\a\cdot\b)^2}}$, $\x_\c^\pm = \pm\frac{\c}{|\c|}$ and $\x\in\Sph \setminus \lbrace   \x_a^\pm, \x_\b^\pm, \x_\c^\pm \rbrace  $. Then the family $\lbrace   u^\a(\x), u^\b(\x) \rbrace  $ is bracket-generating.
\end{lemma}

\begin{proof}
    The two vector fields $u^\a(\x)$ and $u^\b(\x)$ span $T_\x\Sph$ if they are linearly independent. Their  cross product reads
     \begin{equation}
        u_{\a}(\x)\times u_{\b}(\x)=(\a\cdot\x)(\b\cdot\x)[(\a\times\x)\times(\b\times\x)]=(\a\cdot\x)(\b\cdot\x)(\c\cdot \x)\x,
    \end{equation}
    which is non-zero provided $\a\cdot\x\neq 0$, $\b\cdot\x\neq0$ and $\c\cdot\x\neq 0$. We shall now assume at least one of them is zero, otherwise $\lbrace   u^\a(\x), u^\b(\x) \rbrace  $ is already bracket-generating. We next record
    \begin{align}
        [u^\a,u^\b](\x) = (\b\cdot\x)(\c\cdot\x)(\a\times\x) + (\c\cdot\x)(\a\cdot\x)(\b\times\x) + (\a\cdot\x)(\b\cdot\x)(\c\times\x)
    \end{align}
    and we argue by cases.
    \begin{itemize}
    \item[--] If $\a\cdot\x = 0$, then $[u^\a,u^\b](\x)=(\b\cdot\x)(\c\cdot\x)(\a\times\x)$ and 
    \begin{align}
        u^\b(\x) \times [u^\a, u^\b](\x) = -(\b\cdot\x)^2(\c\cdot\x)^2 \x.
    \end{align}
    If $\b\cdot\x=0$, then $\x=\x_\c^\pm$, a contradiction. Likewise, if $\c\cdot\x=0$ and $\x\in \Sph$, since $\a\cdot\c = \b\cdot\c=0$ we have that $\x=\x_\b^\pm$, a contradiction as well.
    \item[--] If $\b\cdot\x=0$ then $[u^\a,u^\b](\x)=(a\cdot\x)(\c\cdot\x)(\b\times\x)$ and $u^\a(\x) \times [u^\a, u^\b](\x) = (\a\cdot\x)^2(\c\cdot\x)^2 \x$. As before, $\a\cdot\x=0$ leads to $\x=\x_\c^\pm$ and $\c\cdot\x=0$ leads to $\x=\x_\a^\pm$, a contradiction as well.
    \item[--] If $\a\cdot\x\neq 0$, $\b\cdot\x\neq0$ and $\c\cdot\x=0$, we note that $[u^\a,u^\b](\x)=(\a\cdot\x)(\b\cdot\x)(\c\times\x)$ and 
    \begin{align}
        u^\a(\x)\times [u^\a,u^\b](\x) = (\a\cdot\x)^2(\b\cdot\x)( (\b\cdot\x) - (\a\cdot\b)(\a\cdot\x)) \x,
    \end{align}
    which is never-vanishing for all $\x\neq \pm \a$, since $0\leq|\a\cdot\b|<1$. Moreover,
    \begin{align}
        u^\b(\x)\times [u^\a,u^\b](\x) = (\b\cdot\x)^2(\a\cdot\x)( (\a\cdot\x) - (\a\cdot\b)(\b\cdot\x)) \x,
    \end{align}
    which is again never-vanishing for all $\x\neq \pm \b$. Hence, either $\lbrace   u^\a(\x), [u^\a,u^\b](\x)\rbrace   $ or $\lbrace   u^\b(\x), [u^\a,u^\b](\x)\rbrace   $ have rank 2.
    \end{itemize}
    Therefore, $\lbrace   u^\a(\x), u^\b(\x) \rbrace   $ is bracket-generating away from $\x_\a^\pm$, $\x_\b^\pm$ and $\x_\c^\pm$.
\end{proof}

\begin{remark}
    We note here that when $\a\cdot\b=0$, the axes of rotation are orthogonal, then $\x_\a^\pm = \pm \a$ and $\x_\b^\pm = \pm \b$, we recover the extra fixed points of the dynamical system.
\end{remark}

\begin{proposition}\label{prop:smallsetOPP}
The one-point process is strongly aperiodic, topologically irreducible and admits an open small set.
\end{proposition}

\begin{proof}
    The strong aperiodicity is a consequence of noting that $f_\omega(p)=p$ for $\omega=(0,0)$ and all $p\in \mathsf{X}$ and the continuity of $f_\omega$ with respect to $\omega\in [-N,N]^2$. The existence of open small sets is a consequence of Lemma \ref{lemma:onepointbracket} after choosing $\omega=(0,0)$ for some $\x\in\Sph \setminus \lbrace   \x_a^\pm, \x_\b^\pm, \x_\c^\pm \rbrace  $ with $\a\cdot\x\neq0$, $\b\cdot\x\neq 0$ and $\c\cdot\x \neq 0$. If $\a\cdot \b =0$, we have that $\mathrm{M}=\mathsf{X}$ so that from Lemma \ref{lemma:onepointbracket} and Theorem \ref{thm:RashevskyChow} we conclude the topological irreducibility of the random dynamical system. If $\a\cdot\b\neq 0$, the dynamical system is topologically irreducible in $\mathrm{M}=\Sph \setminus \lbrace   \x_a^\pm, \x_\b^\pm, \x_\c^\pm \rbrace  $. To show the topological irreducibility on $\mathsf{X}$, we note that $u^\a(\x_\a^\pm)\neq 0$ and $u^\b(\x_\b^\pm)\neq 0$, and thus there exists $\omega_\a, \omega_\b\in [-N,N]$ such that $f_{(\omega_\b,0)}(\x_\b^\pm)\in \mathrm{M}$ and $f_{(0,\omega_\a)}(\x_\a^\pm)\in \mathrm{M}$.
\end{proof}

\subsection{The Lyapunov-drift function}
We now construct the Lyapunov-drift function for the one-point process associated to the rotation axis $\a=(0,1,0)$ and $\b=(0,0,1)$. We comment on the case of general axes of rotation later. Let $p=(x,y,z)$, $p_1 = f_{\omega_1}^3(p)$ and $p_2=f_{\omega_2}^2(p_1)$. Due to \eqref{eq:flowmapfE} we recall that 
\begin{align}
    \begin{pmatrix}
        x_1 \\ y_1 \\ z_1
    \end{pmatrix} = \begin{pmatrix}
        x\cos(\omega_1 z) - y \sin(\omega_1 z) \\
        x\sin(\omega_1 z) + y \cos(\omega_1 z) \\
        z
    \end{pmatrix}, \quad \begin{pmatrix}
        x_2 \\ y_2 \\ z_2 
    \end{pmatrix} = \begin{pmatrix}
        x_1\cos(\omega_2 y_1) + z_1 \sin(\omega_2 y_1) \\
        y_1 \\
        -x_1\sin(\omega_2 y_1) + z_1 \cos(\omega_2 y_1) \\
    \end{pmatrix}.
\end{align}
We shall obtain a drift function as a combination of local drifts tailored for each family of fixed points 
\begin{align}
    F_\sfx = \lbrace   \pm(1,0,0) \rbrace,\quad F_\mathsf{y}= \lbrace   \pm(0,1,0) \rbrace  ,\quad F_\sfz = \lbrace   \pm (0,0,1) \rbrace  , 
\end{align}
since the dynamics near each  fixed point is qualitatively different.

\subsubsection{The $F_\sfx$ fixed points} We begin with finding a drift function for the dynamics near the fixed points $(1,0,0)$ and $(-1,0,0)$, which are hyperbolic stagnation points of $f_\omega$. For $p=(x,y,z)\in \Sph$ we define $r_\sfx(p)=\max \lbrace |y|, |z| \rbrace$ and for $R>0$ we set
\begin{align}
    \mathsf{Q}_\sfx(R)= \lbrace   p=(x,y,z)\in \Sph : 0 < r_\sfx(p) < R \rbrace 
\end{align}
and, for all $p\in \mathsf{Q}_\sfx(R)$, we define $V_\mathsf{x}(p) := r_\sfx(p)^{-\alpha}$, for some $\alpha\in (0,1)$ to be determined. Before showing that $V_\mathsf{x}$ constitutes a valid drift function in a neighbourhood of $F_\sfx$, we first prove a technical intermediate result.

\begin{lemma}\label{lemma:prob_bound}
    Let $\ep>0$ and $N>0$. Let $p_1=(x_1,y_1,z_1)= f_{\omega_1}^3(p)$, for $p=(x,y,z)\in \mathsf{Q}_\sfx(R)$, for some $R>0$ small. Then,
    \begin{equation}
        \PP(\lbrace  \omega_1\in[-N,N] : |y_1|\leq \varepsilon\rbrace  )\leq \frac{2\varepsilon}{N|z|}.
    \end{equation}
\end{lemma}

\begin{proof}
    We start by noticing that the probability is equal to the length $l$ of the interval of values of $\omega$'s for which the inequality holds true, hence at first we can bound the probability with $l/2N$.
    We proceed now to give an upper bound for $l$.
    \begin{align}
        |y_1|&=|x\sin(\omega_1 z) + y \cos(\omega_1 z)|\\
        &=\sqrt{x^2+y^2}\l|\sin\l(\omega_1 z+\arctan\l(\frac{y}{x}\r)\r)\r| 
    \end{align}
From this, we can deduce 
\begin{align}
    |y_1|\leq \varepsilon &\iff \l|\sin\l(\omega_1 z+\arctan\l(\frac{y}{x}\r)\r)\r|\leq \frac{\varepsilon}{\sqrt{x^2+y^2}}.
\end{align}
Now, for $y\geq 0$ we have $\arctan\left( \frac{y}{x} \right)\in [0,\pi]$ and for $NR\leq \frac{\pi}{2}$ we see that
\begin{align}
    -\frac{\pi}{2}\leq \omega_1 z + \arctan\left( \frac{y}{x} \right) \leq \frac{3\pi}{2}.
\end{align}
Moreover, $\frac{\ep}{\sqrt{x^2 + y^2}} \leq 2\ep$ for $R>0$ small enough. Hence, $|y_1|\leq \ep$ if and only if either
\begin{align}
    \left| \omega_1 z+\arctan\l(\frac{y}{x}\r) \right| \leq \arcsin\left( \frac{\varepsilon}{\sqrt{x^2+y^2}} \right) 
\end{align}
or
\begin{align}
    \left| \omega_1 z+\arctan\l(\frac{y}{x}\r)-\pi \right| \leq \arcsin\left( \frac{\varepsilon}{\sqrt{x^2+y^2}} \right). 
\end{align}
As a result,
\begin{align}
    \left| \left\lbrace   \omega_1\in[-N,N] : \l|\sin\l(\omega_1 z+\arctan\l(\frac{y}{x}\r)\r)\r|\leq \frac{\varepsilon}{\sqrt{x^2+y^2}} \right\rbrace   \right| &\leq \frac{2}{|z|}\arcsin\l(\frac{\varepsilon}{\sqrt{x^2+y^2}}\r) \leq \frac{4\ep}{|z|}
\end{align}
Here we have used that the size of the interval is independent of shifts, and that, for small $\varepsilon$ it holds $\arcsin(\varepsilon)\leq 2\varepsilon$.
\end{proof}

We are now in position to prove a local drift for $V_\sfx$ following the ideas of \cite{C2023}*{Lemma 13}.

\begin{lemma}\label{lemma:localxlyap}
There exists $s_0>0$ small and $N>1$ large such that for all $p\in \mathsf{Q}_\sfx (s_0)$ and $\omega=(\omega_1,\omega_2)\in [-N,N]^2$ there holds
\begin{align}
    \mathbb{E}\left[ V_\sfx (f_\omega(p)) \right] \leq \gamma V_\sfx (p),
\end{align}
for $\gamma=\frac{95}{100}<1$.
\end{lemma}

\begin{proof}
Let $N>1$ and $s_0>0$ to be determined along the proof. First of all, for all $\delta_1>0$ small there exists $c_1>0$ sufficiently small such that for $p\in \mathsf{Q}_\sfx (s_0)$ and $\omega=(\omega_1,\omega_2)\in [-N,N]^2$ we have that $p_1=f_{\omega_1}^3(p)$ satisfies
\begin{align}\label{eq:smallp1}
    (1-|x_1|) + |y_1| + |z_1| \leq \delta_1
\end{align}
provided that $Ns_0\leq c_1$. Similarly, for all $\delta_2>0$ there exists $c_2>0$ sufficiently small such that 
for all $p\in \mathsf{Q}_\sfx (s_0)$ and all $\omega=(\omega_1,\omega_2)\in [-N,N]^2$ we have that $p_2=f_{\omega_2}^2(p_1)$  satisfies
\begin{align}\label{eq:smallp2}
    (1-|x_2|) + |y_2| + |z_2| \leq \delta_2
\end{align}
if $N^2 s_0\leq c_2$. In what follows, we shall assume that \eqref{eq:smallp1} and \eqref{eq:smallp2} hold for $\delta_1,\delta_2>0$ sufficiently small and that $c_1,c_2>0$ are sufficiently small as well. We observe that $y_2=y_1$ and we distinguish two cases according to the value of $V_\sfx(p)$ 

\diampar{Case $0\leq |y|\leq |z|$} Hence, $V_\sfx (p) = |z|^{-\alpha}$. 
\begin{enumerate}
\item Let $E_1 = \lbrace   \omega_1\in [-N,N] : |y_1| \leq \frac{|z|}{4N} \rbrace  $. Since
\begin{align}
        y_1 = x\sin(\omega_1 z) + y \cos(\omega_1 z)
\end{align}
we easily see that $\mathbb{P}(E_1) \leq \frac{1}{2N^2}$. Moreover, for $(\omega_1,\omega_2)\in E_1\times [-N,N]$ we now have $|\sin (\omega_2 y_1)|\leq \frac{|z|}{4}$ so that 
\begin{align}
    |z_2| \geq |z||\cos(\omega_2 y_1)| - |x_1||\sin(\omega_2 y_1)| \geq \frac{|z|}{2}    
\end{align}
and thus $V_\sfx (p_2) \leq 2^\alpha V_\sfx (p)$. Here we used that $|\cos(\omega_2 y_1)|\geq \sqrt{1-(\tfrac{|z|}{4})^2} \geq \frac34$ for $|z|\leq 1$.

\item Let $E_2= \lbrace   \omega_1\in [-N,N] : \frac{|z|}{4N} < |y_1| \leq 2|z| \rbrace   $. Now, we have $\mathbb{P}(E_2) \leq \frac{4}{N}$ and we further observe that $V_\sfx (p_2)\leq (4N)^\alpha V_\sfx (p)$ for all $(\omega_1,\omega_2)\in E_2\times [-N,N]$.

\item Let $E_3 = \lbrace   \omega_1\in [-N,N] : 2|z| < |y_1| \rbrace  $. We note that for all $(\omega_1,\omega_2)\in E_3\times [-N,N]$ we have $V_\sfx (p_2) \leq 2^{-\alpha} V_\sfx (p)$.
\end{enumerate}
Therefore, we have that
\begin{equation}\label{eq:firstEbound}
\begin{split}
    \mathbb E \left[ V_\sfx (f_{\omega}(p))\right] &= \mathbb E \left[ V_\sfx (f_{\omega}(p)) | E_1 \right]\mathbb{P}(E_1) + \mathbb E \left[ V_\sfx (f_{\omega}(p)) | E_2 \right]\mathbb{P}(E_2) + \mathbb E \left[ V_\sfx (f_{\omega}(p)) | E_3 \right]\mathbb{P}(E_3) \\
    &\leq \left( \frac{2^{\alpha-1}}{N^2} + \frac{4^{1+\alpha}}{N^{1-\alpha}} + 2^{-\alpha} \right) V_\sfx (p).
\end{split}
\end{equation}

\diampar{Case $0\leq |z| \leq |y|$} We now have $V_\sfx (p) = |y|^{-\alpha}$ and we consider the next four settings.
\begin{enumerate}
    \item Let $E_1 = \lbrace   \omega_1\in [-N,N] : |y_1| \leq \frac{|z|}{4N} \rbrace  $. As before, we have $\mathbb{P}(E_1) \leq \frac{1}{2N^2}$ and $|z_2|\geq \frac{|z|}{2}$ as well. Additionally, we note that for $E_1$ to be non-empty we should have
    \begin{align}
        -\frac{|z|}{4N} -x\sin(\omega_1 z) \leq y \cos (\omega_1 z) \leq  \frac{|z|}{4N} -x\sin(\omega_1 z).
    \end{align}
    In particular, we observe that $|y\cos(\omega_1 z)| \leq \frac{|z|}{4N} +|\sin(\omega_1 z)|\leq 2N|z|$ and since $|\cos(\omega_1 z)|\geq \frac34$ because $|z|\leq s_0$ and $Ns_0\leq c_1$ is small enough, we further deduce that $\frac{3|y|}{16N}\leq \frac{|z|}{2} \leq |z_2|$. Therefore, $V_\sfx (p_2)\leq (\tfrac{16N}{3})^{\alpha} V_\sfx (p)$ for $(\omega_1,\omega_2)\in E_1 \times [-N,N]$.
    \item We now define $E_2= \lbrace   \omega_1\in [-N,N] : \frac{|z|}{4N} < |y_1| \leq \frac{|y|}{N^\frac12} \rbrace   $, so that $\mathbb{P}(E_2)\leq \frac{2}{N^\frac32}\frac{|y|}{|z|}$. For $E_2$ to be non-empty, we now need
    \begin{align}
        -\frac{|y|}{N^\frac12} -x\sin(\omega_1 z) \leq y \cos (\omega_1 z) \leq  \frac{|y|}{N^\frac12} -x\sin(\omega_1 z)
    \end{align}
    which gives $|y\cos(\omega_1 z)|\leq N|z| + \frac{|y|}{N^\frac12}$. Therefore, we have that $\left( \frac{3}{4}-N^{-\frac12}\right) |y| \leq N|z|$ and thus $|y| \leq 2N |z|$ for $N$ large enough. Hence, $\mathbb{P}(E_2)\leq \frac{4}{N^\frac12}$, $|y_2| = |y_1| \geq \frac{|z|}{4N} \geq \frac{|y|}{8N^2}$ and thus $V_\sfx (p_2) \leq (8N^2)^{\alpha}V_\sfx (p)$, for all $(\omega_1,\omega_2)\in E_2\times [-N,N]$.
    \item We next set $E_3 = \lbrace   \omega_1,\omega_2\in [-N,N] : |y_1| > N^{-\frac12}|y| \text{ and }  |\sin(\omega_2N^{-\frac12}|y|)|\leq 4|y| \rbrace  $. Then, we must have $|\omega_2|\leq 8N^\frac12$ so that $\mathbb{P}(E_3) \leq \frac{8}{N^{\frac12}}$ and $V_\sfx (p_2) \leq N^{\alpha/2}V_\sfx (p)$.
    \item Finally, for $E_4 = (E_1 \cup E_2 \cup E_3)^c$ we have that $|\sin(\omega_2 N^{-\frac12}|y|)| > 4|y|$ and $|y_1| > N^{-\frac12}|y|$. In particular, we deduce that 
    \begin{align}
        |z_2| \geq |x_1| |\sin(\omega_2 y_1)| - |z| \geq \frac34 |\sin(\omega_2N^{-\frac12}|y|)| - |y| \geq 2|y|
    \end{align} because $|x_1|\geq \frac34$, $|\sin(z)|$ is increasing in $z\in(0,\pi/2)$ and $|z|\leq |y|$. Thus, $V_\sfx (p_2) \leq 2^{-\alpha} V_\sfx(p)$ in $E_4$.
\end{enumerate}
Now, there holds
\begin{equation}\label{eq:secondEbound}
\begin{split}
    \mathbb E \left[ V_\sfx(f_{\omega}(p))\right] &= \mathbb E \left[ V_\sfx(f_{\omega}(p)) | E_1 \right]\mathbb{P}(E_1) + \mathbb E \left[ V_\sfx(f_{\omega}(p)) | E_2 \right]\mathbb{P}(E_2) \\
    &\quad+ \mathbb E \left[ V_\sfx(f_{\omega}(p)) | E_3 \right]\mathbb{P}(E_3) + \mathbb E \left[ V_\sfx(f_{\omega}(p)) | E_4 \right]\mathbb{P}(E_4)\\
    &\leq \left( \frac{16^\alpha}{2\cdot 3^\alpha}\frac{1}{N^{2-\alpha}} + \frac{4\cdot 8^\alpha}{N^{\frac12-2\alpha}} + \frac{8}{N^{\frac12-\alpha/2}} + 2^{-\alpha} \right) V_\sfx(p)
\end{split}
\end{equation}
For $\alpha=\frac18$, $N>1$ sufficiently large and $s_0 = \frac{c_0}{2N_0^2}$ with $c_0=\min(c_1,c_2)$ small enough we see from \eqref{eq:firstEbound} and \eqref{eq:secondEbound} that $\mathbb E \left[ V_\sfx(f_{\omega}(p))\right] \leq \frac{95}{100}V_\sfx(p)$.
\end{proof}

\subsubsection{The $F_\sfz$ fixed points}
Set here $r_\sfz(p)=\sqrt{x^2 + y^2}$ and
\begin{align}
    \mathsf{Q}_\sfz(R)= \lbrace   p\in \Sph : 0 < r_\sfz(p) < R \rbrace  .
\end{align}
for $R>0$. For $p\in \mathsf{Q}_\sfz(R)$ we define $V_\sfz(p)= r_\sfz^{-\alpha}(p)$. We observe that $V_\sfz(f_{\omega_1}^3(p))=V_\sfz(p)$, for all $p\in \Sph$ and $\omega_1\in [-N,N]$, since $f_{\omega_1}^3$ is a rotation along the $z$-axis and it preserves the distance to the $z$-axis.

\begin{lemma}\label{lemma:localzlyap}
There exists $s_0>0$ small and $N>1$ large such that for all $p\in \mathsf{Q}_{\sfz}(s_0)$ and $\omega=(\omega_1,\omega_2)\in [-N,N]^2$ there holds
\begin{align}
    \mathbb{E}\left[ V_\sfz(f_\omega(p)) \right] \leq \gamma V_\sfz(p),
\end{align}
for $\gamma=\frac{9}{10}<1$.
\end{lemma}

\begin{proof}
To ease notation, let $r=r_\sfz(p)$. We assume that $y_1 > 0$ and $s_0< \frac14$ small enough such that $z\geq 1-s_0$ for all $p\in \mathsf{Q}_\sfz(s_0)$. Let $N>1$ and $M>1$, to be determined along the proof. Near the north and south poles $(0,0,1)$ and $(0,0,-1)$, the first motion $f^3_{\omega_1}$ rotates the initial point $p$ along the $z$-axis, thus preserving $V_{\sfz}$. The second motion $f^2_{\omega_2}$ is shearing in the $x$-direction while preserving the $y$-component. We argue near the north pole $(0,0,1)$ and according to the position of the midpoint $p_1=(x_1,y_1,z_1) = f_{\omega_1}^3(p)$.

\begin{enumerate}
    \item For the set $E_1=\lbrace   \omega_1\in[-N,N] : 0\leq N^{-\frac12} |x_1| \leq y_1 \rbrace  $ we define the following events.
    \begin{enumerate}[label=(\theenumi.\arabic*)]
        \item Let $E_{1,1}=\lbrace   \omega_2\in[-N,N] : |\omega_2| > 2\frac{M+1}{1-\delta_0}N^\frac12\rbrace  $. For $\omega_2\in E_{1,1}$, we have that $|x_2| \geq \frac12(1-\delta_0)|\omega_2| y_1 - |x_1| \geq \frac12(1-\delta_0)|\omega_2| y_1 - N^\frac12 y_1 \geq  M N^\frac12 y_1$, from which we see that
        \begin{align}
            r_2 \geq y_1\sqrt{1+M^2N} \geq \sqrt{\frac{1+M^2N}{1+N}}r_1
        \end{align}
        and thus $V_\sfz(f_\omega(p)) \leq \left( \frac{1+N}{1+M^2N}\right)^{\frac{\alpha}{2}} V_\sfz(p)$. 
        \item Set next $E_{1,2}= (E_{1,1})^c$. We observe that $\mathbb{P}(E_{1,2}) = 2\frac{M+1}{1-\delta_0} N^{-\frac12}$ and for $\omega_2\in E_{1,2}$, we further observe that $r_2\geq y_1 \geq \frac{r_1}{\sqrt{1+N}}$ so that now $V_\sfz(f_\omega(p)) \leq (1+N)^{\frac{\alpha}{2}}V_\sfz(p)$.
    \end{enumerate}
    \item For the set $E_2 = \lbrace   \omega_1\in [-N,N] : \frac{|x_1|}{2N} \leq y_1 \leq \frac{|x_1|}{N^\frac12} \rbrace  $, we have that $r_2\geq y_1\geq \frac{r_1}{\sqrt{1+4N^2}}$ and thus $V_\sfz(f_\omega(p)) \leq (1+4N^2)^{\frac{\alpha}{2}}V_\sfz(p)$. Moreover, $\mathbb{P}(E_2) \lesssim N^{-\frac12}$ since $E_2$ is a subset of  $\lbrace   \omega_1\in[-N,N] : 0\leq \beta_1 \leq N^{-\frac12} \rbrace   $, where $\beta_1 = \arctan(\frac{y_1}{|x_1|})$. In turn, $\beta_1 = \beta + \omega_1 z$ and since $\mathbb{P}(\lbrace   \omega_1\in [0,2\pi] : 0< \beta + \omega_1 z < N^{-\frac12} \rbrace   ) \lesssim N^{-\frac12}$ we observe that we may increase the amplitude interval from $[0, 2\pi]$ to $[-2\pi k, 2\pi k]$ and yet, due to periodicity and the uniform law, $\mathbb{P}(E_2)\leq C_1 N^{-\frac12}$, for some uniformly bounded $C_1>0$. We next set
    \begin{enumerate}[label=(\theenumi.\arabic*)]
        \item $E_{2,1} = \lbrace   \omega_2\in [-N,N] : |x_2| \leq \frac{y_1}{2} \rbrace  $. Arguing as in Lemma \ref{lemma:prob_bound} we can show that $\mathbb{P}(E_{2,1}) \leq N^{-1}$. Moreover, $V_\sfz (f_\omega(p)) \leq  (1+4N^2)^{\frac{\alpha}{2}}V_\sfz(p)$.
        \item $E_{2,2} = \lbrace   \omega_2\in [-N,N] : |x_2| \geq \frac{y_1}{2} \rbrace  $. Here, we immediately have $V_\sfz(f_\omega(p)) \leq 4^\alpha V_\sfz(p)$.
    \end{enumerate}
    
    \item For the set $E_3 = \lbrace   \omega_1\in [-N,N] : 0 < y_1 < \frac{|x_1|}{2N} \rbrace   $ we have
    \begin{align}
        |x_2|\geq \frac{3}{4}|x_1| - |\omega_2| y_1 \geq \frac{3}{4}|x_1| - \frac{|\omega_2|}{2N}|x_1|\geq \frac{|x_1|}{4}
    \end{align}
    from which we deduce that $r_2 \geq \frac{r_1}{4}$ and thus $V_\sfz(f_\omega(p)) \leq 4^\alpha V_\sfz(p)$. Furthermore, we have $\mathbb{P}(E_3) \leq C_2 N^{-1}$ for some $C_2>0$ uniformly bounded.
\end{enumerate}

Arguing identically for $y_1<0$ using the symmetry under $\pi$-rotation along the $z$-axis of the flow maps, we find that
\begin{align}
    \mathbb E \left[ V_\sfz(f_{\omega}(p))\right] &= 2\mathbb E \left[ V_\sfz(f_{\omega}(p)) | E_{1,1} \right]\mathbb{P}(E_{1,1}) + 2\mathbb E \left[ V_\sfz(f_{\omega}(p)) | E_{1,2} \right]\mathbb{P}(E_{1,2}) \\
   &\quad + 2 \mathbb E \left[ V_\sfz(f_{\omega}(p)) | E_{2,1} \right]\mathbb{P}(E_{2,1}) + 2\mathbb E \left[ V_\sfz(f_{\omega}(p)) | E_{2,2} \right]\mathbb{P}(E_{2,2}) \\
   &\quad + 2 \mathbb E \left[ V_\sfz(f_{\omega}(p)) | E_3 \right]\mathbb{P}(E_3)
\end{align}
and gathering all the estimates above, we reach 
\begin{align}
     \frac{\mathbb E \left[ V(f_{\omega}(p))\right]}{2V(p)} &\leq  \left( \frac{1+N}{1+M^2N} \right)^\frac{\alpha}{2} + 2\frac{M+1}{1-\delta_0}\frac{(1+N)^\frac{\alpha}{2}}{N^\frac12} + C_1\frac{(1+4N^2)^\frac{\alpha}{2}}{N^\frac12} + C_2 \frac{4^\alpha}{N^\frac12} + C_3 \frac{4^\alpha}{N}
\end{align}
For $\alpha = \frac18\in (0,\frac12)$ and $N\geq 1$ we see that $\frac{1+N}{1+M^2N}\leq \frac{2}{M^2}$ and thus
\begin{align}
    2 \left( \frac{1+N}{1+M^2N} \right)^\frac{\alpha}{2} \leq 2 \frac{2^\frac{1}{16}}{M^\frac18} = \frac45
\end{align}
for $M= (5/2)^8 \sqrt{2}$. Next, we choose $N>1$ large enough such that
\begin{align}
    2\frac{M+1}{1-\delta_0}\frac{(1+N)^\frac{\alpha}{2}}{N^\frac12} + C_1\frac{(1+4N^2)^\frac{\alpha}{2}}{N^\frac12} + C_2 \frac{4^\alpha}{N^\frac12} + C_3 \frac{4^\alpha}{N} \leq \frac{1}{20}.
\end{align}
Therefore, we have
\begin{align}
\mathbb E \left[ V_\sfz(f_{\omega}(p))\right] \leq \frac{9}{10} V_\sfz(p).
\end{align}
\end{proof}

\begin{figure}[!h]
\centering
\includegraphics[width=0.7\linewidth]{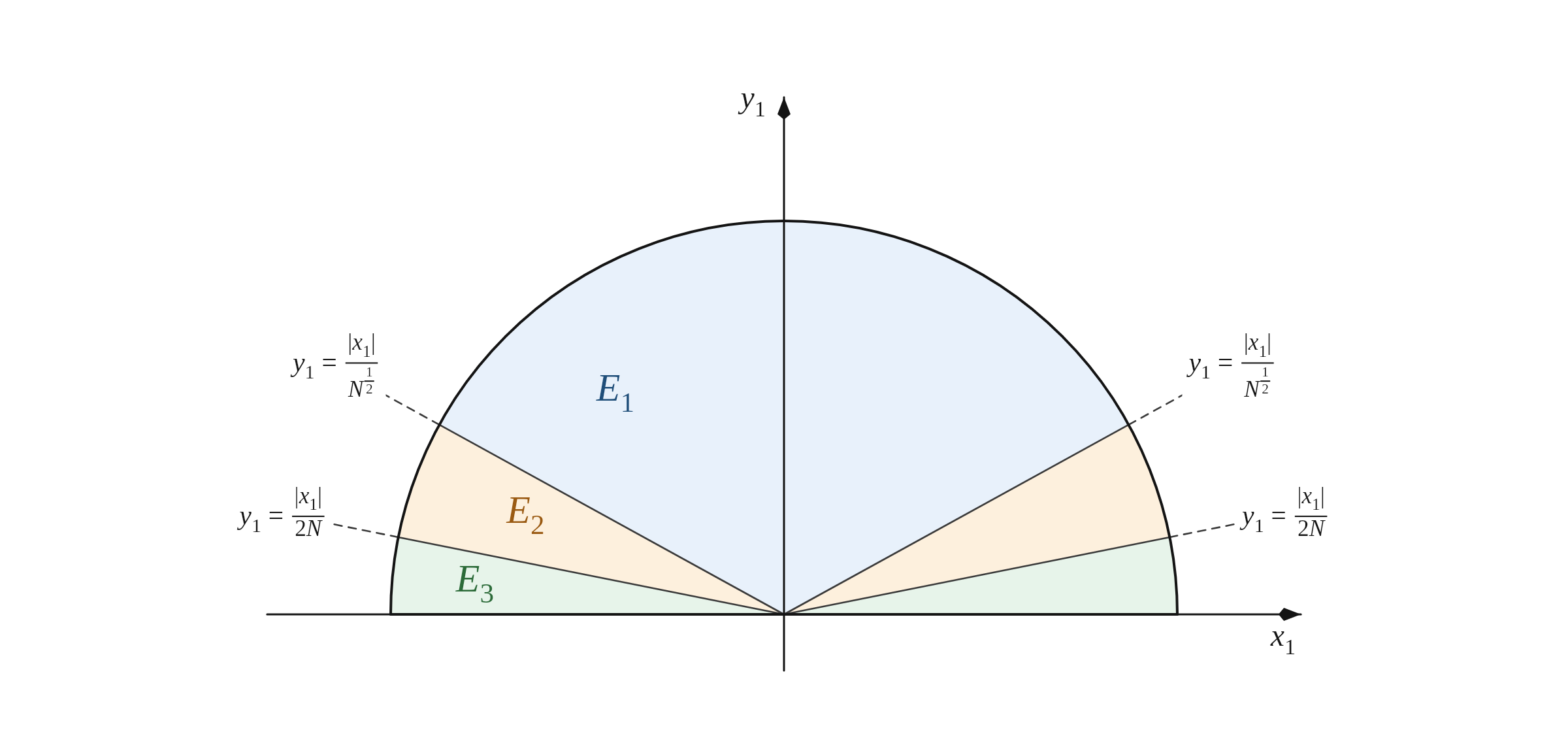}
\caption{ \small Division of the $(x_1,y_1)$ upper half-plane in three regions $E_1$, $E_2$, and $E_3$ for the $F_\sfz$ fixed points.}
\label{fig:sectors_Fz}
\end{figure}

\subsubsection{The $F_\mathsf{y}$ fixed points}

For $p=(x,y,z)\in \Sph$ we now define $r_\mathsf{y}(p) = \sqrt{x^2 +z^2}$ and 
\begin{align}
    \mathsf{Q}_\mathsf{y}(R) = \lbrace   p\in \Sph : 0< r_\mathsf{y}(p) < R \rbrace 
\end{align}
for $R>0$. For $p\in \mathsf{Q}_\mathsf{y}(R)$ we similarly set $\widetilde{V}_\mathsf{y}(p)=r_\mathsf{y}^{-\alpha}(p)$ and for $N>0$ we further set $ {V}_\mathsf{y}(p) = \widetilde{V}_\mathsf{y}(f_N^3(p))$. 

\begin{lemma}\label{lemma:localylyap}
There exists $s_0>0$ small and $N>1$ large such that for all $p\in \mathsf{Q}_{\sfy}(s_0)$ and $\omega=(\omega_1,\omega_2)\in [-N,N]^2$ there holds
\begin{align}
    \mathbb{E}\left[  {V}_\mathsf{y}(f_\omega(p)) \right] \leq \gamma  {V}_\mathsf{y}(p),
\end{align}
for $\gamma=\frac{9}{10}<1$.
\end{lemma}

\begin{proof}
    Let $p\in \Sph$, $\overline{p} = f_N^3(p)$ and $q_{\omega_1}=f_{\omega_1}^3(p)$. Hence, $p= f_{-N}^3(\overline{p})$ and $q_{\omega_1}= f_{\omega_1-N}^3(\overline{p})$. Moreover, $ {V}_\mathsf{y}(p) = \widetilde{V}_\mathsf{y}(\overline{p})$ and $ {V}_\mathsf{y}(f_\omega(p)) = \widetilde{V}_\mathsf{y}(f_N^3\circ f_{\omega_2}^2(q_{\omega_1}))$. Arguing as in the proof of Lemma \ref{lemma:localzlyap}, we see that
    \begin{align}
        \mathbb{E}_{\omega_2}\left[  {V}_\mathsf{y}(f_{\omega}(p)) \right] &=  \mathbb{E}_{\omega_2}\left[  \widetilde{V}_\mathsf{y}(f_N^3\circ f_{\omega_2}^2(q_{\omega_1})) \right] \\
        &\leq 2\left( \left( \frac{1+N}{1+M^2N} \right)^\frac{\alpha}{2} + C_1\frac{(1+4N^2)^\frac{\alpha}{2}}{N^\frac12} + C_2 \frac{4^\alpha}{N}  \right) \widetilde{V}_\mathsf{y}(q_{\omega_1}).
    \end{align}
    Next, for $\widetilde\omega_1 = N-\omega_1\in[0,2N]$ we observe that $q_{\omega_1} = \mathsf{R}_\sfz( f_{\widetilde\omega_1}^3(\mathsf{R}_\sfz \overline{p}))$, where $\mathsf{R}_\sfz(p) = (x,y,-z)$, for $p=(x,y,z)$. Since $\widetilde{V}_\mathsf{y}(\mathsf{R}_\sfz(p)) = \widetilde V_\mathsf{y}(p)$ for all $p\in \Sph$, we deduce from the arguments of Lemma \ref{lemma:localzlyap} concerning the shear motion that
    \begin{align}
        \mathbb{E}_{\omega_1}\left[ \widetilde{V}_\mathsf{y}(q_{\omega_1}) \right] &= \mathbb{E}_{\widetilde\omega_1}\left[ \widetilde{V}_\mathsf{y} \circ \mathsf{R}_\sfz \circ f_{\widetilde\omega_1}^3\circ \mathsf{R}_\sfz (\overline{p}) \right] \\
        &= \mathbb{E}_{\widetilde\omega_1}\left[ \widetilde{V}_\mathsf{y} \circ f_{\widetilde\omega_1}^3\circ \mathsf{R}_\sfz (\overline{p}) \right] \\
        &\leq 2\left( \left(\frac{1+N}{1+M^2N} \right)^{\frac{\alpha}{2}} + 2\frac{M+1}{1-\delta_0}\frac{(1+N)^{\frac{\alpha}{2}}}{N^{\frac12}} + \frac{(1+4N^2)^{\frac{\alpha}{2}}}{N} + 4^\alpha \right) \widetilde{V}_\mathsf{y}(\mathsf{R}_\sfz(\overline{p})) \\
        &\leq 4^{\alpha+1}\widetilde{V}_\mathsf{y}(\overline{p}) \\
        &= 4^{\alpha+1}  V_\mathsf{y}(p),
    \end{align}
    for $M$ and $N$ sufficiently large and $\alpha\in(0,1)$. Note that here we cannot appeal to the probability of being at $E_1$, $E_2$ or $E_3$ since $\mathsf{R}_\sfz (\overline{p})$ is given and does not experience rotating motion, it only undergoes shear motion. Therefore,
    \begin{align}
        \mathbb{E}_{\omega_1,\omega_2}\left[  V_\mathsf{y}(f_{\omega}(p)) \right] &\leq 2 \left(  \left( \frac{1+N}{1+M^2N} \right)^\frac{\alpha}{2} + C_1\frac{(1+4N^2)^\frac{\alpha}{2}}{N^\frac12} + C_2 \frac{4^\alpha}{N}  \right) 4^{\alpha+1} V_\mathsf{y}(p) \\
        &\leq \frac{9}{10} V_\mathsf{y}(p),
    \end{align}
    for $\alpha=1/8$, $M> 4\sqrt{2}10^8$ and $N>1$ large enough. 
\end{proof}

\subsubsection{Construction of the drift}
We next assemble the partial drifts we defined for $F_\sfx$, $F_\sfz$ and $F_\mathsf{y}$ and construct the Lyapunov-drift function for the one-point process. Let $s_0>0$ as in Lemmas \ref{lemma:localxlyap} - \ref{lemma:localylyap} and  $\chi:[0,\infty)\rightarrow[0,1]$ be a smooth non-negative bump function such that
\begin{align}
    \chi(r) = \begin{cases}
        1, & r\leq s_0, \\
        0, & r \geq 2s_0.
    \end{cases}
\end{align}
For $p=(x,y,z)\in \Sph$, recall $r_\sfx(p) = \max \lbrace |y|, |z| \rbrace$, $r_\mathsf{y}(p) = \sqrt{x^2 + z^2}$ and $r_\sfz(p) =\sqrt{x^2 + y^2}$ and also 
\begin{align}
    \mathsf{Q}_\sigma(s) = \lbrace   p\in \Sph : r_\sigma(p) < s \rbrace  , 
\end{align}
with the sets $\sigma\in \lbrace   \sfx, \mathsf{y}, \sfz \rbrace  =\Sigma$ denote the spherical caps centred at the six fixed points. For $s>0$ small enough, the $\mathsf{Q}_\sigma(s)$ are mutually disjoint. We then define
\begin{equation}\label{eq:defdriftOPPV}
\begin{split}
\mathrm{V}(p) &:= \chi(r_\sfx)(1-\chi(r_\mathsf{y}))(1-\chi(r_\sfz)) V_\sfx(p) + \chi(r_\mathsf{y})(1-\chi(r_\sfz)) {V}_\mathsf{y}(p) + \chi(r_\sfz)V_\sfz(p)  \\
&\quad+ \hat c (1-\chi(r_\sfx)) (1-\chi(r_\mathsf{y}))  (1-\chi(r_\sfz)) 
\end{split}
\end{equation}
where we choose $\hat c >0$ so that $\mathrm{V} \geq 1$ and $V_\sfx$, $V_\sfz$ and ${V}_\mathsf{y}$ are given as in Lemma \ref{lemma:localxlyap}, \ref{lemma:localzlyap} and \ref{lemma:localylyap}, respectively. We further observe that for all $\sigma\in \lbrace   \sfx, \mathsf{y}, \sfz \rbrace  $ there exists $0<r_0 < s_0$ independent of $\sigma$ such that $f_\omega(p)\in \mathsf{Q}_\sigma(s_0)$ for all $p\in \mathsf{Q}_\sigma(r_0)$ and all $\omega=(\omega_1,\omega_2)\in[-N,N]^2$.

\begin{proposition}\label{prop:lyapdriftopp}
Let $K(r_0)= \cap_{\sigma\in \lbrace   \sfx, \mathsf{y}, \sfz \rbrace  } \mathsf{Q}_\sigma(r_0)^c$. There exists $\beta>0$ such that 
\begin{align}
    \mathbb{E}\left[\mathrm{V}(f_\omega(p)) \right] \leq \frac{95}{100} \mathrm{V}(p) + \beta \mathbf{1}_{K(r_0)}(p).
\end{align}
for all $p\in X$.
\end{proposition}

\begin{proof}
If $p\in \mathsf{Q}_\sigma(r_0)\subset \mathsf{Q}_\sigma(s_0)$, for some $\sigma\in \lbrace   \sfx, \mathsf{y}, \sfz \rbrace  $ we then have that $f_{\omega}(p)\in \mathsf{Q}_\sigma(s_0)$ for all $\omega=(\omega_1,\omega_2)\in [-N,N]^2$ so that $\chi(r_{\sigma'})=0$ for all $\sigma'\neq \sigma$ and from Lemma \ref{lemma:localxlyap}, \ref{lemma:localzlyap} and \ref{lemma:localylyap} we have that
\begin{align}
\mathbb E \left[ \mathrm{V}(f_\omega(p)) \right ] = \mathbb E \left[ V_\sigma(f_{\omega}(p)) \right ] \leq \gamma V_\sigma(p) = \gamma \mathrm{V}(p).    
\end{align}  
On the other hand,  for all $p\in K$ a compact set, we observe that $f_\omega(p)\not \in F$, for all values $\omega=(\omega_1,\omega_2)\in [-N,N]^2$ lying in a compact set as well. Hence, we set $\hat\beta = \max_{p\in K, \omega\in[-N,N]^2} \mathrm{V}(f_\omega(p))$ and we conclude that $\mathbb E \left[ \mathrm{V}(f_\omega(p)) \right] \leq \hat\beta+\hat c$. For $\beta:=\hat\beta+\hat c$ the proposition is proved.
\end{proof}

\begin{remark}
    In case the two axes of rotation are non-orthogonal, namely $\a\cdot \b \neq 0$, there are only two fixed points for the one-point process, namely $\x=\frac{\a\times\b}{|\a\times\b|}$ and $\x=-\frac{\a\times\b}{|\a\times\b|}$. These correspond to the $F_{\sfx}$ points already considered, since both $u^\a$ and $u^\b$ generate a (non-orthogonal yet also non-parallel) shearing motion around them, and simple adjustments to the arguments presented in Lemma \ref{lemma:localxlyap} provide the existence of local drift functions around these points and then of a global drift function gluing the local contributions appropriately. We omit the details.
\end{remark}

\section{The projective process}\label{sec:PJP}
To use the LARC arguments introduced in Section \ref{sec:LARC}, we lift the pointwise projective process to the unit tangent bundle $UT(\mathsf{X})$. For an axis $\a\in\Sph$ and its respective velocity field $u^{\a}(\x)=(\a\cdot\x)(\a\times\x)$ we define
\begin{equation}
    \hat{u}_{\a}(\x,v)=\frac{\d}{\d \omega} \hat f^{\a}_\omega(\x,v)|_{\omega=0}=(u^{\a}(\x),V_{\a}(\x,v)),
\end{equation}
where the tangential velocity field is given by 
\begin{equation}\label{eq:vectorprojectiv}
    V_{\a}(\x,v)=D_\x u^{\a}(\x)v-(D_\x u^{\a}(\x)v\cdot v)v
\end{equation}
and the differential is given by
\begin{equation}\label{eq:Du_a}
    D_\x u^{\a}(\x)v=(\a\cdot v)(\a\times\x)+(\a\cdot\x)(\a\times v).
\end{equation}
One can check that $\hat{u}_{\a}$ generates a flow in $UT(\mathsf{X})$ by showing $|v(t)|^2=|v(0)|^2=1$ and that $\x(t)\cdot v(t)=\x(0)\cdot v(0)=0$ for all times. This flow is the one associated to $\hat f^{\a}_\omega(\x,v)$.

To show that there exists an open small set for the projective process, we appeal to the Lie algebra rank condition framework. First, we show there exists $\x_0\in \mathsf{X}$ and $v_0\in UT_{\x_0}\mathsf{X}$ for which the family $\lbrace  \hat u_\a,\hat u_\b\rbrace  $ is bracket generating at $(\x_0,\bv_0)$ and next we define the appropriate flow maps for which we can use Propositions \ref{prop:word-construction} and \ref{prop:LARCPJP}. 

To use the abstract LARC framework, let $u\in \mathfrak{X}(\Sph)$ be divergence-free and let $\phi_t^u$ be the unique volume-preserving flow such that
\begin{align}
\left\lbrace
\begin{aligned}
        \frac{\d}{\d t}\phi_t^u(\x) &= u(\phi_t^u(\x)),  \\
        \phi_0^u(\x) &= \x. 
        \end{aligned}
        \right.
\end{align}
The associated projective process is then given by
\begin{align}
    \hat{\phi}_t^u(\x,v) = \left(\phi_t^u(\x), \frac{D_\x\phi_t^u(\x)v}{ |D_\x\phi_t^u(\x)v|} \right),
\end{align}
for all $\x\in\mathsf{X}$ and all $v\in UT_{\x}\mathsf{X}$. Since both $v\in UT_{\x}\mathsf{X}$ and $\frac{D_\x\phi_t^u(\x)v}{ |D_\x\phi_t^u(\x)v|}\in UT_{\phi_t^u(\x)}\mathsf{X}$ live in different spaces that depend on $x,\phi_t^u(\x)\in \mathsf{X}$, we use a local orthonormal reference frame $E_0$, see Section \ref{sec:DiffGeom}, to fix the domain and codomain of the map $\hat\phi_t^u$. More precisely, let $\x\in U_3$ and let $\cU\subset U_3$ be an open set on $\mathsf{X}$ with $\x\in \cU$. Since $UT(\cU)\cong \cU\times \mathbb{S}^1$, we describe $\hat\phi_t^u$ by its action on  $\cU\times \mathbb{S}^1$. To do so, let $\x\in\mathsf{X}$, let $w\in \mathbb{S}^1$ and $v=E_0(\x)w\in T_{\x}\mathsf{X}$. Since $\frac{D_\x\phi_t^u(\x)v}{ |D_\x\phi_t^u(\x)v|}\in UT_{\phi_t^u(\x)}\mathsf{X}$ and $\phi_t^u(\x)\in \cU$ for $|t|$ small enough, there exists $w_t^u(\x,w)\in \mathbb{S}^1$ such that $\frac{D_\x\phi_t^u(\x)E_0(\x)w}{ |D_\x\phi_t^u(\x)E_0(\x)w|} = E_0(\phi_t^u(\x))w_t^u(\x,w)$, that is,
\begin{align}
    w_t^u(\x,w) = E_0(\phi_t^u(\x))^T \frac{D_\x\phi_t^u(\x)E_0(\x)w}{ |D_\x\phi_t^u(\x)E_0(\x)w|}\in \mathbb{S}^1.
\end{align}
Define the projective flow map  $\widetilde{\hat\phi_t^u}:\cU\times \mathbb{S}^1\rightarrow \cU\times \mathbb{S}^1$ by
\begin{align}
    \widetilde{\hat\phi_t^u}(\x,w) := \left( \phi_t^u(\x), E_0(\phi_t^u(\x))^T \frac{D_\x\phi_t^u(\x)E_0(\x)w}{ |D_\x\phi_t^u(\x)E_0(\x)w|} \right)
\end{align}
together with its associated projective velocity field
\begin{align}
    \widetilde{\hat u}(\x,w):= \frac{\d}{\d t}\Big|_{t=0}  \widetilde{\hat\phi_t^u}(\x,w) = \left( u(\x), \mathsf{w}_{u}(\x,w) \right), 
\end{align}
where, for $\x\in \mathsf{X}$ and $w\in \mathbb{S}^1\subset\R^2$, we define $\mathsf{w}_u(\x,w)\in T_{w}\mathbb{S}^1\subset\R^2$ by
\begin{align}
    \mathsf{w}_u(\x,w)= E_0(\x)^T V_{u}(\x,E_0(\x)w) - E_0(\x)^T (\nabla_u E_0(\x))w,
\end{align}
with
\begin{align}
    V_u(\x,v) := \nabla_v u(\x) - (v\cdot \nabla_v u (\x)) v
\end{align}
for $v=E_0(\x)w$. Since $v\in T\mathsf{X}$, we have that $v\cdot \nabla_v u(\x) = v(\x) \cdot D_\x u(\x) v(\x)$. Together with Lemma \ref{lemma:framecovariantD} and $E_0(\x)^TE_0(\x) = I_2$, it follows that
\begin{align}\label{eq:formulawu}
    \mathsf{w}_u(\x,w) = E_0(\x)^T D_\x u(\x) E_0(\x)w -  \left(v(\x) \cdot D_\x u(\x) v(\x) \right) w - E_0(\x)^T(\nabla_u E_0(\x))w.
\end{align}
The next result shows that there exists $\x_0\in \mathsf{X}$ and $v_0\in UT_{\x_0}\mathsf{X}$ (equivalently $w_0\in \mathbb{S}^1$) for which the family $\lbrace  \widetilde{\hat u}_\a,\widetilde{\hat u}_\b\rbrace  $ is bracket generating at $(\x_0,v_0)$.

\begin{proposition}\label{prop:LARCPJP}
Assume $\a=(0,a_2,a_3)\in \Sph$ with $a_2\neq 0$ and $\b=(0,0,1)$. Let $\x_0=(\frac{1}{\sqrt{2}},\frac{1}{\sqrt{2}},0)\in \mathsf{X}$ and  $v_0=(0,0,1)$.  Then the family $\lbrace  \widetilde{\hat u}_\a,\widetilde{\hat u}_\b\rbrace  $ is bracket generating at $(\x_0,v_0)$.
More precisely, 
\begin{align}
    \hat X_1:=\widetilde{\hat u}_\a,\quad 
\hat X_2:=\widetilde{\hat u}_\b,\quad
\hat X_3:=[\hat X_2, \hat X_1]
\end{align}
are linearly independent vectors at $(\x_0,v_0)$.
\end{proposition}

\begin{proof}
We recall that $E_0(\x_0) = (e_1, e_2)$, with $e_1=e_1(\x_0) = (0,0,1)$ and $e_2=e_2(\x_0) = \left( \frac{1}{\sqrt{2}}, - \frac{1}{\sqrt{2}}, 0 \right)$. Then, $v_0 = e_1= E_0(\x_0)w_0$, for $w_0=(1,0)\in \mathbb{S}^1\subset\R^2$ and $(\nabla_u E_0)(\x_0)=0$ for all $u\in T\mathsf{X}$ due to Lemma \ref{lemma:covariantframe}. Next, from \eqref{eq:formulawu} we note that 
\begin{align}
    u^{\a}(\x_0)=\frac{a_2}{2}(-a_3,a_3,-a_2) = -\frac{a_2^2}{2}e_1 - a_2a_3\frac{\sqrt{2}}{2} e_2 \in T_{\x_0}\Sph
\end{align}
and
\begin{align}
    \mathsf{w}_{u_\a}(\x_0,w_0) = (0,\frac{a_2^2-2a_3^2}{2})
    \in T_{w_0}\mathbb{S}^1 \subset \R^2
\end{align}
while
\begin{align}
u^\b(\x_0) = 0, \quad \mathsf{w}_{u_\b}(\x_0,w_0)=(0,-1)\in T_{w_0}\mathbb{S}^1\subset\R^2.
\end{align}
Identifying $T_{w_0}\mathbb{S}^1 = \lbrace   0 \rbrace  \times \R \cong \R$, we have
\begin{align}
    \hat X_1 = \left( -\frac{a_2^2}{2}e_1 - a_2a_3\frac{\sqrt{2}}{2}e_2, \frac{a_2^2-2a_3^2}{2} \right), \quad \hat X_2 = \left( 0, -1 \right)
\end{align}
Next, we compute $\hat X_3 =[\hat X_2,\hat X_1]$. In fact, for our purposes it is enough to compute $[u^\b, u^{\a}](\x_0)$. Since  $u^\b(\x_0)=0$, we have
        \begin{equation}
            [u^\b, u^{\a}](\x_0)=-D_\x u^\b(\x_0)u^{\a}(\x_0)=\frac{1}{2\sqrt{2}}(-a_2^2,a_2^2,0) = -\frac{a_2^2}{2} e_2
        \end{equation}
and we write $\hat X_3 = (-\frac{a_2^2}{2\sqrt{2}} e_2, w_*)$, for some $w_*\in \R$. To see that $\lbrace   \hat X_1, \hat X_2, \hat X_3 \rbrace  $ are linearly independent, assume there exist $c_1,c_2,c_3\in \R$ such that $c_1\hat X_1 + c_2\hat X_2 + c_3 \hat X_3 = 0$. Since $\hat X_1$ is the only vector with non-trivial $e_1$ component (recall $a_2\neq 0$), we deduce that $c_1=0$. Next, since $u^\b(\x_0) = 0$ we must have $c_3 = 0$ since $\hat X_3$ has non-trivial $e_2$ component. Finally, we obtain $c_2 = 0$ as $\hat X_2$ has non-zero tangential vector at $w_0\in \mathbb{S}^1$. Hence, $\lbrace  \widetilde{\hat u}_{\a}, \widetilde{\hat u}_{\b} \rbrace  $ is bracket generating at $(\x_0,v_0)$ and the proof is concluded. 
\end{proof}

Since $\cU\times \mathbb{S}^1\subset \Sph\times \mathbb{S}^1\subset \R^3\times \R^2$ once we identify $\mathbb{S}^1$ as the unit circle of $\R^2$ centred at the origin, for $U_0\subset\R^3$, with $\cU\subset U_0$ and $U_0\cap \lbrace   \lambda \mathbf{e}_3 : \lambda\in \R \rbrace = \emptyset$, we define the extension  $\hat v_u:U_0\times\R^2\subset \R^3\times \R^2 \rightarrow \R^3\times \R^2$ of $\widetilde{\hat u}$ by
\begin{align}
    \hat v_u(\y,w) = (u(\y), \mathsf{w}_u(\y,w))
\end{align}
for $\y\in U_0$ and $w\in \R^2$, where now $\Pi_\y  = Id - \frac{\y \y^T}{\Vert \y \Vert^2}$ and the covariant derivative $\nabla$ is extended accordingly. Clearly, this extension is such that $\hat v_u(\x,w) = \widetilde{\hat u}(\x,w)$, for all $\x\in \Sph$ and all $w\in \mathbb{S}^1$ and thus if $\hat\Phi_t^u:U_0\times \R^2 \rightarrow U_0\times \R^2$ denotes the flow map defined by $\hat v_u$, namely 
\begin{align}
    \begin{cases}
        \frac{\d}{\d t}\hat \Phi_t^u(\y,w) = \hat v_u(\hat \Phi_t^u(\y,w)), & \\
        \hat\Phi_0^u(\y,w) = (\y,w) &
    \end{cases}
\end{align}
then $\hat\Phi_t^u(\x,w)= \widetilde{\hat\phi_t^u}(\x,w)$, for all $\x\in \Sph$, $w\in\mathbb{S}^1$ and all $t\in \R$ for which $\phi_t^u(\x)\in \cU$. 

\begin{proposition}
    The projective process admits an open small set.
\end{proposition}

\begin{proof}
Let $\cV = \lbrace   \hat{v}_{u_\a}, \hat v_{u_\b} \rbrace  $ and $\hat\Phi_t^{ u_\a}$ and $\hat\Phi_t^{u_\b}$ be their respective flow maps in $U_0\times\R^2$. The vector fields $\hat{v}_{u_\a}, \hat v_{u_\b}, [\hat{v}_{u_\a}, \hat v_{u_\b} ]  $ are linearly independent  at $\x_0=(\frac{1}{\sqrt{2}},\frac{1}{\sqrt{2}},0)$ and  $v_0=(0,0,1)$, corresponding to $w_0=(1,0)\in \mathbb{S}^1$, due to Proposition \ref{prop:LARCPJP}. Hence, thanks to Proposition \ref{prop:word-construction} there exist words $\varpi_i$, for $i=1,2,3$ associated to each $\hat{v}_{u_\a}, \hat v_{u_\b}$, and $[ \hat{v}_{u_\a}, \hat v_{u_\b} ]$, and there exists $\ep_0>0$ such that $D_\varpi \hat\Phi_{\varpi_*}(\x_0,w_0)$ has at least rank 3, for $\varpi_* = (\varpi_1(a_1^*), \varpi_2(a_2^*), \varpi_3(a_3^*))\in [-\ep_0,\ep_0]^{d_0}$, for some $d_0>1$ and all $a_*=(a_1^*,a_2^*,a_3^*)\in \R^3$ with $0< \Vert a_* \Vert < \ep_0$.

On the other hand, for fixed $\x_0=(\frac{1}{\sqrt{2}},\frac{1}{\sqrt{2}},0)$ and  $v_0=(0,0,1)$, namely $w_0=(1,0)\in \mathbb{S}^1$, we have that $\hat\Phi_t^{ u}(\x_0,w_0)\in \cU\times \mathbb{S}^1$ with $\cU\subset\Sph$, for $t\in \R$ small enough and thus  
\begin{align}
    D_\varpi \hat\Phi_{\varpi_*}^{u} (\x_0, w_0):\R^{d_0} \rightarrow T_{\hat\Phi_{\varpi_*}^{u} (\x_0, w_0)}(\cU\times \mathbb{S}^1)\cong \R^2\times\R
\end{align}
has at least rank 3 due to Proposition \ref{prop:word-construction}. Hence, $\hat\Phi_{\varpi}^{u} (\x_0, w_0)$ is a submersion at $\varpi = \varpi_*$ and with Lemma \ref{lemma:blackboxsmallset} we conclude that the projective process admits an open small set.
\end{proof}

To deduce the irreducibility of the projective process we inspect the behaviour of the projective flow maps, and in particular how the differential maps act on the tangent bundle. Reducing the dynamics to a suitable point on the sphere, an interesting ``windshield wiper'' behaviour arises: alternating the differentials of the flow maps for specific times moves vectors around the chosen base point and spans its whole tangent space. We carefully rely on this intrinsic behaviour and show that Lemma \ref{lemma:blackboxirreducibility} can be directly applied to prove topological irreducibility of the projective process. 
\begin{proposition}\label{prop:topirredPJP}
    The projective process is topologically irreducible.
\end{proposition}

\begin{proof}
We shall show that the projective process is in fact exactly controllable, that is, for all $(\x,v_0), (\y,v_1)\in \mathsf{Y}=UT(\mathsf{X})$ there exist $n\geq 1$ and a sequence of random amplitudes $\underline{\omega}^n$ such that $f_{\underline{\omega}^n}(\x,v_0) = (\y, v_1)$, to then apply Lemma \ref{lemma:blackboxirreducibility}. Upon noting that $(\hat f_\omega^\a)^{-1} = \hat f_{-\omega}^\a$ and $(\hat f_\omega^\b)^{-1} = \hat f_{-\omega}^\b$ for all $\omega\in \R$, it is enough to map both $(\x,v_0)$ and $(\y,v_1)$ to some common point $(\x_0, v)\in \mathsf{Y}$. For convenience, we choose $\x_0=\frac{1}{\sqrt{2}}(1,1,0)$ and $v=(0,0,1)\in UT_{\x_0}\mathsf{X}$ and we assume without loss of generality that we already have $\x=\x_0$ and $v_0\in UT_{\x_0}\mathsf{X}$. The proof is concluded if we can find $\underline{\omega}$ such that $\hat f_{\underline{\omega}}(\x_0,v_0) = (\x_0,v)$. 

First, we shall only consider $\omega_\a$ and $\omega_\b$ such that $\x_0 = f_{\omega_\a}^\a(\x_0) = f_{\omega_\b}^\b(\x_0)$. Since $u^\b(\x_0)=0$ we note that $f_{\omega_\b}^\b(\x_0) = \x_0$, for all $\omega_\b\in \R$. Similarly, $f_{\omega_\a}^\a(\x_0) = \x_0$ if $\theta(\x_0)\omega_\a = 2\pi k$, for some $k\in \Z$, where we recall that $\theta(\x_0) = \a\cdot\x_0 = \frac{a_2}{\sqrt{2}}$. Hence, we have  $f_{\omega_\a}^\a(\x_0) = \x_0$ for all $\omega_\a(k)= \frac{2\sqrt{2}\pi}{a_2}k$, with $k\in \Z$. Restricted to such $\omega_\a$, in the orthonormal reference frame $E_0(\x_0)$ given by
    \begin{align}
        e_1(\x_0)=(0,0,1), \quad  e_2(\x_0) =  \left(\frac1{\sqrt2},-\frac1{\sqrt2},0\right)
    \end{align}
the differentials $D_\x f^\b_{\omega_\b}(\x_0)$ and $D_\x f^{\a}_{\omega_\a}(\x_0)$ are given by the $SL(2,\R)$ matrices
\begin{align}
    \mathrm{g}_\b(\omega_\b) = \begin{pmatrix}
        1 & 0 \\
        -\omega_\b & 1
    \end{pmatrix}, \quad \mathrm{g}_\a(\omega_\a) = \begin{pmatrix}
        1 - \frac{\omega_\a a_2 a_3}{\sqrt{2}} & \frac{a_2^2\omega_\a}{2} \\
        -\a_3^2\omega_\a & 1 + \frac{\omega_\a a_2 a_3}{\sqrt{2}}
    \end{pmatrix}.
\end{align}
By definition, they are such that 
\begin{align}
    D_\x f_{\omega_{\bfe}}^{\bfe}(\x_0) v = E_0(\x_0) \mathrm{g}_\bfe(\omega_\bfe) w
\end{align}
for $v=E_0(\x_0)w\in T_{\x_0}\mathsf{X}$ with $w\in \R^2$ and $\bfe\in \lbrace   \a, \b \rbrace  $. Now, in view of the orthonormal basis $E_0(\x_0)$, showing that for all non-zero $v\in T_{\x_0}\mathsf{X}$ there exists $\omega_\a(k),  \omega_\b\in \R$ such that $D_\x f_{\omega_\b}^\b(\x_0) D_\x f_{\omega_\a(k)}^\a(\x_0) v = \lambda e_1(\x_0)$ for some $\lambda>0$ is equivalent to showing that for all non-zero $w=(w_1,w_2)\in \R^2$ there exists $\omega_\a(k), \omega_\b\in \R$ such that $\mathrm{g}_\b(\omega_\b)\mathrm{g}_\a(\omega_\a(k))w = \lambda(1,0)$, for some $\lambda>0$. 

To achieve this controllability in $\R^2$, let $w\in \R^2\setminus\lbrace   0 \rbrace  $, $w^{(1)}= \mathrm{g}_\a(\omega_\a(k))w$ and observe that
\begin{align}
    w^{(1)} = \begin{pmatrix}
        \left( 1 - \frac{\omega_\a a_2 a_3}{\sqrt{2}} \right) w_1 +\frac{a_2^2}{{2}}\omega_\a(k)w_2 \\
        -w_1 a_3^2\omega_\a(k) + \left( 1 + \frac{\omega_\a a_2 a_3}{\sqrt{2}} \right) w_2
    \end{pmatrix}.
\end{align}
Note that if $\frac{a_2}{2}w_2 - \frac{a_3}{\sqrt{2}} w_1= 0$ then $w^{(1)}=w$ and additionally $w_1\neq 0$ since otherwise $w_2=0$ as $a_2\neq 0$. Moreover, for any $\omega_\b\neq 0$ we would have
\begin{align}
    w^{(2)} = \mathrm{g}_\b(\omega_\b) w^{(1)} =\mathrm{g}_\b(\omega_\b) w = \begin{pmatrix}
        w_1 \\
        w_2 - w_1\omega_\b
    \end{pmatrix}
\end{align}
and, in particular $\frac{a_2}{2}w_2^{(2)} - \frac{a_3}{\sqrt{2}} w_1^{(2)} = -\omega_\b\frac{a_2}{2}w_1 \neq 0$. As a result, we may assume without loss that $\frac{a_2}{2}w_2 - \frac{a_3}{\sqrt{2}} w_1\neq 0$. Then, since $a_2\neq0$, it is easy to see that for $|\omega_\a(k)|$ large enough, that is $|k|\in \N$ large enough,  we have $w^{(1)}_1 >0$. As a result,
\begin{align}
    \mathrm{g}_\b(\omega_\b) w^{(1)} = \begin{pmatrix}
        w^{(1)}_1 \\
        w^{(1)}_2 - \omega_\b w^{(1)}_1
    \end{pmatrix} = w^{(1)}_1 \begin{pmatrix}
        1 \\ 0
    \end{pmatrix}
\end{align}
for $\omega_\b = \frac{w^{(1)}_2}{w^{(1)}_1}$. Since we can send any non-zero $w\in \R^2$ to the ray generated by $(1,0)$, we deduce that the map $\mathrm{g}_\b(\omega_\b)\mathrm{g}_\a(\omega_\a(k))$ is exactly controllable in $\R^2\setminus \lbrace   0 \rbrace  $. Consequently, $\hat f_{\underline{\omega}}(\x_0,v)$ is also exactly controllable as a map from $\lbrace   \x_0 \rbrace  \times UT_{\x_0}(\mathsf{X})$ to itself for $\omega_\a=\omega_\a(k)$ as the normalizing factor in the projective process does not influence the direction of the final vector but only ensures it has unit norm.
\end{proof}
 
We finish this section recording that the projective process admits a functional satisfying the Lyapunov-drift condition.

\begin{proposition}
    The function $\widehat{\mathrm{V}}:UT(\mathsf{X})\rightarrow [1,\infty)$ given by $\widehat{\mathrm{V}}({p},v):=\mathrm{V}({p})$ as defined in \eqref{eq:defdriftOPPV} for all $({p},v)\in UT(\mathsf{X})$ satisfies the drift condition for the projective chain.
\end{proposition}

\section{Positivity of the top Lyapunov exponent}\label{sec:toplyapexp}
In this section we show that the top Lyapunov exponent $\lambda_1$ associated to the one-point process $f_\omega$ is positive. First, we prove the following equivalence statement.
\begin{proposition}\label{prop:joint-surj}
Let $V,U,W$ be (finite-dimensional) real vector spaces and let $L:V\to U, K:V\to W$ be linear maps.  Define the joint linear map $F:V\to U\times W, F(v):=(Lv,Kv).$ Then the following are equivalent:
\begin{enumerate}
\item[(i)] $F$ is surjective.
\item[(ii)] $L$ is surjective and the restricted map $K|_{\ker L}:\ker L\to W$ is surjective (equivalently, $K(\ker L)=W$).
\end{enumerate}
\end{proposition}

\begin{proof}[Proof of Proposition \ref{prop:joint-surj}]
\emph{(ii)$\Rightarrow$(i).}
Assume that $L$ is surjective and that $K|_{\ker L}$ is surjective.
Let $(u,w)\in U\times W$ be arbitrary. Since $L$ is surjective, there exists $v_0\in V$ such that $Lv_0=u$.
Consider the element $w-Kv_0\in W$. By surjectivity of $K|_{\ker L}$, there exists $v_1\in\ker L$ such that $Kv_1=w-Kv_0$.
Then
\begin{equation}
F(v_0+v_1)=(L(v_0+v_1),K(v_0+v_1))=(Lv_0+Lv_1, Kv_0+Kv_1)=(u,w),
\end{equation}
since $Lv_1=0$ for $v_1\in\ker L$. Hence $F$ is surjective.

\emph{(i)$\Rightarrow$(ii).}
Assume that $F$ is surjective. Composing with the projection
$\pi_U:U\times W\to U$ yields $\pi_U\circ F=L$, hence $L$ is surjective.
Next, let $w\in W$ be arbitrary and consider the point $(0,w)\in U\times W$. By surjectivity of $F$, there exists $v\in V$ such that $F(v)=(0,w)$.
Thus $Lv=0$, so $v\in\ker L$, and simultaneously $Kv=w$. Therefore every $w\in W$ lies in $K(\ker L)$, i.e.\ $K|_{\ker L}$ is surjective.
\end{proof}

With Propositions \ref{prop:blackbox_lyapexponent} and \ref{prop:joint-surj}, to ensure the positivity of $\lambda_1$ it is enough to find $\x_0\in \mathsf{X}$, $n\geq 1$ and $\underline{\omega}_*^n$ for which $\widehat{\Psi}_{\x_0}:\Omega_0^n\rightarrow SL(\mathsf{X})$ is a submersion at $\underline{\omega}^n = \underline{\omega}_*^n$. As $\widehat{\Psi}_{\x}$ takes values in $SL(\mathsf{X})$, we first derive expressions for $\widehat{\Psi}_{\x}$ that will be suitable for computations. As mentioned in Section \ref{sec:DiffGeom}, the main difficulty in the spherical case compared to the flat-geometric settings of \cites{BCZG2023, C2023, coti2026three} is the non-parallelizability of $\Sph$, which we overcome by working with a local smooth orthonormal reference frame in a neighbourhood of a given $\x_0\in \mathsf{X}$. For convenience, we consider the orthonormal and unit-area reference frame $E_0(\x):= ( e_1(\x), e_2(\x))\in(T_\x\mathsf{X})^2$ over $\cU_3=\mathsf{X}\setminus \lbrace   \pm \mathbf{e}_3 \rbrace  $ and we recall that
\begin{equation}\label{eq:recallframeE}
    e_1(\x) = \frac{1}{r}(-zx,-zy,r^2), \quad e_2(\x) = \frac{1}{r}(y,-x,0),\quad r:=\sqrt{x^2+y^2}.
\end{equation}
Recall further that for any open subset $\cU\subset \cU_3$, we can identify the principal $SL(2,\R)$-bundle of oriented area-preserving frames $\mathcal{F}^{SL}(\mathsf{X})$ on $\cU$ by  
\begin{align}
    \mathcal{F}^{SL}(\Sph)\big|_\cU \cong \cU\times SL(2,\R).
\end{align}

Let $u\in \mathfrak{X}(\Sph)$ be divergence-free and let $\phi_t^u$ be the associated volume-preserving flow on $\Sph$. 
If $\x\in \cU_3$ then $\phi_t^u(\x)\in \cU_3$ for all $0\leq t \leq t_0$, for some small $t_0(\x)$. Since $D_\x\phi_t^u:T_{\x}\mathsf{X} \rightarrow T_{\phi_t^u}\mathsf{X}$ with $\text{det}(D_\x\phi_t^u)=1$ because $u\in \mathfrak{X}(\Sph)$ and hence we define $\widetilde{\widetilde{\phi_t^u}}:\mathcal{F}^{SL}(\mathsf{X})\to \mathcal{F}^{SL}(\mathsf{X})$ by 
\begin{align}
    \widetilde{\widetilde{\phi_t^u}}(\x,E) := (\phi_t^u(\x), D_\x\phi_t^u(\x)E)
\end{align}
where, for $E=(e_1,e_2)\in (T_\x\mathsf{X})^2$, 
\begin{equation}
    D_\x\phi_t^u(\x)E := (D_\x\phi_t^u(\x)e_1, D_\x\phi_t^u(\x)e_2)\in (T_{\phi_t^u(\x)}(\mathsf{X}))^2.
\end{equation}
While $\widetilde{\widetilde{\phi_t^u}}$ maps $\mathcal{F}^{SL}(\mathsf{X})$ into itself, the main obstacle in employing it is that for any unit-area frame $E$, the basis $E(\x)$ and $D_\x\phi_t^u E(\x)$ live in $T_\x\mathsf{X}$ and $T_{\phi_t^u(\x)}\mathsf{X}$, respectively, which are in general distinct tangent spaces. More crucially, $D_\x\phi_t^u E\in (T_{\phi_t^u(\x)}\mathsf{X})^2$, which depends on $\phi_t^u(\x)$. In order for the final target space to be independent of $\phi_t^u(\x)$, we use the reference frame $E_0$ as follows: since $E(\x)$ is a unit-area frame on $\cU$, for all $\x\in \cU_3$ there exists $\mathsf{g}_\x\in SL(2,\R)$ such that 
\begin{align}
    E(\x) = E_0(\x) \mathsf{g}_\x.
\end{align}
Similarly, for all $\phi_t^u(\x)\in \cU$, $D_\x\phi_t^u(\x)E(\x)$ is a unit-area basis on $T_{\phi_t^u(\x)}\mathsf{X}$ and thus there exists $\mathsf{h}\in SL(2,\R)$ with
\begin{align}
    D_\x\phi_t^u(\x)E(\x) = E_0(\phi_t^u(\x))\mathsf{h}_{\phi_t^u(\x)}.
\end{align}
As $E_0$ is an orthonormal frame, $E_0^{-1}:=E_0^T$ is a left inverse of $E_0$ when viewed as matrices and we see that 
\begin{align}
    \mathrm{h}_{\phi_t^u(\x)} = E_0(\phi_t^u(\x))^T D_\x\phi_t^u(\x)E= E_0(\phi_t^u(\x))^T D_\x\phi_t^u(\x)E_0(\x)\mathrm{g}_\x  \in SL(2,\R).
\end{align} 
Consequently, we define $\widetilde{\phi_t^u}:\cU\times SL(2,\R)\rightarrow \cU_3\times SL(2,\R)$ by
\begin{align}\label{eq:deftildephi}
        \widetilde{\phi_t^u}(\x, \mathrm{g}) = (\phi_t^u(\x), E_0(\phi_t^u(\x))^T D_\x\phi_t^u(\x)E_0(\x)\mathrm{g}).
\end{align}

To understand the vector field associated to the lifted process  $\widetilde{\phi_t^u}$, we recall that if $w$  denotes a smooth vector field in $\R^3$ such that $w|_{\Sph}\in \mathfrak{X}(\Sph)$ and $v\in T_x\Sph$, the Levi-Civita covariant derivative $\nabla_v w$ on $\Sph$ is then given by
\begin{align}
    \nabla_v w(\x) = \Pi_\x(D_\x w(\x) v(\x)).
\end{align}
for all $\x\in \Sph$. Then, 
\begin{align}\label{eq:deftildeu}
    \widetilde{u}(\x,\mathrm{g}) &:= \frac{\d}{\d t}\Big|_{t=0}\widetilde{\phi_t^u}(\x,\mathrm{g})  = (u(\x), \Omega_u(\x)\mathrm{g}),
\end{align}
with
\begin{align}
    \Omega_u(\x) &= E_0(\x)^T(\nabla_{E_0(\x)} u)(\x) - E_0(\x)^T(\nabla_u E_0)(\x).
\end{align}
Here, for $E_0=(e_1,e_2)$ given by \eqref{eq:recallframeE} we denote
\begin{align}
    (\nabla_{E_0}u)(\x) &:= (\nabla_{e_1} u(\x), \nabla_{e_2} u(\x))
\end{align}
and we recall that
\begin{align}
        \nabla_u E_0(\x) &:=(\nabla_u e_1(\x), \nabla_u e_2(\x)). 
\end{align}
We observe that $\Omega_u$ not only encodes the variation of $u$, but also of the fixed frame $E_0$ under the action of $u$. On a parallelizable manifold, one can choose $E_0(\x)\equiv E_0$ and thus $\nabla_u E_0 \equiv 0$. In view of Lemma \ref{lemma:framecovariantD}, 
\begin{align}
\Omega_u(\x)= E_0(\x)^T(D_\x u)(\x)E_0(\x) - E_0(\x)^T(D_\x E_0)(\x)u(\x),
\end{align}
with $D_\x E_0(\x)u(\x) = (D_\x e_1(\x)u(\x), D_\x e_2(\x)u(\x))$. Let $\widetilde{u}_\a=(u^\a, \Omega_{u_\a})$ and $\widetilde{u}_\b = (u^\b, \Omega_{u_\b})$ be the vector fields associated to $\widetilde{\phi_t^{u_\a}}$ and $\widetilde{\phi_t^{u^\b}}$ respectively. We now show that a set of certain commutators of the velocity fields $\widetilde{u}_\a$ and $\widetilde{u}_\b$  satisfy the LARC. For convenience, set $\x_0=\left(\frac1{\sqrt2},\frac1{\sqrt2},0\right)$ and define
\begin{equation}
\bfe_1:=e_1(\x_0)=(0,0,1), \quad \bfe_2:=e_2(\x_0) =  \left(\frac1{\sqrt2},-\frac1{\sqrt2},0\right)
\end{equation}
so that $E_0(\x_0)=(\bfe_1,\bfe_2)$ is a positively oriented unit-area orthonormal basis of $T_{\x_0}\Sph$. From Lemma \ref{lemma:covariantframe} we now have
\begin{align}
    \Omega_u(\x_0)= E_0(\x_0)^T(D_\x u)(\x_0)E_0(\x_0).
\end{align}
for all $u\in \mathfrak{X}(\Sph)$.

\begin{proposition}\label{prop:LARCtoplyap}
Assume $\a=(0,a_2,a_3)\in \Sph$ with $a_2\neq 0$ and $\b=(0,0,1)$. Let $\widetilde{u}_\a$ and $\widetilde{u}_\b$ be given by \eqref{eq:deftildeu}. Then the family $\lbrace  \widetilde u_\a,\widetilde u_\b\rbrace  $ satisfies the Lie algebra rank condition at $(\x_0,Id_2)$.
More precisely, 
\begin{align}
    X_1:=\widetilde u_\a,\quad  X_2:=\widetilde u_\b,\quad X_3:=[X_1,X_2],\quad X_4:=[X_1,X_3],\quad X_5:=[X_2,X_4],
\end{align}
are linearly independent at $(\x_0,Id_2)$.
\end{proposition}

\begin{proof}
Set $w_1(\x)=u^\a(\x)$, $w_2(\x) = u^\b(\x)$ and define $w_3(\x):=[w_1,w_2](\x)$, $w_4(\x) =[w_1,w_3](\x)$ and $w_5(\x):=[w_2,w_4](\x)$. To compute $[v,w]=[v,w]_{\Sph}$ for $v,w\in \mathfrak{X}(\Sph)$, which denotes the commutator of vector fields on $\Sph$, we note that $[v,w]_\Sph(\x) = \Pi_\x[v,w]_{\R^3}(\x)$, for all $v_1,v_2\in \mathfrak{X}(\Sph)$ and all $\x\in \Sph$. However, since $v,w\in \mathfrak{X}(\Sph)$ are already tangent to the sphere, it follows that $[v,w]_{\R^3}\in\mathfrak{X}(\Sph)$ is also tangent to $\Sph$ and thus $[v,w]_\Sph =[v,w]_{\R^3}$.  Next, since Lemma \ref{lemma:tildecommutator} yields 
\begin{align}
    [\widetilde{v},\widetilde{w}] = \widetilde{[v,w]}
\end{align}
for all $v,w\in \mathfrak{X}(\Sph)$ we have that $X_i = \widetilde{w_i}$, for $i=1,...,5$. Hence, there holds
\begin{align}
    X_i(\x_0,Id) = (w_i(\x_0),\Omega_i(\x_0)), \quad \Omega_i(\x_0) = E_0^T(\x_0)(D_\x w_i)(\x_0)E_0(\x_0) 
\end{align}
for all $i=1,...,5$. We compute\footnote{See Appendix \ref{app:diffgeom} for the computations.}
\begin{align}
    w_1(\x_0) = - \frac{a_2^2}{2}\bfe_1-a_2a_3\frac{\sqrt{2}}{2}\bfe_2 , \quad w_2(\x_0) = 0, \quad w_3(\x_0) = \frac{a_2^2}{2}\bfe_2, \quad w_4(\x_0) = -\frac{3\sqrt{2}}{4}a_2^3a_3\bfe_2, \quad w_5(\x_0) = 0, 
\end{align}
and
\begin{align}
    \Omega_1(\x_0) = \begin{pmatrix}
        -\frac{a_2 a_3}{\sqrt{2}} &  0 \\ \frac{a_2^2}{2}- a_3^2 & \frac{a_2 a_3}{\sqrt{2}}
    \end{pmatrix}&, \quad \Omega_2(\x_0) = \begin{pmatrix}
        0 & 0 \\ -1 & 0
    \end{pmatrix}, \quad \Omega_3(\x_0) = \begin{pmatrix}
        0 & 0 \\ \sqrt{2}a_2 a_3 & 0
    \end{pmatrix}, \\
    \Omega_4(\x_0) &= a_2^2\begin{pmatrix}
        -\frac{a_2 a_3}{2\sqrt{2}} & -2a_2^2 \\  2a_2^2 - 3a_3^2 & \frac{a_2 a_3}{2\sqrt{2}}
    \end{pmatrix}, \quad \Omega_5(\x_0) = \begin{pmatrix}
2a_2^4 & 0\\
-\frac{a_2^3a_3}{\sqrt2} & -2a_2^4
\end{pmatrix}.
\end{align}

 Suppose now that $\sum_{i=1}^5 c_i {X}_i(\x_0,Id)=0$, for some $c_1,...,c_5$. Recall that $a_2\neq 0$ and $\lbrace   \bfe_1, \bfe_2 \rbrace  $ constitute an orthonormal basis of $T_{\x_0}\Sph$. Since $w_1(\x_0)$ is the only vector with non-trivial $\bfe_1$ component, we deduce that $c_1=0$. Next, we observe that $\Omega_4$ is the only element with non-trivial $\bfe_1 \otimes \bfe_2$ component, and thus we must have $c_4 =0$. In turn, since $w_2(\x_0)=w_5(\x_0) = 0$, we deduce that $c_3=0$. Finally, $c_5=0$ because $\Omega_5(\x_0)$ has non-zero $\bfe_2 \otimes \bfe_2$ and $\bfe_1 \otimes \bfe_1$ components, and this ultimately forces $c_2=0$ as well. Therefore, all $c_i=0$ and the set $\lbrace   {X}_i(\x_0,Id) \rbrace_{i=1,...,5}$ is linearly independent. 
\end{proof}

With Proposition \ref{prop:LARCtoplyap} at hand, we now extend $\widetilde{\phi_t^u}$ from $\cU\times SL(2,\R)$ to an ambient Euclidean space in order to use Proposition \ref{prop:word-construction} from the LARC framework. Note that $\cU\times SL(2,\R)\subset \Sph\times SL(2,\R)\subset \R^3\times \R^4$ once we identify $SL(2,\R)$ as a subset of $\R^4$ via the isomorphism $\iota:\R^4\rightarrow M^{2\times2}(\R)$ given by $\iota(h) = \mathrm{h}:=\begin{pmatrix}
    h_1 & h_2 \\
    h_3 & h_4
\end{pmatrix}$ for $h=(h_1,h_2,h_3,h_4)\in \R^4$. For $\cU_0\subset\R^3$ with $\cU\subset \cU_0$ and $\cU_0\cap \lbrace   \lambda \mathbf{e}_3 : \lambda\in \R \rbrace   = \emptyset$ we define the extension  $v_u:\cU_0\times\R^4\subset \R^3\times \R^4 \rightarrow \R^3\times \R^4$ of $\widetilde{u}$ by
\begin{align}
    v_u(\y,h) = (u(\y), \iota^{-1}\circ \Omega_u(\y) \iota(h))
\end{align}
for all $(\y,{h})\in \R^3\times\R^4$, where $\Pi_\y w = w - (w\cdot\y)\frac{\y}{\Vert \y \Vert^2}$ for $w\in \R^3$ and the covariant derivative $\nabla_w$ is extended accordingly. Moreover, we denote $\Phi_t^u:\cU_0\times \R^4 \rightarrow \cU_0\times \R^4$ the flow map defined by $v_u$, namely 
\begin{align}
    \begin{cases}
        \frac{\d}{\d t}\Phi_t^u(\y,{h}) = v_u(\Phi_t^u(\y,{h})), & \\
        \Phi_0^u(\y,{h}) = (\y,{h}) &
    \end{cases}
\end{align}
In particular, since $v_u(\x,{g}) = \widetilde{u}(\x,\mathrm{g})$, for all $\x\in \Sph$ and $g\in \R^4$ with $\iota(g)=\mathrm{g}\in SL(2,\R)$, there holds $\Phi_t^u(\x,{g})= \widetilde{\phi_t^u}(\x,\mathrm{g})$, for all $t\in\R$, $\x\in \Sph$ and $\iota(g)=\mathrm{g}\in SL(2,\R)$ for which $\phi_t^u(\x)\in \cU_3$.

\begin{proposition}\label{prop:positivetoplyap}
The top Lyapunov exponent is positive.
\end{proposition}

\begin{proof}
Let $\cV = \lbrace   v_{u^\a}, v_{u^\b} \rbrace  $ and $\Phi_t^{v_{u^\a}}$, $\Phi_t^{v_{u^\b}}$ their associated flow maps. Let $g_0=(1,0,0,1)$ such that $\iota(g_0) = Id\in SL(2,\R)$. Since $v_{u}(\x_0,g_0) = \widetilde u(\x_0,Id)$ and the set $\lbrace   X_i(\x_0,Id) \rbrace_{i=1}^5\subset \R^3\times\R^4$ is linearly independent due to Proposition \ref{prop:LARCtoplyap}, we obtain from Proposition \ref{prop:word-construction} that there exists $\ep_0>0$ and words $\varpi_i$ associated to each $X_i$  such that $D_{\varpi}\Phi_{\varpi_*}(\x_0,Id)$ has rank at least $5$, for $\varpi_*=(\varpi_1(a_1^*), ..., \varpi_5(a_5^*))\in [-\ep_0,\ep_0]^{d_0}$, for some $d_0\in \N$ and some $a_*=(a_1^*,...,a_5^*)\in \R^5$ with $0<\Vert a_*\Vert \leq \ep_0$. On the other hand, for fixed $(\x,\mathrm{g})\in \cU\times SL(2,\R)$ we have that $\Phi(\x,\iota^{-1}(\mathrm{g})):[-\ep_0,\ep_0]^{d_0}\rightarrow (\cU_3\subset\Sph)\times \iota^{-1}(SL(2,\R))$, after choosing $\ep_0$ smaller, if necessary. In particular,
\begin{align}
    D_\varpi\Phi_{\varpi_*}(\x_0,g_0):\R^{d_0}\rightarrow T_{\Phi_{\varpi_*}(\x_0,\iota^{-1}(Id))}(\cU_3\times \iota^{-1}(SL(2,\R)) )
\end{align}
with $\text{dim }T_{\Phi_{\varpi_*}(\x_0,\iota^{-1}(Id))}(\cU_3\times \iota^{-1}(SL(2,\R)) ) = 2+3=5$ as 
\begin{align}
    \iota^{-1}(SL(2,\R)) = \lbrace   h=(h_1,h_2,h_3,h_4)\in \R^4 : h_1h_4-h_2h_3 = 1 \rbrace  .
\end{align}
Since $D_\varpi\Phi_{\varpi_*}(\x_0,g_0)$ has rank at least 5, we conclude that $D_\varpi\Phi_{\varpi_*}(\x_0,g_0)$ is surjective and thus $\Phi_{\varpi}(\x_0,g_0)$ is a submersion at $\varpi=\varpi_*$. In view of the definition of $\Phi_\varpi$, we have from Propositions \ref{prop:blackbox_lyapexponent} and \ref{prop:joint-surj} that the top Lyapunov exponent is positive.
\end{proof}

\section{The two point process}\label{sec:TPP}

We start the discussion relative to the two-point process by introducing the two point vector field
\begin{equation}\label{eq:deftwopointvelfield}
    u^{(2)}_{\a}(\x,\y)=(u^{\a}(\x),u^{\a}(\y)),\quad \x,\y\in\Sph,
\end{equation}
which generates the two-point flow map
\begin{equation}
    f^{(2),\a}_t(\x,\y)=(f^{\a}_t(\x),f^{\a}_t(\y)), \quad t\in\RR.
\end{equation}

The goal of this section is to prove that the two-point process $f_\omega^{(2)}(\x,\y) = (f_\omega(\x), f_\omega(\y))$ is $\mathrm{W}$-geometrically ergodic. It is readily seen that on the product space $\mathsf{X}\times \mathsf{X}$, the diagonal set 
\begin{align}
    \Delta = \lbrace (\x,\y)\in \mathsf{X} \times \mathsf{X} : \x=\y \rbrace
\end{align}
is invariant under $f_\omega^{(2)}$ and thus there does not exist a unique ergodic stationary measure for the two-point process on $\mathsf{X}\times \mathsf{X}$. The following result describes the set of symmetries that are almost-surely invariant for the two-point process.

\begin{proposition}\label{prop:symmetries}
    Let $\a,\b\in \Sph$ be two non-parallel and non-orthogonal axes. The family $\cR\subset SO(3)$ of rotations of the sphere that are invariant under the two-point process $f_\omega^{(2)}$ is $\cR=\lbrace I,\mathcal{R}_{\c}\rbrace$, where $I$ is the identity and $\mathcal{R}_{\c}$ is the rotation of angle $\pi$ around the axis $\c=\frac{\a\times\b}{|\a\times\b|}$.
    Moreover, if $\a\cdot\b=0$, then $\cR=\lbrace   I, \mathcal{R}_{\c}, \mathcal{R}_{\a}, \mathcal{R}_{\b}\rbrace  $.
\end{proposition}

\begin{proof}
    Let $S\in SO(3)$ such that $f_\omega(S\x)=Sf_\omega(\x)$ almost surely in $\omega$ for all $\x\in\Sph$. In particular, due to the randomness of $\omega_1,\omega_2$, we require further that $u^{\a}(S\x)=Su^{\a}(\x)$ and $u^{\b}(S\x)=Su^{\b}(\x)$ for all $\x\in\Sph$.
    This forces $S$ to preserve the zero sets of both $u^\a$ and $u^\b$, hence $S$ must satisfy 
    \begin{equation}
        S\a=\eps_{\a}\a, \quad S\b=\eps_{\b}\b, \qquad \eps_{\a},\eps_{\b}\in\lbrace 1,-1\rbrace  .
    \end{equation}
    The four cases are the following:
    \begin{itemize}
        \item[--] Case $(\eps_{\a},\eps_{\b})=(1,1)$: the two points are fixed and linearly independent, so $S=I$.
        \item[--] Case $(\eps_{\a},\eps_{\b})=(1,-1)$: the point $\a$ is fixed while the point $\b$ is sent to $-\b$ by $S$. This is possible only if $\a\perp\b$ via a $\pi$ rotation around the $\a$ axis. Hence, only if $\a\perp\b$ we have $S=\cR_{\a}$.
        \item[--] Case $(\eps_{\a},\eps_{\b})=(-1,1)$: analogously to the previous case, only if $\a\perp\b$ then $S=\cR_{\b}$.
        \item[--] Case $(\eps_{\a},\eps_{\b})=(-1,-1)$: in this case we exchange both $\a$ and $\b$ to $-\a$ and $-\b$ respectively, in the plane spanned by the two axes. This is a $\pi$ rotation around the $\c=\frac{\a\times\b}{|\a\times\b|}$ axis, hence $S=\cR_{\c}$.
    \end{itemize}
\end{proof}

For the most demanding case $\a=(0,1,0)$ and $\b=(0,0,1)$ we define $\Sigma := \lbrace \sfx, \sfy, \sfz \rbrace $, $\Sigma_\mathsf{D}:= \Sigma\cup \lbrace \mathsf{D} \rbrace$ and the set
\begin{equation}
    \cS:=\bigcup_{\sigma\in \lbrace   \sfx, \mathsf{y}, \sfz, \mathsf{D} \rbrace  } \mathrm{R}_\sigma, \qquad \mathrm{R}_\sigma = \lbrace   (p,q)\in  \mathsf{X}\times \mathsf{X} :   q=\mathcal{R}_\sigma p \rbrace,
\end{equation}
with $\cR_\mathsf{D}$ the identity map. $\cS$ denotes the subset of $\mathsf{X}\times \mathsf{X}$ that is almost-surely invariant under the two-point process while $\mathrm{R}_\sigma$ and $\mathrm{R}_{\sigma'}$ are mutually disjoint, for all $\sigma,\sigma'\in \Sigma_\mathsf{D}$ with $\sigma\neq \sigma'$. 
We now show that the two-point process is $\mathrm{W}$-geometrically ergodic on the phase space $\mathsf{X}^{(2)}=\mathsf{X}\times \mathsf{X}\setminus \cS$ using the abstract Harris Theorem \ref{thm:abstractHarris}.
For this, we first prove that the two-point chain admits an open small set in Proposition \ref{prop:smallsetTPP}. Next, we show the topological irreducibility in Proposition \ref{prop:irredTPP}. Lastly, since the phase space $\mathsf{X}^{(2)}$ is not compact, we construct in Proposition \ref{prop:globaldriftTPP} a function $\widetilde{\mathrm{W}}$ that satisfies a Lyapunov-Foster drift condition for the two-point process.

\subsection{Open small set}
We rely on the LARC framework.
\begin{proposition}\label{prop:LARCTPP}
For non-parallel axes $\a,\b\in\Sph$, there exists a point $(\x_0,\y_0)\in \mathsf{X}^{(2)}$ such that \linebreak $\Lie_{(\x_0,\y_0)}(\lbrace   u^{(2)}_{\a},u^{(2)}_{\b}\rbrace )$ has dimension 4. More precisely, 
\begin{align}
    X^{(2)}_1(\x,\y) &:= u_\a^{(2)}(\x,\y), \quad X^{(2)}_2(\x,\y) := u_\b^{(2)}(\x,\y), \\ 
    X^{(2)}_3(\x,\y) &:= [X^{(2)}_1,X^{(2)}_2](\x,\y), \quad X^{(2)}_4(\x,\y) := [X^{(2)}_1,X^{(2)}_3](\x,\y), 
\end{align}
are linearly independent for some $(\x_0,\y_0)\in \mathsf{X}^{(2)}$.
\end{proposition}

\begin{proof}
Let $w_1:=u^\a$, $w_2:=u^\b$, $w_3:= [w_1,w_2]$ and $w_4=[w_1,w_3]$. First, for $\a=(0,1,0)$ and $\b=(0,0,1)$ we have
\begin{equation}
    w_3(\x)=(x(y^2-z^2),y(z^2-x^2),z(x^2-y^2)),
\end{equation}
and hence
\begin{equation}
    D_\x w_3(\x)=
    \begin{pmatrix}
        y^2-z^2 & 2xy & -2xz \\
        -2xy & z^2-x^2 & 2yz \\
        2xz & -2yz & x^2-y^2
    \end{pmatrix}.
\end{equation}
With this, we compute
\begin{equation}
    w_4(\x) =D_\x w_3(\x)u^\a(\x)-D_\x u^\a(\x)w_3(\x)=-(2yz(2z^2-1),4xy^2z,2xy(2x^2-1)).
\end{equation}
At $\x_0=(\frac{\sqrt{3}}{2},0,\frac12)$ and $\y_0=(\frac34,\frac{\sqrt{3}}{4},\frac12)$ we have $(\x_0,\y_0)\in \mathsf{X}^{(2)}$ from which
\begin{align}
    X^{(2)}_1(\x_0,\y_0)&= \left( 0,0,0, \frac{\sqrt{3}}{8},0,-\frac{3\sqrt{3}}{16} \right) ,\\
    X^{(2)}_2(\x_0,\y_0)&=\left(0,\frac{\sqrt{3}}{4},0 , -\frac{\sqrt{3}}{8},\frac38,0\right),\\
    X^{(2)}_3(\x_0,\y_0)&=\left(-\frac{\sqrt{3}}{8},0,\frac38 , -\frac{3}{64},-\frac{5\sqrt{3}}{64},\frac{3}{16}\right),\\
    X^{(2)}_4(\x_0,\y_0)&=\left( 0,0,0 , \frac{\sqrt{3}}{8},-\frac{9}{32},-\frac{3\sqrt{3}}{64}\right).\\
\end{align}
These vectors are linearly independent: if $\sum_{i=1}^4 c_i X^{(2)}_i = 0$, from the first three components we first observe that First, $c_2=c_3=0$. Then, as $X^{(2)}_4$ has non-zero fifth component we deduce that $c_4=0$ and finally $c_1 = 0$.

For the case of non-orthogonal axes, recall that we may assume without loss that $\a=(0,a_2,a_3)$ with $a_2,a_3 \neq 0$ and $a_2^2+a_3^2=1$ while $\b = e_3$. For $\x_0=(\frac{\sqrt{3}}{2},-\frac{a_3}{2},\frac{a_2}{2})$ we now set $\y_0=\left( \frac{1}{\sqrt{2}}, \frac{1}{\sqrt{2}}, 0 \right)$. There holds $(\x_0,\y_0)\in \mathsf{X}^{(2)}$ and we have
\begin{align}
    X^{(2)}_1(\x_0,\y_0)&= \left( 0,0,0, -\frac{a_2a_3}{2}, \frac{a_2a_3}{2}, -\frac{a_2^2}{2} \right) ,\\
    X^{(2)}_2(\x_0,\y_0)&=\left( \frac{a_2a_3}{4}, \frac{a_2\sqrt{3}}{4}, 0, 0 , 0, 0\right),\\
    X^{(2)}_3(\x_0,\y_0)&=\left( -\frac{a_2^2\sqrt{3}}{8}, -\frac{3a_2^2a_3}{8}, \frac{3a_2^3}{8}, a_2^2\frac{\sqrt{2}}{4}, -a_2^2\frac{\sqrt{2}}{4}, 0 \right),\\
    X^{(2)}_4(\x_0,\y_0)&=\left( 0,0,0 , -\frac{3a_2^3a_3}{4}, \frac{3a_2^3a_3}{4}, 0 \right).\\
\end{align}
These vectors are linearly independent as well: If $\sum_{i=1}^4 c_i X^{(2)}_i= 0$, since $X^{(2)}_3$ is the only vector with non-trivial third component we have that $c_3=0$. Next, we deduce that $c_2=0$ since $a_2\neq 0$ and both $X^{(2)}_1$ and $X^{(2)}_4$ have trivial projection onto the first three components. Similarly, $c_1=0$ since the sixth component of $X^{(2)}_1$ is non-zero and lastly $c_4=0$ because $X^{(2)}_4\neq 0$ as we assumed that $a_3\neq 0$.
\end{proof}

The existence of an open small set for the two-point process is obtained from the LARC.
For $u\in \mathfrak{X}(\Sph)$, let $\phi_t^u$ denote the flow map of $u$ on $\Sph$ and let $\overline{u}$ denote the smooth radial extension of $u$ in an open spherical shell $U_0\subset\R^3$ containing $\Sph$. For the sake of simplicity, we still denote it $u$. We next define the extension $v^{(2)}_u:U_0\times U_0\rightarrow \R^3\times\R^3$ by
\begin{align}
    v^{(2)}_u(\x,\y) = (u(\x),u(\y)),
\end{align}
for all $\x,\y\in U_0$. Let $\Phi^{(2),u}_t(\x,\y):U_0\times U_0\rightarrow U_0\times U_0$ be the flow map defined by $v^{(2)}_u$, that is
\begin{align}
    \begin{cases}
        \frac{\d}{\d t} \Phi^{(2),u}_t(\x,\y) =  v_u^{(2)}(\Phi^{(2),u}_t(\x,\y)), & \\
        \Phi^{(2),u}_0(\x,\y) = (\x,\y) &
    \end{cases}
\end{align}
and observe that $\Phi^{(2),u}_t(\x,\y)= (\phi_t^u(\x),\phi_t^u(\y))$ for all $\x,\y\in \Sph$ and all $t\in \R$.

\begin{proposition}\label{prop:smallsetTPP}
    The two-point process admits an open small set.
\end{proposition}

\begin{proof}
    We shall use Lemma \ref{lemma:blackboxsmallset}. For this, let $\cV = \lbrace   v_{u^\a}^{(2)}, v_{u^\b}^{(2)} \rbrace  $ and $\Phi^{(2),u^\a}_t$, $\Phi^{(2),u^\b}_t$, their associated flow maps. From Proposition \ref{prop:LARCTPP} we obtain that $\lbrace   X^{(2)}_i(\x,\y) \rbrace_{i=1}^4$ is a set of linearly independent vector fields at some $(\x_0,\y_0)$. Thus, from Proposition \ref{prop:word-construction} we derive the existence of a word $\lbrace   \varpi \rbrace_{i=1}^4$ and $\ep_0>0$ such that $D_\varpi \Phi^{(2)}_{\varpi_*}(\x_0,\y_0)$ has rank at least 4, for $\varpi_* = (\varpi_1(a_1^*),..., \varpi_4(a_4^*))\in [-\ep_0, \ep_0]^{d_0}$, for some $d_0>1$ and all $a_*=(a_1^*,...,a_4^*)$ with  $0 < \Vert a_* \Vert < \ep_0$. Additionally, since $\Phi_t^{(2),u}(\x_0,\y_0)\in  \mathsf{X}^{(2)}$ for all $t\in \R$, for $u\in \lbrace   u^\a, u^\b \rbrace  $, we have that
    \begin{align}
        D_\varpi \Phi^{(2)}_{\varpi_*}(\x_0,\y_0) :\R^{d_0}\rightarrow T_{\Phi^{(2)}_{\varpi_*}(\x_0,\y_0)}  \mathsf{X}^{(2)} 
    \end{align}
    has rank at least 4. Since $\text{dim }T_{\Phi^{(2)}_{\varpi_*}(\x_0,\y_0)} \mathsf{X}^{(2)} = 4$, we conclude that $\Phi^{(2)}_{\varpi}(\x_0,\y_0)$ is a submersion at $\varpi=\varpi_*$ and Lemma \ref{lemma:blackboxsmallset} yields the existence of an open small set.
\end{proof}

\subsection{Topological irreducibility}
Since the set $\cS$ is composed of mutually disjoint subsets that are each invariant under the two-point process, we prove the topological irreducibility of the two-point chain away from $\cS$. We argue for the orthogonal case $\a=(0,1,0)$ and $\b=(0,0,1)$.

\begin{proposition}\label{prop:irredTPP}
The two-point process is topologically irreducible in $\mathsf{X}^{(2)}$.
\end{proposition}

\begin{proof}
    Let $(q_1, q_2)\in \mathsf{X}^{(2)}$ and $(p_1,p_2)\in  \mathsf{X}^{(2)}$. For all $\varepsilon>0$ we shall find a sequence of times $\omega_n$ such that $(f_{\omega_n}(p_1), f_{\omega_n}(p_2))\in (B_{\ep}(q_1), B_{\ep}(q_2))$. We shall achieve this approximate controllability in six steps, taking advantage of the stagnation sets of each $u^\a$ and $u^\b$.

    \noindpar{Step 1} Since the one-point process is exactly controllable, we can find $\omega^1 =(\omega^1_{1},\omega^1_{2})$ such that $\mathsf{p}_1 := f_{\omega^1}(p_1) = \frac{1}{\sqrt{2}}(0,1,1)$. Given $\mathsf{p}_2:=f_{\omega^1}(p_2)$, we denote $\mathsf{p}_2=(x,y,z)$.

    \noindpar{Step 2} Let $\omega_S=\frac{\pi\sqrt{2}}{2}$ and let $\mathsf{p}_j^S= f_{\omega_S}^2(\mathsf{p}_j)$ for $j=1,2$. We have that
    \begin{align}
        \mathsf{p}_1^S = \frac{1}{\sqrt{2}}\begin{pmatrix}
            1 \\ 1 \\ 0
        \end{pmatrix}, \quad \mathsf{p}_2^S = \begin{pmatrix}
        {x}\cos(\omega_S{y}) + {z}\sin(\omega_S{y}) \\
        {y} \\
        -{x}\sin(\omega_S{y}) + {z}\cos(\omega_S{y}) 
        \end{pmatrix}
    \end{align}
    Similarly, let $\omega_N=\frac{3\pi\sqrt{2}}{2}$ and $\mathsf{p}_j^N= f_{\omega_N}^2(\mathsf{p}_j)$ for $j=1,2$. There now holds
    \begin{align}
        \mathsf{p}_1^N = \frac{1}{\sqrt{2}}\begin{pmatrix}
            -1 \\ 1 \\ 0
        \end{pmatrix}, \quad \mathsf{p}_2^N = \begin{pmatrix}
        {x}\cos(\omega_N{y}) + {z}\sin(\omega_N{y}) \\
        {y} \\
        -{x}\sin(\omega_N{y}) + {z}\cos(\omega_N{y}) 
        \end{pmatrix}.
    \end{align}
    
    \noindpar{Step 3} We now claim that since $(p_1,p_2)\in \mathsf{X}^{(2)}$, we have either $(\mathsf{p}_2^S)_z\neq 0$ or $(\mathsf{p}_2^N)_z \neq 0$. Indeed, assume towards a contradiction that both are zero. Then, for $\alpha = \omega_S {y}$ we have $\omega_N{y} = 3\alpha$ and
    \begin{align}
        \begin{pmatrix}
            -\sin(\alpha) & \cos(\alpha) \\
            -\sin(3\alpha) & \cos(3\alpha)
        \end{pmatrix} \begin{pmatrix}
            {x} \\ {z}
        \end{pmatrix} = 0.
    \end{align}
The determinant of the above linear system is $\sin(3\alpha)\cos(\alpha) - \sin(\alpha)\cos(3\alpha) =\sin(2\alpha)$.

\diampar{Non-zero determinant} We thus obtain the trivial solution ${x} = {z} = 0$, which forces ${y}=\pm 1$, namely $\mathsf{p}_2\not \in \mathsf{X}$. 

\diampar{Zero determinant} We thus have $2\alpha = \pi k$ for some $k\in \Z$, namely
\begin{align}
    {y} = \frac{k}{\sqrt{2}}
\end{align}
for some $k\in \Z$. Since $\mathsf{p}_2\in \mathsf{X}$, we have ${y}\in(-1,1)$ so that $|k|\leq 1$.
\begin{itemize}
    \item[--] If $k=0$, this gives ${y} = 0$, so that $\mathsf{p}_2 =\mathsf{p}_2^S = \mathsf{p}_2^N$. If ${z} =0$, we then see that $|{x}| = 1$ and thus $\mathsf{p}_2\not \in \mathsf{X}$.
    \item[--] If $k=\pm 1$, we then find that ${y} = \pm \frac{1}{\sqrt{2}}$ and thus $\sin(\omega_S{y})= -\sin(\omega_N{y}) = \pm 1$. In particular,
    \begin{align}
        (\mathsf{p}_2^S)_z = \mp {x}, \quad (\mathsf{p}_2^N)_z = \pm {x}
    \end{align}
    from which we deduce that ${x}=0$. Hence, we obtain that ${z} = \pm\frac{1}{\sqrt{2}}$ and thus
    \begin{align}
        \mathsf{p}_2^S,\mathsf{p}_2^N \in \left\lbrace   \frac{1}{\sqrt{2}}(1,1,0), \frac{1}{\sqrt{2}}(-1,1,0),  \frac{1}{\sqrt{2}}(1,-1,0),  \frac{1}{\sqrt{2}}(-1,-1,0)\right\rbrace  .
    \end{align}
    We shall argue for $\mathsf{p}_2^S$. Now, the first possibility gives $\mathsf{p}_2^S=\mathsf{p}_1^S$, a contradiction with $(\mathsf{p}_1^S,\mathsf{p}_2^S)\not \in \Delta$. The second option gives a contradiction with $(\mathsf{p}_1^S,\mathsf{p}_2^S)\not \in \mathrm{R}_\sfz$, the third possibility contradicts $(\mathsf{p}_1^S,\mathsf{p}_2^S)\not \in \mathrm{R}_\mathsf{y}$ and the last possibility contradicts $(\mathsf{p}_1^S,\mathsf{p}_2^S)\not \in \mathrm{R}_\sfx$.
\end{itemize}
Hence, we have that either $(\mathsf{p}_2^S)_z\neq 0$ or $(\mathsf{p}_2^N)_z \neq 0$. We shall assume from now on that $(\mathsf{p}_2^S)_z\neq 0$ and we set $\widetilde{\mathsf{p}}_1=\mathsf{p}_1^S$ and $\widetilde{\mathsf{p}}_2=\mathsf{p}_2^S$.

\noindpar{Step 4} Thanks to the exact controllability of the one-point process, there exists $\omega^0 =(\omega^0_{1},\omega^0_{2})$ such that $f_{\omega^0}(\widetilde{\mathsf{p}}_1)=q_1$. Let $\mathsf{q}_2 = f_{\omega^0}^{-1}(q_2)$ denote the pre-image of $q_2$ by the map $f_{\omega^0}$. By definition, it is such that $f_{\omega^0}(\mathsf{q}_2) = q_2$ and by continuity of $f_{\omega^0}$, for all $\varepsilon>0$ there exists $\delta>0$ such that $f_{\omega^0}(B_\delta(\mathsf{q}_2)) \subset B_{\varepsilon}(q_2)$. We remark here that $f^{-1}_{\omega_0} = (f_{\omega_1^0}^3)^{-1}\circ (f_{\omega_2^0}^2)^{-1}$ where for fixed $\omega_1^0$ and $\omega_2^0$ both $f_{\omega_1^0}^3$ and $f_{\omega_2^0}^2$ are invertible.

\noindpar{Step 5} We shall now approximate $\mathsf{q}_2$ from $\widetilde{\mathsf{p}}_2$. 
Since $(\widetilde{\mathsf{p}}_2)_z\neq 0$, using the rotation dynamics of $f^3$ around the $z$-axis, there exists $\omega_1^2\in \R$ such that for $\omega^2=(\omega_1^2,0)$ we have $2\pi\sqrt{2}(f_{\omega^2}(\widetilde{\mathsf{p}}_2))_y\not\in \Q$. Note that since $(\widetilde{\mathsf{p}}_1)_z =0$, we preserve  $f_{\omega^2}(\widetilde{\mathsf{p}}_1)=\mathsf{p}_1^2$.

\diampar{Case 1} Assume first that $(\mathsf{q}_2)_z\neq 0$ and assume further that $(\mathsf{q}_2)_z>0$. Since $2\pi\sqrt{2}(f_{\omega^2}(\widetilde{\mathsf{p}}_2))_y\not\in \Q$, for $\omega^3_2= 2\pi\sqrt{2}$ and $\omega^3 = (0,\omega_2^3)$ we have $f_{\omega^3}^n(\widetilde{\mathsf{p}}_1) = \widetilde{\mathsf{p}}_1$, for all $n\geq 1$, while
\begin{align}
    0<(\mathsf{q}_2)_z-\delta/100 < (f_{\omega^3}^m\circ f_{\omega^2}(\widetilde{\mathsf{p}}_2))_z < (\mathsf{q}_2)_z
\end{align}
for some $m\geq 1$ and $\delta/100< (\mathsf{q}_2)_z/2$ small enough. Let $\overline{p}_2 = f_{\omega^3}^m\circ f_{\omega^2}(\widetilde{\mathsf{p}}_2)$, then, for $\delta>0$ smaller if necessary, there exists $\omega^4=(\omega_1^4,0)$ such that 
\begin{align}
    (f_{\omega^4}(\overline{\mathsf{p}}_2))_y\in B_{\delta/25}((\mathsf{q}_2)_y), \quad (f_{\omega^4}(\overline{\mathsf{p}}_2))_x\in B_{\delta/25}((\mathsf{q}_2)_x)
\end{align}
Note that since $(\widetilde{\mathsf{p}}_1)_z =0$, we preserve  $f_{\omega^4}(\widetilde{\mathsf{p}}_1)=\mathsf{p}_1^2$. 

\diampar{Case 2} Assume next that $(\mathsf{q}_2)_z=0$. Then, because $2\pi\sqrt{2}(f_{\omega^2}(\widetilde{\mathsf{p}}_2))_y\not\in \Q$, for $\omega^3_2= 2\pi\sqrt{2}$ and $\omega^3 = (0,\omega_2^3)$ we have $f_{\omega^3}^n(\widetilde{\mathsf{p}}_1) = \widetilde{\mathsf{p}}_1$, for all $n\geq 1$ and
\begin{align}
    0<\delta/200 < |(f_{\omega^3}^m \circ f_{\omega^2}(\widetilde{\mathsf{p}}_2))_z| <  \delta/100
\end{align}
for some $m\geq 1$. Hence, for $\overline{p}_2 = f_{\omega^3}^m\circ f_{\omega^2}(\widetilde{\mathsf{p}}_2)$, there exists $\omega^4=(\omega_1^4,0)$ such that
\begin{align}
    (f_{\omega^4}(\overline{\mathsf{p}}_2))_y\in B_{\delta/25}((\mathsf{q}_2)_y), \quad (f_{\omega^4}(\overline{\mathsf{p}}_2))_x\in B_{\delta/25}((\mathsf{q}_2)_x)
\end{align}
and we still preserve  $f_{\omega^4}(\widetilde{\mathsf{p}}_1)=\mathsf{p}_1^2$ as $(\widetilde{\mathsf{p}}_1)_z =0$.

\noindpar{Step 6} As a result, $f_{\omega^4}(\overline{\mathsf{p}}_2)\in B_\delta (\mathsf{q}_2)$. We now apply $f_{\omega^0}$ to both $\widetilde{\mathsf{p}}_1$ and  $f_{\omega^4}(\overline{\mathsf{p}}_2)$, thus obtaining $f_{\omega_0}(\widetilde{\mathsf{p}}_1) = q_1$ and $f_{\omega^0}\circ  f_{\omega^4}(\overline{\mathsf{p}}_2) \in B_{\varepsilon}(q_2)$. With this, the proof of the proposition is complete.
\end{proof}

\subsection{Drift function for the two-point process}

We have seen in Proposition \ref{prop:symmetries} above that the symmetry set $\cS$ is almost-surely invariant under the two-point process and the measure supported on $\cS$ is invariant. Thus, $P^{(2)}$ cannot be uniformly geometrically ergodic in $\mathsf{X}\times \mathsf{X}$ with respect to $\pi^{(2)}$. Instead, we work in the non-compact phase space $\mathsf{X}^{(2)}$, which requires the introduction of a Lyapunov function $\widetilde{\mathrm{W}}$ that satisfies a Lyapunov-Foster drift condition in order to use the abstract Harris Theorem \ref{thm:abstractHarris}.

To construct this Lyapunov function, we note that both the one-point and projective processes are uniformly geometrically ergodic. Moreover, we have further shown in Proposition \ref{prop:positivetoplyap} that the top Lyapunov exponent is positive. Therefore, appealing to Proposition \ref{prop:blackboxdriftTPP}, there exists $\xi,\gamma \in(0,1)$, $s_*>0$ and $\psi: X\times X\setminus\Delta\rightarrow [1,\infty)$ such that
\begin{align}
    W(p,q) := d_\Sph(p,q)^{-\xi}\psi(p,q)
\end{align}
satisfies
\begin{align}
    \mathbb{E}\left[ W(f_\omega(p), f_\omega(q)) \right] \leq \gamma W(p,q),
\end{align}
for all $(p,q)\in \Delta(s_*)=\lbrace (p,q)\in \mathsf{X}\times \mathsf{X} :  0 < d_\Sph(p,q) < s_* \rbrace$. We next define
\begin{align}
    W_\sigma(p,q) := W(p,\mathcal{R}_\sigma q) 
\end{align}
for $\sigma\in \Sigma$, so that we likewise have
\begin{align}
    \mathbb{E}\left[ W_\sigma(f_\omega(p), f_\omega(q)) \right] \leq \gamma W_\sigma(p,q).
\end{align}
provided that $0 < d_\Sph(p,\mathcal{R}_\sigma q) < s_*$. For $p,q\in \mathsf{X}$, let 
\begin{align}
    \tilde r_\sigma(p,q) := d_{\Sph}(p,\mathcal{R}_\sigma q)
\end{align}
for $\sigma \in \Sigma_\mathsf{D}$ and further define the local Lyapunov function $\mathrm{W}$ for the two-point process by
\begin{align}
    \mathrm{W}(p,q) &:= \chi (\tilde r_\mathsf{D}(p,q))(1-\chi(\tilde r_\sfx(p,q))) (1-\chi(\tilde r_\mathsf{y}(p,q)))(1-\chi(\tilde r_\sfz(p,q)))W(p,q) \\
    &\quad + \chi (\tilde r_\sfx(p,q)) (1-\chi(\tilde r_\mathsf{y}(p,q)))(1-\chi(\tilde r_\sfz(p,q)))W_\sfx(p,q) \\
    &\quad + \chi (\tilde r_\mathsf{y}(p,q))(1-\chi(\tilde r_\sfz(p,q)))W_\mathsf{y}(p,q) \\
    &\quad + \chi(\tilde r_\sfz(p,q))W_\mathsf{z}(p,q),
\end{align}
where $\chi$ is the bump function introduced for the one-point drift function. To construct a function $\widetilde{\mathrm{W}}$ that satisfies the Lyapunov-Foster drift condition for the two-point chain, we note that any compact subset $\widetilde{K}$ of $\mathsf{X}^{(2)}$ avoids not only the invariant set $\cS$, but also avoids the fixed points in both coordinates, $F\times \mathsf{X} \cup \mathsf{X} \times F$. In particular, away from any such compact set $\widetilde{K}$ we shall prove a contraction of $\widetilde{\mathrm{W}}$ in expectation. 
A priori, no such behaviour for $\mathrm{W}$ is guaranteed for points close to  $F\times \mathsf{X} \cup \mathsf{X} \times F$, in contrast with the drift achieved by $\mathrm{W}$ near $\cS$ (combine its definition and Proposition~\ref{prop:blackboxdriftTPP}). To overcome this obstruction, we use the drift function $\mathrm{V}$ for the one-point process and we prove the following.

\begin{proposition}\label{prop:globaldriftTPP}
There exists $\nu >0$ small enough such that 
\begin{align}\label{eq:globaldriftTPP}
    \widetilde{\mathrm{W}}(p,q) := \max \lbrace \mathrm{W}(p,q), \nu \mathrm{V}(p), \nu \mathrm{V}(q) \rbrace
\end{align}
satisfies the Lyapunov-drift condition. More precisely, there exists $\gamma\in(0,1)$, $\beta>0$, $s_0>0$, $r_0 \leq s_0/2$ and $\ep>0$ such that for the compact set 
    \begin{align}
        \widetilde K = \lbrace   (p,q)\in \mathsf{X}^{(2)} : p, q \in K(\ep) \text{ and } \tilde r_\sigma(p,q) \geq r_0, \text{ for all }\sigma\in \Sigma_\sfD \rbrace  
    \end{align}
there holds
    \begin{align}
        \mathbb{E}\left[ \widetilde{\mathrm{W}}(f_\omega(p), f_\omega(q)) \right] \leq \gamma \widetilde{\mathrm{W}}(p,q) + \beta \mathbf{1}_{\widetilde{K}}(p,q)
    \end{align}
    for all $(p,q)\in \mathsf{X}^{(2)}$.
\end{proposition}

\begin{proof}
Fix $s_0>0$ sufficiently small, let $r_0<\frac{s_0}{2}$ small enough such that $f_\omega (p)\in \mathsf{Q}_\sigma(s_0/2)$ whenever $p\in \mathsf{Q}_\sigma(r_0)$ , for all $\sigma\in \Sigma$, and $\tilde r_\sigma(f_\omega(p),f_\omega(q)) < s_0$ whenever $\tilde r_\sigma(p,q) < r_0$, for all $\sigma\in \Sigma_\sfD$. Moreover, let $\eta\in (0,r_0)$ and $\ep\in (0,\eta)$. We first show that 
\begin{align}\label{eq:localdrifttildermW}
        \mathbb{E}\left[ \widetilde{\mathrm{W}}(f_\omega(p), f_\omega(q)) \right] \leq \gamma \widetilde{\mathrm{W}}(p,q)
    \end{align}
for all $(p,q)\not\in \widetilde{K}$, for $\gamma\in(0,1)$ given the maximum contraction constant of the local drift functions $W$ and $\mathrm{V}$.

\diampar{Case 1} Assume $p,q\in K^c(\eta)$, that is $p\in \mathsf{Q}_{\sigma_p}(\eta)$ and $q\in \mathsf{Q}_{\sigma_q}(\eta)$. 
    \begin{indentblock}[1]
    \diampar{Case 1.1} Assume further that $\sigma_p\neq \sigma_q$. Then, $\widetilde{r}_\mathsf{D}(p,q)\geq 1$ for $\eta$ small enough and, $\mathcal{R}_{\sigma'} p\in \mathsf{Q}_{\sigma_p}(\eta)$ for all $\sigma'\in \Sigma$. In particular, $\widetilde{r}_{\sigma'}(p,q)\geq 1$ for all $\sigma'\in \Sigma$, for $\eta$ small enough. Further choosing $\eta$ smaller in terms of $N$ if necessary, we also have $\widetilde{r}_{\sigma}(f_\omega(p), f_\omega(q))\geq \frac12$ for all $\sigma\in \Sigma_{\mathsf{D}}$. For $s_0$ small enough, we have that $\mathrm{W}(f_\omega(p),f_\omega(q))=\mathrm{W}(p,q)=0$ and thus
    \begin{align}
        \EE \left[ \widetilde{\mathrm{W}}(f_\omega(p),f_\omega(q)) \right] = \nu\EE \left[ \max\lbrace \mathrm{V}(f_\omega(p)), \mathrm{V}(f_\omega(q)) \rbrace \right] \leq \nu \gamma \max\lbrace \mathrm{V}(p), \mathrm{V}(q) \rbrace = \gamma \widetilde{\mathrm{W}}(p,q).
    \end{align}
    \end{indentblock}

    \begin{indentblock}[1]
    \diampar{Case 1.2} Suppose instead that $\sigma_p=\sigma_q$. For $\eta<r_0$ small enough we either have $\widetilde{r}_\mathsf{D}\leq r_0$ or $\widetilde{r}_\mathsf{D}\geq 1$.
    \end{indentblock}
    
    \begin{indentblock}[2]
    \diampar{Case 1.2.1} If $\tilde r_\mathsf{D}(p,q) \leq r_0$ then $p,q$ belong to the same connected component of $\mathsf{Q}_\sigma(\eta)$, they are close to the same fixed point $p_F\in F $ and $\tilde r_{\tilde\sigma} (p,q) \geq 1$ for $\tilde\sigma \neq \sigma, \mathsf{D}$. Hence, for $r_0$ small enough, we further have that $\tilde r_{\tilde\sigma} (f_\omega(p),f_\omega(q)) \geq \frac12$, for all $\omega\in[-N,N]^2$, for $\tilde\sigma \neq \sigma,  \mathsf{D}$. Moreover, 
    \begin{align}
        \tilde r_\sigma (f_\omega(p), f_\omega(q)) = d_\Sph(f_\omega(q), \mathcal{R}_\sigma f_\omega(p)) &\leq d_\Sph(f_\omega(q), p_F) + d_\Sph(p_F, \mathcal{R}_\sigma f_\omega(p)) \\
        &= d_\Sph(f_\omega(q), p_F) + d_\Sph(p_F,  f_\omega(p)) \\
        &\leq r_0
    \end{align}
    for $\eta$ small enough so that $f_{\omega}(p),f_{\omega}(q)\in \mathsf{Q}_\sigma(r_0/2)$ for all $\omega\in[-N,N]$ whenever $p,q\in \mathsf{Q}_\sigma(\eta)$. Therefore, 
    \begin{align}
        \mathrm{W}(f_\omega(p), f_\omega(q)) = W_\sigma(f_\omega(p),f_\omega(q))
    \end{align}
    for all $\omega\in[-N,N]^2$ and thus we obtain \eqref{eq:localdrifttildermW}.
    \end{indentblock}

    \begin{indentblock}[2]
    \diampar{Case 1.2.2} If $\tilde r_\mathsf{D}(p,q)\geq 1$ then $\tilde r_\mathsf{D}(f_\omega(p), f_\omega(q)) \geq \frac12$ for $\eta$ and $r_0$ small enough. In particular, $f_\omega(p)$ and $f_\omega(q)$ belong to different connected components of $\mathsf{Q}_\sigma(r_0)$, each close to $p_F\in F$ and $q_F\in F$, respectively, and thus $\tilde r_\sigma(f_\omega(p), f_\omega(q))\geq \frac14$ as well, for all $\omega\in[-N,N]^2$. Since $q_F= \mathcal{R}_{\tilde\sigma} p_F$ for $\tilde\sigma \neq \sigma$, we have
    \begin{align}
        \tilde r_{\tilde\sigma} (f_\omega(p), f_\omega(q)) \leq r_0,
    \end{align}
    for $\tilde\sigma\neq \sigma$, $\tilde\sigma\neq \mathsf{D}$, for all $\omega\in[-N,N]^2$ and thus ${\mathrm{W}}(f_\omega(p),f_\omega(q)) = W_{\tilde\sigma}(f_\omega(p),f_\omega(q))$, for all $\omega\in[-N,N]^2$, for some $\tilde\sigma\in \lbrace   \sfx, \mathsf{y}, \sfz \rbrace  $, with $\tilde\sigma\neq\sigma$. In particular, we conclude that \eqref{eq:localdrifttildermW} holds.
    \end{indentblock}
    \diampar{Case 2} Assume now that $p\in K^c(\eta)$ and $q\in K(\eta)$. Hence, $p\in \mathsf{Q}_{\sigma_p}(\eta)$  for some $\sigma_p\in \Sigma$ and $q\in \mathsf{Q}^c_{\sigma}(\eta)$, for all $\sigma \in \Sigma$. 

    \begin{indentblock}[1]
    \diampar{Case 2.1} Assume there exists $\tilde\sigma \in \Sigma_\sfD$ such that $\tilde r_{\tilde\sigma}(p,q) < r_0$. Then, $\tilde r_{\tilde\sigma}(f_\omega(p),f_\omega(q)) \leq s_0$ for all $\omega\in[-N,N]^2$ and there exists $\varpi\in \Sigma_D$ such that $\mathrm{W}(f_\omega(p),f_\omega(q)) = W_\varpi(f_\omega(p), f_\omega(q))$ for all $\omega\in[-N,N]^2$. Hence, $\mathbb{E}\left[ \mathrm{W}(f_\omega(p),f_\omega(q)) \right] \leq \gamma \mathrm{W}(p,q)$. Next, let
    \begin{align}
        \beta_0 = \max _{\sigma\in \Sigma} \lbrace   \mathrm{V}(f_\omega(q)) : q\in \mathsf{Q}^c_\sigma(\eta), \omega\in[-N,N]^2 \rbrace  \sim \eta^{-\alpha} 
    \end{align}
    and note that $\beta_0 \leq C_N\eta^{-\frac18}$, for $\eta$ small enough in view of Lemma \ref{lemma:phikappaLip}. Therefore,
\begin{align}
    \EE [\nu\mathrm{V}(f_\omega(q))]\leq \nu C_N \eta^{-\frac18} \leq \gamma W(p,q) 
\end{align}
once we set $\nu=\frac14 \frac{\gamma\eta^\frac18  \min \psi}{r_0^\xi C_N}$. Consequently, $\mathbb{E}\left[ \widetilde{\mathrm{W}}(f_\omega(p), f_\omega(q)) \right] \leq \gamma \widetilde{\mathrm{W}}(p,q)$ as well.
\end{indentblock}

\begin{indentblock}[1]
\diampar{Case 2.2} On the other hand, if $\tilde r_{\tilde\sigma}(p,q) \geq r_0$ for all $\widetilde{\sigma}\in \Sigma_\sfD$ and $(p,q)\not\in \widetilde{K}$ we must have $p\in \mathsf{Q}_{\sigma_p}(\ep)$, for some $\sigma_p\in \Sigma$ as $q\in K(\eta) \subset K(\ep)$. Hence, set
\begin{align}\label{eq:defrho0}
        \rho_0 := \min_{ \sigma\in \lbrace   \mathsf{D}, \sfx, \mathsf{y}, \sfz \rbrace  ,} \lbrace   \tilde r_\sigma(f_\omega(p),f_\omega(q)) : \omega\in [-N,N]^2, \tilde r_\sigma(p,q)\geq r_0 \rbrace  .
    \end{align}
and observe that $\mathrm{W}(f_\omega(p),f_\omega(q))\leq \rho_0^{-\xi}\max\psi$. Recall also that since $q\in \mathsf{Q}_\sigma(\eta)^c$ we have
\begin{equation}\nu\mathrm{V}(f_\omega(q))\leq \nu C_N\eta^{-\alpha}\leq \gamma r_0^{-\xi}\min\psi.
\end{equation}
In particular,
\begin{align}
    \max \lbrace \rho_0^{-\xi}\max\psi, \gamma r_0^{-\xi}\min\psi \rbrace \leq \gamma \nu V(p) \leq \gamma \widetilde{\mathrm{W}}(p,q)
\end{align}
provided that 
\begin{align}
    \ep \leq \left( \frac{\nu}{C_N\max \lbrace \rho_0^{-\xi}\max\psi, \gamma r_0^{-\xi}\min\psi \rbrace}\right) ^{\frac{1}{\alpha}}
\end{align}
as $\mathrm{V}(p)\geq C_N\ep^{-\alpha}$ for some universal $C_N>0$. Together with $\EE [\mathrm{V}(f_\omega(p))] \leq \gamma \mathrm{V}(p)$, we reach \eqref{eq:localdrifttildermW}.
\end{indentblock}

\diampar{Case 3} Assume finally that $p,q\in K(\eta) \subset K(\ep)$. Hence, $p,q\in \mathsf{Q}_\sigma^c(\eta)$ for all $\sigma\in \Sigma$ and we must then have $\tilde r_\sigma (p,q) < r_0$, for some $\sigma\in \Sigma_\mathsf{D}$. Therefore, $\tilde r_\sigma( f_\omega(p), f_\omega(q)) < s_0$ for all $\omega\in [-N,N]^2$ and there exists $\sigma'\in \Sigma_\mathsf{D}$ such that $\mathrm{W}(f_\omega(p),f_\omega(q)) = W_{\sigma'}(f_\omega(p),f_\omega(q))$, for all $\omega\in [-N,N]^2$. Thus, $\mathbb{E}\left[ {\mathrm{W}}(f_\omega(p),f_\omega(q)) \right] \leq \gamma \mathrm{W}(p,q)$. Moreover, recall that $\max \lbrace \nu\mathrm{V}(f_\omega(p)), \nu \mathrm{V}(f_\omega(q)) \rbrace \leq \nu C_N\eta^{-\alpha}\leq \gamma W(p,q)$. As a result, \eqref{eq:localdrifttildermW} holds.

\noindpar{End of proof} For $(p,q)\in \widetilde{K}$, we have that $\mathrm{W}(f_\omega(p),f_\omega(q)) \leq \rho^{-\xi}\max\psi$ and $\mathrm{V}(f_\omega(p)) \leq C_N \ep^{-\alpha}$ and similarly for $\mathrm{V}(f_\omega(q))$, for all $\omega\in[-N,N]^2$ in view of Lemma \ref{lemma:phikappaLip}. Hence, $\EE[\widetilde{\mathrm{W}}(f_\omega(p), f_\omega(q))]\leq \beta$, for some $\beta>0$.
\end{proof}

\section{The stochastic two-point process}\label{sec:stochTPP}
In this section we derive the main estimates for the stochastic flow map $f_\omega^\kappa= f_{\omega_2}^{\a,\kappa}\circ f_{\omega_1}^{\b,\kappa}$ that are later used to prove the existence of a $\kappa$-independent open small set and a drift function for the stochastic two-point process. Such function is defined in Section \ref{sec:driftstochTPP} and combines a local drift function $\mathrm{W}_m$ for the stochastic two-point process near the diagonal and the symmetry sets, together with a suitably modified drift function $\mathrm{V}_\kappa$ for the one-point process. In the remainder of the manuscript we assume without loss of generality that the two rotation axes are orthogonal and correspond as usual to $\a=(0,1,0)$ and $\b=(0,0,1)$.

The local drift function $\mathrm{W}_m$ for the stochastic two-point process reduces to the local drift function $\mathrm{W}$ for the deterministic two-point process, which satisfies a Lyapunov-drift condition for the deterministic $P^{(2)}$ Markov transition kernel. Using perturbative arguments and stochastic stability results, we show in Section~\ref{ssec:localstochTPP} that the local drift function $\mathrm{W}$ also satisfies a Lyapunov-drift condition for the stochastic $P^{(2)}_\kappa$ Markov transition kernel near the diagonal, for all $\kappa>0$ sufficiently small.

On the other hand, the drift function $\mathrm{V}_\kappa$ for the one-point process localises on each fixed point of the deterministic dynamics. Introducing a cutoff at scale $\sqrt{\kappa}$, which corresponds to the size of the effective dynamics coming from the Brownian motion, we show in Section \ref{sec:driftstochOPP} the existence of functions $V_{\sigma,\kappa}$, with $\sigma\in \lbrace \sfx, \sfy, \sfz \rbrace$ that satisfy a Lyapunov-drift condition for the stochastic $P_\kappa$ Markov transition kernel.

\subsection{Stochastic stability for SDEs}
Let $\kappa\geq 0$ and $u\in \mathfrak{X}(\Sph)$. We define $v_\kappa(\x):= u(\x) - 2\kappa\x$ and consider the stochastic flow $\phi_t^{\kappa}(\x)$ defined by the stochastic differential equation
\begin{align}
    \d \phi_t^\kappa(\x) &= v_\kappa( \phi_t^\kappa(\x) )\d t + \sqrt{2\kappa}\Pi_{\phi_t^\kappa(\x)}\d \bB_t, \\
    \phi_0^\kappa(\x) &= \x.
\end{align}

For simplicity, we denote $\phi_t=\phi_t^0$ the associated deterministic flow. In this section we obtain several stochastic stability estimates for $\phi_t^\kappa$. We begin by showing the standard fact that the deviation of $\phi_t^\kappa$ from the deterministic evolution $\phi_t$ is no larger, up to a constant, than $\sqrt{\kappa t}$ in expectation.

\begin{lemma}\label{lemma:EXPphikappaphi}
    There exists $C_u=C(\Vert u \Vert_{C^1})>0$ such that 
    \begin{align}
        \EE_\bB \left| \phi_t^{\kappa}(\x) - \phi_t^{0}(\x) \right|^2 \leq C_u{\kappa t},
    \end{align}
    for all $\kappa\leq 1$ and all $t\leq 1$.
\end{lemma}
\begin{proof}
    We observe that
    \begin{align}
        \left| \phi_t^{\kappa}(\x) - \phi_t^{0}(\x) \right| &\leq \int_0^t \left| v_\kappa(\phi_s^{\kappa}(\x)) - v_\kappa(\phi_s^{0}(\x)) \right| \d s + \sqrt{2\kappa} \left| \int_0^t \left( \Pi_{\phi_s^{\kappa}(\x)} - \Pi_{\phi_s^{0}(\x)} \right) \d \bB_s \right|    \\
        &\quad + {2\kappa}\int_0^t \left|\phi_s^{0}(\x)\right| \d s + \sqrt{2\kappa} \left| \int_0^t\Pi_{\phi_s^{0}(\x)} \d \bB_s \right|. 
    \end{align}
    Since $\left| \phi_s^{0}(\x) \right| =1 $ for all $t\geq 0$ and $\Vert \Pi \Vert_{C^0}\lesssim 1$, for $h(t):= \EE_{\bB} \left| \phi_t^{\kappa}(\x) - \phi_t^{0}(\x) \right|^2$ we obtain
    \begin{align}
        h(t) \leq C\kappa t + \left( \Vert v_\kappa \Vert_{C^1}^2 + 2\kappa\Vert \Pi \Vert_{C^1}^2 \right) \int_0^t h(s) \d s.
    \end{align}
    Further noting that $\Vert \Pi \Vert_{C^1}\lesssim 1$ and $\Vert v_\kappa \Vert_{C^1}\lesssim 1$ for all $\kappa\leq 1$ we conclude with Gr\"{o}nwall's inequality that
    \begin{align}
        h(t) \leq C\kappa t\exp(Ct),
    \end{align}
    for some $C>1$. With this the lemma is proved.
\end{proof}

The next Lemma records the usual Lipschitz property of $\phi_t^\kappa$ in expectation and the bi-Lipschitz property of $\phi_t$. 

\begin{lemma}\label{lemma:phikappaLip}
There exists $C_0>0$ such that for all $\kappa\leq 1$ and $\x,\y\in \Sph$ we have
\begin{align}
    \EE_\bB \left|\phi_t^{\kappa}(\x) - \phi_t^{\kappa}(\y) \right|^2  \leq C_0 |\x - \y|^2,
\end{align}
for all $t\in[0,1]$. Moreover,
\begin{align}
     C_0^{-1} |\x - \y|^2\leq \left|\phi_t^{0}(\x) - \phi_t^{0}(\y) \right|^2  \leq C_0 |\x - \y|^2,
\end{align}
for all $t\in[0,1]$.
\end{lemma}

\begin{proof}
Let $v_\kappa(\x) =u(\x) -2\kappa\x$, we observe that
    \begin{align}
        \phi_t^{\kappa}(\x) = \x + \int_0^t  v_\kappa(\phi_s^{\kappa}(\x))\d s + \int_0^t \sqrt{2\kappa}\Pi_{\phi_s^{\kappa}(\x)}\d \bB_s,
    \end{align}
and similarly for $\phi_t^{\kappa}(\y)$. Hence, 
\begin{align}
    \left| \phi_t^{\kappa}(\x) - \phi_t^{\kappa}(\y) \right|^2 &\lesssim |\x-\y|^2 + \left| \int_0^t  \left( v_\kappa(\phi_s^{\kappa}(\x)) - v_\kappa(\phi_s^{\kappa}(\y))\right) \d s\right|^2 \\
    &\quad + 2\kappa\left| \int_0^t  \left( \Pi_{\phi_s^{\kappa}(\x)}-\Pi_{\phi_s^{\kappa}(\y)}\right)\d \bB_s\right|^2.
\end{align}
In particular, for $h(t) := \EE_\bB\left| \phi_t^{\kappa}(\x) - \phi_t^{\kappa}(\y) \right|^2$ we have 
\begin{align}
    h(t) \lesssim h(0) + \Vert v_\kappa \Vert_{C^1}^2 \int_0^t h(s) \d s + 2\kappa\Vert \Pi \Vert_{C^1}^2 \int_0^t h(s) \d s,
\end{align}
due to the It\^{o} isometry $\EE_\bB \left| \int_0^t \mathbf{G}(s) \d\bB_s \right|^2 = \int_0^t \Vert \mathbf{G}(s) \Vert^2 \d s$. Using Gr\"onwall's lemma, we reach
\begin{align}
    \EE_\bB\left| \phi_t^{\kappa}(\x) - \phi_t^{\kappa}(\y) \right|^2 \lesssim |\x-\y|^2,
\end{align}
for all $0\leq t \leq 1$, as $\Vert v_\kappa \Vert_{C^1} + \Vert \Pi \Vert_{C^1}\lesssim 1$ uniformly for all $0\leq \kappa\leq 1$. In particular, for $\kappa=0$ we obtain a uniform Lipschitz bound for $\phi_t^0$ for all $t\in[0,1]$. We finish the proof by showing that $\phi_t^0$ is in fact bi-Lipschitz. To that purpose, note that
\begin{align}
    \phi_{1-t}^{0}(\x) - \phi_{1-t}^{0}(\y) &= \phi_{1}^{0}(\x) - \phi_{1}^{0}(\y) - \int_{1-t}^1  \left( u(\phi_s^{0}(\x)) - u(\phi_s^{0}(\y)) \right)\d s 
\end{align}
so that for $g(t):=\left| \phi_{1-t}^{0}(\x) - \phi_{1-t}^{0}(\y) \right|^2$ we now have
\begin{align}
    g(t) \lesssim g(0) + \Vert u \Vert_{C^1}^2  \int_0^t g(s) \d s,
\end{align}
from which we deduce with Gr\"{o}nwall inequality that $g(t) \lesssim g(0)$, that is, $|\x-\y|^2 \lesssim |\phi_{1-t}^0(\x)-\phi_{1-t}^0(\y)|^2$ for all $t\in[0,1]$.
\end{proof}

The next result is key for proving drift-function estimates for the stochastic two-point process. A similar estimate was derived for the Kraichnan model in \cites{coti2024gaussian}, see also \cites{falkovich2013single, zel1984kinematic}.

\begin{lemma}\label{lemma:EXPphikappainv}
    There exists $C_0>0$ such that for all $\xi\in(0,1]$, $\kappa\leq 1$ and $\x,\y\in \Sph$ with $\x\neq\y$ we have
    \begin{align}
        \EE_\bB \left[ \frac{1}{|\phi_1^{\kappa}(\x) - \phi_1^{\kappa}(\y)|^\xi} \right] \leq \frac{C_0^\xi}{|\x-\y|^\xi}.
    \end{align}
\end{lemma}

\begin{proof}
    For convenience, let $Y_t =  \phi_t^{\kappa}(\x) - \phi_t^{\kappa}(\y)$. It satisfies the stochastic differential equation
    \begin{align}
        d Y_t &= (v_\kappa(\phi_t^{\kappa}(\x)) - v_\kappa(\phi_t^{\kappa}(\y))) \d t + \sqrt{2\kappa}\left( \Pi_{\phi_t^{\kappa}(\x)} - \Pi_{\phi_t^{\kappa}(\y)} \right) \d \bB_t  \\
        &= b(t) \d t + G(t) \d \bB_t,
    \end{align}
    and we note that
    \begin{align}\label{eq:lipvkappa}
        \left| b^i(t) \right| + \left| G^{ij}(t) \right| \leq \left( \Vert v_\kappa \Vert_{C^1} + {\sqrt{2\kappa}}\Vert \Pi \Vert_{C^1} \right) \left| Y_t \right| \leq C\left| Y_t \right|,
    \end{align}
    for all $i,j=1,2,3$, for some $C>1$ uniformly for all $\kappa\leq 1$. We also define the process $r_t^\ep=(|Y_t|^2+\ep^2)^{-1/2}$ and the function $g_\ep(x) = (|x|^2+\ep^2)^{-1/2}$ for $x\in \R^3$ and $\ep>0$. Since
    \begin{align}
        \partial_i g_\ep(x) = -\frac{x_i}{(|x|^2+\ep^2)^\frac32}, \quad \partial_i\partial_j g_\ep(x) = 3\frac{x_ix_j}{(|x|^2 + \ep^2)^\frac52} - \frac{1}{(|x|^2 + \ep^2)^\frac32}\delta_{ij},
    \end{align}
    we obtain from It\^{o}'s Lemma that
    \begin{align}
        \d r_t^\ep = -\sum_{i=1}^3 \frac{Y_t^i}{(|Y_t|^2 + \ep^2)^\frac32}\d Y_t^i +\frac12\sum_{i,j=1}^3 \left(3\frac{Y_t^iY_t^j}{(|Y_t|^2 + \ep^2)^\frac52} - \frac{1}{(|Y_t|^2 + \ep^2)^\frac32}\delta_{ij}\right)\sum_{k=1}^3 G^{ik}(t)G^{jk}(t) \d t.
    \end{align}
    Hence, there holds
    \begin{align}
        \EE_\bB r_t^\ep &= r_0^\ep - \EE_\bB\int_0^t \sum_{i=1}^3 \frac{Y_s^i}{(|Y_s|^2 + \ep^2)^\frac32}b^i(s) \d s \\
        &\quad+\frac12 \EE_\bB\int_0^t \sum_{i,j=1}^3 \left(3\frac{Y_s^iY_s^j}{(|Y_s|^2 + \ep^2)^\frac52} - \frac{1}{(|Y_s|^2 + \ep^2)^\frac32}\delta_{ij}\right)\sum_{k=1}^3 G^{ik}(s)G^{jk}(s) \d s,
    \end{align}
    where we have used the martingale property. Since $r_t^\ep > 0$ for all $t\in[0,1]$, there holds $\EE_\bB [|r_t^\ep|] = \EE_\bB [r_t^\ep]$ and thus
    \begin{align}
        \EE_\bB [r_t^\ep] \leq r_0^\ep + C \int_0^t \EE_\bB [r_s^\ep] \d s
    \end{align}
    due to \eqref{eq:lipvkappa}. For $h_\ep(t) = \EE_\bB r_t^\ep$ from Gr\"{o}nwall inequality we deduce that
    \begin{align}
        h_\ep(t) \leq C h_\ep(0) \exp(Ct) \leq C_0|Y_0|^{-1},
    \end{align}
    and thus the lemma for $\xi=1$ follows after using Fatou's Lemma for $\ep\rightarrow 0$. Finally, for $\xi\in (0,1)$ the function $x\mapsto x^\xi$ is concave and thus Jensen's inequality yields
\begin{align}
    \EE_\bB \left[ \frac{1}{|\phi_1^{\kappa}(\x) - \phi_1^{\kappa}(\y)|^\xi} \right] \leq \left( \EE_\bB \left[ \frac{1}{|\phi_1^{\kappa}(\x) - \phi_1^{\kappa}(\y)|} \right] \right)^\xi\leq \frac{C_0^\xi}{|\x-\y|^\xi}.
\end{align}
\end{proof}

Using triangular inequality and the stochastic stability of Lemma \ref{lemma:EXPphikappaphi}, we easily obtain the following.

\begin{lemma}\label{lemma:EXPcrudevartheta}
Let $u\in \mathfrak{X}(\Sph)$. For all $\x,\y\in \Sph$ and $t\in [0,1]$ define
    \begin{align}
        \vartheta_t^{\kappa}(\x,\y) := \left| \phi_t^{\kappa}(\x) - \phi_t^{\kappa}(\y) - (\phi_t^{0}(\x) - \phi_t^{0}(\y) ) \right|.
    \end{align}
    Then, there exists $C>0$ such that
    \begin{align}
        \EE_\bB \vartheta_t^{\kappa} (\x,\y) \leq C \sqrt{\kappa},
    \end{align}
    for all $\kappa\leq 1$ and all $t\in[0,1]$.
\end{lemma}

The next lemma improves the estimate on the deviation of the separation vector under the random flow, $X_t^\kappa:=\phi_t^{\kappa}(\x) - \phi_t^{\kappa}(\y)$ from the separation vector under the deterministic flow, $X_t:=\phi_t^{0}(\x) - \phi_t^{0}(\y)$. 
\begin{lemma}\label{lemma:EXPvartheta}
Let $u\in \mathfrak{X}(\Sph)$. For all $\x,\y\in \Sph$ and $t\in [0,1]$ define
    \begin{align}
        \vartheta_t^{\kappa}(\x,\y) := \left| \phi_t^{\kappa}(\x) - \phi_t^{\kappa}(\y) - (\phi_t^{0}(\x) - \phi_t^{0}(\y) ) \right|.
    \end{align}
    Then, there exists $C>0$ such that
    \begin{align}
        \EE_\bB \vartheta_t^{\kappa} (\x,\y) \leq C \left( \sqrt{\kappa} + |\x-\y| \right) \left|\phi_1^{0}(\x) - \phi_1^{0}(\y)\right|,
    \end{align}
    for all $\kappa\leq 1$.
\end{lemma}

\begin{proof}
    Note that 
    \begin{align}
        \vartheta_t^{\kappa}(\x,\y) &\leq \left| \int_0^t \left( v_\kappa(\phi_s^{\kappa}(\x)) - v_\kappa(\phi_s^{\kappa}(\y)) - \left( v_\kappa(\phi_s^{0}(\x)) - v_\kappa(\phi_s^{0}(\y)) \right) \right) \d s \right| \\
        &\quad + \sqrt{2\kappa} \left| \int_0^t \left( \Pi_{\phi_s^{\kappa}(\x)} - \Pi_{\phi_s^{\kappa}(\y)} - \left( \Pi_{\phi_s^{0}(\x)} - \Pi_{\phi_s^{0}(\y)} \right) \right) \d \bB_s \right| \\
        &\quad + 2\kappa \left| \int_0^t \left( \phi_s^{0}(\x) - \phi_s^{0}(\y) \right) \d s \right| + \left| \sqrt{2\kappa} \int_0^t \left( \Pi_{\phi_s^{0}(\x)} - \Pi_{\phi_s^{0}(\y)}\right) \d \bB_s\right|.
    \end{align}
Setting $h(t) := \EE_\bB (\vartheta_t^{\kappa}(\x,\y))^2$ and using It\^{o}'s isometry we obtain
\begin{align}
    h(t) &\leq \Vert v_\kappa \Vert_{C^1}^2 \int_0^t h(s) \d s + 2\kappa \Vert \Pi \Vert_{C^1}^2\int_0^t h(s) \d s \\
    &\quad + \int_0^t \EE_\bB \left| D_\x v_\kappa(\xi_1(s)) -  D_\x v_\kappa(\xi_2(s))\right|^2 \left| \phi_s^{0}(\x) - \phi_s^{0}(\y) \right|^2 \d s \\
    &\quad + 2\kappa \int_0^t \EE_\bB \left|D_\x \Pi_{\tilde\xi_1(s)} - D_\x \Pi_{\tilde\xi_2(s)}\right|^2 \left| \phi_s^{0}(\x) - \phi_s^{0}(\y) \right|^2 \d s \\
    &\quad + 4\kappa^2 \int_0^t \left| \phi_s^{0}(\x) - \phi_s^{0}(\y) \right|^2 \d s + 2\kappa \Vert \Pi \Vert_{C^1}^2 \int_0^t  \left| \phi_s^{0}(\x) - \phi_s^{0}(\y) \right|^2 \d s 
\end{align}
for some $\xi_1(s), \tilde \xi_1(s), \xi_2(s), \tilde\xi_2(s)\in \R^3$ with
\begin{align}
    \left| \xi_1(s) - \xi_2(s) \right| \leq \max \lbrace    \left| \phi_s^{\kappa}(\x) - \phi_s^{0}(\x) \right|,  \left| \phi_s^{\kappa}(\y) - \phi_s^{0}(\y) \right|,  \left| \phi_s^{\kappa}(\x) - \phi_s^{0}(\y) \right|,  \left| \phi_s^{\kappa}(\y) - \phi_s^{0}(\x) \right| \rbrace  
\end{align}
and likewise for $\left| \tilde\xi_1(s) - \tilde\xi_2(s) \right|$. In view of Lemma \ref{lemma:EXPphikappaphi} and Lemma \ref{lemma:phikappaLip}, we deduce that
\begin{align}
   \EE_\bB \left| \xi_1(s) - \xi_2(s) \right|^2 \leq C(\sqrt{\kappa } + |\x-\y| )^2
\end{align}
for some $C>1$. Therefore, noting that $\Vert v_\kappa\Vert_{C^2}^2 + 2\kappa\Vert \Pi \Vert_{C^2}^2\lesssim 1$ uniformly for all $\kappa\leq 1$, Lemma~\ref{lemma:phikappaLip} gives
\begin{align}
    h(t) &\leq C\int_0^t h(s) \d s + C (\sqrt{\kappa t} + |\x-\y| )^2 \int_0^t \left| \phi_s^{0}(\x) - \phi_s^{0}(\y) \right|^2 \d s \\
    &\leq C\int_0^t h(s) \d s + C (\sqrt{\kappa t} + |\x-\y| )^2 \left| \phi_1^{0}(\x) - \phi_1^{0}(\y) \right|^2.
\end{align}
As a result, Gr\"{o}nwall inequality yields
\begin{align}
    h(t) \leq C (\sqrt{\kappa t} + |\x-\y| )^2 \left| \phi_1^{0}(\x) - \phi_1^{0}(\y) \right|^2 \exp(Ct)
\end{align}
and the lemma follows {from $\EE_\bB\vartheta_t^{\kappa}\leq \sqrt{h(t)}$}.
\end{proof}

We finish the subsection by specializing to the flow maps considered in this manuscript. Recall that, for $N>1$ and for $\omega=(\omega_1,\omega_2)\in[-N,N]^2$, we denote with $f_\omega^\kappa := f_{\omega_2}^{2,\kappa}\circ f_{\omega_1}^{3,\kappa}$ and $f_\omega :=  f_{\omega_2}^{2}\circ f_{\omega_1}^{3}$ the stochastic and deterministic one-point processes, respectively, arising from \eqref{eq:introSDEITO}. The usual Gr\"{o}nwall-type arguments and Burkh\"{o}lder inequality yield the following result.

\begin{lemma}\label{lemma:EXPsupomegaLipTPP}
    Let $\x,\y\in \Sph$ and $n\geq 1$. Then, there exists $C_{N,n}>0$ such that
    \begin{align}\label{eq:unifomegaLipx}
        \EE_\bB \left[ \sup_{\underline{\omega}\in [-N,N]^{2n}} |f_{\underline{\omega}}^\kappa(\x) - f_{\underline{\omega}}^\kappa(\y)| \right] &\leq C_{N,n} |\x-\y|
        \end{align}
        and
        \begin{align}
        \EE_\bB \left[ \sup_{{\underline{\omega}}\in [-N,N]^{2n}} |f_{\underline{\omega}}^\kappa(\x) - f_{\underline{\omega}}(\x)| \right] &\leq C_{N,n}  \sqrt{\kappa},
    \end{align}
    for all $\kappa\in (0,1)$.
\end{lemma}

\begin{proof}
    We show the first inequality, as the ideas to prove both bounds are similar. Fix $\x,\y\in \Sph$ and $\kappa\in(0,1)$, let $g_{\underline{\omega}}(t) = \phi_t^\kappa(\x;{\underline{\omega}}) - \phi_t^\kappa(\y;{\underline{\omega}})$, where $\phi_t^\kappa(\cdot,{\underline{\omega}}):= \phi_t^\kappa(\cdot)$ solves the It\^{o} differential equation \eqref{eq:introSDEITO} with vector field $u(t,\cdot, \underline{\omega};\a,\b)$. In particular, $g_{\underline{\omega}}(2n) = f_{\underline{\omega}}^\kappa(\x) - f_{\underline{\omega}} ^\kappa(\y)$. There holds
    \begin{align}
        g_{\underline{\omega}}(s) &= \x - \y + \int_0^s \left( u(r,\phi_{r}^\kappa(\x,{\underline{\omega}});\a,\b) - u(r,\phi_{r}^\kappa(\y,{\underline{\omega}});\a,\b) \right) \d r + 2\kappa \int_0^s g_{\underline{\omega}}(r) \d r \\
        &\quad +\sqrt{2\kappa}\int_0^s \left( \Pi_{\phi_{r}^\kappa(\x,{\underline{\omega}})} - \Pi_{\phi_{r}^\kappa(\y,{\underline{\omega}})} \right) \d \bB_r.
    \end{align}
    Let $p\geq 1$ and $G_{\underline{\omega}}(t) = \EE_\bB \left[ |\sup_{0\leq s \leq t}|g_{\underline{\omega}}(s)||^p\right]$. Using Burkholder-Davis-Gundy inequality,
    \begin{align}
        \EE_{\bB} \left[ \left( \sup_{0\leq s \leq t} \left| \int_0^s \left( \Pi_{\phi_{r}^\kappa(\x,{\underline{\omega}})} - \Pi_{\phi_{r}^\kappa(\y,{\underline{\omega}})} \right) \d \bB_r \right| \right)^p \right] &\lesssim \EE_{\bB} \left[\left( \int_0^t \left( \Pi_{\phi_{r}^\kappa(\x,{\underline{\omega}})} - \Pi_{\phi_{r}^\kappa(\y,{\underline{\omega}})} \right)^2 \d r \right)^{p/2} \right] \\
        &\lesssim \EE_{\bB}\left[ \left( \int_0^t |g_{\underline{\omega}}(r)|^2 \d r\right)^{p/2} \right] \\
        &\lesssim \EE_{\bB}\left[ \left( \int_0^t (\sup_{0\leq s \leq r}|g_{\underline{\omega}}(s)|)^2 \d r\right)^{p/2} \right] \\
        &\lesssim \EE_{\bB}\left[ \int_0^t \left(\sup_{0\leq s \leq r}|g_{\underline{\omega}}(s)|\right)^p \d r \right] \\
        &\lesssim \int_0^t G_{\underline{\omega}}(r) \d r ,
    \end{align}
    while we also have
    \begin{align}
        \EE_{\bB} \left[ \left( \sup_{0\leq s \leq t} \left| \int_0^s \left( u(r,\phi_{r}^\kappa(\x,{\underline{\omega}});\a,\b) - u(r,\phi_{r}^\kappa(\x,{\underline{\omega}});\a,\b) \right) \d r \right| \right)^p \right] \lesssim \int_0^t G_{\underline{\omega}}(r) \d r ,
    \end{align}
    and similarly for the remaining contribution. Hence, using Gr\"{o}nwall on $G_{\underline{\omega}}(t)$ we conclude that
    \begin{align}\label{eq:pmomentG}
       \sup_{t\in[0,2n]} G_{\underline{\omega}}(t)\leq C_N |\x-\y|^p,
    \end{align}
    for all $\underline{\omega}\in [-N-1,N+1]^{2n}$, for some $C_N>0$ independent of $\underline{\omega}$.
    
    Now, to exchange the supremum in ${\underline{\omega}}$ and the expectation in $\bB$ we use the Sobolev inequality. For this, we consider first $\partial_j\phi_t^\kappa(\x;\underline{\omega}):= \partial_{\omega_j}\phi_t^\kappa(\x;\underline{\omega})$ and $g_{\underline{\omega}}^{j}=\partial_{\omega_j}g_{\underline{\omega}}(t)$, where $\underline{\omega}=(\omega_1,..., \omega_{2n})$ and $1\leq j \leq 2n$. There holds
    \begin{align}
        \partial_j \phi_t^\kappa(\x;\underline{\omega}) &= \int_0^t \left( \left( Du(s,\phi_s^\kappa(\x,\underline{\omega});\a,\b) - 2\kappa Id \right) \partial_j \phi_s^\kappa(\x;\underline{\omega}) \right) \d s + \int_0^t \partial_{\omega_j}u(s,\phi_s^\kappa(\x,\underline{\omega});\a,\b) \d s \\
        &\quad + \sqrt{2\kappa}\int_0^t D\Pi_{\phi_s^\kappa(\x,\underline{\omega})} \partial_j\phi_s^\kappa(\x,\underline{\omega}) \d \bB_s.
    \end{align}
    We next note that $\partial_{\omega_j}u(s,\phi_s^\kappa(\x,\underline{\omega});\a,\b)= u^\sigma(\phi_s^\kappa(\x,\underline{\omega}))1_{[t_j,t_j+1)}(s)$ for some $\sigma\in \lbrace \a, \b \rbrace$ and some $t_j\in \N$ with $0\leq t_j \leq 2n-1$ and $u^\sigma$ uniformly bounded. Hence, a Gr\"{o}nwall-type argument and Burkholder-Davis-Gundy inequality yields
    \begin{align}\label{eq:pmomentphij}
        \sup_{\underline{\omega}\in [-N-1,N+1]^{2n}}\EE_\bB (\sup_{t\in[0,2n]} |\partial_j \phi_t^\kappa(\x,\underline{\omega})|)^p \leq C,
    \end{align}
    for some $C>0$. On the other hand,
    \begin{align}
        g_{\underline{\omega}}^j(t) &= \int_0^t (\partial_{\omega_j}u(s,\phi_t^\kappa(\x,\underline{\omega}),{\underline{\omega}};\a,\b) - \partial_{\omega_j}u(s,\phi_t^\kappa(\y,\underline{\omega}),{\underline{\omega}};\a,\b) ) \d s \\
        &\quad + \int_0^t \left( D u(s,\phi_s^\kappa(\x,\underline{\omega}),{\underline{\omega}};\a,\b) \partial_j\phi_s^\kappa(\x,\underline{\omega}) - D u(s,\phi_s^\kappa(\y,\underline{\omega}),{\underline{\omega}};\a,\b) \partial_j\phi_s^\kappa(\y,\underline{\omega}) \right) \d s \\
        &\quad + 2\kappa\int_0^t g_{\underline{\omega}}^j(s) \d s + \sqrt{2\kappa}\int_0^t \left( D\Pi_{\phi_s^\kappa(\x,\underline{\omega})} \partial_j \phi_s^\kappa(\x,\underline{\omega}) - D\Pi_{\phi_s^\kappa(\y,\underline{\omega})} \partial_j \phi_s^\kappa(\y,\underline{\omega}) \right) \d \bB_s,
    \end{align}
    and we note that for $G_{\underline{\omega}}^j(t) := \EE_\bB \left[ |\sup_{0\leq s \leq t}|g_{\underline{\omega}}^j(s)||^p\right]$ we have
    \begin{align}
        \EE_{\bB} \left[ \left| \sup_{0\leq s \leq t} \int_0^s (\partial_{\omega_j}u(r,\phi_r^\kappa(\x,\underline{\omega}),{\underline{\omega}};\a,\b) - \partial_{\omega_j}u(r,\phi_r^\kappa(\y,\underline{\omega}),{\underline{\omega}};\a,\b) ) \d r \right|^p \right] \lesssim \int_0^t G_{\underline{\omega}}(r) \d r \lesssim |\x-\y|^p,
    \end{align}
    while also
    \begin{align}
        \EE_{\bB}& \left[ \left| \sup_{0\leq s \leq t} \int_0^s \left( D u(r,\phi_r^\kappa(\x,\underline{\omega}),{\underline{\omega}};\a,\b) \partial_j\phi_r^\kappa(\x,\underline{\omega}) - D u(s,\phi_r^\kappa(\y,\underline{\omega}),{\underline{\omega}};\a,\b) \partial_j\phi_r^\kappa(\y,\underline{\omega}) \right) \d r \right|^p \right] \\
        &\quad \lesssim \EE_{\bB}\left[ \int_0^t |g_{\underline{\omega}}(r)|^p |\partial_j \phi_r^\kappa(\x,\underline{\omega})|^p \d r \right] + \EE_{\bB}\left[ \int_0^t |g_{\underline{\omega}}^j(r)|^p \d r \right] \\
        &\quad \lesssim \EE_{\bB}\left[ \int_0^t (\sup_{0\leq s \leq r}|g_{\underline{\omega}}(r)|)^p (\sup_{0\leq s \leq r}| \partial_j \phi_r^\kappa(\x,\underline{\omega})|)^p \d r \right] + \EE_{\bB}\left[ \int_0^t (\sup_{0\leq s \leq r}|g_{\underline{\omega}}^j(r)|)^p \d r \right] \\ 
        &\quad \lesssim |\x-\y|^p + \int_0^t G_{\underline{\omega}}^j(r) \d r.
    \end{align}
    In the above we used that $E(|XY|)\leq E(X^2)^\frac12 E(Y^2)^\frac12$ and that \eqref{eq:pmomentG} and \eqref{eq:pmomentphij} hold for all $p\geq 1$. Appealing again to the Burkholder-Davis-Gundy inequality for the martingale term and the Gr\"{o}nwall inequality, we conclude that
    \begin{align}\label{eq:pmomentGj}
        \sup_{\underline{\omega}\in[-N-1,N+1]^{2n}} \sup_{t\in[0,2n]} G_{\underline{\omega}}^j(t) \lesssim |\x-\y|^p.
    \end{align}
We are now in position to prove \eqref{eq:unifomegaLipx}. Since $[-N,N]^{2n}\Subset [-N-1,N+1]^{2n}$ we note that for all $\underline{\omega}\in [-N,N]^{2n}$ and all $t\in[0,2n]$, we have the Sobolev inequality
\begin{align}
    \sup_{\underline{\omega}\in [-N,N]^{2n}}|g_{\underline{\omega}}(t)| \lesssim \left( \int_{[-N-1,N+1]^{2n}} |g_{\underline{\varpi}}(t)|^p \d \underline{\varpi} + \sum_{j=0}^{2n} \int_{[-N-1,N+1]^{2n}}|g_{\underline{\varpi}}^j(t)|^p \d \underline{\varpi} \right)^\frac{1}{p}
\end{align}
for all $p>2n$. Hence, in view of \eqref{eq:pmomentG} and \eqref{eq:pmomentGj},
\begin{align}
    \left(\EE_{\bB}\left[ \sup_{\underline{\omega}\in [-N,N]^{2n}} |g_{\underline{\omega}}(t)| \right] \right)^p &\leq \EE_{\bB}\left[ \sup_{\underline{\omega}\in [-N,N]^{2n}} |g_{\underline{\omega}}(t) |^p\right]  \\
    &\lesssim \int_{[-N-1,N+1]^{2n}} \EE_{\bB}\left [|g_{\underline{\varpi}}(t)|^p\right] \d \underline{\varpi} + \sum_{j=0}^{2n} \int_{[-N-1,N+1]^{2n}} \EE_\bB \left[|g_{\underline{\varpi}}^j(t)|^p \right]\d \underline{\varpi}  \\
    &\lesssim |\x-\y|^p
\end{align} for all $t\in[0,2n]$. The bound \eqref{eq:unifomegaLipx} then follows.
\end{proof}

The ideas presented in the proof of Lemma \ref{lemma:EXPsupomegaLipTPP} also show the Lipschitz continuity of $f_\omega^\kappa$ with respect to the random-amplitude noise $\omega$. 

\begin{lemma}\label{lemma:EXPunifomegaTPP}
Let $n\geq 1$ and $N>1$. There exists $C_{N,n}>0$ such that 

\begin{align}
    \EE_\bB \left[ \sup_{ \underline{\varpi}\in [-N,N]^{2n}} |D_{\underline{\omega}}f_{\underline{\varpi}}^\kappa(\x)| \right] \leq C_{N,n} 
\end{align}
In particular, for all $\kappa\in[0,1]$ and all $\x\in \mathsf{X}$,
\begin{align}
    \EE_\bB \left[ \sup_{\underset{\underline{\omega}\neq \underline{\varpi}}{\underline{\omega},\underline{\varpi}\in [-N,N]^{2n}}} \frac{|f_{\underline{\omega}}^\kappa(\x) - f_{\underline{\varpi}}^\kappa(\x)|^2}{|\underline{\omega}-\underline{\varpi}|_\infty^2} \right] \leq C_{N,n}.
\end{align}
\end{lemma}

\begin{proof}
    We obtain first uniform-in-time bounds on $p$-moments of $D_{\underline{\omega}}\phi_t^\kappa(\x,\underline{\varpi})$ that are uniform in $\varpi\in [-N-1,N+1]^{2n}$. These are then used to obtain uniform-in-time bounds on $p$-moments of $D^2_{\underline{\omega}}\phi_t^\kappa(\x,\underline{\varpi})$ that are also uniform in $\varpi\in [-N-1,N+1]^{2n}$. The Sobolev inequality then yields the claimed estimate. We omit the routine details.
\end{proof}

\subsection{Local stochastic drift function for the two-point process}\label{ssec:localstochTPP}
We begin by recalling from Proposition \ref{prop:blackboxdriftTPP} that 
\begin{equation}
W(\x,\y) = d_{\Sph}(\x,\y)^{-\xi}\psi(\x,\theta(\x,\y)),
\end{equation}
for $\theta(\x,\y) := \frac{\mathrm{exp}_{\x}^{-1}(\y)}{d_\Sph(\x,\y)}$, some $\psi:\mathsf{X}\times T\mathsf{X}\rightarrow \R_+$, and some $\xi\in(0,1)$. Here $\mathrm{exp}_\x:T_\x\mathsf{X}\to\mathsf{X}$ denotes the exponential map at time $1$. We note that for $\x,\y\in \Sph$, its distance $d_\Sph(\x,\y)$ on $\Sph$ is given by $d_\Sph(\x,\y)=\arccos(\x\cdot\y)$, while if we view $\x,\y\in \R^3$, then $d_{\R^3}(\x,\y) = |\x-\y|$. The two distances are related by
\begin{align}
    |\x-\y|= 2\sin \left( \frac{d_{\Sph}(\x,\y)}{2} \right), \quad d_{\Sph}(\x,\y)= 2\arcsin \left( \frac{|\x-\y|}{2} \right), \quad 
\end{align}
and since $|\x-\y|\leq 2$ and $d_{\Sph}(\x,\y)\leq \pi$ for all $\x,\y\in \Sph$ we have the expansions
\begin{equation}\label{eq:metricequiv}
\begin{split}
    |\x-\y| &= d_{\Sph}(\x,\y) - \mathfrak{R}_\Sph(\x,\y) d_{\Sph}(\x,\y)^3, \\
    d_{\Sph}(\x,\y) &= |\x-\y| - \mathfrak{R}_{\R^3}(\x,\y) |\x-\y|^3,
\end{split}
\end{equation}
with $\Vert \mathfrak{R}_\Sph \Vert_{L^\infty(\Sph\times\Sph)} + \Vert \mathfrak{R}_{\R^3} \Vert_{L^\infty(\Sph\times\Sph)} \leq C$, for some $C>0$. Moreover,
\begin{equation}\label{eq:inversemetricequiv}
\begin{split}
    |\x-\y|^{-\alpha} &= d_{\Sph}(\x,\y)^{-\alpha} - \widetilde{\mathfrak{R}}_\Sph(\x,\y) d_{\Sph}(\x,\y)^{2-\alpha}, \\
    d_{\Sph}(\x,\y)^{-\alpha} &= |\x-\y|^{-\alpha} - \widetilde{\mathfrak{R}}_{\R^3}(\x,\y) |\x-\y|^{2-\alpha},
\end{split}
\end{equation}
for $0< \alpha\leq 1$, where now  $\Vert \widetilde{\mathfrak{R}}_\Sph \Vert_{L^\infty(\Delta(s_0))} + \Vert \widetilde{\mathfrak{R}}_{\R^3} \Vert_{L^\infty(\Delta(s_0))} \leq \widetilde{C}$ for some $\widetilde{C}>0$ and some $s_0>0$. In particular, for all $\ep>0$ there exists $\delta>0$ such that $|\x-\y|< \ep$ for all $d_\Sph(\x,\y)< \delta$. Then, in view of \eqref{eq:inversemetricequiv}, for all $\xi\in (0,1]$ there holds
\begin{align}\label{eq:inversespheremetric}
    d_\Sph(\x,\y)^{-\xi} \leq C_2 |\x-\y|^{-\xi},\quad |\x-\y|^{-\xi} \leq C_3 d_\Sph(\x,\y)^{-\xi},
\end{align}
for all $d_\Sph(\x,\y)< \delta$, for some universal constants $C_2,C_3>0$.

We first show that the local drift function for the deterministic two-point process is stable under the stochastic dynamics if $\kappa$ is small enough.
    \begin{proposition}\label{prop:stochdriftlocalTPP}
        Let $\ep_0>0$. There exists $s_*>0$ and $\kappa_*>0$ such that
            \begin{align}
              {\EE_\bB \left| W(f_\omega^\kappa(\x),f_\omega^\kappa(\y)) - W(f_\omega(\x),f_\omega(\y)) \right|} < \ep_0{W(\x,\y)}
    \end{align}
    for all $(\x,\y)\in \Delta(\ep)$, all $\omega\in [-N,N]^2$,  all $\kappa\in[0,\kappa_*)$ and all $0<\ep<s_*$.
    \end{proposition}
\begin{proof}
    Fix $\ep_0>0$. Note that
    \begin{align}
        \left| W(f_\omega^\kappa(\x),f_\omega^\kappa(\y)) - W(f_\omega^0(\x),f_\omega^0(\y)) \right|  &= \left| (Y_1^\kappa)^{-\xi} \psi^\kappa - (Y_1)^{-\xi}\psi^0\right| \\
        &\leq \psi^\kappa \left| (Y_1^\kappa)^{-\xi} - (Y_1)^{-\xi} \right| + (Y_1)^{-\xi}\left| \psi^\kappa - \psi^0 \right|,
    \end{align}
    where 
    \begin{align}
        Y_1^\kappa := d_\Sph(f_\omega^\kappa(\x),f_\omega^\kappa(\y)), \quad Y_1 := d_\Sph(f_\omega(\x),f_\omega(\y))
    \end{align}
    and for $\theta(\x',\y') := \frac{\mathrm{exp}_{\x'}^{-1}(\y')}{d_\Sph(\x',\y')}$ we denote
    \begin{align}
        \psi^\kappa = \psi(f_\omega^\kappa(\x), \theta(f_\omega^\kappa(\x), f_\omega^\kappa(\y))), \quad \psi^0 = \psi(f_\omega(\x), \theta(f_\omega(\x), f_\omega(\y))).
    \end{align}
    Let $\eta:= \frac{Y_1}{2}$ and consider the event $E:= \lbrace   Y_1^\kappa \geq \eta \rbrace  $. The mean value Theorem then yields
    \begin{align}
        | |Y_1^\kappa|^{-\xi} - |Y_1|^{-\xi} | \leq \xi \left( |Y_1^\kappa| \wedge |Y_1| \right)^{-(\xi+1)} |Y_1^\kappa - Y_1|,
    \end{align}
    where $a\wedge b:= \min(a,b)$.  Hence, on $E$ we have 
    \begin{equation}
        \left( |Y_1^\kappa| \wedge |Y_1| \right)^{-(\xi+1)}\leq \eta^{-(\xi+1)}.
    \end{equation}
    We introduce also 
    \begin{align}
        X_1^\kappa := |f_\omega^\kappa(\x) - f_\omega^\kappa(\y)|, \quad X_1 := |f_\omega(\x) - f_\omega(\y) |.
    \end{align}
    From \eqref{eq:metricequiv} and noting that $|X_1^\kappa|\leq 2$ because $f_\omega^\kappa(\x),f_\omega^\kappa(\y)\in \Sph$ for all $\kappa\geq 0$, we obtain
    \begin{align}
        \left| Y_1^\kappa - Y_1 \right| \leq |X_1^\kappa - X_1 | + C\left( |X_1^\kappa|^2 + |X_1|^2 \right)
    \end{align}
    and since $|X_1^\kappa - X_1 | \leq \vartheta_1^\kappa$ we deduce from Lemmas \ref{lemma:phikappaLip} and \ref{lemma:EXPvartheta} that
    \begin{align}
        \EE_\bB \left| Y_1^\kappa - Y_1 \right| \lesssim \left( \sqrt{\kappa} + |\x-\y|\right) |\x- \y| + |\x - \y|^2.
    \end{align}
    Therefore 
    \begin{align}
        \EE_\bB \left[ |\psi^\kappa||Y_1^\kappa|^{-\xi} - |Y_1|^{-\xi} | \mathbf{1}_E \right] \lesssim \eta^{-(\xi+1)}|\x-\y|\left( \sqrt{\kappa} + |\x-\y| \right).
    \end{align}
    An application of \eqref{eq:inversespheremetric} and Lemma \ref{lemma:phikappaLip} together with $1+\xi>0$ and $\inf \psi >0$ yields
    \begin{align}
        \eta^{-(\xi+1)} \lesssim |Y_1|^{-(\xi+1)}\lesssim |X_1|^{-(\xi+1)}\lesssim |\x-\y|^{-(\xi+1)}\lesssim |\x-\y|^{-1}\d_{\Sph}(\x,\y)^{-\xi}\lesssim |\x-\y|^{-1}W(\x,\y).
    \end{align}
    As a result, we conclude that
    \begin{align}\label{eq:Exp_on_E_StochTPP}
         \EE_\bB \left[ |\psi^\kappa||Y_1^\kappa|^{-\xi} - |Y_1|^{-\xi} | \mathbf{1}_E \right]  \lesssim \left( \sqrt{\kappa} + |\x-\y| \right)W(\x,\y).
    \end{align}

    On the other hand, for the event $E^c=\lbrace   |Y_1^\kappa|< \eta \rbrace   $, we  observe that $\vartheta_1^\kappa\geq |Y_1| - |Y_1^\kappa| \geq \eta$ and thus $E^c\subseteq \lbrace   \vartheta_1^{\kappa}  \geq \eta \rbrace  $. In particular, Markov inequality yields
    \begin{align}\label{eq:Prob_on_Ec_StochTPP}
        \PP(E^c) \leq \PP \left( \vartheta_1^{\kappa}  \geq \eta\right) \leq \eta^{-1}\EE_\bB \vartheta_1^{\kappa} \leq C \left( \sqrt{\kappa} + |\x-\y|\right),
    \end{align}
    after noting that $\eta^{-1}\lesssim |\x-\y|^{-1}$ in view of \eqref{eq:inversespheremetric} and Lemma \ref{lemma:phikappaLip}. Moreover,
    \begin{align}
        \EE_\bB \left[ |\psi^\kappa| |Y_1^\kappa|^{-\xi} - |Y_1|^{-\xi} | \mathbf{1}_{E^c} \right] \lesssim  \EE_\bB[|Y_1^\kappa|^{-\xi}\mathbf{1}_{E^c}] + \EE_\bB[|Y_1|^{-\xi} \mathbf{1}_{E^c}] \lesssim {|\x-\y|^{-\xi}}\lesssim W(\x,\y), 
    \end{align}
    where we used Lemma~\ref{lemma:EXPphikappainv} and the  H\"{o}lder inequality 
    \begin{equation}\label{eq:Exp_on_Ec_StochTPP}
    \EE_\bB[|Y_1^\kappa|^{-\xi} \mathbf{1}_{E^c}]\leq \left(\EE_\bB [|Y_1^\kappa|^{-1}] \right)^{\xi} (\PP(E^c))^{1-\xi}
    \end{equation} with $p=\frac{1}{\xi}>1$ and $q=\frac{1}{1-\xi}>1$. Hence, combining \eqref{eq:Exp_on_E_StochTPP}, \eqref{eq:Prob_on_Ec_StochTPP}, and \eqref{eq:Exp_on_Ec_StochTPP},
    \begin{equation}\label{eq:EXPA}
    \begin{split}
        \EE_\bB \left[ |\psi^\kappa| |Y_1^\kappa|^{-\xi} - |Y_1|^{-\xi} | \right] &= \EE_\bB \left[ |\psi^\kappa| |Y_1^\kappa|^{-\xi} - |Y_1|^{-\xi} | \mathbf{1}_E \right] + \EE_\bB \left[ |\psi^\kappa| |Y_1^\kappa|^{-\xi} - |Y_1|^{-\xi} | \mathbf{1}_E^c \right] \\
        &\lesssim (\sqrt{\kappa} + |\x-\y|)^{1-\xi} W(\x,\y).
        \end{split}
    \end{equation}

    We next argue for $\EE_\bB \left[ |Y_1|^{-\xi}|\psi^\kappa - \psi^0| \right]$.     First, the function $\psi$ is uniformly continuous, namely for all  \linebreak $(\x',\theta_\x'),(\y',\theta_\y')\in UT\Sph$ and all $\epsilon_0>0$ there exists $\delta_0>0$ such that $|\psi(\x',\theta_{\x}')-\psi(\y',\theta_{\y}')|< \epsilon_0$ whenever $d((\x',\theta_\x'),(\y',\theta_\y'))< \delta_0$. In our setting,
    \begin{align}
        d((\x',\theta_\x'),(\y',\theta_\y')) \leq d_\Sph(f_\omega^\kappa(\x), f_\omega(\x)) + \left| \theta(f_\omega^\kappa(\x), f_\omega^\kappa(\y)) - \theta(f_\omega(\x), f_\omega(\y)) \right|
    \end{align}
    after denoting
    \begin{equation}
    (\x',\theta_\x') := (f_\omega^\kappa(\x), \theta(f_\omega^\kappa(\x), f_\omega^\kappa(\y)), \quad (\y',\theta_\y') :=(f_\omega(\x), \theta(f_\omega(\x), f_\omega(\y)).
    \end{equation}
    Moreover, for $\z,\z'\in \Sph$, let $\varphi(\z,\z')= \mathrm{exp}_\z^{-1}(\z')$ so that $\mathrm{exp}_\z(\varphi(\z,\z'))=\z'$, by definition of exponential map, and $\Vert \varphi(\z,\z') \Vert = d_\Sph(\z,\z')$.
    On $\Sph$, we have that
    \begin{align}
        \z' = \cos(d_\Sph(\z,\z'))\z + \sin(d_\Sph(\z,\z'))\frac{\varphi(\z,\z')}{d_\Sph(\z,\z')}
    \end{align}
    and, since $\theta(\z,\z') = \frac{\varphi(\z,\z')}{d_\Sph(\z,\z')}$, it holds 
    \begin{align}
        f_{\omega}^\kappa(\y) &= \cos(Y_1^\kappa)f_\omega^\kappa(\x) + \sin (Y_1^\kappa) \theta(f_\omega^\kappa(\x),f_\omega^\kappa(\y)), \\
        f_{\omega}(\y) &= \cos(Y_1)f_\omega(\x) + \sin (Y_1) {\theta(f_\omega(\x),f_\omega(\y))}.
    \end{align}
    Subtracting leads to
    \begin{align}
        \theta(f_\omega^\kappa(\x),f_\omega^\kappa(\y)) - \theta(f_\omega(\x),f_\omega(\y)) & = \frac{ (f_\omega^\kappa(\y) - f_\omega^\kappa(\x)) - (f_\omega(\y) - f_\omega(\x))}{\sin(Y_1) }  + \frac{(1-\cos(Y_1))(f_\omega^\kappa(\x) - f_\omega(\x))}{\sin(Y_1) } \\
        &\quad - \frac{f_\omega^\kappa(\x) \left( \cos(Y_1^\kappa) - \cos(Y_1) \right)}{\sin(Y_1) }   + \frac{\theta(f_\omega^\kappa(\x),f_\omega^\kappa(\y)) \left( \sin(Y_1^\kappa) - \sin(Y_1) \right)}{\sin(Y_1)}.
    \end{align}
    For $d_{\Sph}(\x,\y)<s_*$ and $s_*>0$ small enough, combining the trivial inequality 
    \begin{equation}
    Y_1\lesssim X_1\lesssim |\x-\y|\lesssim d_{\Sph}(\x,\y)<\frac{\pi}{4},
    \end{equation}
    with \eqref{eq:inversespheremetric} gives $\sin^{-1}(Y_1)\leq 2(Y_1)^{-1} \lesssim |\x-\y|^{-1}$ and $|1-\cos(Y_1)|\lesssim |Y_1|^2$.
    Therefore, since $\Vert f_\omega^\kappa(\x)\Vert = \Vert \theta(f_\omega^\kappa(\x),f_\omega^\kappa(\y)) \Vert = 1$, a direct application of the mean value theorem and Lemmas \ref{lemma:EXPphikappaphi} and \ref{lemma:EXPvartheta} concludes that
    \begin{equation}\label{eq:EXPthetaDifference}
    \begin{split}
        \EE_\bB \left| \theta(f_\omega^\kappa(\x),f_\omega^\kappa(\y)) - \theta(f_\omega(\x),f_\omega(\y)) \right| &\lesssim |\x-\y|^{-1}\EE_\bB[\vartheta_1^\kappa] + \EE_\bB [\left| f_\omega^\kappa(\x) - f_\omega(\x) \right|] \\
        &\lesssim \sqrt{\kappa} + |\x-\y|.
    \end{split}
    \end{equation}
    Note that $|Y_1|^{-\xi}\lesssim W(\x,\y)$ follows from Lemma \ref{lemma:phikappaLip} and consider now the set $F = \lbrace   d((\x',\theta_\x'),(\y',\theta_\y')) < \delta_0 \rbrace  $. Then,
    \begin{align}
        \EE_\bB \left[|Y_1|^{-\xi}|\psi^\kappa - \psi^0| \big| F \right] \lesssim \epsilon_0 W(\x,\y) \leq \frac{\ep_0}{8}W(\x,\y),
    \end{align}
    for $\epsilon_0$ sufficiently small. On its complement, $F^c$, from
    \begin{align}
        \lbrace   d((\x',\theta_\x'),(\y',\theta_\y'))  \geq \delta_0 \rbrace   \subseteq  \left\lbrace   d_{\Sph}(f_\omega^\kappa(\x), f_\omega(\x)) \geq \frac{\delta_0}{2} \right\rbrace   \bigcup \left\lbrace   |\theta(f_\omega^\kappa(\x),f_\omega^\kappa(\y)) - \theta(f_\omega(\x),f_\omega(\y))| \geq \frac{\delta_0}{2} \right\rbrace  
    \end{align}
    we have $\PP(F^c) \lesssim_{\delta_0}\sqrt{\kappa} + |\x-\y|$ due to Lemma \ref{lemma:phikappaLip} and \eqref{eq:EXPthetaDifference}. 
    In addition, Lemma~\ref{lemma:phikappaLip} implies
    \begin{equation}
    |Y_1|^{-\xi}|\psi^\kappa - \psi^0|\lesssim |\x-\y|^{-\xi}\lesssim W(\x,\y)
    \end{equation}
    from which it follows
    \begin{equation}\label{eq:EXPB}
    \begin{split}
        \EE_\bB \left[|Y_1|^{-\xi}|\psi^\kappa - \psi^0| \right] &= \EE_\bB \left[|Y_1|^{-\xi}|\psi^\kappa - \psi^0| \big| F \right]\PP(F) + \EE_\bB \left[|Y_1|^{-\xi}|\psi^\kappa - \psi^0|\big| F^c \right]\PP(F^c) \\
        &\lesssim_{\delta_0}(\sqrt{\kappa}+ |\x-\y|) W(\x,\y).
    \end{split}
    \end{equation}
     Gathering \eqref{eq:EXPA} and \eqref{eq:EXPB}, the proposition is proved after choosing $s_*>0$ and $\kappa_*>0$ small enough.
\end{proof}

\section{Drift function for the stochastic one-point process}\label{sec:driftstochOPP}
In this section we adapt the drift function $\mathrm{V}$ for the one-point process to the presence of the stochastic term coming from the diffusion operator. We thus define $\mathrm{V}_\kappa$ in terms of modified local drift functions
\begin{align}
\begin{array}{ll}
    V_{\sfx,\kappa}(p):=\max\{r_\sfx(p),\sqrt{\kappa}\}^{-\alpha}, & r_\sfx(p):=\max\{|y|,|z|\}, \\
    V_{\sfy,\kappa}(p):=\max\{r_\sfy(f_N^3(p)),\sqrt{\kappa}\}^{-\alpha}, & r_\sfy(p):=\sqrt{x^2 + z^2}, \\
    V_{\sfz,\kappa}(p):=\max\{r_\sfz(p),\sqrt{\kappa}\}^{-\alpha}, & r_\sfz(p):=\sqrt{x^2 + y^2},
\end{array}
\end{align}
for $p=(x,y,z)\in \Sph$, and some $\alpha\in (0,\frac14)$ fixed. We show that each $V_{\sigma,\kappa}$ satisfies a Lyapunov-drift condition for the stochastic flow map $f_\omega^\kappa$, for $\sigma\in \Sigma$. To do so, we first show in Lemma \ref{lemma:driftboundawayx}, \ref{lemma:driftboundawayz} and \ref{lemma:driftboundawayy} that on a compact set $\mathsf{Q}_\sigma(s_0)^c$ away from $F_\sigma$ the drift function $V_{\sigma,\kappa}$ is bounded in expectation, for $\kappa>0$ small enough. For points $p\in \mathsf{Q}_\sigma(s_0)$ we next distinguish three regions $\mathcal{K}_{A_0}$, $\mathcal{K}_{A_1}\setminus \mathcal{K}_{A_0}$ and $\mathsf{Q}_\sigma(s_0)\setminus \mathcal{K}_{A_1}$ as portrayed in Figure~\ref{fig:three_zones_core}, for some $A_1>A_0>1$ to be determined, where for $A>1$ we denote 
\begin{align}
    \mathcal{K}_A = \lbrace p=(x,y,z)\in \Sph : r_\sigma(p) \leq \sqrt{\kappa}A \rbrace.
\end{align}
We call $\mathcal{K}_{A_1}$ the stochastic regime, where the effects of the Brownian motion are of the same order as those of the velocity field. In this region, we further distinguish the inner core $\mathcal{K}_{A_0}$, where the stochasticity of the Brownian motion is enough to move the process quantitatively away from the fixed point with sufficiently high probability, and the annular region $\mathcal{K}_{A_1}\setminus \mathcal{K}_{A_0}$, where the shearing mechanism is still sufficiently strong to produce a drift estimate for $\mathrm{V}_{\sfx,\kappa}$. Finally, $\mathsf{Q}_\sigma(s_0)\setminus \mathcal{K}_{A_1}$ denotes the inviscid dominated regime, where the stochastic evolution can be understood as a small perturbation of the deterministic dynamics and the arguments of Lemmas \ref{lemma:localxlyap}, \ref{lemma:localzlyap} and \ref{lemma:localylyap} can be applied.

\begin{figure}[!ht]
    \centering
    \includegraphics[width=.7\linewidth]{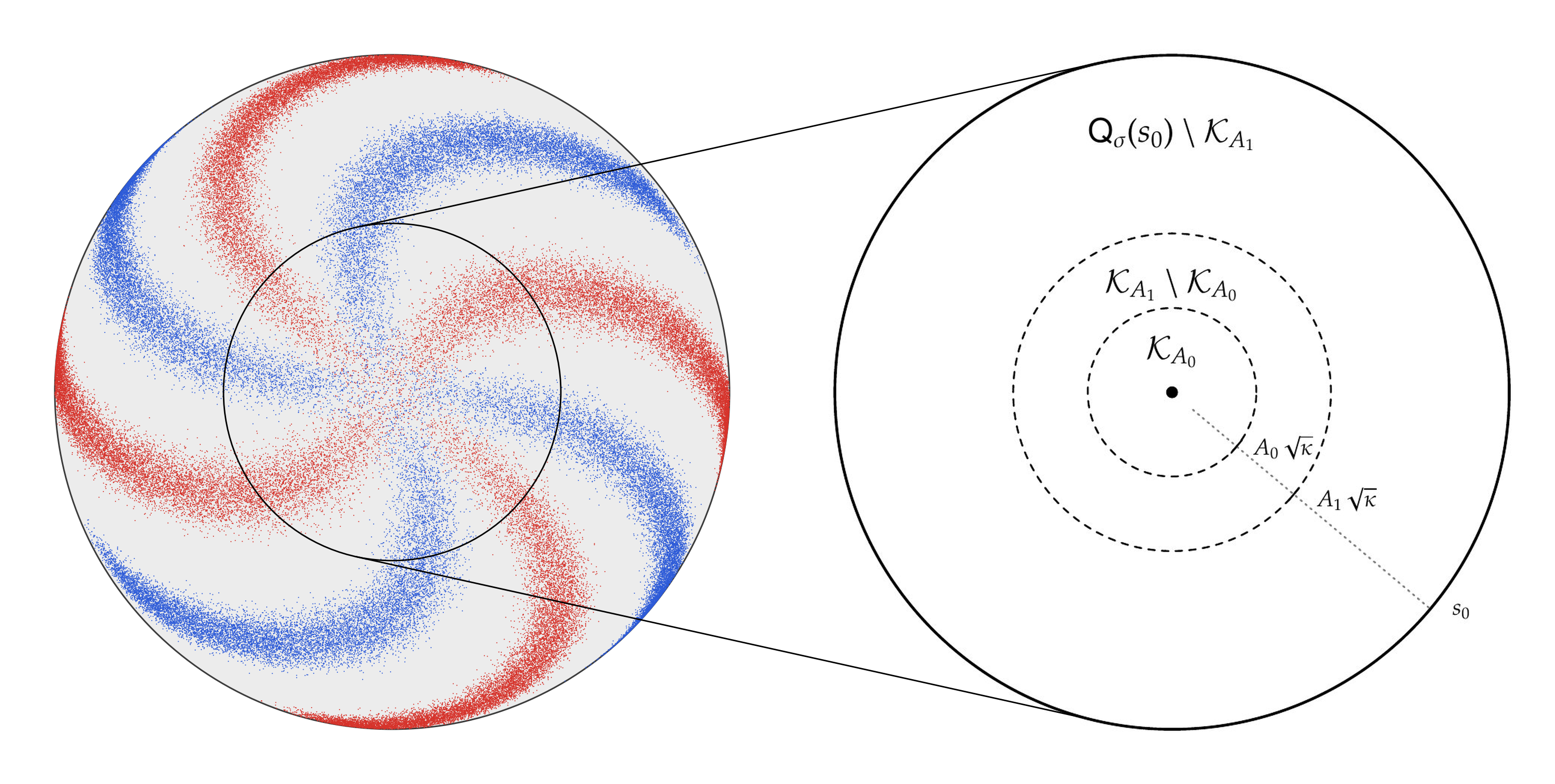}
    \caption{Stochastic dynamics near fixed points $F_\sigma$. On the left, the configuration after one iteration of $10^5$ tracers on $\Sph$, shown from the south pole, with $\omega_1\approx 0.91$, $\omega_2\approx -0.47$ and  $\kappa=5\times10^{-4}$.
    On the right, the corresponding three-zone decomposition around the fixed point.}
    \label{fig:three_zones_core}
\end{figure}

Recall from Section~\ref{sec:intro} that, for $\omega=(\omega_1,\omega_2)\in[-N,N]^2$, $t\in[0,2]$ and $p\in \Sph$, the associated Rossby-Haurwitz vector field is
\begin{align}\label{eq:defin_u_omega_t02}
    u_\omega(t,p) := \omega_1u^3(p)\mathbf{1}_{[0,1)}(t) + \omega_2u^2(p)\mathbf{1}_{[1,2)}(t).
\end{align}
Moreover, $\phi_t^\kappa$ satisfies
\begin{equation}\label{eq:piecewise_sde_fx}
\left\lbrace
     \begin{aligned}
     \dd \phi_t^\kappa &= \bigl(u_\omega(t,\phi_t^\kappa)-2\kappa \phi_t^\kappa\bigr)\,\dd t + \sqrt{2\kappa}\,\Pi_{\phi_t^\kappa}\,\dd \bB_t,  \\
     \phi_0^\kappa &= p, 
    \end{aligned}
    \right.
\end{equation}
where $p\in \Sph$, while $\phi_t$ solves the corresponding $\kappa=0$ deterministic equation,
\begin{align}\label{eq:piecewise_ode_fx}
\left\lbrace
\begin{aligned}
    \dot \phi_t &= u_\omega(t,\phi_t), \\ 
    \phi_0 &= p.
    \end{aligned}
    \right.
\end{align}
Finally, recall also that
\begin{align}
    f_\omega^\kappa(p):=\phi_2^\kappa, \qquad f_\omega(p):=\phi_2.
\end{align}

Before specializing to each set of fixed points $F_\sigma$, we begin by applying Lemma \ref{lemma:EXPsupomegaLipTPP} to record the stability in expectation of $f_\omega^\kappa$ with respect to $f_\omega$.

\begin{lemma}\label{lemma:noisy_two_step_L2}
There exists $C_N=C(N)>0$ such that 
\begin{align}
    \EE_\bB |\phi^\kappa_1-\phi_1|^2 \leq C_N\kappa,  \qquad \EE_\bB |\phi_2^\kappa-\phi_2|^2 \leq C_N\kappa.
\end{align}
for all $\kappa\in(0,1]$, all $p\in \Sph$, and all $\omega=(\omega_1,\omega_2)\in[-N,N]^2$,
\end{lemma}

\begin{proof}
The first inequality follows directly from Lemma~\ref{lemma:EXPphikappaphi} with $u=\omega_1 u^3$ and $t=1$, after noticing that $C_u$ is uniform in $\omega_1\in[-N,N]$ since $|\omega_1|\leq N$. Hence, for some $C_{1,N}>0$, 
\begin{equation}
\EE_\bB |\phi_1^\kappa-\phi_1|^2 \leq C_{1,N}\kappa.
\end{equation}

For the second inequality, we introduce an intermediate point $\widehat \phi_2 := f_{\omega_2}^{2}(\phi_1^\kappa)$, which represent the evolution under the deterministic flow $f^2_{\omega_2}$ of the random endpoint $\phi_1^\kappa$.
Then 
\begin{equation}
|\phi_2^\kappa-\phi_2|^2 \leq 2|\phi_2^\kappa-\widehat \phi_2|^2 + 2|\widehat \phi_2-\phi_2|^2,
\end{equation}
where conditioning on $\phi_1^\kappa$ and using Lemma~\ref{lemma:EXPphikappaphi} with $u=\omega_2 u^2$ and $t=1$ leads to the bound 
\begin{equation}
\EE_\bB\left[\,|\phi_2^\kappa-\widehat \phi_2|^2 \,\big|\, \phi_1^\kappa \right] \leq C_{1,N}\kappa,
\end{equation}
thus $\EE_\bB |\phi_2^\kappa-\widehat \phi_2|^2 \leq C_{1,N}\kappa$.

For the second term, since $f_{\omega_2}^{2}$ is a smooth diffeomorphism of $\Sph$ and $\omega_2\in[-N,N]$, there exists $ C_{2,N}>0$ such that
\begin{equation}
|\widehat \phi_2-\phi_2| = |f_{\omega_2}^{2}(\phi_1^\kappa)-f_{\omega_2}^{2}(\phi_1)| \leq C_{2,N} |\phi_1^\kappa-\phi_1|.
\end{equation}
As a result, $\EE_\bB |\widehat \phi_2-\phi_2|^2 \leq C_{1,N} C_{2,N}^2 \kappa$ and combining the two estimates proves the lemma.
\end{proof}

\subsection{Stochastic drift function for $F_\sfx$ fixed points}\label{ssec:viscous_drift_Fx}
We now turn our attention to the fixed points present in the inviscid dynamics.
Starting from $F_\sfx=\{(\pm1,0,0)\}$, we discuss the drift condition for the viscous Lyapunov function 
\begin{align}
    V_{\sfx,\kappa}(p):=\max\{r_\sfx(p),\sqrt{\kappa}\}^{-\alpha}, \qquad r_\sfx(p):=\max\{|y|,|z|\},
\end{align}
where $p=(x,y,z)\in \Sph$, and throughout we fix $\alpha\in\Bigl(0,\frac14\Bigr)$. 

We work around the fixed point $(1,0,0)$, noting that all arguments apply to its opposite pole by symmetry. We define, for $y^2+z^2<s_0^2$, the local chart and its corresponding neighbourhood respectively 
\begin{align}
    \Psi_{\sfx}(y,z) := \bigl(\sqrt{1-y^2-z^2},\,y,\,z\bigr),\qquad \mathsf{Q}_{\sfx}(s_0) := \Psi_{\sfx}\bigl(\{(y,z)\in\RR^2:\ y^2+z^2<s_0^2\}\bigr).
\end{align}

The goal of this subsection is the following Lyapunov-Foster drift estimate.

\begin{proposition}\label{prop:driftkappax}
There exists $s_0$ small enough such that there exist $\gamma\in(0,1)$, $\beta_\sfx\geq 1$ and $\kappa_0>0$ for which
\begin{align}
    \EE_{\omega,\bB}\left[V_{\sfx,\kappa}(f_\omega^\kappa(p))\right]  \leq  \gamma\,V_{\sfx,\kappa}(p) + \beta_\sfx \mathbf{1}_{\mathsf{Q}_\sfx(s_0)^c}(p)
\end{align}
for all $p\in \Sph$ and for all $\kappa\in(0,\kappa_0)$.
\end{proposition}

As we previously commented, the proof of the proposition is divided according to the distance of $p$ to the fixed points $F_\sfx$, given by $r_\sfx(p)$, relative to the stochastic noise, which is of order $\sqrt{\kappa}$.

From Lemma~\ref{lemma:noisy_two_step_L2} it follows immediately a uniform bound in $\kappa\in(0,1)$ for the expected value of $V_{\sfx,\kappa}(\phi_2^\kappa)$ for points outside the small cap $\mathsf{Q}_{\sfx}(s_0)^c$.

\begin{lemma}\label{lemma:driftboundawayx}
    There exists $C>0$ and $\kappa_0>0$ such that
    \begin{align}
        \EE [V_{\sfx,\kappa}(\phi_2^\kappa)] \leq C
    \end{align}
    for all $p\in \mathsf{Q}_{\sfx}(s_0)^c$ and all $\kappa\in(0,\kappa_0)$.
\end{lemma}

\begin{proof}
    Since the mapping $(p,\omega)\mapsto r_\sfx(\phi_2(p,\omega))$ is continuous on the compact set $\mathcal{K}:=\mathsf{Q}_\sfx(s_0)^c\times[-N,N]^2$, there exists $c_0>0$ such that $r_\sfx(\phi_2)\geq c_0$ for all points in $\mathcal{K}$.
    Now, using that the map $r_\sfx$ is Lipschitz continuous on $\Sph$ with constant $1$, we define
    \begin{equation}
        G:=\{|\phi^\kappa_2-\phi_2|\leq \delta_0\}, \qquad \delta_0:=\frac{c_0}{2}.
    \end{equation}
    On $G$ we use Lipschitz continuity to write 
    \begin{equation}
        V_{\sfx,\kappa}(\phi_2^\kappa) \leq r_\sfx(\phi_2^\kappa)^{-\alpha} \leq (r_\sfx(\phi_2)-|\phi_2^\kappa-\phi_2|)^{-\alpha} \leq  \left(\frac{c_0}{2}\right)^{-\alpha},
    \end{equation}
    while on $G^c$, using Chebyshev's inequality and Lemma~\ref{lemma:noisy_two_step_L2}, we estimate
    \begin{align}
    \PP_{\omega,\bB}(G^c)\leq\delta_0^{-2}\EE_{\omega,\bB}|\phi_2^\kappa-\phi_2|^2\leq C_N\delta_0^{-2}\kappa.
    \end{align}
    Finally, since on $G^c$ holds the trivial bound $V_{\sfx,\kappa}(\phi_2^\kappa)\leq \kappa^{-\alpha/2}$, we deduce
    \begin{align}
    \EE_{\omega,\bB}[V_{\sfx,\kappa}(\phi_2^\kappa)] \leq \EE_{\omega,\bB}[V_{\sfx,\kappa}(\phi_2^\kappa)|G]\PP_{\omega,\bB}(G)+\EE_{\omega,\bB}[V_{\sfx,\kappa}(\phi_2^\kappa)|G^c]\PP_{\omega,\bB}(G^c) \leq  \left(\frac{c_0}{2}\right)^{-\alpha} + C_N\delta_0^{-2}\kappa^{1-\alpha/2},
\end{align}
which, since $\alpha<2$, is uniform in $\kappa\in(0,1)$.
\end{proof}

\subsubsection{Stochastic regime} 
We consider here those $p\in \mathsf{Q}_\sfx(s_0)$ such that $r_\sfx(p)\leq A \sqrt{\kappa}$, for some $A>1$ to be determined. The point $p$ is now at distance of order $\sqrt{\kappa}$ to the fixed point $(1,0,0)$ and since the velocity field vanishes linearly at $(1,0,0)$ we observe that diffusion and advection are of the same order. To precisely describe the motion in this regime, we rescale our variables.

For $A>1$ we define the compact set
\begin{align}
    K_A:=\lbrace (y,z)\in\R^2:\ \max\lbrace  |y|,|z|\rbrace  \leq A \rbrace  ,
\end{align}
and, for some $\alpha\in(0,1)$, the functions
\begin{align}
    g_\sfx(y,z):=h(y,z)^{-\alpha}, \quad h(y,z) := \max \lbrace |y|,|z|, 1 \rbrace.
\end{align}
We have the trivial bounds $A^{-\alpha}\leq g_\sfx(y,z) \leq 1$, for all $(y,z)\in K_A$ and for $p=\Psi_\sfx(\sqrt{\kappa}y,\sqrt{\kappa}z)\in \mathsf{Q}_\sfx(s_0)$, we have 
\begin{align}\label{eq:exact_scaling_identity}
    V_{\sfx,\kappa}(p)=\kappa^{-\alpha/2}g_\sfx(y,z).
\end{align}
Starting from $u_\omega$ defined in \eqref{eq:defin_u_omega_t02}, we introduce the rescaled two-step process
\begin{align}
    \widetilde{f}_\omega^\kappa(y,z)= \frac{1}{\sqrt{\kappa}}\left( (f_\omega^\kappa(\Psi_\sfx(\sqrt{\kappa}y,\sqrt{\kappa}z)))_y, (f_\omega^\kappa(\Psi_\sfx(\sqrt{\kappa}y,\sqrt{\kappa}z)))_z \right),
\end{align}
which can be rewritten as
\begin{align}
    \widetilde{f}_\omega^{\kappa}(y,z) := \kappa^{-\frac12}((\phi_2^{\kappa})_y,(\phi_2^{\kappa})_z), 
\end{align}
where $\phi_2^{\kappa}$ is the solution at time $2$ of the stochastic differential equation \eqref{eq:piecewise_sde_fx} with initial datum $\phi_0^\kappa=\Psi_\sfx(\sqrt{\kappa}y,\sqrt{\kappa}z)$.

For $(y,z)\in K_A$ we have the following equivalence result.

\begin{lemma}\label{lemma:boundary_layer_reduction}
Let $A>1$ and $\alpha\in(0,1)$. The following statements are equivalent:
\begin{enumerate}
    \item there exists $\gamma\in(0,1)$ such that
    \begin{align}\label{eq:physical_core_drift}
        \EE_{\omega,\bB}\left[V_{\sfx,\kappa}(\phi_2^\kappa)\right]\leq\gamma V_{\sfx,\kappa}(p)
    \end{align}
    for all sufficiently small $\kappa$ and all $p\in \mathsf{Q}_\sfx(s_0)$ with $r_\sfx(p)\leq A\sqrt{\kappa}$,
    \item There exists $\gamma\in(0,1)$ such that
    \begin{align}\label{eq:scaled_core_drift}
         \EE_{\omega,\bB}[g_\sfx(\widetilde{f}_\omega^\kappa(y,z))]\leq\gamma g_\sfx(y,z),
    \end{align}
    for all sufficiently small $\kappa$ and all $(y,z)\in K_A$,
\end{enumerate}
\end{lemma}

\begin{proof}
For $p=\Psi_\sfx(\sqrt{\kappa}y,\sqrt{\kappa}z)$, we have $r_\sfx(p)\leq A\sqrt{\kappa}$ if and only if $(y,z)\in K_A$. By definition of $\widetilde{f}_\omega^\kappa(y,z)$, we have
\begin{align}
    \EE_{\omega,\bB}\left[V_{\sfx,\kappa}(\phi_2^\kappa)\right] &= \kappa^{-\alpha/2}  \EE[g_\sfx(\widetilde{f}_\omega^\kappa(y,z))]
\end{align}
for all $(y,z)\in K_A$. In view of \eqref{eq:exact_scaling_identity}, the lemma is proved.
\end{proof}

Thanks to the above lemma, we next aim to show $g_\sfx$ satisfies the drift condition \eqref{eq:scaled_core_drift} for $\widetilde{f}_\omega^\kappa$. To that purpose, we first obtain the exact stochastic differential equations satisfied by the rescaled one-point process.
We denote it as $\widetilde \phi_t^\kappa:=(\widetilde y_t^\kappa, \widetilde z_t^\kappa)$ for consistency, hence having $\widetilde{f}_\omega^\kappa=\widetilde\phi_2^\kappa$.
Using the explicit formulas
\begin{align}\label{eq:u2u3-explicit-boundary}
    u^3(x,y,z)=z(-y,x,0), \qquad u^2(x,y,z)=y(z,0,-x).
\end{align}
we obtain for $t\in[0,1)$ and $t\in[1,2)$ respectively  
\begin{equation}\label{eq:explicit_sde_sfx}
\left\lbrace\begin{aligned}
    d \widetilde y_t^{\kappa} &= \left(\omega_1 x_t^{\kappa} \widetilde z_t^{\kappa} - 2\kappa \widetilde y_t^{\kappa}  \right) \d t  + \sqrt{2}e_2 \cdot \Pi_{\phi_t^{\kappa}}\d \bB_t, \\
    d \widetilde z_t^{\kappa} &= -2\kappa \widetilde z_t^{\kappa} \d t + \sqrt{2} e_3\cdot  \Pi_{\phi_t^{\kappa}}\d \bB_t,
\end{aligned} \right.
\qquad 
\left\lbrace\begin{aligned}
    d \widetilde y_t^{\kappa} &= - 2\kappa \widetilde y_t^{\kappa} \d t  + \sqrt{2}e_2 \cdot \Pi_{\phi_t^{\kappa}}\d \overline\bB_{t}, \\
    d \widetilde z_t^{\kappa} &= -\left( \omega_2 x_t^{\kappa} \widetilde y_t^{\kappa}   + 2\kappa \widetilde z_t^{\kappa} \right) \d t + \sqrt{2} e_3\cdot  \Pi_{\phi_t^{\kappa}}\d \overline\bB_{t},
\end{aligned} \right.
\end{equation}
where $\overline\bB_t=\bB_{t}-\bB_1$, $x_t^\kappa=(\phi_t^\kappa)_x=\sqrt{1-(y_t^\kappa)^2-(z_t^\kappa)^2}$ for $t\in[1,2)$, and  
\begin{equation}\label{eq:explicit_PI_phikappat}
    \Pi_{\phi_t^{\kappa}}= 
    \begin{pmatrix}
        1-( x_t^\kappa)^2  &  -x_t^\kappa y_t^\kappa  &  -x_t^\kappa z_t^\kappa \\
          -x_t^\kappa y_t^\kappa & 1-( y_t^\kappa)^2  &  -y_t^\kappa z_t^\kappa \\
           -x_t^\kappa z_t^\kappa  &  -y_t^\kappa z_t^\kappa & 1-( z_t^\kappa)^2
    \end{pmatrix}.
\end{equation}
The aim is to compare $\widetilde \phi_t^\kappa$ with the corresponding limiting solution $\widetilde \phi_t:=(\widetilde y_t, \widetilde z_t)$ at time $t=2$ (i.e. $\widetilde{f}^\kappa_\omega$ and $\widetilde{f}_\omega(y,z)$ respectively), which satisfies, for $t\in[0,1)$ and $t\in[1,2)$ respectively, 
\begin{equation}\label{eq:explicit_ode_sfx}
\left\lbrace\begin{aligned}
    d  \widetilde y_t &= \omega_1  \widetilde z_t  \d t  + \sqrt{2} \d B_t^2, \\
   d \widetilde z_t &=  \sqrt{2} \d B_t^3,
\end{aligned}\right.
\qquad 
\left\lbrace\begin{aligned}
    d \widetilde y_t &=  \sqrt{2}\d \overline{B}_{t}^2, \\
    d \widetilde z_t &= -\omega_2  \widetilde y_t\d t + \sqrt{2} \d \overline{B}_t^3.
\end{aligned}\right.
\end{equation}
The comparison between $\widetilde\phi^\kappa_t$ and $\widetilde \phi_t$ is perturbative only as long as the rescaled state remains moderate in size.
This motivates the introduction of the following stopping time\footnote{The exponent $\frac14$ is chosen to balance perturbative remainders in \eqref{eq:explicit_sde_sfx} with the exit event contribution.}
\begin{align}\label{def:stopping_time}
    \tau_\kappa := \inf\bigl\{ t\in[0,2):\ |\widetilde \phi^\kappa_{t}|\geq \kappa^{-\frac14} \bigr\}\wedge 2.
\end{align}
Indeed, if $|\widetilde \phi^\kappa_{t}|\leq \kappa^{-\frac14}$, for $\kappa$ small enough
\begin{align}\label{eq:stopping_time_chart_permanence}
    x_t^\kappa=\sqrt{1-\kappa |(\widetilde y^\kappa_{t},\widetilde z^\kappa_{t})|^2}\geq \frac12,
\end{align}
namely the path $t\mapsto \phi_t^\kappa$ starting from $(y,z)\in K_A$ remains in the chart around $(1,0,0)$.

\begin{lemma}\label{lemma:EXPdiffrescaledlimiting}
Let $A,N>1$. There exists a constant $C_{A,N}>0$ and $\kappa_0>0$ such that for all $\omega\in[-N,N]^2$ and all $(y,z)\in K_A$ there holds
\begin{align}
    \EE\bigl[|\widetilde{f}^\kappa_\omega(y,z)-\widetilde{f}_\omega(y,z)|\bigr]\leq C_{A,N}\kappa^{1/4},
\end{align}
for all $\kappa\in(0,\kappa_0)$.
\end{lemma}

\begin{proof}
Fix $\kappa\in(0,\kappa_0)$, where $\kappa_0>0$ is chosen small enough so that \eqref{eq:stopping_time_chart_permanence} applies. Let $\tau_\kappa$ be the stopping time introduced in \eqref{def:stopping_time} and decompose
\begin{align}
    \EE\Bigl[|\widetilde{f}^\kappa_\omega-\widetilde{f}_\omega|\Bigr]
    &=\EE\Bigl[|\widetilde{f}^\kappa_\omega-\widetilde{f}_\omega| \Big| \tau_\kappa \geq 2 \Bigr]\PP(\tau_\kappa \geq 2)  + \EE\Bigl[  |\widetilde{f}^\kappa_\omega-\widetilde{f}_\omega|  \big| \tau_\kappa<2  \Bigr] \PP(\tau_\kappa < 2) \\
    &\leq\EE\Bigl[|\widetilde\phi^\kappa_2-\widetilde\phi_2| \Bigr]  + \EE\Bigl[  \left(\sup_{0\leq t\leq 2}|\widetilde\phi^\kappa_2| + \sup_{0\leq t\leq 2}|\widetilde\phi_2| \right) 1_{\lbrace \tau_\kappa<2 \rbrace} \Bigr] \\
    &\leq \EE\Bigl[\sup_{0\leq t\leq 2}|\widetilde\phi_t^\kappa-\widetilde\phi_t|^2\Bigr]^{1/2} + \Bigl( \EE \Bigl[\sup_{0\leq t\leq 2}|\widetilde\phi^\kappa_t|^2\Bigr]^\frac12 + \EE \Bigl[\sup_{0\leq t\leq 2}|\widetilde\phi_t|^2\Bigr]^\frac12 \Bigr)
    \PP(\tau_\kappa<2)^\frac12.
\end{align}
A standard Burkholder-Davis-Gundy argument and Gr\"onwall inequality applied to the processes $\widetilde \phi_t^\kappa$, $ \widetilde \phi_t$, and $ \widetilde \phi^\kappa_{t\wedge\tau_\kappa}-\widetilde \phi_{t\wedge\tau_\kappa}$ yields
\begin{align}\label{eq:BDGforsupdifference}
    \EE\Bigl[\sup_{0\leq t\leq 2}|\widetilde \phi_t^\kappa|^2\Bigr] \leq C_{A,N}, \qquad \EE\Bigl[\sup_{0\leq t\leq 2}|\widetilde \phi_t|^2\Bigr] \leq C_{A,N},\qquad  \EE\Bigl[\sup_{0\leq t\leq 2}|\widetilde \phi^\kappa_{t\wedge\tau_\kappa}-\widetilde \phi_{t\wedge\tau_\kappa}|^2\Bigr]\leq C_{A,N}\kappa^{1/2},
\end{align}
where $\widetilde\phi^\kappa_{t\wedge\tau_\kappa}$ and $\widetilde \phi_{t\wedge\tau_\kappa}$ satisfy the stopped version of \eqref{eq:explicit_sde_sfx} and \eqref{eq:explicit_ode_sfx} respectively.
Indeed, with the choice of $\tau_\kappa$ all the remainder terms are of order $\kappa^{\frac14}$ while $u_\omega$ is globally Lipschitz with Lipschitz constant bounded by $N$ uniformly in $t\in[0,2]$. 

We estimate
\begin{align}
    \PP(\tau_\kappa<2) = \PP\Bigl(\sup_{0\leq t\leq 2}|\widetilde\phi^\kappa_t|\geq \kappa^{-1/4}\Bigr) \leq \kappa^{1/2} \EE\Bigl[\sup_{0\leq t\leq 2}|\widetilde\phi^\kappa_t|^2\Bigr] \leq C_{A,N}\kappa^{1/2},
\end{align}
which combined with \eqref{eq:BDGforsupdifference} concludes the proof.
\end{proof}
It follows immediately from the above proposition that, for small $\kappa>0$, establishing a drift condition for $\widetilde{f}_\omega$ is enough to derive a drift for $\widetilde{f}^\kappa_\omega$.

\begin{proposition}\label{prop:stabilitydriftrescaledx}
Let $\alpha,\gamma\in(0,1)$ and $A,N>1$. Then there exists $\kappa_0=\kappa_0(\alpha,\gamma,A,N)>0$ such that
\begin{align}
    \EE\bigl[|g_\sfx(\widetilde{f}_\omega^\kappa(y,z))-g_\sfx(\widetilde{f}_\omega(y,z))|\bigr] \leq \frac{1-\gamma}{2}g_\sfx(y,z)
\end{align}
for all $(y,z)\in K_A$ and all $\kappa\in(0,\kappa_0)$.
\end{proposition}

\begin{proof}
The proposition follows directly from combining $g_\sfx$ being globally Lipschitz on $\RR^2$ with Lemma~\ref{lemma:EXPdiffrescaledlimiting} and the trivial bound $g_\sfx(y,z)\geq A^{-\alpha}$ on $K_A$ to obtain
\begin{align}
    \EE\bigl[|g_\sfx(\widetilde{f}_\omega^\kappa(y,z))-g_\sfx(\widetilde{f}_\omega(y,z))|\bigr] \leq C_{\alpha,A,N}A^\alpha \kappa^{1/4} g_\sfx(y,z).
\end{align}
Choosing now $\kappa_0>0$ such that $C_{\alpha,A,N}A^\alpha\kappa_0^{1/4}\leq \frac{1-\gamma}{2}$ concludes the proof.
\end{proof}
Having established that it is enough to obtain a drift estimate for the limiting rescaled process $\widetilde f_\omega$, we integrate explicitly \eqref{eq:explicit_ode_sfx} to write, for $(y,z)\in K_A$,
\begin{equation}\label{eq:integrated_explicit_ode_sfx}
\begin{aligned}
    \left( \widetilde y_1, \widetilde z_1\right) &= \left( y + \omega_1 z + \sqrt{2}\, B_1^2 + \omega_1 \sqrt{2}\int_0^1 B_s^3\,\dd s, \quad z + \sqrt{2}\,B_1^3 \right),\\
    \left( \widetilde y_2, \widetilde z_2 \right) &= \left(  \widetilde y_1 + \sqrt{2}\,\overline B_1^2, \quad \widetilde z_1 - \omega_2 \widetilde y_1 - \omega_2 \sqrt{2}\int_0^1 \overline B_s^2\,\dd s + \sqrt{2}\,\overline B_1^3 \right)
\end{aligned}
\end{equation}
where $\omega_1,\omega_2\in[-N,N]$ uniformly and $B_t=(B_t^2,B^3_t)$ and $\overline B_t=(\overline B_t^2,\overline B_t^3)$ are independent two-dimensional Brownian motions, and all these objects are mutually independent.
To prove a drift estimate for $\widetilde f_\omega$, we divide the stochastic region $K_A=K^\sfx_{A_1}$ into an inner core $K^\sfx_{A_0}$ and an annular core $K^\sfx_{A_1}\setminus K^\sfx_{A_0}$, for suitably chosen $1 <A_0 <A_1$.

\subsubsection{The annular core $K^\sfx_{A_1}\setminus K^\sfx_{A_0}$} Consider here $1<A_0<A_1$ to be determined. The following abstract lemma records stretching estimates of shearing motion with random amplitudes. 

\begin{lemma}\label{lemma:EEmaxshearX}
Let $\alpha\in(0,1)$. Then there exists $C_\alpha>0$ such that for all $a,b\in\RR$ and all $N\geq 1$,
\begin{align}
    \EE_\omega\Bigl[\min\{1,(N|a+\omega b|)^{-\alpha}\}\Bigr] \leq C_\alpha \min\bigl\{1,(N|a|)^{-\alpha},(N^2|b|)^{-\alpha}\bigr\},
\end{align}
where $\omega\sim\mathrm{Unif}([-N,N])$.
\end{lemma}

\begin{proof}
    Set
    \begin{equation}
        I(a,b,N):=\frac1{2N}\int_{-N}^{N}\min\{1,(N|a+\omega b|)^{-\alpha}\}\,\dd\omega
    \end{equation}
    and note that for $a=b=0$ the estimate is trivial. We split the proof into two cases based on the sizes of $N|b|$ and $|a|$.
    
\diampar{Case 1} Assume $2N|b|\leq |a|$ and note that for every $\omega\in[-N,N]$, it holds $|a+\omega b|\geq \frac{|a|}{2}$ so that 
\begin{equation}
    I(a,b,N) \leq \frac1{2N}\int_{-N}^{N}\min\{1,(N|a+\omega b|)^{-\alpha}\}\,\dd\omega \leq \min\{1,(N|a|/2)^{-\alpha}\} \leq 2^\alpha \min\{1,(N|a|)^{-\alpha}\}.
\end{equation}

\diampar{Case 2} Consider now $2N|b|> |a|$, which in particular implies $b\neq 0$.  Since the function
\begin{equation}
u\longmapsto \min\{1,(N|u|)^{-\alpha}\}
\end{equation}
is even and non-increasing on $[0,\infty)$, among all intervals of fixed length $2N|b|$ the integral is maximised by the interval centred at the origin. Therefore
\begin{align}
    I(a,b,N)\leq\frac1{2N|b|}\int_{-N|b|}^{N|b|}\min\{1,(N|u|)^{-\alpha}\}\,\dd u=\frac1{2N^2|b|}\int_{-N^2|b|}^{N^2|b|}\min\{1,|v|^{-\alpha}\}\,\dd v,
\end{align}
upon changing variables $v=Nu$.
Computing the last integral explicitly leads to $I(a,b,N)\leq 1$ when $N^2|b|\leq 1$ and $I(a,b,N)\leq C_\alpha(N^2|b|)^{-\alpha}$ when $N^2|b|\geq 1$, hence 
\begin{equation}
    I(a,b,N)\leq C_\alpha\min\{1,(N^2|b|)^{-\alpha}\}.
\end{equation}
Combining the two cases concludes the proof.
\end{proof}

We immediately apply this lemma to determine the threshold $A_0>1$ such that $g_\sfx$ satisfies a local Lyapunov-drift condition for the limiting rescaled dynamics in the annular region $(K^\sfx_{A_0})^c$.

\begin{proposition}\label{prop:driftlimrescaledXannular}
Fix $\alpha\in(0,1)$. Then there exist $A_0=A_0(\alpha)>1$ and $ N_\infty=N_\infty(\alpha)\in\NN$ such that for every $N\geq N_\infty$ and every $(y,z)\in\RR^2$ with $h(y,z)\geq A_0$, there holds
\begin{align}
    \EE [g_\sfx(\widetilde f_\omega(y,z))]\leq \frac12 g_\sfx(y,z).
\end{align}
\end{proposition}

\begin{proof}
Fix $(y,z)\in\RR^2$ and let $h_0:=h(y,z)\geq 1$. Let $\widetilde{f}_{\omega_1}(y,z) = (\widetilde y_1,\widetilde z_1)$ and $\widetilde{f}_\omega(y,z) = (\widetilde y_2,\widetilde z_2)$ be given by \eqref{eq:integrated_explicit_ode_sfx}. Since $g_\sfx(\widetilde y_2,\widetilde z_2)\leq \min\{1,|\widetilde z_2|^{-\alpha}\}$, it suffices to estimate $|\widetilde z_2|$. Set
\begin{align}
    \Theta_1 :=\widetilde y_1+\sqrt{2}\int_0^1 \overline B_s^2\,\dd s, \quad \Theta_2 := \widetilde z_1+\sqrt{2}\,\overline B_1^3.
\end{align}
Then $\widetilde z_2=\Theta_2-\omega_2 \Theta_1$ and we obtain
\begin{align}
    \EE_{\omega_2}[g_\sfx(\widetilde f_\omega(y,z))] \leq \EE_{\omega_2}[\min\{1,|\Theta_2-\omega_2\Theta_1|^{-\alpha}\}] \leq C_{1,\alpha}\, \EE\Bigl[\min\{1,(N|\Theta_1|)^{-\alpha}\}\Bigr],
\end{align}
where the last inequality follows from the same argument used in the proof of Lemma~\ref{lemma:EEmaxshearX}.
We next note that $\Theta_1 = \Lambda_1 + \omega_1 \Lambda_2$, for 
\begin{align}
    \Lambda_1 := y+\sqrt{2}\,B_1^2+\sqrt{2}\int_0^1 \overline B_s^2\,\dd s, \quad \Lambda_2 := z+\sqrt{2}\int_0^1 B_s^3\,\dd s.
\end{align}
Setting $G_1:=\sqrt{2}\,B_1^2\sim \cN(0,2)$, $G_2:=\sqrt{2}\int_0^1 \overline B_s^2\,\dd s\sim \cN(0,2/3)$ and $G_3:=\sqrt{2}\int_0^1 B_s^3\,\dd s\sim \cN(0,2/3)$, we consider
\begin{align}
    F:=\left\lbrace |G_1| + |G_2| + |G_3| \leq \frac{h_0}{4} \right\rbrace.
\end{align}
On $F$ we have $|\Lambda_1-y|\leq \frac{h_0}{4}$ and also $|\Lambda_2-z|\leq \frac{h_0}{4}$ and therefore $\max\{|\Lambda_1|,|\Lambda_2|\}\geq \frac{3h_0}{4}$. Letting $A_0>1$ we have $h_0\geq 1$ and for $N>1$ we see that Lemma~\ref{lemma:EEmaxshearX} implies the existence of $C_{2,\alpha}>0$, depending only on $\alpha$, such that
\begin{align}
    \EE_{\omega_1,\bB}\Bigl[\min\{1,(N|\Theta_1|)^{-\alpha}\}| F \Bigr] = \EE_\bB \Bigl[\EE_{\omega_1}\bigl[\min\{1,(N|\Lambda_1+\omega_1 \Lambda_2|)^{-\alpha}\} | F \bigr] \Bigr] \leq C_{2,\alpha}N^{-\alpha}h_0^{-\alpha}.
\end{align}
On the other hand, 
\begin{align}
    F^c\subset\Bigl\{|G_1|>\frac{h_0}{12}\Bigr\}\cup\Bigl\{|G_2|>\frac{h_0}{12}\Bigr\}\cup\Bigl\{|G_3|>\frac{h_0}{12}\Bigr\}
\end{align}
and since each $G_j$ is a Gaussian random variable, there exist universal constants $c,C>0$ such that $\PP(F^c)\leq C e^{-c h_0^2}$. In particular, we may choose $A_0>1$ large enough depending only on $c$, $C$, $C_{1,\alpha}$ so that $C_{1,\alpha}\PP(F^c)\leq \frac14 h_0^{-\alpha}$ for all $ h_0\geq A_0$. Summing up, for $h_0\geq A_0$, we deduce that
\begin{align}
    \EE[g_\sfx(\widetilde{f}_\omega(y,z))] &\leq C_{1,\alpha}\EE_{\omega_1,\bB}\Bigl[\min\{1,(N|\Theta_1|)^{-\alpha}\}|F\Bigr]+C_{1,\alpha}\EE_{\omega_1,\bB}\Bigl[\min\{1,(N|\Theta_1|)^{-\alpha}\}|F^c\Bigr]\PP(F^c) \\
    &\leq C_{1,\alpha}C_{2,\alpha}N^{-\alpha} h_0^{-\alpha}+C_{1,\alpha}\PP(F^c)\\
    &\leq C_{1,\alpha}C_{2,\alpha}N^{-\alpha} h_0^{-\alpha}+\frac14 h_0^{-\alpha}.
\end{align}
Choosing $N_\infty\in\NN$ so large that $C_{1,\alpha}C_{2,\alpha}N^{-\alpha}\leq \frac14$ for all $N\geq N_\infty$, the lemma is proved.
\end{proof}

\subsubsection{The inner core $K^\sfx_{A_0}$}
Once $A_0$ is determined by Proposition \ref{prop:driftlimrescaledXannular} above, we now turn to studying the dynamics in the core $K^\sfx_{A_0}$. Before proving a drift estimate for $\widetilde f_\omega$ in this region, we first need some preliminary lemmas.

\begin{lemma}\label{lemma:bounded_density_truncated_reciprocal}
Let $\zeta$ be a real-valued random variable with probability density function $\rho:\R\rightarrow[0,\infty)$ satisfying $\|\rho\|_{L^\infty(\R)}\leq M$. Then, for every $c\in(0,1]$,
\begin{align}\label{eq:truncated_reciprocal_bound}
    \EE\left[\min\left\lbrace 1,\frac{c}{|\zeta|}\right\rbrace  \right] \leq \left(1+2M+2M\log(c^{-1})\right)c.
\end{align}
\end{lemma}

\begin{proof}
Splitting the expectation into the three regions $\lbrace  |\zeta|\leq c\rbrace  $, $\lbrace   c<|\zeta|\leq 1\rbrace  $, and $\lbrace  |\zeta|>1\rbrace  $ we obtain
\begin{align}
    \EE\left[\min\left\lbrace  1,\frac{c}{|\zeta|}\right\rbrace  \right]
    &=\PP(|\zeta|\leq c) +  c\,\EE\left[|\zeta|^{-1}\big|  c<|\zeta|\leq 1\right]\PP(c<|\zeta|\leq 1) + c\,\EE\left[|\zeta|^{-1}  \big| |\zeta|>1 \right]\PP(|\zeta|>1) \notag\\
    &\leq \PP(|\zeta|\leq c)  + c\int_{\lbrace  c<|x|\leq 1\rbrace  } \frac{\rho(x)}{|x|}\,\dd x + c\,\PP(|\zeta|>1).
\end{align}
Since $\rho\leq M$, we have $\PP(|\zeta|\leq c)\leq 2Mc$, while
\begin{align}
    \int_{\lbrace c<|x|\leq 1\rbrace  } \frac{\rho(x)}{|x|}\,\dd x \leq M\int_{\lbrace c<|x|\leq 1\rbrace  } \frac{1}{|x|}\,\dd x =  2M\int_c^1 \frac{1}{x}\,\dd x = 2M\log(c^{-1}).
\end{align}
Finally, $c\,\PP(|\zeta|>1)\leq c$ and the combination of these estimates completes the proof.
\end{proof}

The next lemma shows that the shearing motion localises points with small probability.

\begin{lemma}\label{lemma:averaged_affine_small_ball}
Let $\omega\sim \mathrm{Unif}[-N,N]$ with $N\geq1$, and let $\eta\sim \mathcal{N}(0,\upsilon^2)$ be independent of $\omega$. Let $\mathcal H$ be a $\sigma$-algebra independent of $(\omega,\eta)$, and let $\Lambda_1,\Lambda_2$ be real-valued, $\mathcal H$-measurable random variables. Then, for every $r\in(0,N]$,
\begin{align}\label{eq:averaged_affine_small_ball}
    \PP\left(|\Lambda_1+\omega(\Lambda_2+\eta)|\leq r \,\big|\, \mathcal H\right) \leq  \frac{2r}{N}\left(1+\upsilon^{-1} +\upsilon^{-1}\log\frac{N}{r}\right) \qquad\text{a.s.,}
\end{align}
In particular, $ \PP\left(|\Lambda_1+\omega(\Lambda_2+\eta)|\leq r \right)  \leq  \frac{2r}{N}\left(1+\upsilon^{-1} +\upsilon^{-1}\log\frac{N}{r}\right)$.
\end{lemma}

\begin{proof}
Note that the set $a\in[-N,N]$ for which $|\Lambda_1+a(\Lambda_2+\eta)|\leq r$ has Lebesgue measure at most $\min\left\lbrace 2N,\frac{2r}{|\Lambda_2+\eta|}\right\rbrace$. Since $\omega\sim \text{Unif}[-N,N]$, we have
\begin{align}
    \PP\left(|\Lambda_1+\omega(\Lambda_2+\eta)|\leq r \,\big|\, \mathcal H,\eta\right) \leq  \min\left\lbrace 1,\frac{r}{N|\Lambda_2+\eta|}\right\rbrace  .
\end{align}
Next,
\begin{align}
    \PP\left(|\Lambda_1+\omega(\Lambda_2+\eta)|\leq r \,\big|\, \mathcal H\right) = \EE_\eta\left [\PP\left(|\Lambda_1+\omega(\Lambda_2+\eta)|\leq r \,\big|\, \mathcal H,\eta\right)\right]  \leq\EE_{\Lambda_2+\eta} \left[\min\left\lbrace 1,\frac{r}{N|\Lambda_2+\eta|}\right\rbrace \right]
\end{align}
and applying Lemma~\ref{lemma:bounded_density_truncated_reciprocal} with $c=\frac{r}{N}\in(0,1]$, $\zeta=\Lambda_2+\eta\sim \mathcal{N}(\Lambda_2,\upsilon^2)$ whose probability density function $\varrho_{\Lambda_2+\eta}$ satisfies $\Vert \varrho_{\Lambda_2+\eta} \Vert_{L^\infty}\leq (2\pi\upsilon^2)^{-\frac12}$ we obtain
\begin{align}
    \PP\left(|\Lambda_1+\omega(\Lambda_2+\eta)|\leq r \,\big|\, \mathcal H\right)   &\leq \left(1+2(2\pi\upsilon^2)^{-\frac12}+2(2\pi\upsilon^2)^{-\frac12}\log(\frac{N}{r})\right)\frac{r}{N},
\end{align}
as claimed. Finally, set $E= \lbrace  |\Lambda_1+\omega(\Lambda_2+\eta)|\leq r \rbrace  $ and note that
\begin{align}
    \PP(E) = \EE[1_E] = \EE[\EE[1_E|\mathcal H]] = \EE[\PP(E|\mathcal{H})].
\end{align}
The lemma follows as the bound for $\PP(E|\mathcal{H})$ is deterministic and independent of $\mathcal{H}$.
\end{proof}

We are now in position to prove a drift for the limiting rescaled process. 
\begin{proposition}\label{prop:driftlimrescaledXinner}
    Let $A_0>1$ and $\alpha\in(0,1)$. There exists $N_0=N_0(A_0)\in \NN$ and $\gamma\in(0,1)$ such that 
    \begin{align}
        \EE [ g_\sfx(\widetilde{f}_\omega(y,z))] \leq \gamma g_\sfx(y,z),
    \end{align}
    for all $N\geq N_0$ and all $(y,z)\in K^\sfx_{A_0}$.
\end{proposition}

\begin{proof}
    We begin by recalling that ${A_0}^{-\alpha}\leq g_\sfx(y,z) \leq 1$ and $1\leq h(y,z) \leq A_0$, for all $(y,z)\in K^\sfx_{A_0}$. In particular, there holds
    \begin{align}
         g_\sfx(\widetilde{f}_\omega(y,z)) \leq {A_0}^\alpha g_\sfx(y,z)
    \end{align}
    for all $(y,z)\in K^\sfx_{A_0}$. Let $(y,z)\in K^\sfx_{A_0}$ and set $h_0=h(y,z)\in [1,A_0]$, we define the set $E:=\lbrace |\widetilde z_2| \leq 2h_0 \rbrace$. We have
\begin{align}
\widetilde z_2 = \left(\widetilde z_1+\sqrt{2}\,\overline B_1^3\right)-\omega_2\left(\widetilde y_1+\sqrt{2}\int_0^1 \overline B_s^2\,\dd s\right).
\end{align}
For the $\sigma$-algebra $\mathcal H_1:=\sigma(\omega_1,(B_t)_{t\in[0,1]},\overline B_1^3)$ we observe that $\widetilde y_1$ and $\widetilde z_1+\sqrt{2}\,\overline B_1^3$ are $\mathcal H_1$-measurable, while $\eta:=\sqrt{2}\int_0^1 \overline B_s^2\,\dd s\sim \mathcal{N}(0,2/3)$ is independent of $\mathcal H_1$ and of $\omega_2$. Using Lemma~\ref{lemma:averaged_affine_small_ball} on $\mathcal H_1$, with
\begin{align}
\Lambda_1=\widetilde z_1+\sqrt{2}\,\overline B_1^3,\qquad \Lambda_2=\widetilde y_1
\end{align}
and threshold $r=2h_0\leq 2A_0\leq N$ we obtain
\begin{align}
    \PP(E) \leq C\,\frac{2h_0}{N}\left(1+\log\frac{N}{2h_0}\right) \leq  CA_0\frac{1+\log N}{N}.
\end{align}
On the other hand, on $E^c$ we have that $|\widetilde{z}_2| \geq 2h_0$ and $g_\sfx(\widetilde y_2, \widetilde z_2) \leq 2^{-\alpha}g_\sfx(y,z)$. Combining both estimates, we obtain
\begin{align}
    P_\sfx g_\sfx(y,z) &= \PP(E)\EE[g_\sfx(\widetilde y_2, \widetilde z_2) | E] + \PP(E^c)\EE[g_\sfx(\widetilde y_2, \widetilde z_2) | E^c] \\
    &\leq \left( CA_0^{\alpha + 1}\frac{1 + \log (N)}{N} + 2^{-\alpha} \right) g_\sfx(y,z)
\end{align}
and the lemma follows, for $\gamma:=\frac{1+2^{-\alpha}}{2}<1$, after choosing $N_0=N_0(A_0,\alpha)$ such that $CA_0^{\alpha + 1}\frac{1 + \log (N)}{N}\leq \frac{1-2^{-\alpha}}{2}$ for all $N\geq N_0$ large enough.
\end{proof}

In view of Propositions \ref{prop:stabilitydriftrescaledx}, \ref{prop:driftlimrescaledXannular}, \ref{prop:driftlimrescaledXinner}, and Lemma~\ref{lemma:boundary_layer_reduction}, the following result is immediate.
\begin{corollary}\label{cor:driftrescaledx}
Let $A_1>A_0$. Then, there exists $\gamma\in(0,1)$ and $\kappa_0>0$ such that
\begin{align}
    \EE \left[ V_{\sfx,\kappa}(f_\omega^\kappa(p))\right] \leq \gamma V_{\sfx,\kappa}(p)
\end{align}
for all $p\in \mathsf{Q}_\sfx(s_0)$ satisfying $r_\sfx(p)\leq A_1\sqrt{\kappa}$ and all $\kappa\in (0,\kappa_0)$.
\end{corollary}
\begin{proof}
    Lemma \ref{lemma:boundary_layer_reduction} allows us to work directly on $g_\sfx$, for which 
    \begin{equation}
        \EE[g_\sfx(\widetilde\phi_2^\kappa)] =\EE[g_\sfx(\widetilde\phi_2^\kappa)-g_\sfx(\widetilde\phi_2)]+\EE[g_\sfx(\widetilde\phi_2)].
    \end{equation}
    The second term is $\gamma$-contractive either by Proposition \ref{prop:driftlimrescaledXinner}, if $(y,z)\in K^\sfx_{A_0}$, or $\frac12$-contractive by Proposition \ref{prop:driftlimrescaledXannular}, if $(x,y)\in K^\sfx_{A_1}\setminus K^\sfx_{A_0}$. 
    Applying then Proposition \ref{prop:stabilitydriftrescaledx} with  $\bar\gamma=\frac{1+\max\{\gamma,\frac12\}}{2}$ concludes the proof.
\end{proof}

\subsubsection{Inviscid-dominated regime}
We consider here the region where the inviscid dynamics dominate the stochastic processes by assuming that $r_\sfx(p)\geq A_1\sqrt{\kappa}$, for some $A_1\geq A_0$ to be determined. We show that a version of Lemma \ref{lemma:localxlyap} holds for all $p\in \mathsf{Q}_\sfx(s_0)$.

\begin{proposition}\label{prop:driftanuularx}
Let $\alpha=\frac18$. Then there exist $N_1>1$, $A_1=A_1(N_1)>A_0$, $s_0>0$, $\kappa_0>0$ and $\gamma\in(0,1)$ such that, 
\begin{align}\label{eq:noisy_annulus_drift}
    \EE_{\omega,\bB}\left[V_{\sfx,\kappa}(\phi_2^\kappa)\right] \leq \gamma V_{\sfx,\kappa}(p),
\end{align}
for all $\kappa\in(0,\kappa_0)$ and all $p\in \mathsf{Q}_\sfx(s_0)$ such that $r_\sfx(p)\geq A_1\sqrt{\kappa}$.
\end{proposition}

\begin{proof}
Fix $\eta>0$ small, to be chosen during the proof, and set for convenience $r:=r_\sfx(p)=\max\lbrace  |y|,|z|\rbrace$ and note that $V_{\sfx,\kappa}(p) = r^{-\alpha}$ for all $p\in \mathsf{Q}_\sfx(s_0)$ with $r_\sfx(p)\geq A_1\sqrt{\kappa}\geq \sqrt{\kappa}$. We define the set
\begin{align}
    G:=\left\lbrace   |\phi_1^\kappa-\phi_1|\leq \eta r,\ \ |\phi_2^\kappa-\phi_2|\leq \eta r \right\rbrace 
\end{align}
and denote $r_2=r_\sfx(\phi_2)$ and $r_1=r_\sfx(\phi_1)$. Since $r_\sfx$ is $1$-Lipschitz continuous we have $|r_\sfx(\phi_2^\kappa) - r_\sfx(\phi_2)|\leq |\phi_2^\kappa - \phi_2|\leq \eta r$ on $G$. We distinguish the two cases as in the deterministic case (see Lemma~\ref{lemma:localxlyap}).

\diampar{Case $|y|\leq |z|$} Here $r=|z|$ and $V_{\sfx,\kappa}(p)=|z|^{-\alpha}$. We recall the splitting $E_1$, $E_2$, and $E_3$ for which
\begin{enumerate}[label=--]
    \item On $E_1$ there holds $r_2 \geq \frac{r}{2}$ and $\PP(E_1)\leq \frac{1}{2N^2}$.
    \item On $E_2$ there holds $r_2 \geq \frac{r}{4N}$ and $\PP(E_2)\leq \frac{4}{N}$.
    \item On $E_3=(E_1\cup E_2)^c$ there holds $r_2 \geq 2r$.
\end{enumerate}
Setting $E=E_3$, we observe that
\begin{equation}
\begin{aligned}
    r_\sfx(p_2^\kappa) \geq (2 - \eta) r \mathbf{1}_{E\cap G} + \left( \frac{1}{4N} - \eta \right) r \mathbf{1}_{E^c\cap G} \geq \frac{3r}{2} \mathbf{1}_{E\cap G} + \frac{r}{8N}\mathbf{1}_{E^c\cap G}
\end{aligned}
\end{equation}
for $\eta= \frac{1}{8N}$ and $\PP(E^c)\leq \frac{8}{N}$ for $N$ large enough. Thus, 
\begin{align}
    \EE_{\omega,\bB}\left[ V_{\sfx,\kappa}(\phi_2^\kappa)|G \right] \leq \left( \left( \frac{2}{3} \right)^ \alpha + 8\frac{(8N)^\alpha}{N} \right)V_{\sfx,\kappa}(p) \leq \gamma V_{\sfx,\kappa}(p)
\end{align}
for some $\gamma\in ((2/3)^\alpha,1)$, for $\alpha=\frac18$ and $N_1$ large enough.

\diampar{Case  $|z|\leq |y|$} Now $r=|y|$ and $V_{\sfx,\kappa}(p)=|y|^{-\alpha}$. We again recall the splitting used in Lemma \ref{lemma:localxlyap}, for which
\begin{enumerate}[label=--]
    \item On $E_1$ there holds $r_2 \geq \frac{r}{8N}$ and $\PP(E_1)\leq \frac{1}{2N^2}$.
    \item On $E_2$ there holds $r_2 \geq \frac{r}{8N^2}$ and $\PP(E_2)\leq \frac{4}{\sqrt{N}}$.
    \item On $E_3$ there holds $r_2 \geq \frac{r}{\sqrt{N}}$ and $\PP(E_3)\leq \frac{8}{\sqrt{N}}$.
    \item On $E_4=(E_1\cup E_2\cup E_3)^c$ there holds $r_2 \geq 2r$.
\end{enumerate}
Hence, setting $E=E_4$ leads to 
\begin{align}
    r_\sfx(\phi_2^\kappa) \geq (2 - \eta) r \mathbf{1}_{E\cap G} + \left( \frac{1}{8N^2} - \eta \right) r \mathbf{1}_{E^c\cap G} \geq \frac{3r}{2} \mathbf{1}_{E\cap G} + \frac{r}{16N^2}\mathbf{1}_{E^c\cap G}
\end{align}
for $\eta=\frac{1}{16N^2}$ and $\PP(E^c) \leq \frac{16}{\sqrt{N}}$. Therefore,
\begin{align}
    \EE_{\omega,\bB}\left[ V_{\sfx,\kappa}(\phi_2^\kappa)|G\right] \leq \left( \left( \frac{2}{3} \right)^ \alpha + 16\frac{(16N^2)^\alpha}{\sqrt{N}} \right)V_{\sfx,\kappa}(p) \leq \gamma V_{\sfx,\kappa}(p)
\end{align}
for some $\gamma\in ((2/3)^\alpha,1)$, as $\alpha=\frac18$ and choosing $N_1$ large enough. 

Finally, the probability of the bad set $G^c$ is estimated using Lemma~\ref{lemma:noisy_two_step_L2} and Markov's inequality as
\begin{equation}
    \PP_{\omega,\bB}(G^c)  \leq \frac{C_N\sqrt{\kappa}}{\eta r},
\end{equation}
which combined with the trivial bound $V_{\sfx,\kappa}\leq \kappa^{-\alpha/2}$ leads to 
\begin{align}\label{eq:bad_event_contribution}
    \EE_{\omega,\bB}\left[V_{\sfx,\kappa}(\phi_2^\kappa)|G^c\right]\PP(G^c)  \leq \frac{C_N}{\eta}\kappa^{(1-\alpha)/2}r^{-1} = \frac{C_N}{\eta}\left(\frac{\sqrt{\kappa}}{r}\right)^{1-\alpha}r^{-\alpha} \leq \frac{C_N}{\eta {A_1}^{1-\alpha}}V_{\sfx,\kappa}(p).
\end{align}
Choosing $A_1=A_1(N_1)>A_0$ large enough so that $\frac{C_N}{\eta A^{1-\alpha}_1} \leq \frac{1-\gamma}{2}$ concludes the proof.
\end{proof}
We can now conclude the local and global Lyapunov estimates.

\begin{proposition}\label{prop:driftkappaxlocal}
There exist $N>1$, $s_0>0$, $\gamma\in(0,1)$, and $\kappa_0>0$ such that
\begin{align}
    \EE_{\omega,\bB}\left[V_{\sfx,\kappa}(f_\omega^\kappa(p))\right] \leq \gamma V_{\sfx,\kappa}(p)
\end{align}
for all $p\in \mathsf{Q}_\sfx(s_0)$ and all $\kappa\in(0,\kappa_0)$.
\end{proposition}
\begin{proof}
Choose $A_0$ and $N_\infty$ from Proposition~\ref{prop:driftlimrescaledXannular}. Then choose $N$ so large that
\begin{align}
    N\geq\max\{N_\infty,\;N_0(A_0,\alpha),\;N_{\mathrm{out}}\}.
\end{align}
For this $N$, let $A_1=A_1(N), s_0=s_{\mathrm{out}}(N)$, and let $\gamma_{\mathrm{in}},\gamma_{\mathrm{out}}$ and $\kappa_{\mathrm{in}},\kappa_{\mathrm{out}}$ be given by Proposition~\ref{prop:driftanuularx} and Corollary~\ref{cor:driftrescaledx} respectively. Define
\begin{align}
    \kappa_0:=\min\left\{\kappa_{\mathrm{in}}, \kappa_{\mathrm{out}}, \frac{s_0^2}{A_1^2}\right\}>0,\qquad \gamma:=\max\{\gamma_{\mathrm{in}},\gamma_{\mathrm{out}}\}\in(0,1),
\end{align}
and fix $p\in \mathsf{Q}_\sfx(s_0)$ and $\kappa\in(0,\kappa_0)$. Since either $p\in  K^\sfx_{A_1}$ or $p\in  (K^\sfx_{A_1})^c$ holds, the claim follows from Corollary~\ref{cor:driftrescaledx} and Proposition~\ref{prop:driftanuularx}.
\end{proof}

\begin{proof}[Proof of Proposition~\ref{prop:driftkappax}]
Combining Proposition~\ref{prop:driftkappaxlocal} and Lemma~\ref{lemma:driftboundawayx} according to whether $p\in \mathsf{Q}_\sfx(s_0)$ or not yields
\begin{align}
    \EE_{\omega,\bB}\left[V_{\sfx,\kappa}(f_\omega^\kappa(p))\right] \leq \gamma V_{\sfx,\kappa}(p)+\beta_\sfx\mathbf{1}_{\mathsf{Q}_\sfx(s_0)^c}(p)
\end{align}
for all $p\in \mathsf{X}$ and all $\kappa\in(0,\kappa_0)$.
\end{proof}

\subsection{Stochastic drift function for $F_\sfz$ fixed points}\label{ssec:viscous_drift_Fz}
We adapt the analysis of Section~\ref{ssec:viscous_drift_Fx} to the fixed points $F_\sfz = \lbrace (0,0,\pm 1)\rbrace$.
We define $f_\omega^\kappa$, $f_\omega$, $\phi_1^\kappa$, $\phi_2^\kappa$, $\phi_1$, and $\phi_2$ exactly as in Section~\ref{ssec:viscous_drift_Fx}, and we further recall that $r_\sfz(p):=\sqrt{x^2+y^2}$ for $p=(x,y,z)\in\Sph$. We introduce the diffusive Lyapunov function
\begin{equation}
    V_{\sfz,\kappa}(p) := \max\lbrace r_\sfz(p),\sqrt{\kappa}\rbrace^{-\alpha},
\end{equation}
for some $\alpha\in(0,1)$ to be chosen later. Moreover, we define the local chart 
\begin{equation}
    \Psi_\sfz(x,y):=(x,y,\sqrt{1-x^2-y^2}),\qquad \mathsf{Q}_\sfz(s_0):=\Psi_\sfz(\{(x,y)\in\RR^2\,:\,0 < x^2+y^2<s_0^2\}),
\end{equation}
and we remark that by symmetry we will be considering only the point $(0,0,1)$. 

The main goal of the section is to prove the following proposition.
\begin{proposition}\label{prop:driftkappaz}
There exists $s_0>0$ small enough such that there exist $\gamma\in(0,1)$, $\beta_\sfz\geq 1$, and $\kappa_0>0$ for which
\begin{align}
    \EE_{\omega,\bB}\!\left[V_{\sfz,\kappa}(f_\omega^\kappa(p))\right] \leq \gamma\,V_{\sfz,\kappa}(p)+\beta_\sfz\,\mathbf 1_{\mathsf{Q}_{\sfz}(s_0)^c}(p)
\end{align}
for all $p\in \mathsf{X}$ and all $\kappa\in(0,\kappa_0)$.
\end{proposition}
In particular, Lemma~\ref{lemma:noisy_two_step_L2} being independent of the fixed points implies that a version of Lemma~\ref{lemma:driftboundawayx} holds near $F_\sfz$ after replacing $r_\sfx$ with $r_\sfz$ in the proof. 
We record it here for completeness.
\begin{lemma}\label{lemma:driftboundawayz}
There exists $C>0$ such that
\begin{align}
    \EE\left[V_{\sfz,\kappa}(\phi_2^\kappa)\right]\leq C,
\end{align}
for all $p\in \mathsf{Q}_{\sfz}(s_0)^c$ and all $\kappa\in(0,1)$.
\end{lemma}

As in the $F_\sfx$ case, the proof of Proposition \ref{prop:driftkappaz} is divided according to the size of the distance of $p$ to the fixed points $F_\sfz$, measured by $r_\sfz(p)$, relative to the stochastic noise, which is of order $\sqrt{\kappa}$.
The key simplification in the present setting is that the first map $f^3_{\omega_1}$ preserves $r_\sfz$ exactly.

\subsubsection{Stochastic regime}

We consider here those $p\in \mathsf{Q}_{\sfz}(s_0)$ such that $r_\sfz(p)\leq A\sqrt{\kappa}$, for some $A>1$ to be determined. The point $p$ is now at distance of order $\sqrt{\kappa}$ from the fixed point $(0,0,1)$, and since the velocity field vanishes linearly at $(0,0,1)$, diffusion and advection are of the same order. To precisely describe the motion in this regime, we rescale our variables.

For $A>1$, define the compact set
\begin{align}
    K_A^\sfz := \{(x,y)\in\RR^2:\ x^2+y^2\leq A^2\},
\end{align}
and, for some $\alpha\in(0,1)$, the functions
\begin{align}
    h_\sfz(x,y):=\max\{\sqrt{x^2+y^2},1\}, \qquad g_\sfz(x,y):=h_\sfz(x,y)^{-\alpha}.
\end{align}
We have the trivial bounds $A^{-\alpha}\leq g_\sfz(x,y)\leq 1$ for all $(x,y)\in K_A^\sfz$. Moreover, for $p=\Psi_\sfz(\sqrt{\kappa}x,\sqrt{\kappa}y)\in\mathsf{Q}_{\sfz}(s_0)$,
\begin{align}\label{eq:exact_scaling_identity_z}
    V_{\sfz,\kappa}(p)=\kappa^{-\alpha/2}g_\sfz(x,y).
\end{align}
Starting from $u_\omega$ defined in \eqref{eq:defin_u_omega_t02}, we introduce the rescaled two-step process
\begin{align}\label{eq:def_rescaled_process_z}
    \widetilde f_\omega^\kappa(x,y) := \frac{1}{\sqrt{\kappa}} \left( (f_\omega^\kappa(\Psi_\sfz(\sqrt{\kappa}x,\sqrt{\kappa}y)))_x,\, (f_\omega^\kappa(\Psi_\sfz(\sqrt{\kappa}x,\sqrt{\kappa}y)))_y \right),
\end{align}
which can be rewritten as
\begin{align}
    \widetilde f_\omega^\kappa(x,y) = \kappa^{-1/2}\bigl((\phi_2^\kappa)_x,(\phi_2^\kappa)_y\bigr),
\end{align}
where $\phi_2^\kappa$ is the solution at time $2$ of \eqref{eq:piecewise_sde_fx} with initial datum $\phi_0^\kappa=\Psi_\sfz(\sqrt{\kappa}x,\sqrt{\kappa}y)$.
From \eqref{eq:exact_scaling_identity_z} we immediately have the following equivalence result for $(x,y)\in K_A^\sfz$.
\begin{lemma}\label{lemma:boundary_layer_reduction_z}
Let $A>1$ and $\alpha\in(0,1)$. The following statements are equivalent:
\begin{enumerate}
    \item there exists $\gamma\in(0,1)$ such that
    \begin{align}\label{eq:physical_core_drift_z}
        \EE_{\omega,\bB}\left[V_{\sfz,\kappa}(f_\omega^\kappa(p))\right] \leq \gamma V_{\sfz,\kappa}(p)
    \end{align}
    for all sufficiently small $\kappa$ and all $p\in \mathsf{Q}_{\sfz}(s_0)$ with $r_\sfz(p)\leq A\sqrt{\kappa}$;
    \item there exists $\gamma\in(0,1)$ such that
    \begin{align}\label{eq:scaled_core_drift_z}
         \EE_{\omega,\bB}[g_\sfz(\widetilde f_\omega^\kappa(x,y))]\leq\gamma g_\sfz(x,y)
    \end{align}
    for all sufficiently small $\kappa$ and all $(x,y)\in K_A^\sfz$.
\end{enumerate}
\end{lemma}

The equations for the rescaled process $\widetilde \phi^\kappa_t=(\widetilde x^\kappa_t,\widetilde y^\kappa_t)$ are obtained from \eqref{eq:piecewise_sde_fx} and read, for $t\in[0,1)$ and $t\in[1,2)$ respectively, 
\begin{equation}\label{eq:explicit_sde_sfz}
\left\lbrace\begin{aligned}
    d \widetilde x_t^{\kappa} &= \left(-\omega_1 \widetilde y_t^{\kappa}  z_t^{\kappa} - 2\kappa \widetilde x_t^{\kappa}  \right) \d t + \sqrt{2}e_1 \cdot \Pi_{\phi_t^{\kappa}}\d \bB_t, \\
    d \widetilde y_t^{\kappa} &= \left(\omega_1 \widetilde  x_t^{\kappa}  z_t^{\kappa} - 2\kappa \widetilde y_t^{\kappa}  \right) \d t+ \sqrt{2} e_2\cdot  \Pi_{\phi_t^{\kappa}}\d \bB_t,
\end{aligned}\right.
\qquad 
\left\lbrace\begin{aligned}
    d \widetilde x_t^{\kappa} &= \left(\omega_2 \widetilde y_t^{\kappa}  z_t^{\kappa} - 2\kappa \widetilde x_t^{\kappa}  \right) \d t + \sqrt{2}e_1 \cdot \Pi_{\phi_t^{\kappa}}\d \overline\bB_{t}, \\
    d \widetilde y_t^{\kappa} &= - 2\kappa \widetilde y_t^{\kappa} \d t + \sqrt{2} e_2\cdot  \Pi_{\phi_t^{\kappa}}\d \overline\bB_{t},
\end{aligned}\right.
\end{equation}
where $\overline\bB_t=\bB_{t}-\bB_1$ for $t\in[1,2)$ and $\Pi_{\phi^\kappa_t}$ is as in \eqref{eq:explicit_PI_phikappat}. On the other hand, the limiting rescaled process $\widetilde \phi_t=(\widetilde x_t,\widetilde y_t)$ satisfies, again for $t\in[0,1)$ and $t\in[1,2)$ respectively
\begin{equation}\label{eq:explicit_ode_sfz}
\left\{
\begin{aligned}
   d \widetilde x_t &= -\omega_1 \widetilde y_t  \d t  + \sqrt{2}\d B_t^1, \\
    d \widetilde y_t &= \omega_1 \widetilde  x_t \d t  + \sqrt{2}\d B_t^2, 
\end{aligned}\right.
\qquad 
\left\{\begin{aligned}
    d \widetilde x_t &= \omega_2 \widetilde y_t   \d t  + \sqrt{2}\d \overline B_t^1, \\
    d \widetilde y_t &=  \sqrt{2}\d \overline B_t^2. 
\end{aligned}\right.
\end{equation}
As in the $F_\sfx$ section, we next compare the exact rescaled process $\widetilde f_\omega^\kappa:=\widetilde\phi^\kappa_2$ with a limiting rescaled process $\widetilde f_\omega:=\widetilde\phi_2$.

\begin{lemma}\label{lemma:EXPdiffrescaledlimitingz}
Let $A,N>1$. There exists a constant $C_{A,N}>0$ and $\kappa_0>0$ such that for all $\omega\in[-N,N]^2$ and all $(x,y)\in K_A^\sfz$ there holds
\begin{align}
    \EE\bigl[|\widetilde f_\omega^\kappa(x,y)-\widetilde f_\omega(x,y)|\bigr] \leq C_{A,N}\kappa^{1/4}
\end{align}
for all $\kappa\in(0,\kappa_0)$.
\end{lemma}

\begin{proof}
The proof is identical in structure to that of Lemma~\ref{lemma:EXPdiffrescaledlimiting}. For the rescaled trajectory $\widetilde \phi_t^\kappa$ define the stopping time
\begin{align}
    \tau_\kappa := \inf\{t\in[0,2]: |\widetilde \phi_t^\kappa|\geq \kappa^{-1/4}\}\wedge 2.
\end{align}
On $[0,\tau_\kappa]$, the spherical trajectory remains in the north chart and $z_t^\kappa\geq \frac12$. Consequently the exact rescaled coefficients are $O(\kappa^{\frac14})$-perturbations of the limiting coefficients in \eqref{eq:explicit_sde_sfz}, uniformly in $\omega\in[-N,N]^2$. A standard Burkholder--Davis--Gundy argument and Gr\"onwall's inequality then yields the analogue of \eqref{eq:BDGforsupdifference}.
Combining this with the same maximal moment bound used in the $F_\sfx$ case and the estimate $\PP(\tau_\kappa<2)\lesssim \kappa^{\frac14}$, one obtains the claim exactly as in the proof of Lemma~\ref{lemma:EXPdiffrescaledlimiting}.
\end{proof}

\begin{proposition}\label{prop:stabilitydriftrescaledz}
Let $\alpha,\gamma\in(0,1)$ and $A,N>1$. Then there exists $\kappa_0=\kappa_0(\alpha,\gamma,A,N)>0$ such that
\begin{align}
    \EE\bigl[|g_\sfz(\widetilde f_\omega^\kappa(x,y))-g_\sfz(\widetilde f_\omega(x,y))|\bigr] \leq \frac{1-\gamma}{2}g_\sfz(x,y)
\end{align}
for all $(x,y)\in K_A^\sfz$ and all $\kappa\in(0,\kappa_0)$.
\end{proposition}
\begin{proof}
The proof is identical to that of Proposition~\ref{prop:stabilitydriftrescaledx}, since $g_\sfz$ is globally Lipschitz on $\RR^2$ and $g_\sfz(x,y)\geq A^{-\alpha}$ on $K_A^\sfz$.
\end{proof}

Moving on to determine a drift for the limiting process, we record that the explicit solution of the first limiting map can be written for all $t\in[0,1]$ as
\begin{align}
    (\widetilde x_t,\widetilde y_t) = R_{\omega_1}(x,y)+G^{(1)},\qquad R_\theta=
\begin{pmatrix}
\cos\theta & -\sin\theta \\
\sin\theta & \cos\theta
\end{pmatrix},
\end{align}
and $G^{(1)}\sim \cN(0,2I_2)$ is independent of $\omega_1$. For the second map,
\begin{align}
    \widetilde y_2  &= \widetilde y_1+\sqrt{2}\,\overline B_1^2,  \\ 
    \widetilde x_2 &= \widetilde x_1+\sqrt{2}\,\overline B_1^1 + \omega_2\left( \widetilde y_1+\sqrt{2}\int_0^1 \overline B_s^2\,\dd s \right).
\end{align}
We will use this to mirror the same annular/inner-core decomposition as in the $F_\sfx$ section.

\subsubsection{The annular core $K_{A_1}^\sfz\setminus K_{A_0}^\sfz$}

We begin with the large-rescaled region.
\begin{lemma}\label{lemma:EEmaxrotationZ}
Let $\alpha\in(0,1)$. Then there exists $C_\alpha>0$ such that for all $h>0$, all $c\in\RR$ with $|c|\leq h/2$, all $\theta\in\RR$, and all $N\geq 1$,
\begin{align}
    \EE_\omega\left[\min\{1,(N|h\sin(\omega+\theta)+c|)^{-\alpha}\}\right] \leq C_\alpha \min \lbrace 1, (Nh)^{-\alpha}\rbrace,
\end{align}
where $\omega\sim\mathrm{Unif}([-N,N])$.
\end{lemma}

\begin{proof}
    First, the constant upper bound is immediate for any $C_\alpha>1$. Setting
    \begin{equation}
        I:=\frac{1}{2N}\int_{-N}^{N}\min\{1,(N|h\sin(\omega+\theta)+c|)^{-\alpha}\}\,\dd\omega,
    \end{equation}
    we immediately have from $\min\{1,x^{-\alpha}\}\leq x^{-\alpha}$ that
    \begin{equation}
        I\leq N^{-\alpha}h^{-\alpha}\frac{1}{2N}\int_{-N}^{N}\bigl|\sin(\omega+\theta)+a\bigr|^{-\alpha}\,\dd\omega,\qquad a:=\frac{c}{h}.
    \end{equation}
    Using that the density of $(\omega +\theta) \!\mod 2\pi$ is uniformly bounded in $N$ and $\theta$, the integral further reduces to 
    \begin{equation}
        I\leq C_0N^{-\alpha}h^{-\alpha}\int_0^{2\pi}|\sin(u) + a|^{-\alpha}\,\dd u.
    \end{equation}
    By assumptions, $|a|\leq \frac12$ implies that the integral can be bounded uniformly in $a$.
    Indeed for each such  $a$, the function $u\mapsto \sin u+a$ has exactly two zeros $u_1(a),u_2(a)$ in $[0,2\pi)$, and these are uniformly non-degenerate since
\begin{equation}
|\cos u_j(a)|=\sqrt{1-a^2}\geq\frac{\sqrt3}{2},\qquad j=1,2.
\end{equation}
Hence, fixing $\delta:=\sqrt3/4$, the mean value theorem gives
\begin{equation}
|\sin u+a|\geq\frac{\sqrt3}{4}\,|u-u_j(a)|
\end{equation}
whenever $|u-u_j(a)|\leq \delta$, so the contribution of each $\delta$-neighbourhood is bounded by a constant depending only on $\alpha$, since $\alpha<1$.
On the complement of these two neighbourhoods, the continuous function $(a,u)\mapsto |\sin u+a|$ is strictly positive on a compact set, hence bounded below by some $m>0$ uniformly in $|a|\leq 1/2$. Therefore the remaining contribution is bounded by $2\pi m^{-\alpha}$. This proves
\begin{equation}
\sup_{|a|\leq 1/2}\int_0^{2\pi}|\sin u+a|^{-\alpha}\,\dd u<\infty,
\end{equation}
and concludes the proof.
\end{proof}

We now determine the threshold $A_0>1$ such that $g_\sfz$ satisfies a local Lyapunov drift condition for the limiting rescaled dynamics in the annular region $(K_{A_0}^\sfz)^c$.

\begin{proposition}\label{prop:driftlimrescaledZannular}
Fix $\alpha\in(0,1)$. Then there exist $A_0=A_0(\alpha)>1$ and $N_\infty=N_\infty(\alpha)\in\NN$ such that for every $N\geq N_\infty$ and every $(x,y)\in\RR^2$ with $h_\sfz(x,y)\geq A_0$, there holds
\begin{align}
    \EE[g_\sfz(\widetilde f_\omega(x,y))] \leq \frac12\,g_\sfz(x,y).
\end{align}
\end{proposition}

\begin{proof}
Fix $(x,y)\in\RR^2$ and let $h_0:=h_\sfz(x,y)\geq 1$. Since $h_0\geq A_0>1$, we may write $ (x,y)=h_0(\cos\phi,\sin\phi) $ for some $\phi\in\RR$, and $ \widetilde x_2=\Theta_2+\omega_2\Theta_1$ after setting
\begin{align}
    \Theta_1:= \widetilde y_1+\sqrt{2}\int_0^1 \overline B_s^2\,\dd s, \qquad \Theta_2 := \widetilde x_1+\sqrt{2}\,\overline B_1^1.
\end{align}
Since $g_\sfz(\widetilde x_2,\widetilde y_2)\leq \min\{1,|\widetilde x_2|^{-\alpha}\}$, the same one-dimensional averaging argument used in the proof of Lemma~\ref{lemma:EEmaxshearX} implies the existence of $C_{1,\alpha}>0$ such that
\begin{align}
    \EE_{\omega_2}[g_\sfz(\widetilde f_\omega(x,y))] \leq C_{1,\alpha}\,\min\{1,(N|\Theta_1|)^{-\alpha}\},
\end{align}
thus,
\begin{align}
    \EE[g_\sfz(\widetilde f_\omega(x,y))] \leq C_{1,\alpha}\,\EE\bigl[\min\{1,(N|\Theta_1|)^{-\alpha}\}\bigr].
\end{align}
Next, using the representation of $\widetilde y_1$, we may write
\begin{align}
    \Theta_1=h_0\sin(\omega_1+\phi)+G,
\end{align}
where $G$ is a centred Gaussian random variable, independent of $\omega_1$, with variance independent of $h_0$.
On $F:=\{|G|\leq h_0/2\}$, Lemma~\ref{lemma:EEmaxrotationZ} yields
\begin{align}
    \EE_{\omega_1,\bB}\Bigl[\min\{1,(N|\Theta_1|)^{-\alpha}\}\,\big|\,F\Bigr]  \leq  C_{2,\alpha}N^{-\alpha}h_0^{-\alpha}.
\end{align}
On the other hand, since $G$ is Gaussian, there exist universal constants $c,C>0$ such that $\PP(F^c)\leq C e^{-c h_0^2}$. In particular, we may choose $A_0>1$ large enough so that $C_{1,\alpha}\PP(F^c)\leq \frac14 h_0^{-\alpha}$ for all $h_0\geq A_0$. Summing up, for $h_0\geq A_0$,
\begin{align}
    \EE[g_\sfz(\widetilde f_\omega(x,y))] &\leq C_{1,\alpha}C_{2,\alpha}N^{-\alpha}h_0^{-\alpha} + C_{1,\alpha}\PP(F^c) \\
    &\leq C_{1,\alpha}C_{2,\alpha}N^{-\alpha}h_0^{-\alpha} + \frac14 h_0^{-\alpha}.
\end{align}
Choosing $N_\infty\in\NN$ so large that $C_{1,\alpha}C_{2,\alpha}N^{-\alpha}\leq \frac14$ for all $N\geq N_\infty$, the proposition is proved.
\end{proof}

\subsubsection{The inner core $K_{A_0}^\sfz$}

Once $A_0$ is determined by Proposition~\ref{prop:driftlimrescaledZannular} above, we
now turn to studying the dynamics in the core $K_{A_0}^\sfz$.

\begin{proposition}\label{prop:driftlimrescaledZinner}
Let $A_0>1$ and $\alpha\in(0,1)$. There exists $N_0=N_0(A_0)\in\NN$ and $\gamma\in(0,1)$ such that
\begin{align}
    \EE[g_\sfz(\widetilde f_\omega(x,y))] \leq \gamma g_\sfz(x,y)
\end{align}
for all $N\geq N_0$ and all $(x,y)\in K_{A_0}^\sfz$.
\end{proposition}

The proof is analogous to that of Proposition~\ref{prop:driftlimrescaledXinner}, since the last motion of $\widetilde{f}_\omega$ is shearing. 
In view of Propositions~\ref{prop:stabilitydriftrescaledz}, \ref{prop:driftlimrescaledZannular}, \ref{prop:driftlimrescaledZinner}, and Lemma~\ref{lemma:boundary_layer_reduction_z}, the following result is immediate.

\begin{corollary}\label{cor:exact_core_fixed_Fz}
Let $A_1>A_0$. Then there exist $\gamma\in(0,1)$ and $\kappa_0>0$ such that
\begin{align}
    \EE_{\omega,\bB}[V_{\sfz,\kappa}(f_\omega^\kappa(p))]  \leq  \gamma V_{\sfz,\kappa}(p)
\end{align}
for all $p\in\mathsf{Q}_{\sfz}(s_0)$ satisfying $r_\sfz(p)\leq A_1\sqrt{\kappa}$ and all $\kappa\in(0,\kappa_0)$.
\end{corollary}

\subsubsection{Inviscid-dominated regime}
We consider here the region where the inviscid dynamics dominate the stochastic processes by assuming that $r_\sfz(p)\geq A_1\sqrt{\kappa}$, for some $A_1\geq A_0$ to be determined. 
We show that a version of Lemma~\ref{lemma:localzlyap} holds for all $p\in \mathsf{Q}_{\sfz}(s_0)$.

\begin{remark}
    The invariance of $r_\sfz$ via the first rotation simplifies the $F_\sfz$ analysis relative to the $F_\sfx$ in a fundamental way.
    Near $F_\sfx$, the first deterministic step $f_{\omega_1}^3$ rotates the $(x,y)$-plane and thereby mixes the coordinates of $r_\sfx = \max\lbrace|y|,|z|\rbrace$, so the good event must control both $|\phi_1^\kappa - \phi_1|$ and $|\phi_2^\kappa - \phi_2|$.
    Near $F_\sfz$, the first step leaves $r_\sfz$ invariant deterministically for every realisation of $\omega_1$.
    As a result, the three-case decomposition can be defined entirely in terms of the deterministic midpoint $\phi_1$ (for which $r_\sfz(\phi_1) = r_\sfz(p)$ exactly), and the only perturbation that must be controlled explicitly is the endpoint deviation $|\phi_2^\kappa - \phi_2|$.
    The good event therefore reduces to a single condition.
\end{remark}

\begin{proposition}\label{prop:driftannularz}
    Fix $\alpha=\frac18$. Then there exist $N>1$, $A>1$, $s_0>0$, $\kappa_0>0$,  and $\gamma\in(0,1)$ such that
    \begin{align}
        \EE_{\omega,\bB}\left[V_{\sfz,\kappa}(\phi_2^\kappa)\right] \leq \gamma\, V_{\sfz,\kappa}(p)
    \end{align}
    for all $\kappa\in(0,\kappa_0)$ and all $p\in\mathsf{Q}_\sfz(s_0)$ with $r_\sfz(p)\geq A\sqrt\kappa$.
\end{proposition}

\begin{proof}
    Fix a small $\eta>0$ (to be determined) and define the good set
    \begin{equation}
        G:=\{|\phi_2^\kappa - \phi_2|\leq \eta r_\sfz(p)\}.
    \end{equation}
On the good event $G$ we follow the splitting presented in Lemma~\ref{lemma:localzlyap}, recalling that 
\begin{equation}
    E_1=\lbrace N^{-\frac12} |x_1| \leq y_1 \rbrace ,\quad E_2= \left\lbrace \frac{|x_1|}{2N} \leq y_1 \leq \frac{|x_1|}{N^\frac12} \right\rbrace  ,\quad E_3= \left\lbrace 0 < y_1 < \frac{|x_1|}{2N} \right\rbrace ,
\end{equation}
and further refining $E_1$ and $E_2$ into
\begin{equation}
    E_{1,1}=\left\lbrace|\omega_2| > 2\frac{M+1}{1-\delta_0}N^\frac12 \right\rbrace, \quad E_{1,2}=E_{1,1}^c,\quad E_{2,1} = \left\lbrace |x_2| \leq \frac{y_1}{2} \right\rbrace, \quad E_{2,2}=E_{2,1}^c,
\end{equation}
where $\delta_0\in(0,1)$ and $M>0$ will be determined later. Let $r_2=r_\sfz(p_2)$ and $r_1=r_\sfz(p_1)$, with $p_2=\phi_2$ and $p_1=\phi_1$. We recall that 
\begin{enumerate}[label=--]
    \item On $E_{1,1}$ there holds $r_2 \geq \sqrt{\frac{1+M^2N}{1+N}}r_1\geq \frac{M}{2}r_1$ for $N\geq 1$.
    \item On $E_{1,2}$ there holds $r_2 \geq \frac{r_1}{\sqrt{1+N}}$ and $\PP(E_{1,2}) = 2\frac{M+1}{1-\delta_0}N^{-\frac12}$.
    \item On $E_2$ there holds $r_2\geq \frac{r_1}{\sqrt{1+4N^2}}$ and $\PP(E_2) \leq C_2N^{-\frac12}$ for some universal $C_2>0$.
    \item On $E_3$ there holds $r_2 \geq \frac{r_1}{4}$ and $\PP(E_3) \leq C_3 N^{-\frac12}$ for some universal $C_3>0$.
\end{enumerate}
and we set $E:=E_{1,1}$ and $E^c=E_{1,2}\cup E_2\cup E_3$. Since $|r_\sfz(\phi_2^\kappa)-r_\sfz(\phi_2)|\leq 4|\phi_2^\kappa - \phi_2|\leq 4\eta r_\sfz(p)$ we now have
\begin{align}
    r_2^\kappa &\geq \left( \frac{M}{2} - 4\eta \right) r_\sfz(p) \mathbf 1_{E} + \left( \frac{1}{\sqrt{1+4N^2}} - 4\eta \right) r_\sfz(p)  \mathbf 1_{E^c} \\
    &\geq \frac{M}{4}r_\sfz(p) \mathbf 1_{E} + \frac{1}{2\sqrt{1+4N^2}}r_\sfz(p)  \mathbf 1_{E^c}
\end{align}
for $4\eta= \frac{1}{2\sqrt{1+4N^2}}$ and $M>4$. For $M=8$ we have $\PP(E^c)\lesssim N^{-\frac12}$ and 
\begin{align}
    \EE_{\bB,\omega}\left[V_{\sfz,\kappa}(\phi_2^\kappa)\mathbf 1_{G}\right]
    &\leq 2^{-\alpha}V_{\sfz,\kappa}(p) + (2\sqrt{1+4N^2})^{\alpha}\PP(E^c)V_{\sfz,\kappa}(p) \\
    &\leq \gamma V_{\sfz,\kappa}(p)
\end{align}
choosing $\alpha = \frac18$ and $N>1$ large enough, for some $\gamma\in (2^{-\alpha},1)$. For $G^c$, we argue as in Proposition \ref{prop:driftanuularx} to conclude that
\begin{align}\label{eq:bad_event_z_contribution}
    \EE_{\omega,\bB}\!\left[V_{\sfz,\kappa}(\phi_2^\kappa)\mathbf 1_{G^c}\right]  &\leq \frac{C_N}{\eta^2A^{2-\alpha}}V_{\sfz,\kappa}(p) \leq \frac{1-\gamma}{2}V_{\sfz,\kappa}(p)
\end{align}
for $A$ large enough. With this, the proof is completed. 
\end{proof}

We can now conclude the local and global Lyapunov estimates, as done for the fixed point $F_\sfx$ in Section~\ref{ssec:viscous_drift_Fx}.

\begin{proposition}\label{prop:driftkappazlocal}
There exist $N>1$, $s_0>0$, $\gamma\in(0,1)$, and $\kappa_0>0$ such that
\begin{align}
    \EE_{\omega,\bB}\left[V_{\sfz,\kappa}(f_\omega^\kappa(p))\right] \leq \gamma V_{\sfz,\kappa}(p)
\end{align}
for all $p\in \mathsf{Q}_{\sfz}(s_0)$ and all $\kappa\in(0,\kappa_0)$.
\end{proposition}

\begin{proof}[Proof of Proposition~\ref{prop:driftkappaz}]
Combining Proposition~\ref{prop:driftkappazlocal} and Lemma~\ref{lemma:driftboundawayz}, according to whether $p\in \mathsf{Q}_{\sfz}(s_0)$ or not, yields
\begin{align}
    \EE_{\omega,\bB}\left[V_{\sfz,\kappa}(f_\omega^\kappa(p))\right] \leq \gamma V_{\sfz,\kappa}(p)+C\mathbf 1_{\mathsf{Q}_{\sfz}(s_0)^c}(p)
\end{align}
for all $p\in \mathsf{X}$ and all $\kappa\in(0,\kappa_0)$.
\end{proof}

\subsection{Stochastic drift function for $F_\mathsf{y}$ fixed points}\label{ssec:viscous_drift_Fy}

We now turn to the fixed points $F_\mathsf{y}=\{(0,\pm1,0)\}$ and throughout this subsection we work near $(0,1,0)$, the point $(0,-1,0)$ being treated identically by symmetry.
As in the deterministic construction of Lemma~\ref{lemma:localylyap}, the natural radial quantity is not measured directly at $p$, but after the fixed post-rotation $f_N^3$.
We keep the notation $ f_\omega^\kappa, f_\omega, \phi_1^\kappa, \phi_1, \phi_2^\kappa, \phi_2$ introduced in Section~\ref{ssec:viscous_drift_Fx} and we further define $r_\mathsf{y}(p):=\sqrt{x^2+z^2}$ for $p=(x,y,z)\in\Sph$ and the functions
\begin{align}
    \widetilde V_{\mathsf{y},\kappa}(p)  := \max\{r_\mathsf{y}(p),\sqrt{\kappa}\}^{-\alpha}, \qquad  V_{\mathsf{y},\kappa}(p) := \widetilde V_{\mathsf{y},\kappa}(f_N^3(p)),
\end{align}
for some $N>1$ and some $\alpha\in(0,\frac14)$ to be determined later. The local chart around $F_\sfy$ reads
\begin{align}
    \Psi_\sfy(x,z):=(x,\sqrt{1-x^2-z^2},z), \qquad \mathsf{Q}_\mathsf{y}(s_0):=\Psi_\sfy(\{(x,z)\in\RR^2 \; :\;\ 0< x^2+z^2<s_0^2\}).
\end{align}

The goal of this subsection is the following Lyapunov--Foster drift estimate.

\begin{proposition}\label{prop:driftkappay}
There exist $N>1$, $s_0>0$, $\gamma\in(0,1)$, $\beta_\mathsf{y}\geq1$, and $\kappa_0>0$ such that
\begin{align}
    \EE_{\omega,\bB}\bigl[V_{\mathsf{y},\kappa}(f_\omega^\kappa(p))\bigr] \leq \gamma\,V_{\mathsf{y},\kappa}(p) + \beta_\mathsf{y}\mathbf 1_{\mathsf{Q}_\mathsf{y}(s_0)^c}(p)
\end{align}
for all $p\in \mathsf{X}$ and all $\kappa\in(0,\kappa_0)$.
\end{proposition}

As in Section~\ref{ssec:viscous_drift_Fx}, the proof is divided according to the size of $r_\mathsf{y}(f_N^3(p))$ relative to the stochastic scale $\sqrt{\kappa}$.
We first record the stability in expectation of the noisy two-step map after the fixed post-rotation.

\begin{lemma}\label{lemma:noisy_two_step_L2_y}
There exists $C_N=C(N)>0$ such that
\begin{align}
    \EE_\bB\bigl|f_N^3(\phi_1^\kappa)-f_N^3(\phi_1)\bigr|^2 + \EE_\bB\bigl|f_N^3(\phi_2^\kappa)-f_N^3(\phi_2)\bigr|^2 \leq C_N\kappa
\end{align}
for all $\kappa\in(0,1]$, all $p\in \Sph$, and all $\omega=(\omega_1,\omega_2)\in[-N,N]^2$.
\end{lemma}

\begin{proof}
Since $f_N^3:\Sph\to \Sph$ is a smooth diffeomorphism, its derivative is uniformly bounded for fixed $N$.
Hence, there exists $C_N>0$ such that
\begin{align}
    \bigl|f_N^3(a)-f_N^3(b)\bigr|
    \leq
    C_N |a-b|,
\end{align}
for all $a,b\in \Sph$. 
Applying this with $a=\phi_1^\kappa$, $b=\phi_1$, and then with $a=\phi_2^\kappa$, $b=\phi_2$, and using Lemma~\ref{lemma:noisy_two_step_L2}, proves the claim.
\end{proof}

We next obtain a uniform bound away from the $F_\mathsf y$ caps.

\begin{lemma}\label{lemma:driftboundawayy}
There exists $C>0$ such that
\begin{align}
    \EE_{\omega,\bB}\bigl[V_{\mathsf{y},\kappa}(f_\omega^\kappa(p))\bigr] \leq C
\end{align}
for all $p\in \mathsf{Q}_\mathsf{y}(s_0)^c$ and all $\kappa\in(0,1)$.
\end{lemma}

\begin{proof}
The proof follows as in Lemma \ref{lemma:driftboundawayx} by considering the map $(p,\omega)\longmapsto r_\mathsf{y}\bigl(f_N^3(f_\omega(p))\bigr)$.
\end{proof}

\subsubsection{Stochastic regime}

We now treat the regime in which $r_\mathsf{y}(f_N^3(p))\lesssim \sqrt{\kappa}$.
As in the $F_\sfx$ and $F_\sfz$ analysis, we pass to rescaled variables, but here the rescaling is performed after conjugation by $f_N^3$.

For $A>1$, define
\begin{align}
    K_A^\mathsf{y} := \{(x,z)\in\RR^2:\ \sqrt{x^2+z^2}\leq A\},
\end{align}
and, for some $\alpha\in(0,\frac14)$,
\begin{align}
    g_\mathsf{y}(x,z):=h_\mathsf{y}(x,z)^{-\alpha}, \qquad h_\mathsf{y}(x,z):=\max\{\sqrt{x^2+z^2},1\},
\end{align}
hence $A^{-\alpha}\leq g_\mathsf{y}(x,z)\leq 1$ for all $(x,z)\in K_A^\mathsf{y}$. If $\overline p=f_N^3(p)=\Psi_\mathsf{y}(\sqrt{\kappa}x,\sqrt{\kappa}z)$, we have
\begin{align}\label{eq:exact_scaling_identity_y}
    V_{\mathsf{y},\kappa}(p) = \kappa^{-\alpha/2}g_\mathsf{y}(x,z).
\end{align}
We introduce the corresponding rescaled two-step process
\begin{align}
    \widetilde f_\omega^\kappa(x,z) := \frac{1}{\sqrt{\kappa}} \left( \bigl(f_N^3\circ f_\omega^\kappa\circ f_{-N}^3(\Psi_\mathsf{y}(\sqrt{\kappa}x,\sqrt{\kappa}z))\bigr)_x,\, \bigl(f_N^3\circ f_\omega^\kappa\circ f_{-N}^3(\Psi_\mathsf{y}(\sqrt{\kappa}x,\sqrt{\kappa}z))\bigr)_z \right).
\end{align}
For this process, the following equivalence holds, whose proof is identical to the one of Lemma~\ref{lemma:boundary_layer_reduction}.

\begin{lemma}\label{lemma:boundary_layer_reduction_y}
Let $A>1$ and $\alpha\in(0,1)$. The following are equivalent:
\begin{enumerate}
    \item there exists $\gamma\in(0,1)$ such that
    \begin{align}\label{eq:physical_core_drift_y}
        \EE_{\omega,\bB}\bigl[V_{\mathsf{y},\kappa}(f_\omega^\kappa(p))\bigr] \leq \gamma\,V_{\mathsf{y},\kappa}(p)
    \end{align}
    for all sufficiently small $\kappa$ and all $p\in \mathsf{Q}_\mathsf{y}(s_0)$ such that $r_\mathsf{y}(f_N^3(p))\leq A\sqrt{\kappa}$;
    \item there exists $\gamma\in(0,1)$ such that
    \begin{align}\label{eq:scaled_core_drift_y}
        \EE_{\omega,\bB}\bigl[g_\mathsf{y}(\widetilde f_\omega^\kappa(x,z))\bigr] \leq \gamma\,g_\mathsf{y}(x,z)
    \end{align}
    for all sufficiently small $\kappa$ and all $(x,z)\in K_A^\mathsf{y}$.
\end{enumerate}
\end{lemma}

We next identify the limiting rescaled process, in the spirit of Sections~\ref{ssec:viscous_drift_Fx} and \ref{ssec:viscous_drift_Fz}. For convenience we write
\begin{align}
    G_1:=\sqrt2\,B_1^1, \qquad G_2:=\sqrt2\,B_1^3, \qquad G_3:=\sqrt2\int_0^1 B_s^3\,\dd s,
\end{align}
where $B^1,B^3$ are independent Brownian motions, so that $G_1\sim \mathcal N(0,2)$, $G_2\sim \mathcal N(0,2)$, and $G_3\sim \mathcal N(0,2/3)$.

\begin{lemma}\label{lemma:limiting_rescaled_dynamics_y}
Let $\omega_1,\omega_2\sim\mathrm{Unif}([-N,N])$ be independent. Let $(G_1,G_2,G_3)$ be as above and let $\xi=(\xi_1,\xi_2)\sim \mathcal N(0,2I_2)$, independent of $(\omega_1,\omega_2,G_1,G_2,G_3)$. Then, for every $(x,z)\in\RR^2$, the limiting rescaled map $\widetilde f_\omega(x,z)=(\widetilde x,\widetilde z)$ is given by
\begin{align}
    \widetilde x_1 &:= x+(N-\omega_1)z+G_1-\omega_1 G_3, \\
    \widetilde z_1 &:= z+G_2,
\end{align}
and
\begin{align}
    \widetilde x &= (\widetilde x_1-N\widetilde z_1)\cos(\omega_2) + (N\widetilde x_1+\widetilde z_1)\sin(\omega_2) + (\xi_1-N\xi_2), \\
    \widetilde z &= -\widetilde x_1\sin(\omega_2)+\widetilde z_1\cos(\omega_2)+\xi_2.
\end{align}
In particular, there exist random variables $\rho_1\geq0$, $\theta_1\in\RR$, and $\zeta\sim \mathcal N(0,2+2N^2)$, with $\zeta$ independent of $(\omega_2,\rho_1,\theta_1)$, such that
\begin{align}\label{eq:rotation_representation_y}
    \widetilde x = \sqrt{1+N^2}\,\rho_1\sin(\omega_2+\theta_1)+\zeta, \qquad \rho_1:=\sqrt{\widetilde x_1^2+\widetilde z_1^2}.
\end{align}
\end{lemma}

\begin{proof}
We argue analogously as in the previous section, using the explicit form of $u_\omega$ together with the conjugation $f_{-N}^3$, the exact rescaling and the resulting limit obtained for $\kappa\to 0$.
The representation \eqref{eq:rotation_representation_y} follows from the identity
\begin{align}
    \bigl(\widetilde x_1-N\widetilde z_1\bigr)^2 + \bigl(N\widetilde x_1+\widetilde z_1\bigr)^2 = (1+N^2)\bigl(\widetilde x_1^2+\widetilde z_1^2\bigr),
\end{align}
together with the independence of $\xi$.
\end{proof}
We now split the stochastic core into annular and inner zones.

\subsubsection{The annular core $K_{A_1}^\mathsf{y}\setminus K_{A_0}^\mathsf{y}$}
Unlike the corresponding annular-core analysis in the $F_\sfx$ and $F_\sfz$ regimes, after conjugation by $f_N^3$ the quantity $r_\sfy(f_N^3(\cdot))$ is no longer adapted to a single coordinate of the rescaled limiting dynamics.
In particular, the deterministic component of the map does not reduce to the same simple shear/rotation structure used previously, and so the proof cannot be obtained by a direct repetition of the arguments from the other sections.
To recover a quantitative drift estimate, we instead exploit more explicitly the Gaussian structure of the limiting rescaled process: the relevant radial variable can be written in terms of linear combinations of Gaussian random variables, and this allows us to control both the averaged rotation effect and the Gaussian tails that appear in the annular regime.
The following two preliminary lemmas isolate precisely these probabilistic features and will be the key ingredients in the proof of Proposition~\ref{prop:driftlimrescaledYannular}.
\begin{lemma}[Gaussian smoothing estimates]\label{lemma:gaussian_smoothing}
Fix $0<\sigma_-\leq \sigma_+<\infty$, and let $G\sim \mathcal N(0,\sigma^2)$ with
$\sigma\in[\sigma_-,\sigma_+]$. Then,
\begin{enumerate}
    \item[(i)] For every $\alpha\in(0,1)$, there exists $C_{\alpha,\sigma_-,\sigma_+}>0$, depending only on $\alpha,\sigma_-,\sigma_+$, such that for all $\lambda\geq 1$ and all $a\in\RR$,
    \begin{align}
        \EE\bigl[\min\{1,(\lambda|a+G|)^{-\alpha}\}\bigr] \leq C_{\alpha,\sigma_-,\sigma_+} \min\{1,(\lambda|a|)^{-\alpha}\}.
    \end{align}

    \item[(ii)] For every $c_0>0$, there exists $c_1=c_1(c_0,\sigma_+)>0$ such that for all $a\in\RR$,
    \begin{align}
        \EE\bigl[e^{-c_0(a+G)^2}\bigr] \leq e^{-c_1 a^2}.
    \end{align}
\end{enumerate}
\end{lemma}

\begin{proof}
We start by proving (i). Since $\min\{1,(\lambda|x|)^{-\alpha}\}\leq \lambda^{-\alpha}|x|^{-\alpha}$, it is enough to obtain
\begin{equation}\label{eq:negmoment_bound_short}
\EE[|a+G|^{-\alpha}] \leq C_{\alpha,\sigma_-,\sigma_+} \min\{1,|a|^{-\alpha}\},
\end{equation}
uniformly for $\sigma\in[\sigma_-,\sigma_+]$. Indeed, for $\lambda\geq1$, it holds
\begin{equation}
\EE[\min\{1,(\lambda|a+G|)^{-\alpha}\}]\leq\lambda^{-\alpha}\EE[|a+G|^{-\alpha}]\leq C_\alpha \lambda^{-\alpha}\min\{1,|a|^{-\alpha}\}\leq C_\alpha \min\{1,(\lambda|a|)^{-\alpha}\}.
\end{equation}
To prove \eqref{eq:negmoment_bound_short} set $t:=\frac{a}{\sigma}$ and write
\begin{equation}
    \EE[|a+G|^{-\alpha}]=\sigma^{-\alpha}F_\alpha(t), \qquad F_\alpha(t):=\int_\RR |u|^{-\alpha}\varphi_1(u-t)\,\dd u,
\end{equation}
where $\varphi_1(v):=(2\pi)^{-1/2}e^{-v^2/2}$ denotes the density of $\sigma^{-1}G\sim N(0,1)$.
It is easy to check that for all $t\in\RR$,
\begin{equation}
F_\alpha(t)\leq C_\alpha(1+|t|)^{-\alpha}.
\end{equation}
Indeed, when $|t|\leq 2$, splitting the integration whether $|u|\leq 4$ or not and using integrability of $|u|^{-\alpha}$ near zero and $\varphi_1$ near $\infty$ respectively, shows that $F_\alpha(t)\leq C_\alpha$ uniformly. On the other hand, when $|t|>2$, integrating where $|u|>|t|/2$ or not and again using the monotonicity and integrability properties of $|u|^{-\alpha}$ and $\varphi_1(t)$ respectively proves that $F_\alpha(t)\leq C_\alpha|t|^{-\alpha}$.

Consequently,
\begin{equation}
\EE[|a+G|^{-\alpha}] \leq C_\alpha \sigma^{-\alpha}(1+|a|/\sigma)^{-\alpha} = C_\alpha (\sigma+|a|)^{-\alpha} \leq C_{\alpha,\sigma_-,\sigma_+}\min \lbrace 1, |a|^{-\alpha} \rbrace
\end{equation}
and using $\sigma\in[\sigma_-,\sigma_+]$, we obtain \eqref{eq:negmoment_bound_short}. 

For point (ii), a direct Gaussian computation yields
\begin{equation}
\EE\bigl[e^{-c_0(a+G)^2}\bigr]=\frac{1}{\sqrt{1+2c_0\sigma^2}}\exp\!\left(-\frac{c_0}{1+2c_0\sigma^2}a^2\right)\leq e^{-c_1a^2},
\end{equation}
where we used that $\sigma\leq \sigma_+$.
\end{proof}

We are now in position to prove a drift estimate for the rescaled process.

\begin{proposition}\label{prop:driftlimrescaledYannular}
Fix $\alpha\in(0,1)$.  Then there exist $A_0=A_0(\alpha)>1$ and $N_\infty=N_\infty(\alpha)\in\NN$ such that, for every $N\geq N_\infty$ and every $(x,z)\in\RR^2$ with $h_\mathsf{y}(x,z)\geq A_0$, there holds
\begin{align}
    \EE_{\omega,\bB}\bigl[g_\mathsf{y}(\widetilde f_\omega(x,z))\bigr] \leq \frac12\,g_\mathsf{y}(x,z).
\end{align}
\end{proposition}
\begin{proof}
Fix $(x,z)\in\RR^2$, and set $h_0:=h_\mathsf y(x,z)=\sqrt{x^2+z^2}$.
Since $h_0\geq A_0>1$, we have $g_\mathsf y(x,z)=h_0^{-\alpha}$ and $g_\mathsf y(\widetilde x,\widetilde z)\leq \min\{1,|\widetilde x|^{-\alpha}\}$, so it suffices to estimate $\widetilde x$.
From Lemma~\ref{lemma:limiting_rescaled_dynamics_y} recall that
\begin{equation}
\widetilde x = \sqrt{1+N^2}\,\rho_1\sin(\omega_2+\theta_1)+\zeta, \qquad \rho_1:=\sqrt{\widetilde x_1^2+\widetilde z_1^2},
\end{equation}
where
\begin{equation}
\widetilde x_1=x+(N-\omega_1)z+G_1-\omega_1G_3, \qquad \widetilde z_1=z+G_2,
\end{equation}
and $\zeta\sim\mathcal N(0,2+2N^2)$ is independent of $(\omega_2,\rho_1,\theta_1)$, and define the good event
\begin{equation}
F_1:=\Bigl\{|\zeta|\leq \frac12\sqrt{1+N^2}\,\rho_1\Bigr\}.
\end{equation}
Conditioning on $\mathcal{H}_5:= \sigma(\omega_1,G_1,G_2,G_3,\zeta)$, on $F_1$ we write
\begin{equation}
\widetilde x = N\Bigl(h\sin(\omega_2+\theta_1)+c\Bigr), \qquad h:=\frac{\sqrt{1+N^2}}{N}\rho_1,\qquad c:=\frac{\zeta}{N},
\end{equation}
and we have $|c|\leq h/2$. Hence, by Lemma~\ref{lemma:EEmaxrotationZ}, $\rho_1\geq|\widetilde x_1|$ and $\sqrt{1+N^2}\geq N$,
\begin{align}
    \EE_{\omega_2}\Bigl[\min\{1,|\widetilde x|^{-\alpha}\}\,\Big|\,\mathcal{H}_5\Bigr]
    &\leq C_\alpha \min \lbrace 1, N^{-\alpha}h^{-\alpha} \rbrace\\
    &= C_\alpha \min \lbrace 1, (\sqrt{1+N^2}\,\rho_1)^{-\alpha} \rbrace\\
    &\leq C_\alpha \min \lbrace1, (N|\widetilde{x}_1|)^{-\alpha} \rbrace.
\end{align}
On the complement $F_1^c$, for $\mathcal{H}_4 = \sigma(\omega_1,G_1,G_2,G_3)$  the Gaussian tail of $\zeta/\sqrt{1+N^2}\sim \cN(0,2)$ leads to
\begin{align}
    \PP\bigl(F_1^c\,\big|\,\mathcal{H}_4\bigr)\leq C e^{-c\rho_1^2},
\end{align}
so that, using $g_\mathsf y\leq 1$ on $F_1^c$,
\begin{align}\label{eq:reduction_annular_y_corrected}
    \EE_{\omega_2,\zeta}\Bigl[g_\mathsf y(\widetilde f_\omega(x,z))\,\Big|\,\mathcal{H}_4 \Bigr]\leq C_\alpha \min\{1,(N|\widetilde x_1|)^{-\alpha}\}+C e^{-c\rho_1^2}.
\end{align}
We now treat separately the two terms appearing in \eqref{eq:reduction_annular_y_corrected} starting with the first one. Write $\mathcal{H}_2=\sigma(G_1,G_3)$ and
\begin{equation}
\widetilde x_1=\Lambda_1+\omega_1\Lambda_2,\qquad\Lambda_1:=x+Nz+G_1,\qquad\Lambda_2:=-(z+G_3),
\end{equation}
so that an application of Lemma~\ref{lemma:EEmaxshearX} gives
\begin{align}
    \EE_{\omega_1}\Bigl[\min\{1,(N|\widetilde x_1|)^{-\alpha}\}\,\Big|\, \mathcal{H}_2\Bigr] \leq C_\alpha\min\bigl\{1,(N|\Lambda_1|)^{-\alpha},(N^2|\Lambda_2|)^{-\alpha}\bigr\}.
\end{align}
Taking expectation in $G_1,G_3$ and using
\begin{equation} \min\{1,(N|\Lambda_1|)^{-\alpha},(N^2|\Lambda_2|)^{-\alpha}\} \leq \min\{1,(N|\Lambda_1|)^{-\alpha}\},
\end{equation}
as well as the analogous bound with $\Lambda_2$, we obtain
\begin{align}
    \EE\Bigl[\EE_{\omega_1}\bigl[\min\{1,(N|\widetilde x_1|)^{-\alpha}\}\,\big|\, \mathcal{H}_2 \bigr]\Bigr]\leq C_\alpha\min\Bigl\{1,\,\EE\bigl[\min\{1,(N|\Lambda_1|)^{-\alpha}\}\bigr],\,\EE\bigl[\min\{1,(N^2|\Lambda_2|)^{-\alpha}\}\bigr]\Bigr\}.
\end{align}
Since $\Lambda_1=x+Nz+G_1$ and $\Lambda_2=-(z+G_3)$, point (i) from Lemma~\ref{lemma:gaussian_smoothing} with exponent $\alpha$ yields
\begin{align}
    \EE\Bigl[\EE_{\omega_1}\bigl[\min\{1,(N|\widetilde x_1|)^{-\alpha}\}\,\big|\,\mathcal{H}_2 \bigr]\Bigr]&\leq C_\alpha \min\bigl\{1,(N|x+Nz|)^{-\alpha},(N^2|z|)^{-\alpha}\bigr\}\\
    &\leq C_\alpha\min\bigl\{(N|x+Nz|)^{-\alpha},(N^2|z|)^{-\alpha}\bigr\}\\
    &\leq C_\alpha 2^\alpha N^{-\alpha} \Bigl(|x+Nz|+N|z|\Bigr)^{-\alpha}
\end{align}
as both $G_1$ and $G_3$ have uniformly bounded variance. Since $N\geq1$,
\begin{equation}
h_0=\sqrt{x^2+z^2}\leq |x|+|z|\leq |x+Nz|+(N+1)|z| \leq 2(|x+Nz|+N|z|),
\end{equation}
therefore,
\begin{align}\label{eq:first_term_annular_y_corrected}
    \EE\Bigl[ \EE_{\omega_1}\bigl[\min\{1,(N|\widetilde x_1|)^{-\alpha}\}\,\big|\,\mathcal{H}_2\bigr] \Bigr] \leq C_\alpha N^{-\alpha}h_0^{-\alpha}.
\end{align}
We move to the second term in \eqref{eq:reduction_annular_y_corrected}, and we use the trivial bound 
\begin{equation}
    e^{-c\rho_1^2}\leq e^{-c(\Lambda_1+\omega_1\Lambda_2)^2}\leq C_{\alpha'} \min\{1,|\Lambda_1+\omega_1\Lambda_2|^{-{\alpha'}}\},\qquad {\alpha'}=\frac{1+\alpha}{2}<1
\end{equation}
to deduce from Lemma~\ref{lemma:EEmaxshearX}
\begin{equation}
    \EE_{\omega_1,G_2}\bigl[e^{-c\rho_1^2}\,\big|\,\mathcal{H}_2\bigr]\leq C_{\alpha'} \min\{1,|\Lambda_1|^{-{\alpha'}},(N|\Lambda_2|)^{-{\alpha'}}\},
\end{equation}
and, mimicking the argument above, we infer 
\begin{align}\label{eq:second_term_annular_y_corrected}
\EE_{\omega_1,G_1,G_2,G_3}\bigl[e^{-c\rho_1^2}\bigr]  \leq C_{\alpha'} h_0^{-{\alpha'}}.
\end{align}
Finally, combining \eqref{eq:reduction_annular_y_corrected}, \eqref{eq:first_term_annular_y_corrected}, and \eqref{eq:second_term_annular_y_corrected}, we conclude that
\begin{align}
    \EE_{\omega,\bB}\bigl[g_\mathsf y(\widetilde f_\omega(x,z))\bigr]
    \leq C_{\alpha}N^{-\alpha}h_0^{-\alpha}+C_{{\alpha'}} h_0^{-{\alpha'}}.
\end{align}
Since ${\alpha'}>\alpha$, there exists $A_0=A_0(\alpha)>1$ such that, for all $h_0\geq A_0$,
\begin{equation}
C_{{\alpha'}} h_0^{-{\alpha'}} \leq \frac14 h_0^{-\alpha},
\end{equation}
and, once $A_0$ is fixed, choose $N_\infty=N_\infty(\alpha)$ so large that, for all $N\geq N_\infty$,
\begin{equation}
C_{\alpha}N^{-\alpha}\leq \frac14.
\end{equation}
Then, for every $N\geq N_\infty$ and every $h_0\geq A_0$,
\begin{equation}
\EE_{\omega,\bB}\bigl[g_\mathsf y(\widetilde f_\omega(x,z))\bigr] \leq \frac12 h_0^{-\alpha}=\frac12 g_\mathsf y(x,z).
\end{equation}
This proves the proposition.
\end{proof}

\subsubsection{The inner core $K_{A_0}^\mathsf{y}$}

Once $A_0$ is fixed by Proposition~\ref{prop:driftlimrescaledYannular}, the dynamics inside $K_{A_0}^\mathsf y$ are handled by the direct Gaussian spreading mechanism.

\begin{proposition}\label{prop:driftlimrescaledYinner}
Let $A_0>1$ and $\alpha\in(0,1)$. Then there exist $N_0=N_0(A_0,\alpha)\in\NN$ and $\gamma\in(0,1)$ such that
\begin{align}
    \EE_{\omega,\bB}\bigl[g_\mathsf{y}(\widetilde f_\omega(x,z))\bigr] \leq \gamma\,g_\mathsf{y}(x,z)
\end{align}
for all $N\geq N_0$ and all $(x,z)\in K_{A_0}^\mathsf{y}$.
\end{proposition}

\begin{proof}
Fix $(x,z)\in K_{A_0}^\mathsf{y}$ and set $h_0:=h_\mathsf{y}(x,z)\in[1,A_0]$ so that $g_\mathsf{y}(x,z)=h_0^{-\alpha}$ and $A_0^{-\alpha}\leq g_\mathsf{y}(x,z)\leq 1$. For $(\widetilde x,\widetilde z)=\widetilde f_\omega(x,z)$  define
\begin{align}
    E:=\{|\widetilde x|>2h_0\}, \qquad E^c=\{|\widetilde x|\leq 2h_0\}.
\end{align}
On $E$, we have $ g_\mathsf{y}(\widetilde x,\widetilde z)\leq 2^{-\alpha}g_\mathsf{y}(x,z)$, while conditionally on
\begin{align}
    \mathcal H:= \sigma\bigl(\omega_1,\omega_2,G_1,G_2,G_3\bigr),
\end{align}
Lemma~\ref{lemma:limiting_rescaled_dynamics_y} implies that $\widetilde x\,|\,\mathcal H\sim\mathcal N(a,2+2N^2)$ for some $\mathcal H$-measurable $a\in\RR$.
Therefore
\begin{align}
    \PP(E^c\,|\,\mathcal H) \leq \frac{4h_0}{\sqrt{2\pi(2+2N^2)}} \leq \frac{4h_0}{N},
\end{align}
and we note that taking expectations removes the conditioning on $\mathcal{H}$.
Using the trivial bound $g_\mathsf{y}(\widetilde x,\widetilde z)\leq 1\leq A_0^\alpha g_\mathsf{y}(x,z)$, we conclude
\begin{align}
    \EE_{\omega,\bB}\bigl[g_\mathsf{y}(\widetilde f_\omega(x,z))\bigr]
    &=\EE\bigl[g_\mathsf{y}(\widetilde f_\omega(x,z)) | E \bigr]\PP(E) + \EE\bigl[g_\mathsf{y}(\widetilde f_\omega(x,z)) | E^c \bigr]\PP(E^c)\\
    &\leq \left( 2^{-\alpha} + 4\frac{A_0^{1+\alpha}}{N} \right)g_\mathsf{y}(x,z).
\end{align}
Choosing $N_0=N_0(A_0,\alpha)$ so large that the coefficient is strictly smaller than $1$ proves the result.
\end{proof}

We next compare the exact rescaled process $\widetilde f_\omega^\kappa$ with the limiting
process $\widetilde f_\omega$, obtaining the same result as Lemma~\ref{lemma:EXPdiffrescaledlimiting}. We omit the proof.
\begin{lemma}\label{lemma:EXPdiffrescaledlimitingy}
Let $A,N>1$. There exists $C_{A,N}>0$ and $\kappa_0>0$ such that
\begin{align}
    \EE\Bigl[ \bigl| \widetilde f_\omega^\kappa(x,z)-\widetilde f_\omega(x,z) \bigr| \Bigr] \leq C_{A,N}\kappa^{1/4}
\end{align}
for all $(x,z)\in K_A^\mathsf{y}$, all $\omega\in[-N,N]^2$, and all $\kappa\in(0,\kappa_0)$.
\end{lemma}

The following proposition is the $F_\sfy$ version of Proposition~\ref{prop:stabilitydriftrescaledx}.

\begin{proposition}\label{prop:stabilitydriftrescaledy}
Let $\alpha,\gamma\in(0,1)$, $A>1$, and $N>1$. Then there exists $\kappa_0=\kappa_0(\alpha,\gamma,A,N)>0$ such that
\begin{align}
    \EE\Bigl[ \bigl| g_\mathsf{y}(\widetilde f_\omega^\kappa(x,z)) - g_\mathsf{y}(\widetilde f_\omega(x,z)) \bigr| \Bigr] \leq \frac{1-\gamma}{2}\,g_\mathsf{y}(x,z)
\end{align}
for all $(x,z)\in K_A^\mathsf{y}$ and all $\kappa\in(0,\kappa_0)$.
\end{proposition}

In view of Propositions~\ref{prop:driftlimrescaledYannular}, \ref{prop:driftlimrescaledYinner}, and \ref{prop:stabilitydriftrescaledy}, the following result is immediate.

\begin{corollary}\label{cor:driftrescaledy}
Let $A_1>A_0$. Then there exist $\gamma\in(0,1)$ and $\kappa_0>0$ such that
\begin{align}
    \EE\bigl[g_\mathsf{y}(\widetilde f_\omega^\kappa(x,z))\bigr] \leq \gamma\,g_\mathsf{y}(x,z)
\end{align}
for all $(x,z)\in K_{A_1}^\mathsf{y}$ and all $\kappa\in(0,\kappa_0)$.
\end{corollary}

Combining Corollary~\ref{cor:driftrescaledy} with Lemma~\ref{lemma:boundary_layer_reduction_y} yields the drift in the physical stochastic core.

\begin{proposition}\label{prop:driftcorey}
Let $A_1>A_0$. Then there exist $\gamma\in(0,1)$ and $\kappa_0>0$ such that
\begin{align}
    \EE_{\omega,\bB}\bigl[V_{\mathsf{y},\kappa}(f_\omega^\kappa(p))\bigr] \leq \gamma\,V_{\mathsf{y},\kappa}(p)
\end{align}
for all $p\in \mathsf{Q}_\mathsf{y}(s_0)$ satisfying $r_\mathsf{y}(f_N^3(p))\leq A_1\sqrt{\kappa}$ and for all $\kappa\in(0,\kappa_0)$.
\end{proposition}

\subsubsection{Inviscid-dominated regime}
We first treat the regime in which the stochastic dynamics $f_\omega^\kappa$ can be seen as a small perturbation of the inviscid dynamics $f_\omega$. 

\begin{proposition}\label{prop:driftannulary}
Fix $\alpha = \frac18$. Then there exist $N>1$, $A>1$, $s_0>0$, $\kappa_0>0$ and $\gamma\in(0,1)$ such that
\begin{align}\label{eq:noisy_annulus_drift_y}
    \EE_{\omega,\bB}\left[V_{\mathsf{y},\kappa}(f_\omega^\kappa(p))\right] \leq \gamma\,V_{\mathsf{y},\kappa}(p)
\end{align}
for all $\kappa\in(0,\kappa_0)$ and all $p\in \Sph$ such that $p\in \mathsf{Q}_\mathsf{y}(s_0)$ and $r_\mathsf{y}(f_N^3(p))\geq A\sqrt{\kappa}$.
\end{proposition}

\begin{proof}

Let $p_0\in \Sph$, $\overline{p} = f_N^3(p_0)$ and $q_{\omega_1}=f_{\omega_1}^3(p_0)$. Hence, $p_0= f_{-N}^3(\overline{p})$, $q_{\omega_1}= f_{\omega_1-N}^3(\overline{p})$, $ {V}_\mathsf{y}(p_0) = \widetilde{V}_\mathsf{y}(\overline{p})$, and $ {V}_\mathsf{y}(f_\omega(p_0)) = \widetilde{V}_\mathsf{y}(f_N^3\circ f_{\omega_2}^2(q_{\omega_1}))$. Let $p_2^\kappa=f_\omega^\kappa(p_0)$ and $p_2=f_\omega(p_0)$ and define the set
\begin{align}
    G := \lbrace |r_{\sfy}(f_N^3(p_2^\kappa))-r_{\sfy}(f_N^3(p_2))| \leq \eta r_\sfy(\overline{p})\rbrace,
\end{align}
for some $\eta>0$ to be chosen later on. Arguing as in the proof of Lemma \ref{lemma:localzlyap}, $f_N^3(p_2) = f_N^3\circ f_{\omega_2}^2(q_{\omega_1})$, and  we see that
\begin{enumerate}[label=--]
    \item On $E_{1,1}^{\omega_2}$ we have $r_\sfy(f_N^3(p_2))\geq \sqrt{\frac{1+M^2N}{1+N}}r_\sfy(q_{\omega_1})$. Note that $E_{1,2}$ is empty since we are choosing $\omega_3=N$.
    \item On $E_{2}^{\omega_2}$ we have $r_\sfy(f_N^3(p_2)) \geq \frac{r_\sfy(q_{\omega_1})}{\sqrt{1+4N^2}}$ and $\PP_{\omega_2}(E_2^{\omega_2})\leq \frac{C_2}{\sqrt{N}}$.
    \item On $E_3^{\omega_2}$ we have $r_\sfy(f_N^3(p_2)) \geq \frac{r_\sfy(q_{\omega_1})}{4}$ and $\PP_{\omega_2}(E_3^{\omega_2})\leq \frac{C_3}{\sqrt{N}}$.
\end{enumerate}
On the other hand, $r_\sfy(q_{\omega_1})= r_\sfy (f_{\widetilde \omega_1}^3 \circ \mathsf{R}_\sfz(\overline{p}))$ for $\widetilde\omega_1=N-\omega_1\in[0,2N]$ and thus, again as in Lemma \ref{lemma:localzlyap}, we have
\begin{enumerate}[label=--]
    \item On $E_{1,1}^{\omega_1}$ there holds $r_\sfy(q_{\omega_1})\geq \sqrt{\frac{1+M^2N}{1+N}}r_\sfy(\mathsf{R}_\sfz(\overline{p}))$.
    \item On $E_{1,2}^{\omega_1}$ there holds $r_\sfy(q_{\omega_1})\geq \frac{r_\sfy(\mathsf{R}_\sfz(\overline{p}))}{\sqrt{1+N}}$ and $\PP_{\omega_1}(E_{1,2}^{\omega_1}) = 2\frac{M+1}{1-\delta_0}N^{-\frac12}$.
    \item On $E_{2,1}^{\omega_1}$ there holds $r_\sfy(q_{\omega_1})\geq \frac{r_\sfy(\mathsf{R}_\sfz(\overline{p}))}{\sqrt{1+4N^2}}$ and $\PP_{\omega_1}(E_{2,1}^{\omega_1})\leq N^{-1}$.
    \item On $E_{2,2}^{\omega_1}$ there holds $r_\sfy(q_{\omega_1})\geq \frac{r_\sfy(\mathsf{R}_\sfz(\overline{p}))}{4}$.
    \item On $E_3^{\omega_1}$ we have $r_\sfy(q_{\omega_1})\geq \frac{r_\sfy(\mathsf{R}_\sfz(\overline{p}))}{4}$ and $\PP(E_3^{\omega_1})\leq \frac{C_3}{N}$.
\end{enumerate}
Set $E_{\omega_2}=E_{1,1}^{\omega_2}$, $E_{\omega_2}^c=E_{2}^{\omega_2}\cup E_{3}^{\omega_2}$,  $E_{\omega_1}=E_{1,1}^{\omega_1}$, $E_{\omega_1}^c=E_{1,2}^{\omega_1}\cup E_{2,1}^{\omega_1} \cup E_{2,2}^{\omega_1} \cup E_{3}^{\omega_1}$, and note further that $r_\sfy\circ \mathsf{R}_\sfz = r_\sfy$. Then, 
\begin{align}
    r_\sfy(f_N^3(p_2)) &\geq \frac{1+M^2N}{1+N} \mathbf{1}_{E_{\omega_1}\cap E_{\omega_2}}r_\sfy(\overline{p}) + \sqrt{\frac{1+M^2N}{(1+N)(1+4N^2)}}(\mathbf{1}_{E_{\omega_1}^c\cap E_{\omega_2}} + \mathbf{1}_{E_{\omega_1}\cap E_{\omega_2}^c} )r_\sfy(\overline{p}) \\
    &\quad+ \frac{1}{1+4N^2}\mathbf{1}_{E_{\omega_1}^c\cap E_{\omega_2}^c}r_\sfy(\overline{p}).
\end{align}
Moreover, we have that $\PP_{\omega_1}(E_{\omega_1}^c) \lesssim_M N^{-\frac12}$ and $\PP_{\omega_2}(E_{\omega_2}^c) \lesssim N^{-\frac12}$. Additionally, $\frac{1+M^2N}{1+N}\geq \frac{M^2}{4}$ so that, for 
\begin{equation}
\eta= \frac12 \min \left\lbrace \frac{1}{\sqrt{1+4N^2}}\frac{M}{2}, \frac{1}{1+4N^2} \right\rbrace,
\end{equation}
we have
\begin{align}
    r_\sfy(f_N^3(p_2))- \eta r_\sfy(\overline{p}) &\geq \frac{M^2}{4} \mathbf{1}_{E_{\omega_1}\cap E_{\omega_2}}r_\sfy(\overline{p}) + \frac{M}{4}\frac{1}{\sqrt{(1+4N^2)}}(\mathbf{1}_{E_{\omega_1}^c\cap E_{\omega_2}} + \mathbf{1}_{E_{\omega_1}\cap E_{\omega_2}^c} )r_\sfy(\overline{p}) \\
    &\quad+ \frac{1}{2+8N^2}\mathbf{1}_{E_{\omega_1}^c\cap E_{\omega_2}^c}r_\sfy(\overline{p}).
\end{align}
Together with 
\begin{equation}
    r_{\sfy}(f_N^3(p_2^\kappa)) \geq r_\sfy(f_N^3(p_2)) - |r_{\sfy}(f_N^3(p_2^\kappa))-r_{\sfy}(f_N^3(p_2))|\geq r_\sfy(f_N^3(p_2))- \eta r_\sfy(\overline{p}),
\end{equation}
on $G$ we obtain 
\begin{align}
    \EE[V_{\sfy,\kappa}(p_2^\kappa) | G] \leq \left( 2\left(\frac{4}{M^2}\right)^{\alpha} + 2C_M\frac{4^\alpha\left(1+4N^2\right)^{\frac{\alpha}{2}}}{M^\alpha\sqrt{N}} + 2C_M\frac{(2+8N^2)^{\alpha}}{N} \right) V_{\sfy,\kappa}(p_0).
\end{align}
Setting $\alpha=\frac18$ and $M=4\left(\frac{8}{3}\right)^8$ we have $2\left(\frac{4}{M^2}\right)^{\frac18}=\frac34$ and thus we can find $N$ large enough so that
\begin{align}
   \EE [V_{\sfy,\kappa}(p_2^\kappa) | G] \leq \gamma V_{\sfy,\kappa}(p_0)
\end{align}
for some $\gamma\in (\frac34,1)$. Finally, since 
\begin{equation}
    \PP(G^c)\leq \frac{C_N}{\eta^2 r_\sfy(\overline{p})^2}\kappa
\end{equation}
and $V_{\sfy,\kappa}\leq \kappa^{-\alpha/2}$, we observe that
\begin{align}\label{eq:bad_event_y_contribution}
    \EE_{\omega,\bB} \left[V_{\sfy,\kappa}(p_2^\kappa) | G^c \right]\PP(G^c)  &\leq \frac{C_N}{\eta^2A^{2-\alpha}}V_{\sfy,\kappa}(p_0) \leq \frac{1-\gamma}{2}V_{\sfy,\kappa}(p_0)
\end{align}
for $A$ large enough, the proposition is proved.
\end{proof}

\subsubsection{Local and global Lyapunov estimates}

We can now combine the stochastic and inviscid regimes, exactly as in Proposition~\ref{prop:driftkappaxlocal} to obtain the following local drift estimate.

\begin{proposition}\label{prop:driftkappaylocal}
There exist $N>1$, $s_0>0$, $\gamma\in(0,1)$, and $\kappa_0>0$ such that
\begin{align}
    \EE_{\omega,\bB}\bigl[V_{\mathsf{y},\kappa}(f_\omega^\kappa(p))\bigr] \leq \gamma\,V_{\mathsf{y},\kappa}(p)
\end{align}
for all $p\in \mathsf{Q}_\mathsf{y}(s_0)$ and all $\kappa\in(0,\kappa_0)$.
\end{proposition}
The proof of the Lyapunov-Foster drift condition concludes the section's argument.
\begin{proof}[Proof of Proposition~\ref{prop:driftkappay}]
Combining Proposition~\ref{prop:driftkappaylocal} with Lemma~\ref{lemma:driftboundawayy}, according to whether $p\in \mathsf{Q}_\mathsf{y}(s_0)$ or not, yields
\begin{align}
\EE_{\omega,\bB}\bigl[V_{\mathsf{y},\kappa}(f_\omega^\kappa(p))\bigr] \leq \gamma\,V_{\mathsf{y},\kappa}(p) + \beta_\mathsf{y}\mathbf 1_{\mathsf{Q}_\mathsf{y}(s_0)^c}(p),
\end{align}
for all $p\in \mathsf{X}$ and all sufficiently small $\kappa$, for some $\beta_\mathsf{y}\geq1$.
\end{proof}

\section{Drift function for the stochastic two-point process}\label{sec:driftstochTPP}
With the local drift $W$ for the two-point process and $V_{\sfx,\kappa}$, $V_{\sfy,\kappa}$ and $V_{\sfz,\kappa}$ for the one-point process near the $F_\sfx$, $F_\sfy$ and $F_\sfz$ fixed points respectively, we are now in position to write a drift function $\mathrm{V}^{(2)}$ that satisfies the Lyapunov-drift condition with uniform-in-$\kappa$ constants $\gamma>0$ and $\beta\geq 1$. 

We first define a drift function $\mathrm{V}_\kappa$ for the stochastic one-point process $f_\omega^\kappa$ as follows.
Recall $\chi:[0,\infty)\rightarrow[0,1]$ is a smooth non-negative bump function such that $\chi(r) =1$ for $r\in [0,s_0]$ and $\chi(r)=0$ for $r\in [2 s_0,+\infty)$.
Recall further that for $p=(x,y,z)\in \Sph$, we have $r_\sfx(p) = \max \lbrace |y|, |z| \rbrace $, $r_\mathsf{y}(p) = \sqrt{x^2 + z^2}$ and $r_\sfz(p) =\sqrt{x^2 + y^2}$ and also that the sets $\mathsf{Q}_\sigma(s) = \lbrace   p\in \Sph : 0 < r_\sigma(p) < s \rbrace$, with $\sigma\in \lbrace   \sfx, \mathsf{y}, \sfz \rbrace  =\Sigma$, denote the spherical caps centred at the six fixed points.
For $s>0$ small enough, the $\mathsf{Q}_\sigma(s)$ are mutually disjoint. We then define
\begin{equation}\label{eq:defdrifOPPVkappa}
\begin{split}
\mathrm{V}_\kappa(p) &:= \chi(r_\sfx)(1-\chi(r_\mathsf{y}))(1-\chi(r_\sfz)) V_{\sfx,\kappa}(p) + \chi(r_\mathsf{y})(1-\chi(r_\sfz)) {V}_{\mathsf{y}, \kappa}(p) + \chi(r_\sfz)V_{\sfz,\kappa}(p)  \\
&\quad+ \hat c (1-\chi(r_\sfx)) (1-\chi(r_\mathsf{y}))  (1-\chi(r_\sfz)) 
\end{split}
\end{equation}
where we choose $\hat c >0$ so that $\mathrm{V}_\kappa \geq 1$.
We further recall that there exists $0<r_0<s_0$ such that, for all $\sigma\in \Sigma$, it holds that $f_\omega(p)\in \mathsf{Q}_\sigma(s_0/2)$ for all $p\in \mathsf{Q}_\sigma(r_0)$.
\begin{lemma}
    Let $K(r_0):= \mathsf{X}\cap_{\sigma\in \lbrace   \sfx, \mathsf{y}, \sfz \rbrace  } \mathsf{Q}_\sigma(r_0)^c$. There exists $\beta>0$, $\gamma'\in (0,1)$ and $\kappa_0>0$ such that 
    \begin{align}
        \EE \left[\mathrm{V}_\kappa(f_\omega^\kappa(p)) \right] \leq \gamma' \mathrm{V}_\kappa(p) + \beta \mathbf{1}_{K(r_0)}
    \end{align}
   for all $\kappa\in (0,\kappa_0)$ and all $p\in \mathsf{X}$.
\end{lemma}

\begin{proof}
    Let $p\in K(r_0)^c$. Then, there exists $\sigma\in \Sigma$ such that $p\in\mathsf{Q}_\sigma(r_0)$ and thus $f_\omega(p)\in \mathsf{Q}_\sigma(s_0/2)$. On 
    \begin{equation}
    E=\lbrace | r_\sigma(f_\omega^\kappa(p)) - r_\sigma(f_\omega(p))| \leq \frac{s_0}{4} \rbrace
    \end{equation}
    there holds $r_\sigma(f_\omega^\kappa(p))\mathbf{1}_E \leq \frac{3s_0}{4}$ and thus $f_\omega^\kappa(p)\in \mathsf{Q}_\sigma(s_0)$. As a result,
    \begin{align}
        \EE[ \mathrm{V}_{\kappa}(f_\omega^\kappa(p)) \mathbf{1}_E] = \EE [V_{\sigma,\kappa}(f_\omega^\kappa(p)) \mathbf{1}_E ] \leq \gamma V_{\sigma,\kappa}(p) = \gamma \mathrm{V}_\kappa(p)
    \end{align}
    in view of Propositions \ref{prop:driftkappax}, \ref{prop:driftkappaz} and \ref{prop:driftkappay}. On the other hand, we observe that
    \begin{align}
        \PP(E^c) \leq \frac{4}{s_0} \EE \left[| r_\sigma(f_\omega^\kappa(p)) - r_\sigma(f_\omega(p))|\right]\leq \frac{4C_{\sigma,N}}{s_0}\sqrt{\kappa}
    \end{align}
    thanks to Lemma \ref{lemma:noisy_two_step_L2} and the Lipschitz continuity of $r_\sigma$, for all $\sigma\in \Sigma$. Together with the trivial bound $V_{\sigma,\kappa}(q)\leq (\sqrt{\kappa})^{-\alpha}$ we obtain
    \begin{align}
        \EE \left[ \mathrm{V}_{\kappa}(f_\omega^\kappa(p)) | E^c \right] {\PP(E^c)}\leq (3(\sqrt{\kappa})^{-\alpha} + \hat c ) \frac{4C_{\sigma,N}}{s_0}\sqrt{\kappa} \leq 4C_{\sigma,N}(3(\sqrt{\kappa})^{-\alpha} + \hat c )\frac{r_0^\alpha}{s_0}\sqrt{\kappa}V_{\sigma,\kappa}(p) \leq \frac{1-\gamma}{2}V_{\sigma,\kappa}(p)
    \end{align}
    for all $\kappa\in(0,\kappa_0)$ for $\kappa_0$ sufficiently small. Here we have used that $r_0^{-\alpha} \leq V_{\sigma,\kappa}(p)$ for all $r_0>\sqrt{\kappa}$ and all $p\in\mathsf{Q}_\sigma(r_0)$. As a result,
    \begin{align}
        \EE \left[ \mathrm{V}_{\kappa}(f_\omega^\kappa(p)) \right] \leq \gamma' \mathrm{V}_\kappa(p)
    \end{align}
    for $\gamma'=\frac{1+\gamma}{2}$, for all $p\in \mathsf{Q}_\sigma(r_0)$, for all $\sigma\in \Sigma$. 

    Next, for $p\in K(r_0)$, we have that $p\in \mathsf{Q}_\sigma(r_0)^c$ for all $\sigma\in \Sigma$ so that 
    \begin{align}
       \EE \left[ \mathrm{V}_\kappa(f_\omega^\kappa(p)) \right] \leq \hat c + \sum_{\sigma\in \Sigma} \EE \left[ V_{\sigma,\kappa}(f_\omega^\kappa(p)) \right] \leq \beta
    \end{align}
    for some $\beta>0$ independent of $\kappa\in(0,1)$ due to Lemma \ref{lemma:driftboundawayx}, \ref{lemma:driftboundawayz} and \ref{lemma:driftboundawayy}. With this, the lemma is proved.
\end{proof}

Define next 
\begin{align}\label{eq:tildermW}
    \widetilde {\mathrm{W}}_\kappa(p,q) = \max \lbrace \mathrm{W}_m(p,q), \nu\mathrm{V}_\kappa(p), \nu \mathrm{V}_\kappa(q) \rbrace,
\end{align}
with
\begin{align}
    \mathrm{W}_m(p,q) &= \chi (\tilde r_\mathsf{D}(p,q))(1-\chi(\tilde r_\sfx(p,q))) (1-\chi(\tilde r_\mathsf{y}(p,q)))(1-\chi(\tilde r_\sfz(p,q)))W(p,q) \\
    &\quad + \chi (\tilde r_\sfx(p,q)) (1-\chi(\tilde r_\mathsf{y}(p,q)))(1-\chi(\tilde r_\sfz(p,q)))W_{m,\sfx}(p, q)   \\
    &\quad + \chi (\tilde r_\mathsf{y}(p,q))(1-\chi(\tilde r_\sfz(p,q))) W_{m,\sfy}(p, q) \\
    &\quad + \chi(\tilde r_\sfz(p,q)) W_{m,\sfz}(p, q),
\end{align}
where $W_{m,\sigma}(p, q) =  \max \lbrace{ W_\sigma(p,q), W(p,q) \rbrace}$ for $\sigma\in \Sigma$. We next show that $\widetilde {\mathrm{W}}_\kappa$ satisfies a Lyapunov-drift condition.

\begin{proposition}\label{prop:tildeWkappadrift}
    There exists $\gamma\in(0,1)$, $\beta>0$, $s_0>0$, $r_0 \leq s_0/2$, $\kappa_0>0$, $\nu>0$ and $\ep>0$ such that for the compact set 
    \begin{align}
        \widetilde K = \lbrace   (p,q)\in \mathsf X^2 : p, q \in K(\ep) \text{ and } \tilde r_\sigma(p,q) \geq r_0, \text{ for all }\sigma\in \Sigma_\sfD \rbrace  
    \end{align}
there holds
    \begin{align}
        \EE \left[\widetilde {\mathrm{W}}_\kappa(f_\omega^\kappa(p),f_\omega^\kappa(q))\right] \leq \gamma \widetilde {\mathrm{W}}_\kappa(p,q) + \beta \mathbf{1}_{\widetilde{K}}(p,q),
    \end{align}
    for all $(p,q)\in \mathsf{X}^{(2)}$ and all $\kappa\in (0,\kappa_0)$.
\end{proposition}

\begin{proof}
 Fix $0< s_0 < s_*$ sufficiently small, let $r_0<\frac{s_0}{2}$ small enough such that $f_\omega (p)\in \mathsf{Q}_\sigma(s_0/2)$ whenever $p\in \mathsf{Q}_\sigma(r_0)$ and $\tilde r_\sigma(f_\omega(p),f_\omega(q)) < s_0$ whenever $\tilde r_\sigma(p,q) < r_0$, for all $\sigma\in \Sigma_\sfD$. Moreover, let $\eta\in (0,r_0)$ and $\ep\in (0,\eta)$. We argue as in Proposition \ref{prop:globaldriftTPP}. First, for $(p,q)\not \in \widetilde K$ we consider three cases.

\diampar{Case 1} Assume first that $p,q\in K^c(\eta)$, so that $p\in \mathsf{Q}_{\sigma_p}(\eta)$ and $q\in \mathsf{Q}_{\sigma_q}(\eta)$.

\begin{indentblock}[1]\diampar{Case 1.1} Assume also that $\sigma_p \neq \sigma_q$. Hence, for $\eta$ small enough we have $\widetilde{r}_\sigma(p,q) \geq 1$ for all $\sigma\in \Sigma_\mathsf{D}$. Recall also that for $\eta$ smaller, if necessary and in terms of $N$, we also have $\widetilde{r}_\sigma(f_\omega(p),f_\omega(q)) \geq \frac12$ for all $\sigma\in \Sigma_\mathsf{D}$. Consider next
\begin{align}
    E = \bigcap_{\sigma\in \Sigma_\mathsf{D}} \left\lbrace |\widetilde {r}_{\sigma}(f_\omega^\kappa(p), f_\omega^\kappa(q)) - \widetilde r_{\sigma}( f_\omega(p), f_\omega(q))| \leq \frac{1}{4} \right\rbrace.
\end{align}
Hence, we have that $\widetilde {r}_{\sigma}(f_\omega^\kappa(p), f_\omega^\kappa(q)) \geq \frac14 \geq 2s_0$ for all $\sigma\in \Sigma_\mathsf{D}$ and thus $\mathrm{W}_m(f_\omega^\kappa(p), f_\omega^\kappa(q)) = \mathrm{W}_m(p,q) = 0$. As a result, we have
\begin{align}
    \EE \left[\widetilde{\mathrm{W}}_\kappa(f_\omega^\kappa(p), f_\omega^\kappa(q)) \mathbf{1}_E\right]\leq \gamma \widetilde{\mathrm{W}}_\kappa(p,q)
\end{align}
arguing as in the proof of Proposition \ref{prop:globaldriftTPP}. On the other hand, for $E^c$ we use Lemma \ref{lemma:EXPphikappainv} to reach 
\begin{align}
    \EE[\mathrm{W}_m(f_\omega^\kappa(p), f_\omega^\kappa(q))\mathbf{1}_{E^c}] \leq C\PP(E^c)^{1-\xi} \leq \kappa^{\frac{1-\xi}{2}} C_N^{-1} C\nu^{-1} \eta^{\alpha} \widetilde{\mathrm{W}}_\kappa(p,q)
\end{align}
since $\widetilde{r}_\sigma (p,q)\geq 1$ and $\mathrm{V}_\kappa(p)\geq C_N \eta^{-\alpha}$, for some $C_N>0$, and noting that
\begin{align}
    \PP(E^c) &\leq 4 \sum_{\sigma\in \Sigma_{\mathsf{D}}} \EE \left[ |\widetilde {r}_{\sigma}(f_\omega^\kappa(p), f_\omega^\kappa(q)) - \widetilde r_{\sigma}( f_\omega(p), f_\omega(q))|\right] \\
    &\leq 4 \sum_{\sigma\in \Sigma_{\mathsf{D}}} \EE \left[ \widetilde {r}_{\sigma}(f_\omega^\kappa(p), f_\omega(p)) + \widetilde r_{\sigma}( f_\omega^\kappa(q), f_\omega(q))\right] \\
    &\leq C\sqrt{\kappa}
\end{align}
due to Lemma \ref{lemma:noisy_two_step_L2}. The same bound holds for $\EE[\widetilde{\mathrm{W}}_\kappa(f_\omega^\kappa(p), f_\omega^\kappa(q))\mathbf{1}_{E^c}]$ and we conclude that
\begin{align}
    \EE[\mathrm{W}_m(f_\omega^\kappa(p), f_\omega^\kappa(q))]  \leq \EE[\mathrm{W}_m(f_\omega^\kappa(p), f_\omega^\kappa(q))\mathbf{1}_E] + \EE[\mathrm{W}_m(f_\omega^\kappa(p), f_\omega^\kappa(q)) \mathbf{1}_{E^c}] \leq \frac{1+\gamma}{2} \widetilde{\mathrm{W}}_\kappa(p,q)
\end{align}
for $\kappa_0\leq \frac{1-\gamma}{2}\frac{C_N\nu}{C \eta^{\alpha}}$ small enough.
\end{indentblock}
\begin{indentblock}[1]\diampar{Case 1.2} Suppose now that $\sigma_p = \sigma_q$. As before, for $\eta<r_0$ small enough we have either $\widetilde{r}_\mathsf{D}(p,q)\geq 1$ or $\widetilde{r}_\mathsf{D}\leq r_0$.
\end{indentblock}
\begin{indentblock}[2]\diampar{Case 1.2.1} If $\widetilde r_{\mathsf{D}}(p,q)\leq r_0$, then $\widetilde r_{\widetilde\sigma}(p,q)\geq 1$ for $\widetilde\sigma\neq\sigma, \mathsf{D}$ and for $\eta\leq r_0/2$ we have $\widetilde{r}_\sigma(p,q) \leq r_0 + 2\eta \leq s_0$ for $r_0\leq s_0/2$. Hence, we now consider 
\begin{align}
    E = \left\lbrace |\widetilde {r}_{\sigma}(f_\omega^\kappa(p), f_\omega^\kappa(q)) - \widetilde r_{\sigma}( f_\omega(p), f_\omega(q))| \leq \frac{\widetilde {r}_{\sigma} (f_\omega(p), f_\omega(q))}{2} \right\rbrace.
\end{align}
On $E$, there holds 
\begin{equation}
\widetilde{r}_{\sigma}(f_\omega^\kappa(p), f_\omega^\kappa(q))\leq\frac{3\widetilde{r}_{\sigma}(f_\omega(p), f_\omega(q))}{2}\leq r_0,
\end{equation}
for $\eta>0$ small enough and all $p,q\in \mathsf{Q}_\sigma(\eta)$. In particular, $\widetilde{r}_{\sigma'}(f_\omega^\kappa(p), f_\omega^\kappa(q))\geq \frac12$ for $\sigma'\neq \sigma, \mathsf{D}$ and thus
\begin{align}
    \mathrm{W}_m(f_\omega^\kappa(p), f_\omega^\kappa(q))= W_{m,\sigma}(f_\omega^\kappa(p), f_\omega^\kappa(q))
\end{align}
hence we obtain 
\begin{equation}
    \EE[ {\mathrm{W}}_m (f_\omega^\kappa(p), f_\omega^\kappa(q)) | E ]\leq \gamma W_m(p,q)
\end{equation}
due to Proposition \ref{prop:stochdriftlocalTPP} for some $\gamma\in(0,1)$, for all $\kappa\in (0,\kappa_0)$. On the other hand, we note from Markov inequality and Lemma \ref{lemma:EXPvartheta} that $\PP(E^c)\leq C \left(\sqrt{\kappa} + \widetilde{r}_\sigma(p,q) \right)$. Recall from Lemma \ref{lemma:EXPphikappainv} that 
\begin{align}
    \EE[W_{\sigma'}(f_\omega^\kappa(p), f_\omega^\kappa(q))] \leq \frac{C}{\widetilde{r}_{\sigma'}(p,q)^{\xi}} 
\end{align}
for all  $\sigma'\in \Sigma_\mathsf{D}$, for some $\xi\in (0,1)$. Moreover, since $\widetilde{r}_{\sigma'}(p,q)\geq 1$, we have that $\EE[W_{\sigma'}(f_\omega^\kappa(p), f_\omega^\kappa(q))] \leq C$ for $\sigma'\neq\sigma ,\mathsf{D}$. Then, 
\begin{align}
    \EE[\mathrm{W}_m(f_\omega^\kappa(p), f_\omega^\kappa(q))|E^c] &\leq \sum_{\sigma'\neq \sigma,\mathsf{D}}\EE[W_{\sigma'}(f_\omega^\kappa(p), f_\omega^\kappa(q))|E^c] + \sum_{\sigma'= \sigma,\mathsf{D}}\EE[W_{\sigma'}(f_\omega^\kappa(p), f_\omega^\kappa(q))|E^c] \\
    &\leq \left(C + 2C \mathrm{W}_m(p,q)\right) \PP(E^c)^{1-\xi} \\
    &\leq \widetilde{C} \mathrm{W}_m(p,q)\PP(E^c)^{1-\xi}
\end{align}
for some universal $\widetilde{C}>0$. Since $\widetilde{r}_\mathsf{D}(p,q)\leq 2r_0$ we conclude that 
\begin{align}
    \EE[\mathrm{W}_m(f_\omega^\kappa(p), f_\omega^\kappa(q))]\leq \EE[\mathrm{W}_m(f_\omega^\kappa(p), f_\omega^\kappa(q))\mathbf{1}_E] + \EE[\mathrm{W}_m(f_\omega^\kappa(p), f_\omega^\kappa(q))\mathbf{1}_{E^c}] \leq \frac{1+\gamma}{2}\mathrm{W}_m(p,q)
\end{align}
after choosing $r_0$ small enough and then $\kappa_0$ small enough. As a result, for $\gamma'=\frac{1+\gamma}{2}$, we obtain 
\begin{align}
    \EE[\widetilde{\mathrm{W}}_m(f_\omega^\kappa(p), f_\omega^\kappa(q))] \leq \gamma' \widetilde{\mathrm{W}}_m(p,q).
\end{align}
\end{indentblock}
\begin{indentblock}[2]\diampar{Case 1.2.2} We now have $\widetilde{r}_\mathsf{D}(p,q)\geq 1$ and $p,q\in \mathsf{Q}_{\sigma}(\eta)$. Hence, due to the symmetries of the sphere and of the rotations $\mathsf{R}_\sigma$, there exists $\sigma'\in \Sigma\setminus\lbrace \sigma \rbrace$ such that $\widetilde{r}_{\sigma'}(p,q)\leq r_0/2$ and thus $\mathrm{W}_m(p,q) = W_{m,\sigma'}(p,q)$. In particular, for $\eta>0$ small enough, we have $\widetilde{r}_{\sigma'}(f_{\omega}(p),f_\omega(q))\leq r_0$ and 
\begin{equation}
\mathrm{W}_m(f_\omega(p),f_\omega(q)) = W_{m,\sigma'}(f_\omega(p),f_\omega(q))
\end{equation}
for all $\omega\in [-N,N]^2$. Consider now
\begin{align}
    E = \left\lbrace |\widetilde {r}_{\sigma'}(f_\omega^\kappa(p), f_\omega^\kappa(q)) - \widetilde r_{\sigma'}( f_\omega(p), f_\omega(q))| \leq \frac{\widetilde {r}_{\sigma'} (f_\omega(p), f_\omega(q))}{2} \right\rbrace
\end{align}
and note that on $E$ there holds $\widetilde {r}_{\sigma'}(f_\omega^\kappa(p), f_\omega^\kappa(q)) \leq 2r_0\leq s_0$. As a result, 
\begin{equation}
\mathrm{W}_m(f_\omega^\kappa(p), f_\omega^\kappa(q)) =W_{m,\sigma'}(f_\omega^\kappa(p), f_\omega^\kappa(q))
\end{equation}
so that  
\begin{equation}
\EE [\mathrm{W}_m(f_\omega^\kappa(p), f_\omega^\kappa(q))\mathbf{1}_E] \leq \gamma \mathrm{W}_m(p,q).
\end{equation}
On the other hand, for $E^c$ we argue as in Case 1.2.1 to obtain that 
\begin{equation}
\EE \left[\mathrm{W}_m(f_\omega^\kappa(p), f_\omega^\kappa(q))\mathbf{1}_{E^c} \right]\leq \widetilde{C}W_m(p,q)\PP(E^c)^{1-\xi}
\end{equation}
and  
\begin{equation}
\PP(E^c) \leq C\left(\sqrt{\kappa} + \widetilde{r}_{\sigma'}(p,q)\right)\leq C(\sqrt{\kappa} + r_0 )
\end{equation}
so that for $r_0>0$ and $\kappa_0>0$ small enough, we conclude first that 
\begin{equation}
\EE \left[{\mathrm{W}}_m(f_\omega^\kappa(p), f_\omega^\kappa(q)) \right] \leq \frac{1+\gamma}{2}{\mathrm{W}}_m(p,q)
\end{equation}
and then 
\begin{equation}
\EE \left[\widetilde{\mathrm{W}}_\kappa(f_\omega^\kappa(p), f_\omega^\kappa(q)) \right] \leq \frac{1+\gamma}{2}\widetilde{\mathrm{W}}_\kappa(p,q).
\end{equation}
\end{indentblock}
\diampar{Case 2} Consider next $p\in K^c(\eta)$ and $q\in K(\eta)$, so that $p\in \mathsf{Q}_{\sigma_p}(\eta)$ and $q\in \mathsf{Q}_\sigma(\eta)^c$, for all $\sigma \in \Sigma$. We distinguish two more cases.

\begin{indentblock}[1]
\diampar{Case 2.1} Assume there exists $\tilde\sigma \in \Sigma_\sfD$ such that $\tilde r_{\tilde\sigma}(p,q) < r_0$. Then, for all $\omega\in[-N,N]^2$ it holds $\tilde r_{\tilde\sigma}(f_\omega(p),f_\omega(q)) \leq s_0/2$, and there exists $\sigma'\in \Sigma_\sfD$ such that $\mathrm{W}_m(f_\omega(p),f_\omega(q)) = W_{\sigma',m}(f_\omega(p), f_\omega(q))$ for all $\omega\in[-N,N]^2$. Consider thus 
\begin{align}
    E_1 &:= \bigcap_{\sigma\in \Sigma} \lbrace |f_\omega^\kappa(q) - f_\omega(q)| \leq \delta_0 \rbrace, \\
    E_2 &:= \left\lbrace |\widetilde {r}_{\sigma'}(f_\omega^\kappa(p), f_\omega^\kappa(q)) - \widetilde r_{\sigma'}( f_\omega(p), f_\omega(q))| \leq \frac{\widetilde {r}_{\sigma'} (f_\omega(p), f_\omega(q))}{2} \right\rbrace
\end{align}
and $E:= E_1\cap E_2$. Here $\delta_0$ is such that $|r_\sigma(p_1) - r_\sigma(p_2)|\leq \frac{r_\sigma(f_\omega(q))}{2}$ for all $p_1,p_2\in \Sph$ such that $|p_1-p_2|\leq \delta_0$ due to the uniform continuity of $r_\sigma$, for all $\sigma\in \Sigma$. Note that 
\begin{equation}
\min_{q\in K(\eta),\, \sigma\in \Sigma,\, \omega\in [-N,N]^2} r_\sigma(f_\omega(q)) >0.
\end{equation}
Then, on $E$, for $p_1=f_\omega^\kappa(q)$ and $p_2=f_\omega(q)$ we have 
\begin{equation}
r_\sigma(f_\omega^\kappa(q)) \geq \frac{r_{\sigma}(f_\omega(q))}{2}\geq C_N^{-1}\eta
\end{equation}
due to Lemma \ref{lemma:EXPphikappaphi}. Hence,
    \begin{align}
         \max _{\sigma\in \Sigma} \lbrace   \mathrm{V}_\kappa(f_\omega(q)) : q\in \mathsf{Q}^c_\sigma(\eta), \omega\in[-N,N]^2 \rbrace  \leq C_N \eta^{-\alpha}
    \end{align}
for $\eta$ small enough. Therefore, fixing $\alpha=\frac18$,
\begin{equation}
\EE [\nu\mathrm{V}_\kappa(f_\omega^\kappa(q)) \mathbf{1}_E]\leq \nu C_N \eta^{-\frac18} \leq \gamma \mathrm{W}_m(p,q)
\end{equation} 
once we set $\nu=\frac14 \frac{\gamma\eta^\frac18  \min \psi}{r_0^\xi C_N}$. 

Next, to estimate the drift associated to the two-point chain, on $E$ we also have $\widetilde{r}_{\sigma'}(f_\omega^\kappa(p), f_\omega^\kappa(q)) \leq 3s_0/4 \leq s_0$, so that 
\begin{equation}
\EE[\mathrm{W}_m(f_\omega^\kappa(p), f_\omega^\kappa(q))\mathbf{1}_E] = \EE[W_{\sigma',m}(f_\omega^\kappa(p), f_\omega^\kappa(q))\mathbf{1}_E]
\end{equation}
and
\begin{align}
    \EE[W_{\sigma'}(f_\omega^\kappa(p), f_\omega^\kappa(q))\mathbf{1}_E] \leq \gamma W_{\sigma'}(p,q).
\end{align}
Moreover, note that 
\begin{equation}
\EE [\widetilde{r}_\mathsf{D}(f_\omega^\kappa(p), f_\omega^\kappa(q))^{-\xi}] \leq C_{\xi}\widetilde{r}_\mathsf{D}(p,q)^{-\xi}
\end{equation}
due to Lemma \ref{lemma:EXPphikappainv}. If $\widetilde{r}_\mathsf{D}(p,q)\geq 4 r_0\left(\frac{C_\xi\max\psi}{\min\psi}\right)^\frac{1}{\xi}$, we have that
\begin{equation}
\mathrm{W}_m(p,q)\leq \widetilde{r}_\mathsf{D}(p,q)^{-\xi} \max \psi \leq \widetilde{r}_{\sigma'}(p,q)^{-\xi}\min\psi \leq \mathrm{W}_{\sigma'}(p,q),
\end{equation}
so that $\mathrm{W}_m(p,q)= W_{\sigma'}(p,q)$. Hence, 
\begin{align}
    \EE[W(f_\omega^\kappa(p),f_\omega^\kappa(q))\mathbf{1}_E] \leq C_\xi \widetilde{r}_\mathsf{D}(p,q)^{-\xi} \max\psi \leq 4^{-\xi} W_{\sigma'}(p,q) \leq \gamma' \mathrm{W}_m(p,q),
\end{align}
for $\gamma'= \max\lbrace \gamma, 4^{-\xi} \rbrace \in (0,1)$. On the other hand, if $\widetilde{r}_\mathsf{D}(p,q)\leq 4 r_0\left(\frac{C_\xi\max\psi}{\min\psi}\right)^\frac{1}{\xi}$, we choose $r_0$ small enough so that $\widetilde{r}_\mathsf{D}(p,q)\leq s_*$ from Proposition \ref{prop:stochdriftlocalTPP}. Hence,
\begin{align}
    \EE[W(f_\omega^\kappa(p),f_\omega^\kappa(q))\mathbf{1}_E] \leq \gamma W(p,q)
\end{align}
and we conclude that 
\begin{equation}
\EE[\widetilde{\mathrm{W}}_\kappa(f_\omega^\kappa(p), f_\omega^\kappa(q))\mathbf{1}_E] \leq \gamma' \widetilde{\mathrm{W}}_\kappa(p,q).
\end{equation}

For the complementary event we write
\begin{align}
    E^c = E_1^c \cup E_2^c = (E_1^c\cap E_2) \cup (E_1^c \cap E_2^c) \cup  (E_1\cap E_2^c) 
\end{align}
Note that from Markov inequality and Lemmas \ref{lemma:noisy_two_step_L2} and \ref{lemma:EXPvartheta} that
\begin{align}
    \PP(E_1^c) \lesssim \delta_0^{-1}\sqrt{\kappa}, \quad \PP(E_2^c) \lesssim (\sqrt{\kappa} + r_0).
\end{align}
We argue for each subset.
\begin{enumerate}[label=--]
    \item For $E_1^c\cap E_2$ we have $\PP(E_1^c\cap E_2) \leq \PP(E_1^c) \lesssim \sqrt{\kappa}$. We now have
    \begin{align}
        \EE [\mathrm{V}_\kappa(f_\omega^\kappa(q))|E_1^c\cap E_2] \leq (\sqrt{\kappa})^{-\alpha} \leq \left( \frac{\eta}{\sqrt{\kappa}}\right)^\alpha \mathrm{V}_\kappa(q)
    \end{align}
    due to the pointwise bound $\mathrm{V}_\kappa(f_\omega^\kappa(q)) \leq (\sqrt{\kappa})^{-\alpha}$. Therefore,
    \begin{equation}
    \EE[\widetilde{\mathrm{W}}_\kappa(f_\omega^\kappa(p), f_\omega^\kappa(q))\mathbf{1}_{E_1^c\cap E_2}] \leq  \sqrt{\kappa} \left( \frac{\eta}{\sqrt{\kappa}}\right)^\alpha\widetilde{\mathrm{W}}_\kappa(p,q).
    \end{equation}
    \item For $E_1^c\cap E_2^c$ we have $\PP(E_1^c\cap E_2^c) \lesssim \sqrt{\kappa}$ again. We now observe that 
    \begin{equation}
    \EE[\mathrm{W}_m(f_\omega^\kappa(p),f_\omega^\kappa(q))\mathbf{1}_{E_1^c\cap E_2^c}] \leq C \mathrm{W}_m(p,q)\PP(E_1^c\cap E_2^c)^{1-\xi}
    \end{equation}
    due to Lemma \ref{lemma:EXPphikappainv} and we still have $\mathrm{V}_\kappa(f_\omega^\kappa(q))  \leq \left( \frac{\eta}{\sqrt{\kappa}}\right)^\alpha \mathrm{V}_\kappa(q)$.
    Hence, 
    \begin{equation}
    \EE[\widetilde{\mathrm{W}}_\kappa(f_\omega^\kappa(p), f_\omega^\kappa(q))\mathbf{1}_{E_1^c\cap E_2}] \leq \left( \left( \frac{\eta}{\sqrt{\kappa}}\right)^\alpha\PP(E_1^c\cap E_2^c) + \PP(E_1^c\cap E_2^c)^{1-\xi} \right) \widetilde{\mathrm{W}}_\kappa(p,q)
    \end{equation}
    as well for $\kappa_0$ small enough.
    \item For $E_1\cap E_2^c$ we have $\PP(E_1\cap E_2^c) \lesssim \sqrt{\kappa} + r_0$.
    We still have 
    \begin{equation}
    \EE [\nu\mathrm{V}(f_\omega^\kappa(q))]\leq \nu C_N \eta^{-\frac18} \leq \gamma \widetilde{\mathrm{W}}_\kappa(p,q) 
    \end{equation}
    and also 
    \begin{equation}
    \EE[\mathrm{W}_m(f_\omega^\kappa(p),f_\omega^\kappa(q))\mathbf{1}_{E_1\cap E_2^c}] \leq C \mathrm{W}_m(p,q)\PP(E_1\cap E_2^c)^{1-\xi}
    \end{equation}
    due to Lemma \ref{lemma:EXPphikappainv}. As a result,
    \begin{align}
        \EE[\widetilde{\mathrm{W}}_\kappa(f_\omega^\kappa(p), f_\omega^\kappa(q))\mathbf{1}_{E_1\cap E_2^c}] \leq (\gamma + C\PP(E_1\cap E_2^c)^{1-\xi}) \widetilde{\mathrm{W}}_\kappa(p,q).
    \end{align}
\end{enumerate}
With this, we conclude that 
\begin{align}
    \EE[\widetilde{\mathrm{W}}_\kappa(f_\omega^\kappa(p), f_\omega^\kappa(q))] &= \EE[\widetilde{\mathrm{W}}_\kappa(f_\omega^\kappa(p), f_\omega^\kappa(q))\mathbf{1}_E] + \EE[\widetilde{\mathrm{W}}_\kappa(f_\omega^\kappa(p), f_\omega^\kappa(q))\mathbf{1}_{E_1^c\cap E_2}] \\
    &\quad + \EE[\widetilde{\mathrm{W}}_\kappa(f_\omega^\kappa(p), f_\omega^\kappa(q))\mathbf{1}_{E_1^c\cap E_2^c}]  + \EE[\widetilde{\mathrm{W}}_\kappa(f_\omega^\kappa(p), f_\omega^\kappa(q))\mathbf{1}_{E_1\cap E_2^c}] \\
    &\leq \frac{1+\gamma'}{2}\widetilde{\mathrm{W}}_\kappa(p,q)
\end{align}
after choosing $r_0$ and $\kappa_0$ small enough, once we note that $\kappa^{\frac12(1-\alpha)}$ can be made arbitrarily small because $\alpha=\frac18 <1$.
\end{indentblock}

\begin{indentblock}[1]
\diampar{Case 2.2 } On the other hand, if $\tilde r_{\tilde\sigma}(p,q) \geq r_0$ for all $\tilde\sigma\in \Sigma_\sfD$ and $(p,q)\not\in \widetilde{K}$ we must have $p\in \mathsf{Q}_{\sigma_p}(\ep)$, for some $\sigma_p\in \Sigma$ as $q\in K(\eta) \subset K(\ep)$. We begin by noting that from Lemma \ref{lemma:EXPphikappainv}, we have $\EE [\mathrm{W}_m(f_\omega^\kappa(p), f_\omega^\kappa(q))] \leq C r_0^{-\xi}$. Consider again
\begin{align}
    E &:= \bigcap_{\sigma\in \Sigma} \lbrace |f_\omega^\kappa(q) - f_\omega(q)| \leq \delta_0 \rbrace,
\end{align}
so that on $E$ there holds $r_\sigma(f_\omega^\kappa(q)) \geq \frac{r_{\sigma}(f_\omega(q))}{2}\geq C_N^{-1}\eta$ due to Lemma \ref{lemma:EXPphikappaphi}. 
Hence, 
\begin{equation}
\EE [\nu\mathrm{V}(f_\omega(q))\mathbf{1}_E]\leq \nu C_N \eta^{-\frac18} \leq C r_0^{-\xi}
\end{equation}
due to the choice of $\nu=\frac{\gamma\eta^\alpha\min\psi}{4C_Nr_0^\xi}$. Since $\ep^{-\alpha}\leq \mathrm{V}_\kappa(p)$ for $\kappa_0<\ep^2$ small enough, we have that
\begin{align}
    Cr_0^{-\xi} \leq \nu\gamma \mathrm{V}_\kappa(p)
\end{align}
provided that $Cr_0^{-\xi} \leq \nu \gamma \ep^{-\alpha}$, that is, $\ep \leq \left( \frac{\nu\gamma r_0^\xi}{C} \right)^{\frac{1}{\alpha}}$. Hence, for $\ep>0$ small enough and $\kappa_0 < \ep^2$ smaller, we find that
\begin{align}
    \EE[\widetilde{\mathrm{W}}_\kappa(f_\omega^\kappa(p), f_\omega^\kappa(q))\mathbf{1}_E] \leq \nu\gamma \mathrm{V}_\kappa(p) \leq \gamma \widetilde{\mathrm{W}}_\kappa(p,q).
\end{align}
On the contrary, on $E^c$ we have $\PP(E^c) \lesssim \sqrt{\kappa}$ and 
\begin{equation}
\EE [\mathrm{V}_\kappa(f_\omega^\kappa(q))|E^c]\PP(E^c) \leq \sqrt{\kappa}(\sqrt{\kappa})^{-\alpha} \leq \sqrt{\kappa}\left( \frac{\ep}{\sqrt{\kappa}}\right)^\alpha \mathrm{V}_\kappa(q),
\end{equation}
while we still have 
\begin{equation}
\EE [\mathrm{W}_m(f_\omega^\kappa(p), f_\omega^\kappa(q))\mathbf{1}_{E^c}] \leq C r_0^{-\xi} \PP(E^c)^{1-\xi}\leq \nu\gamma\mathrm{V}_\kappa(p)
\end{equation}
for the same choice of $\ep$. As a result, 
\begin{equation}
\EE[\widetilde{\mathrm{W}}_\kappa(f_\omega^\kappa(p), f_\omega^\kappa(q))\mathbf{1}_{E^c}] \leq \left({\kappa}^{\frac{1-\alpha}{2}} + \gamma \kappa^{\frac{1-\xi}{2}}\right)\widetilde{\mathrm{W}}_\kappa(p,q).
\end{equation}
With this,
\begin{align}
    \EE[\widetilde{\mathrm{W}}_\kappa(f_\omega^\kappa(p), f_\omega^\kappa(q))] &\leq \EE[\widetilde{\mathrm{W}}_\kappa(f_\omega^\kappa(p), f_\omega^\kappa(q))\mathbf{1}_E] + \EE[\widetilde{\mathrm{W}}_\kappa(f_\omega^\kappa(p), f_\omega^\kappa(q))\mathbf{1}_{E^c}] \\
    &\leq \frac{1+\gamma}{2}\widetilde{\mathrm{W}}_\kappa(p,q)
\end{align}
for $\kappa_0$ small enough since $\alpha=\frac18<1$.
\end{indentblock}
\diampar{Case 3} Assume finally that $p,q\in K(\eta) \subset K(\ep)$. Hence, $p,q\in \mathsf{Q}_\sigma^c(\eta)$ for all $\sigma\in \Sigma$ and we must then have $\tilde r_\sigma (p,q) < r_0$, for some $\sigma\in \Sigma_\mathsf{D}$. Therefore, $\tilde r_\sigma( f_\omega(p), f_\omega(q)) < s_0/2$ for all $\omega\in [-N,N]^2$ and there exists $\sigma'\in \Sigma_\mathsf{D}$ such that $\mathrm{W}_m(f_\omega(p),f_\omega(q)) = W_{\sigma',m}(f_\omega(p),f_\omega(q))$, for all $\omega\in [-N,N]^2$. In particular, $\mathrm{W}_m(p,q) = W_{\sigma',m}(p,q)$.  We consider here 
\begin{align}    
    E_1 &:= \bigcap_{\sigma\in \Sigma} \lbrace |f_\omega^\kappa(p) - f_\omega(p)| \leq \delta_0 \rbrace, \\
    E_2 &:= \bigcap_{\sigma\in \Sigma} \lbrace |f_\omega^\kappa(q) - f_\omega(q)| \leq \delta_0 \rbrace, \\
    E_3 &:= \left\lbrace |\widetilde {r}_{\sigma'}(f_\omega^\kappa(p), f_\omega^\kappa(q)) - \widetilde r_{\sigma'}( f_\omega(p), f_\omega(q))| \leq \frac{\widetilde {r}_{\sigma'} (f_\omega(p), f_\omega(q))}{2} \right\rbrace
\end{align}
where $\delta_0$ is given as before. Then, on $E=E_1\cap E_2\cap E_3$ we have $\mathrm{W}_m(f_\omega^\kappa(p),f_\omega^\kappa(q)) = W_{\sigma',m}(f_\omega^\kappa(p),f_\omega^\kappa(q))$, which yields 
\begin{equation}
\mathbb{E}\left[ {\mathrm{W}_m}(f_\omega^\kappa(p),f_\omega^\kappa(q)) \mathbf{1}_E \right] \leq \gamma \mathrm{W}_m(p,q).
\end{equation}
Additionally,  we recall that 
\begin{equation}
\max \lbrace \nu\mathrm{V}_\kappa(f_\omega^\kappa(p)), \nu \mathrm{V}_\kappa(f_\omega^\kappa(q)) \rbrace \mathbf{1}_E \leq \nu C_N\eta^{-\alpha}\leq \gamma W(p,q).
\end{equation}
As a result, there holds 
\begin{equation}
\EE[\widetilde{\mathrm{W}}_\kappa(f_\omega^\kappa(p), f_\omega^\kappa(q))\mathbf{1}_E] \leq \gamma \widetilde{\mathrm{W}}_\kappa(p,q).
\end{equation}
For the complement $E^c$, we now write $E^c = E_1^c \cup E_2^c \cup E_3^c$ where 
\begin{align}
    E_1^c = (E_1^c \cap E_2\cap E_3) \cup (E_1^c \cap E_2^c\cap E_3) \cup (E_1^c \cap E_2\cap E_3^c) \cup (E_1^c \cap E_2^c\cap E_3^c) 
\end{align}
as well as 
\begin{align}
    E_2^c = (E_1 \cap E_2^c\cap E_3) \cup (E_1^c \cap E_2^c\cap E_3) \cup (E_1 \cap E_2^c\cap E_3^c) \cup (E_1^c \cap E_2^c\cap E_3^c) 
\end{align}
and
\begin{align}
    E_3^c = (E_1 \cap E_2\cap E_3^c) \cup (E_1^c \cap E_2\cap E_3^c) \cup (E_1 \cap E_2^c\cap E_3^c) \cup (E_1^c \cap E_2^c\cap E_3^c) 
\end{align}
We argue for each of these sets $\widetilde{E}$ as in Case 2.1 above, noting that either 
\begin{equation}
\EE[\widetilde{\mathrm{W}}_\kappa(f_\omega^\kappa(p),f_\omega^\kappa(q))|\widetilde{E}] \leq \kappa^{-\frac{\alpha}{2}}\widetilde{\mathrm{W}}_\kappa(p,q)
\end{equation}
with $\PP(\widetilde{E}) \lesssim \sqrt{\kappa}$ or 
\begin{equation}
\EE[\widetilde{\mathrm{W}}_\kappa(f_\omega^\kappa(p),f_\omega^\kappa(q))\mathbf{1}_{\widetilde{E}}] \leq C \widetilde{\mathrm{W}}_\kappa(p,q)\PP(\widetilde{E})^{1-\xi}
\end{equation}
with $\PP(\widetilde{E}) \lesssim \sqrt{\kappa} + r_0$. Since $\alpha=\frac18< 1$, choosing $r_0$ and $\kappa_0$ small enough we conclude that 
\begin{equation}
\EE \left[\widetilde{\mathrm{W}}_\kappa(f_\omega^\kappa(p), f_\omega^\kappa(q)) \right] \leq \frac{1+\gamma}{2}\widetilde{\mathrm{W}}_\kappa(p,q)
\end{equation}
as well.

\noindpar{End of proof} Finally, for $(p,q)\in \widetilde{K}$, we have that $\EE [\mathrm{W}_m(f_\omega^\kappa(p),f_\omega^\kappa(q))] \leq C r_0^{-\xi}$ due to Lemma \ref{lemma:EXPphikappainv}. Moreover, we have that 
\begin{equation}
\EE[\mathrm{V}_\kappa(f_\omega^\kappa(p))] \leq \gamma' \mathrm{V}_\kappa(p) + \beta \leq \widetilde\beta
\end{equation}
for some $\widetilde\beta>0$ for $\kappa_0$ small enough and similarly for $\EE[\mathrm{V}_\kappa(f_\omega^\kappa(q))]$. With this, we conclude that 
\begin{equation}
\EE[\widetilde{\mathrm{W}}_\kappa(f_\omega^\kappa(p), f_\omega^\kappa(q))] \leq \beta',
\end{equation} for some universal $\beta'>0$ for all $(p,q)\in \widetilde{K}$.
\end{proof}

\section{\texorpdfstring{Uniform-in-$\kappa$ exponential mixing and enhanced dissipation}{Uniform-in-κ exponential mixing and enhanced dissipation}}\label{sec:unifkappa}
In this section we culminate the analysis on the stochastic flow map $f_\omega^\kappa$ and we show that the stochastic Markov chain $P_\kappa^{(2)}$ satisfies the hypothesis $(H_\kappa)$ introduced in Section~\ref{sec:asbtractRDS}, namely there exists $\ell\in \N$, $\gamma\in(0,1)$, $\beta>0$, $R> \frac{2\beta}{1-\gamma}$, $\upsilon>0$ and a probability measure $\nu_0$ such that for all $\kappa>0$ sufficiently small there holds 
\begin{align}\label{eq:Hkappadrift_loc}
    P^{(2), \ell}_\kappa \mathrm{V}^{(2)}_\kappa \leq \gamma \mathrm{V}^{(2)}_\kappa + \beta,
\end{align}
and
\begin{align}\label{eq:Hkappasmallset_loc}
    \inf_{(p,q)\in \lbrace \mathrm{V}^{(2)}_\kappa\leq R \rbrace} P^{(2),\ell}_\kappa((p,q), \cdot) \geq \upsilon\nu_0(\cdot),
\end{align}
for some $\mathrm{V}^{(2)}_\kappa:\mathsf{X}^{(2)} \rightarrow [1,\infty)$.

Note that \eqref{eq:Hkappadrift_loc} is satisfied for the drift function $\mathrm{V}_\kappa^{(2)}:=\widetilde{\mathrm{W}}_\kappa$ thanks to Proposition \ref{prop:tildeWkappadrift}.
Hence, this section is devoted to prove the small set condition \eqref{eq:Hkappasmallset_loc}. 
To do so, we follow closely the approach used in \cites{CIS2024}, expanding and adapting the discussion to our sphere setting, where necessary.
More precisely, we start with Lemma~\ref{lemma:quantpreimagefromlocalsubmersion} (a quantitative version of \cite{CIS2024}*{Lemma~5.2}) which is then used in Lemma~\ref{lemma:opensmallsetstochasticTPP} to establish the existence of an open small set for the stochastic two-point process (see \cite{CIS2024}*{Lemma~5.1} for reference).

\begin{lemma}
\label{lemma:quantpreimagefromlocalsubmersion}
Let $f:B_r(0)\subset\RR^N\to\RR^d$ be a $C^1$ map such that $f(0)=0$ and $\rank (Df(0))=d$. Then there exist constants $\delta, s, C>0$ such that, for every Lipschitz map $g:\RR^N\to\RR^d$ satisfying $\norm{f-g}_{L^\infty(B_r(0))}\leq \delta$, one has
\begin{enumerate}
    \item $B_s(0)\subset g(B_r(0))$, and,
    \item  for every Borel set $A\subset\RR^d$,
\begin{equation}\label{eq:quantpreimagefromlocalsubmersion}
    |g^{-1}(A)\cap B_r(0)| \geq C\,\Lip(g)^{-d}|A\cap B_s(0)|.
\end{equation}
\end{enumerate}
\end{lemma}
\begin{proof}
The inclusion $B_s(0)\subset g(B_r(0))$ is precisely Lemma~5.2 of \cites{CIS2024}, after choosing $\delta,s>0$ sufficiently small. We prove only the quantitative estimate.

By the constant-rank theorem, there exists open sets $\mathcal{U},\mathcal{V}\subset \R^d$ and diffeomorphisms 
\begin{align}
    &\psi:B^d_\rho\times B^{N-d}_R\subset \R^N \rightarrow \psi(B_\rho^d\times B_R^{N-d}) \subset B_r(0)\subset \R^N,\\
    &\varphi:\mathcal{U}\subset \R^d\rightarrow \mathcal{V}\subset \R^d,
\end{align}
for $\rho,R>0$ small enough, such that for all $(u,v)\in B^d_\rho\times B_R^{N-d}$ there holds $(\varphi\circ f\circ\psi)(u,v)=u$. Here we denote by $B^d_\rho\subset \R^d$ and $B^{N-d}_R\subset \R^{N-d}$ the balls of radius $\rho$ and $R$ centred at the origin of $\R^d$ and $\R^{N-d}$, respectively.

For $\delta>0$ sufficiently small, $g\circ\psi$ remains in $\mathcal{U}$, and the slice maps 
\begin{equation}
    h_v:B_\rho^d\to\RR^d, \qquad h_v(u):=(\varphi\circ g\circ\psi)(u,v),
\end{equation}
satisfy the covering property 
\begin{equation}\label{eq:slicecoveringcompactproof}
    B_\sigma^d\subset h_v(B_\rho^d),
\end{equation}
for every $v\in B_R^{N-d}$ and for some $0<\sigma<\rho$. This is the degree argument contained in the proof of \cite{CIS2024}*{Lemma 5.2}.
We choose $s>0$ so small that
\begin{equation}
    A_s:=A\cap B_s(0) \quad\Longrightarrow\quad \mathcal{W}:=\varphi(A_s)\subset B_\sigma^d.
\end{equation}
Since the diffeomorphisms $\varphi$ and $\psi$ are fixed on pre-compact sets, there exists $C_1>0$ such that for every $v\in B_R^{N-d}$ it holds $\Lip(h_v)\leq C_1\Lip(g)$.
Moreover, after decreasing $s$ if necessary, the change of variables under $\varphi$ gives
\begin{equation}\label{eq:alphalowercompactproof}
    |\mathcal{W}| =|\varphi(A_s)|\geq C_2 |A_s|
\end{equation}
for some $C_2>0$.
For $v\in B_R^{N-d}$, define $E_v:=\{u\in B_\rho^d:h_v(u)\in \mathcal{W}\}$. 
By \eqref{eq:slicecoveringcompactproof}, the multiplicity function $N(h_v,E_v,y):= \mathcal{H}^0(E_v \cap h_v^{-1}\lbrace y \rbrace)$ satisfies $N(h_v,E_v,y)\geq 1$ for every $y\in \mathcal{W}$, where $\mathcal{H}^0$ denotes the zero-dimensional Hausdorff measure. The area formula on the $d$-dimensional slices \cite{evans2025measure}*{Lemma 5.2} gives
\begin{equation}
    |\mathcal{W}| = \int_{\R^d} \mathbf{1}_{\mathcal{W}}\d y \leq C_d \int_{\R^d} \mathbf{1}_{\mathcal{W}}\d \mathcal H^d(y) \leq C_d \int_{\R^d}N(h_v,E_v, y) \d \mathcal{H}^d(y) = C_d\int_{E_v}|J h_v(u)| \d u,
\end{equation}
where $\mathcal{H}^d$ denotes the $d$-dimensional Hausdorff measure, $C_d>0$ and $J h_v$ is the Jacobian determinant of $h_v$.
Since $h_v$ is Lipschitz, for a.e. $u\in B_\rho^d$
\begin{equation}
    |J h_v(u)|\leq \Lip(h_v)^d\leq C_1^d\,\Lip(g)^d.
\end{equation}
Therefore, for every $v\in B_R^{N-d}$, there holds
\begin{equation}\label{eq:Evboundcompactproof}
    |E_v|\geq C_1^{-d} C_d^{-1}\,\Lip(g)^{-d}|\mathcal{W}|,
\end{equation}
and by Fubini,
\begin{align}\label{eq:fubinicompactproof}
    \left|\left\{(u,v)\in B_\rho^d\times B_R^{N-d}:(\varphi\circ g\circ\psi)(u,v)\in \mathcal{W}\right\}\right|&=\int_{B_R^{N-d}} |E_v|\,\dd v\ge|B_R^{N-d}|\,C_1^{-d}\,\Lip(g)^{-d}|\mathcal{W}|.
\end{align}
Moreover,
\begin{align}
    \psi \left( \left\{(u,v)\in B_\rho^d\times B_R^{N-d}:(\varphi\circ g\circ\psi)(u,v)\in \mathcal{W}\right\} \right) = g^{-1}(A_s)\cap \psi(B^d_\rho \times B^{N-d}_R) \subseteq g^{-1}(A)\cap B_r(0)
\end{align}
since $A_s\subseteq A$ and $\psi(B^d_\rho \times B^{N-d}_R) \subset B_r(0)$,  and thus
\begin{align}
    \left|\left\{(u,v)\in B_\rho^d\times B_R^{N-d}:(\varphi\circ g\circ\psi)(u,v)\in \mathcal{W}\right\}\right| \leq C_\psi^{-1}  \left| g^{-1}(A)\cap B_r(0) \right|.
\end{align}
for some $C_\psi:=\inf_{B_R^d\times B_\rho^{N-d}}|\det D\psi|>0$ since $\psi$ is a diffeomorphism on $B^d_\rho\times B^{N-d}_R$ to its image. Combining this with \eqref{eq:alphalowercompactproof} yields, for some $C>0$,
\begin{equation}
    |g^{-1}(A)\cap B_r(0)| \geq C\Lip(g)^{-d}|A\cap B_s(0)|
\end{equation}
and the proof is finished.
\end{proof}

\begin{lemma}\label{lemma:opensmallsetstochasticTPP}
There exists non-empty open sets $U_1, U_2\subset \mathsf{X}^{(2)}$ and a constant $\psi_0>0$ such that, for some $n\in \N$,
    \begin{align}
        \inf_{(\x,\y)\in U_1} P_\kappa^{(2),n}((x,y), \cdot) \geq \beta_0\mathrm{Leb}|_{U_2}(\cdot)
    \end{align}
    for all $\kappa\in(0,\kappa_0)$, for some $\kappa_0$ small enough.
\end{lemma}

\begin{proof}
    We argue as in \cite{CIS2024}*{Lemma 5.1}.  From the proof of Proposition \ref{prop:smallsetTPP} there exists $(\x_0,\y_0)\in \mathsf{X}^{(2)}$, $n\geq 1$ and $\underline{\omega}_n$ such that the map $\underline\omega\mapsto f_{\underline\omega}^{(2)}(\x_0,\y_0)$ is a submersion at $\underline{\omega}=\underline{\omega}_n$, with $\varrho_n(\underline{\omega})\geq c>0$ for all $\underline{\omega}\in B_\ep(\underline{\omega}_n)$ for some $\ep>0$. Let $\eta>0$ and $(\x,\y)\in B_\eta((\x_0,\y_0))$. Let $\delta>0$ and consider the set 
    \begin{align}
    E &= \left\lbrace \sup_{\underline{\omega}\in [-N,N]^{2n}} |f_{\underline{\omega}}^\kappa(\x) - f_{\underline{\omega}}^\kappa(\x_0)| + |f_{\underline{\omega}}^\kappa(\x_0) - f_{\underline{\omega}}(\x_0)| + |f_{\underline{\omega}}^\kappa(\y) - f_{\underline{\omega}}^\kappa(\y_0)| + |f_{\underline{\omega}}^\kappa(\y_0) - f_{\underline{\omega}}(\y_0)| \leq \frac{\delta}{100} \right\rbrace \\
    & \quad \bigcap \left\lbrace \sup_{ \underset{\omega\neq\varpi}{\underline{\omega},\underline{\varpi}\in[-N,N]^{2n}}} \frac{|f_{\underline{\omega}}^\kappa(\x,\y)-f_{\underline{\varpi}}^\kappa(\x,\y)|}{|\underline{\omega} - \underline{\varpi}|_\infty} \leq M \right\rbrace
\end{align}
and note that $\PP(E^c) \lesssim \sqrt{\kappa} + \eta + M^{-2}$ due to Markov inequality and Lemma~\ref{lemma:EXPsupomegaLipTPP} and Lemma~\ref{lemma:EXPunifomegaTPP} for all $(\x,\y)\in B_\eta((\x_0,\y_0))$. Hence, $\PP(E)\geq \frac34$ for all $(\x,\y)\in U_1:=B_\eta((\x_0,\y_0))$ for some $M>1$ large enough, some $\eta>0$ small enough and $\kappa_0$ small enough. Moreover, on $E$ we have that $|f_\omega^{(2),\kappa}(\x,\y) - f_\omega^{(2)}(\x_0,\y_0)|\leq \frac{\delta}{25}$ for all $(\x,\y)\in U_1$ and all $\underline{\omega}\in[-N,N]^{2n}$ and thus apply \cite{CIS2024}*{Lemma 5.2} to the map $g:\underline{\omega} \mapsto f_{\underline{\omega}}^{(2),\kappa}(\x,\y)$ to conclude that $g(B_\ep(\underline{\omega}_n))\supseteq B_s(f_{\underline{\omega}_n}(\x_0,\y_0))=:U_2$, for some $s>0$. Hence, for all measurable $U\subset \mathsf X\times \mathsf X\setminus\cS$,
\begin{align}
    P^{(2),\kappa}_n((\x,\y),U) &\geq \EE_\bB\left[ \int_{g^{-1}(U)\cap B_\ep(\underline{\omega}_n)} \varrho_n(\underline{\omega}) \d \underline{\omega} \big| E \right] \PP_\bB(E) \\
    &\geq\frac{3c}{4}\EE_\bB\left[ |g^{-1}(U)\cap B_\ep(\underline{\omega}_n)| \big| E \right] \\
    &\geq c' |U\cap B_s(f_{\underline{\omega}_n}^{(2)}(\x_0,\y_0))|
\end{align}
for some $c'>0$. Here we have used that $\varrho_n(\underline{\omega})\geq c$ for all $\underline{\omega}\in B_\ep(\underline{\omega}_n)$ and Lemma \ref{lemma:quantpreimagefromlocalsubmersion}, since $f_\omega^{(2)}$ is a local submersion and  $g$ has a uniform Lipschitz bound on the set $E$.
\end{proof}

We are now in position to verify the Lyapunov-Foster drift condition \eqref{eq:Hkappadrift_loc} and prove small set condition \eqref{eq:Hkappasmallset_loc} in $(H_\kappa)$.
To do so, we follow the two step proof of Lemma 3.1 in \cites{CIS2024}, which we split into three auxiliary lemmas: 
\begin{enumerate}[label=--]
    \item Lemma~\ref{lemma:inviscidTPPcompactreachability} proves that every compact set can be uniformly reached via the inviscid two-point process;
    \item Lemma~\ref{lemma:compactsmallsetstochasticTPP} then translate the above inviscid property into a uniform in $\kappa>0$ small set condition for all compact sets for the stochastic two-point process, combining Lemma~\ref{lemma:inviscidTPPcompactreachability} with Lemma~\ref{lemma:opensmallsetstochasticTPP};
    \item Lemma~\ref{lemma:Wtildekappasublevelcompact} shows that level sets of $\widetilde{\mathrm{W}}_\kappa$ are contained in compact sets.
\end{enumerate}
Once these result are established, Proposition~\ref{prop:kappasmallset} proves that $(H_\kappa)$ holds.

\begin{lemma}
\label{lemma:inviscidTPPcompactreachability}
For every compact set $K\Subset \mathsf{X}^{(2)}$ and every non-empty open set $\mathcal O\subset \mathsf{X}^{(2)}$, there exist $q_K\in\NN$ and $c_K>0$ such that
\begin{equation}\label{eq:inviscidcompactreachcor}
    \inf_{\x^{(2)}\in K} P_0^{(2),q_K}(\x^{(2)},\mathcal O) \geq c_K .
\end{equation}
\end{lemma}

\begin{proof}
The continuity of the inviscid two-point RDS gives the Feller property of $P_0^{(2)}$, Lemma~\ref{lemma:blackboxaperiodic} gives aperiodicity, Lemma~\ref{prop:smallsetTPP} gives the open-small-set structure, and Proposition~\ref{prop:irredTPP} gives topological irreducibility on $\mathsf{X}^{(2)}$. Combined with Theorem~6.2.5~(ii) and Theorem~5.5.7 in~\cites{meyn2012markov} we obtain \eqref{eq:inviscidcompactreachcor} (see~\cites{CIS2024}).
\end{proof}

\begin{lemma}
\label{lemma:compactsmallsetstochasticTPP}
Let $K\Subset \mathsf{X}^{(2)}$ be compact. There exists an integer $\ell_K\in\NN$, a constant $\upsilon_K>0$, a probability measure $\nu_0$ on $\mathsf{X}^{(2)}$, and $\kappa_K>0$ such that, for every $\kappa\in(0,\kappa_K)$,
\begin{equation}\label{eq:compactsmallsetstochasticTPP}
    \inf_{\x^{(2)}\in K} P_\kappa^{(2),\ell_K}(\x^{(2)},\cdot) \geq \upsilon_K\nu_0(\cdot).
\end{equation}
\end{lemma}

\begin{proof}
By Lemma~\ref{lemma:opensmallsetstochasticTPP}, there exist non-empty open sets $U_1,U_2\Subset \mathsf{X}^{(2)}$, an integer $n_\ast\in\NN$, a constant $\beta_\ast>0$, and $\kappa_\ast>0$ such that
\begin{equation}\label{eq:opensmallsetusedcompactlemma}
    \inf_{\x^{(2)}\in U_1} P_\kappa^{(2),n_\ast}(\x^{(2)},\cdot) \geq \beta_\ast\,\Leb(\cdot\cap U_2)
\end{equation}
for every $\kappa\in(0,\kappa_\ast)$. Choose $\x_0^{(2)}\in U_1$ and $r>0$ such that $B_r(\x_0^{(2)})\Subset U_1$. Applying Lemma~\ref{lemma:inviscidTPPcompactreachability} with $\mathcal O:=B_{r/2}(\x_0^{(2)})$, we obtain $q_K\in\NN$ and $c_K>0$ such that
\begin{equation}\label{eq:deterministiccompactreachU1}
    \inf_{\x^{(2)}\in K} P_0^{(2),q_K} \bigl(\x^{(2)},B_{r/2}(\x_0^{(2)})\bigr) \geq c_K .
\end{equation}
We now extend this deterministic lower bound to the stochastic chain.

For each $\x^{(2)}\in K$ the measure $P_\kappa^{(2),q_K}(\x^{(2)},\cdot)$ converges weakly-* to the measure $P_0^{(2),q_K}(\x^{(2)},\cdot)$ as $\kappa\rightarrow 0$ so that there exists $\kappa_0(\x^{(2)})>0$ such that
\begin{align}\label{eq:pointwiseweaktransfer}
\inf_{\kappa<\kappa_0(\x^{(2)})}P_\kappa^{(2),q_K}\bigl(\x^{(2)},B_{r/2}(\x_0^{(2)})\bigr) \geq \frac12 P_0^{(2),q_K}\bigl(\x^{(2)},B_{r/2}(\x_0^{(2)})\bigr) \geq \frac{c_K}{2}.
\end{align}

This pointwise lower bound is stable under small perturbations of the initial condition. Fix $\x^{(2)}\in K$ and define, for $\y^{(2)}$ close to $\x^{(2)}$,
\begin{equation}
    E:=\left\{
        \sup_{\underline{\omega}_{q_K}\in[-N,N]^{2q_K}} d_{\mathsf{X}^{(2)}}\left(f_{\underline{\omega}_{q_K}}^{(2),\kappa}(\x^{(2)}),f_{\underline{\omega}_{q_K}}^{(2),\kappa}(\y^{(2)}) \right) \leq \frac r2
        \right\}.
\end{equation}
By the uniform-in-$\underline{\omega}$ stability estimate in Lemma~\ref{lemma:EXPsupomegaLipTPP}, there exists $\widetilde C=\widetilde C(q_K,N)>0$ such that
\begin{equation}
    \EE_\bB \left[ \sup_{\underline{\omega}_{q_K}\in[-N,N]^{2q_K}} d_{\mathsf{X}^{(2)}}\left(f_{\underline{\omega}_{q_K}}^{(2),\kappa}(\x^{(2)}),f_{\underline{\omega}_{q_K}}^{(2),\kappa}(\y^{(2)})
        \right) \right] \leq \widetilde C\,d_{\mathsf{X}^{(2)}}(\x^{(2)},\y^{(2)}).
\end{equation}
Choose $r_{\x^{(2)}}:= \frac{c_Kr}{8\widetilde C} $, decreasing it if necessary so that $B_{r_{\x^{(2)}}}(\x^{(2)})\Subset X^{(2)}$. Then, for every $\y^{(2)}\in B_{r_{\x^{(2)}}}(\x^{(2)})$, Markov's inequality gives
\begin{equation}
    \PP(E^c) \leq \frac2r \EE_\bB\left[ \sup_{\underline{\omega}_{q_K}\in[-N,N]^{2q_K}} d_{\mathsf{X}^{(2)}}\left( f_{\underline{\omega}_{q_K}}^{(2),\kappa}(\x^{(2)}), f_{\underline{\omega}_{q_K}}^{(2),\kappa}(\y^{(2)}) \right) \right] \leq \frac{c_K}{4}.
\end{equation}
On $E$ one has
\begin{equation}
    \left\{ f_{\underline{\omega}_{q_K}}^{(2),\kappa}(\x^{(2)}) \in B_{r/2}(\x_0^{(2)}) \right\} \subset \left\{ f_{\underline{\omega}_{q_K}}^{(2),\kappa}(\y^{(2)}) \in B_r(\x_0^{(2)})\right\}.
\end{equation}
Combining this inclusion with \eqref{eq:pointwiseweaktransfer}, we obtain
\begin{align}
    P_\kappa^{(2),q_K} \bigl(\y^{(2)},B_r(\x_0^{(2)})\bigr)&\geq P_\kappa^{(2),q_K} \bigl(\x^{(2)},B_{r/2}(\x_0^{(2)})\bigr) - \PP(E^c) \\ 
    &\geq \frac{c_K}{2}-\frac{c_K}{4} = \frac{c_K}{4}.
\end{align}
for $\kappa\in(0,\kappa_0(\x^{(2)}))$. The balls $B_{r_{\x^{(2)}}}(\x^{(2)})$, with $\x^{(2)}\in K$, form an open cover of $K$.
By compactness, choose a finite subcover $K\subset \bigcup_{j=1}^J B_{r_{\x_j^{(2)}}}(\x_j^{(2)})$, and let $\kappa_K':=\min_{1\leq j\leq J}\kappa_{\x_j^{(2)}}$.
Then, for every $\kappa\in(0,\kappa_K')$,
\begin{equation}\label{eq:compactreachU1compactlemma}
    \inf_{\y^{(2)}\in K} P_\kappa^{(2),q_K} \bigl(\y^{(2)},B_r(\x_0^{(2)})\bigr) \geq \frac{c_K}{4}.
\end{equation}
Since $B_r(\x_0^{(2)})\Subset U_1$, this implies
\begin{equation}\label{eq:compactreachU1final}
    \inf_{\y^{(2)}\in K} P_\kappa^{(2),q_K}(\y^{(2)},U_1) \geq a_K, \qquad a_K:=\frac{c_K}{4}.
\end{equation}

Combining \eqref{eq:compactreachU1final} with \eqref{eq:opensmallsetusedcompactlemma}, the Chapman--Kolmogorov argument (see also \cite{meyn2012markov}*{Theorem~3.4.2}) from the proof of Lemma~3.1 in~\cites{CIS2024} gives, for every Borel set $U\subset \mathsf X^{(2)}$ and every $\x^{(2)}\in K$,
\begin{align}
    P_\kappa^{(2),q_K+n_\ast}(\x^{(2)},U)
    &\geq \int_{U_1} P_\kappa^{(2),n_\ast}(\y^{(2)},U) P_\kappa^{(2),q_K}(\x^{(2)},\dd\y^{(2)}) \\
    &\geq \beta_\ast\Leb(U\cap U_2) P_\kappa^{(2),q_K}(\x^{(2)},U_1) \\
    &\geq \beta_\ast a_K\Leb(U\cap U_2).
\end{align}
Since $U_2$ is non-empty and open, $\Leb(U_2)>0$. Define next
\begin{equation}
    \nu_0(U):= \frac{\Leb(U\cap U_2)}{\Leb(U_2)}, \qquad \upsilon_K:=\beta_\ast a_K\Leb(U_2)>0, \qquad \ell_K:=q_K+n_\ast .
\end{equation}
After decreasing $\kappa_K'>0$ so that $\kappa_K'\leq\kappa_\ast$, we obtain \eqref{eq:compactsmallsetstochasticTPP} with $\kappa_K:=\kappa_K'$.
\end{proof}
Together, the two lemmas above constitute the first step in the proof of Lemma~3.1 in~\cites{CIS2024} while the second step is carried out in the following lemma. 

\begin{lemma}
\label{lemma:Wtildekappasublevelcompact}
For every $R>0$ there exist $\kappa_R>0$ and a compact set $\widetilde K_R\Subset \mathsf{X}^{(2)}$ such that
\begin{equation}\label{eq:Wtildekappasublevelcompact}
    \{\widetilde{\mathrm W}_\kappa\leq R\} \subset \widetilde K_R
\end{equation}
for every $\kappa\in(0,\kappa_R)$.
\end{lemma}

\begin{proof}
Choose parameters $ 0<\ep'\leq r_0'\leq \frac{s_0}{2} $ so small that
\begin{equation}\label{eq:compactlevelchoices}
    C(r_0')^{-\xi}>2R, \qquad C(\ep')^{-1/8}>2R.
\end{equation}
Then choose $\kappa_R>0$ so small that $ \sqrt{\kappa}\leq \ep' $ for every $\kappa\in(0,\kappa_R)$.  Define
\begin{equation}
    \widetilde K_R:= \left\{ (p,q)\in \mathsf X^{(2)}: p,q\in K(\ep'), \quad \tilde r_\sigma(p,q)\geq r_0' \text{ for every }\sigma\in\Sigma_{\sfD} \right\}.
\end{equation}
By construction, $\widetilde K_R\Subset \mathsf X^{(2)}$. We prove that $(\widetilde K_R)^c\subset\{\widetilde{\mathrm W}_\kappa>R\}$. Let $(p,q)\notin\widetilde K_R$ so that at least one of the following alternatives holds.
\begin{enumerate}[label=--]
    \item There exists $\sigma\in\Sigma_{\sfD}$ such that $\tilde r_\sigma(p,q)<r_0'$. By the two-point singular part of $\widetilde{\mathrm W}_\kappa$, we get
    \begin{equation}
        \widetilde{\mathrm W}_\kappa(p,q) \geq C(r_0')^{-\xi} > 2R.
    \end{equation}
    \item There exists $\sigma\in\Sigma$ such that $r_\sigma(p)<\ep'$. Since $\sqrt{\kappa}\leq\ep'$, the one-point part of $\widetilde{\mathrm W}_\kappa$ gives
    \begin{equation}
        \widetilde{\mathrm W}_\kappa(p,q) \geq C\max\{\sqrt{\kappa},r_\sigma(p)\}^{-1/8} \geq C(\ep')^{-1/8} > 2R.
    \end{equation}
    \item There exists $\sigma\in\Sigma$ such that $r_\sigma(q)<\ep'$. The same argument gives
    \begin{equation}
        \widetilde{\mathrm W}_\kappa(p,q) \geq C\max\{\sqrt{\kappa},r_\sigma(q)\}^{-1/8} \geq C(\ep')^{-1/8} > 2R.
    \end{equation}
\end{enumerate}
In all three cases, $\widetilde{\mathrm W}_\kappa(p,q)>2R$.  Hence $ (\widetilde K_R)^c \subset \{\widetilde{\mathrm W}_\kappa>R\}$, which is equivalent to $ \{\widetilde{\mathrm W}_\kappa\leq R\} \subset \widetilde K_R $. The proof is complete.
\end{proof}

Checking the $(H_\kappa)$ condition now follows by combining the various ingredients established so far.

\begin{proposition}\label{prop:kappasmallset}
There exists $\ell\in\NN$, $\gamma\in(0,1)$, $\beta>0$, $R>\frac{2\beta}{1-\gamma}$, $\upsilon>0$, and a probability measure $\nu_0$ such that, for all $\kappa>0$ sufficiently small,
\begin{equation}\label{eq:Hkappadriftprop}
    P^{(2),\ell}_\kappa \widetilde{\mathrm W}_\kappa \leq \gamma \widetilde{\mathrm W}_\kappa+\beta,
\end{equation}
and
\begin{equation}\label{eq:Hkappasmallsetprop}
    \inf_{(\x,\y)\in\{\widetilde{\mathrm W}_\kappa\leq R\}} P^{(2),\ell}_\kappa((\x,\y),\cdot) \geq \upsilon\nu_0(\cdot).
\end{equation}
\end{proposition}

\begin{proof}
Let $\gamma_0\in(0,1)$ and $b_0>0$ be the constants given by Proposition~\ref{prop:tildeWkappadrift}, so that, after decreasing $\kappa_0>0$ if necessary,
\begin{equation}\label{eq:onestepdriftWtildekappaprop}
    P_\kappa^{(2)}\widetilde{\mathrm W}_\kappa \leq \gamma_0\widetilde{\mathrm W}_\kappa+b_0
\end{equation}
for every $\kappa\in(0,\kappa_0)$.  Iterating \eqref{eq:onestepdriftWtildekappaprop}, as in the proof of Lemma~3.1 in~\cites{CIS2024}, gives for every $m\in\NN$
\begin{equation}\label{eq:iterateddriftWtildekappaprop}
    P_\kappa^{(2),m}\widetilde{\mathrm W}_\kappa \leq \gamma_0^m\widetilde{\mathrm W}_\kappa + b_0\frac{1-\gamma_0^m}{1-\gamma_0}.
\end{equation}
Choose
\begin{equation}\label{eq:Rchoicekappasmallsetprop}
    R>\frac{2b_0}{1-\gamma_0}.
\end{equation}
By Lemma~\ref{lemma:Wtildekappasublevelcompact}, there exist $\kappa_R>0$ and a compact set $\widetilde K_R\Subset X^{(2)}$ such that $\{\widetilde{\mathrm W}_\kappa\leq R\} \subset \widetilde K_R $ for every $\kappa\in(0,\kappa_R)$.  Applying Lemma~\ref{lemma:compactsmallsetstochasticTPP} to $K=\widetilde K_R$, we obtain $\ell\in\NN$, $\upsilon>0$, a probability measure $\nu_0$, and $\kappa_R'>0$ such that
\begin{equation}
    \inf_{\zeta^{(2)}\in\widetilde K_R} P_\kappa^{(2),\ell}(\zeta^{(2)},\cdot) \geq \upsilon\nu_0(\cdot)
\end{equation}
for every $\kappa\in(0,\kappa_R')$.  Hence \eqref{eq:Hkappasmallsetprop} follows for every $\kappa\in(0,\min\{\kappa_R,\kappa_R', \kappa_0\})$. Finally, define
\begin{equation}
    \gamma:=\gamma_0^\ell, \qquad \beta:=b_0\frac{1-\gamma_0^\ell}{1-\gamma_0}.
\end{equation}
Then \eqref{eq:iterateddriftWtildekappaprop} with $m=\ell$ gives \eqref{eq:Hkappadriftprop}.  Moreover, $\frac{2\beta}{1-\gamma} = \frac{2b_0}{1-\gamma_0}$, and therefore \eqref{eq:Rchoicekappasmallsetprop} gives $R>\frac{2\beta}{1-\gamma}$.
The proof is complete.
\end{proof}

We finish the section by presenting the main ideas behind the proof of Corollary \ref{cor:optimalenhanceddiss}.

\begin{proof}[Proof of Corollary \ref{cor:optimalenhanceddiss}]
    Following the proof of Proposition 2.2 in \cites{navarro2025exponential}, in view of Theorem \ref{thm:mainexpdecaycorr} and Corollary \ref{cor:unifexpmixing}, there exists a deterministic $\lambda>0$ and a random constant $\mathrm{C}_{\underline{\omega},\kappa}'\geq 1$ independent of $\bB$ such that $\EE_{\omega}[|\mathrm{C}_{\underline{\omega},\kappa}'|^2] \leq \mathrm{C}'_2$, for some $\mathrm{C}'_2\geq 1$ uniformly in $\kappa$, for which
    \begin{align}
        \Vert \rho(t,\cdot) \Vert_{L^2} \leq \frac{\mathrm{C}_{\underline{\omega},\kappa}'}{\sqrt{\kappa}}e^{-\lambda t} \Vert \rho_0 \Vert_{L^2}.
    \end{align}
    Note that the trivial monotonicity estimate $\Vert \rho(t,\cdot) \Vert_{L^2} \leq \Vert \rho_0 \Vert_{L^2}$ is sharper for small times $t\lesssim \log \frac{1}{\kappa}$. For the critical time
    \begin{align}
        t_* = \frac{1}{\lambda}\left( 1 + \log\left( \frac{\mathrm{C}_{\underline{\omega},\kappa}'}{\sqrt{\kappa}} \right) \right)
    \end{align}
    there holds $\Vert \rho(t_*,\cdot) \Vert_{L^2} \leq \frac{1}{e}\Vert \rho_0 \Vert_{L^2}$. Using the monotonicity of $\Vert \rho(t,\cdot) \Vert_{L^2}$ and iterating the above estimate we obtain
    \begin{align}
        \Vert \rho(t,\cdot) \Vert_{L^2} \leq e^{-n}\Vert \rho_0\Vert_{L^2} \leq e^{-\mu(\kappa)t}\Vert \rho_0\Vert_{L^2},
    \end{align}
    for all $n\in \N$ and all $t\in [nt_*, (n+1)t_*)$, where we define the random rate $\mu(\kappa)$ as
    \begin{align}
        \mu(\kappa):= \lambda  \left( 1 + \log\left( \frac{\mathrm{C}_{\underline{\omega},\kappa}'}{\sqrt{\kappa}} \right) \right)^{-1}.
    \end{align}
    Since $\kappa\in(0,1)$ and $\mathrm{C}_{\underline{\omega},\kappa}'\geq 1$ we readily obtain $\mu(\kappa)\leq 2\lambda$ almost surely. Moreover, for any fixed $M>1$,
    \begin{align}
        \PP\left( \mu(\kappa) \log(\kappa^{-1}) \leq \frac{2\lambda}{1+M} \right) = \PP \left( |\mathrm{C}_{\underline{\omega},\kappa}'|^2 \geq e^{-2}\kappa^{-M}\right) \leq e^2\kappa^M \mathrm{C}'\leq e^2 \kappa_0^M \mathrm{C}'.
    \end{align}
    This concludes the proof.
\end{proof}

\addtocontents{toc}{\protect\setcounter{tocdepth}{-1}}

\section*{Acknowledgments}
The authors would like to thank V\'ictor Navarro-Fern\'andez and Bernhard Kepka for inspiring and fruitful discussions. The research of ADZ is supported by the SNSF through the grant PCEFP2 203059. The research of MN is partially supported by the European Research Council (ERC) under the European Union’s Horizon 2020 research and innovation programme through the grant agreement 862342.

\medskip
The authors acknowledge the use of AI tools in the preparation of this manuscript.
Some figures were generated with Claude Opus 4.7, while others contain snapshots extracted from a simulation generated using again Claude Opus 4.7. This is available, together with its code, on ADZ's personal webpage under the name \href{https://sites.google.com/view/augustodelzotto/math?authuser=0#h.n3qp573y9509}{Rossby--Haurwitz flows simulation}.
ChatGPT 5.5 Plus was also used as an auxiliary tool for language proofreading before finalization.

\addtocontents{toc}{\protect\setcounter{tocdepth}{1}}

\appendix

\section{Commutators of vector fields}\label{app:diffgeom}
In this section we show that the commutators of the lifted vector fields $\widetilde{u_\a}$ and $\widetilde{u_\b}$ of Proposition \ref{prop:LARCtoplyap} can be computed in terms of $u^\a$ and $u^\b$. 

\subsection{Commutators on smooth manifolds}
We begin by characterising the commutator of vector fields on a smooth manifold $\mathrm{M}$ in terms of the flow maps they locally generate. 

\begin{lemma}[\cite{LeeManifolds}*{Corollary~9.14}]\label{prop:diffeoinvflow}
Let $\mathrm{M}$, $\mathrm{N}$ be smooth manifolds and let $F:\mathrm{M}\rightarrow \mathrm{N}$ be a diffeomorphism. If $Y\in T\mathrm{M}$ is a vector field on $\mathrm{M}$ and $\Phi_s^Y$ is the flow of $Y$, then the flow of the push-forward vector field $F_*Y$ is given by $\Phi_s^{F_*Y} = F \circ\Phi_s^Y \circ F^{-1}$.
\end{lemma}

\begin{definition}
    Let $\mathrm{M}$ be a smooth manifold, $X\in \mathfrak{X}(\mathrm{M})$ be a smooth vector field on $\mathrm{M}$ and $\Phi_t^X$ be the flow associated to $X$. Let $Y\in \mathfrak{X}(\mathrm{M})$. Then, the Lie derivative of $Y$ with respect to $X$ is
    \begin{align}
        \mathcal{L}_X Y = \frac{\d }{\d t}\Big|_{t=0} ( d \Phi_{-t}^X)|_{\Phi_t^X} (Y(\Phi_t^X))
    \end{align}
\end{definition}

\begin{lemma}\label{lemma:commutator}
Let $\mathrm{M}$ be a smooth manifold, $X\in \mathfrak{X}(\mathrm{M})$ be a smooth vector field on $\mathrm{M}$ and $\Phi_t^X$ be the flow associated to $X$. Let $Y\in \mathfrak{X}(\mathrm{M})$. Then, 
\begin{align}
     \mathcal{L}_X Y = \frac{\d }{\d t}\Big|_{t=0} (\Phi_{-t}^X)_*Y.
\end{align}
\end{lemma}

\begin{proof}
Let $\Phi_s^Y$ denote the flow of $Y$, with $\Phi_0^Y = I_\mathrm{M}$. For all $t\in \R$, we recall that 
    \begin{align}
        (\Phi_{-t}^X)_*Y &:= \frac{\d}{\d s}\Big|_{s=0} \left( \Phi_{-t}^X \circ \Phi_s^Y \circ \Phi_{t}^X \right) =(d\Phi_{-t}^X)|_{\Phi_t^X} (Y(\Phi_t^X)).
    \end{align}
Hence, the Lemma follows from taking $\frac{\d}{\d t}\Big|_{t=0}$ above.
\end{proof}

\subsection{The commutator lift and the lifted commutator}
Let $\mathrm{M}$ be a $d$-dimensional manifold. We say that $\mathrm{B}$ is a fiber bundle over $\mathrm{M}$ with model fiber $\mathrm{F}$, if $\mathrm{B}$ is a topological space and there exists a continuous surjective map ${\pi}:\mathrm{B}\rightarrow\mathrm{M}$ such that for each $\mathrm{x}\in \mathrm{M}$ there exists a neighbourhood $\mathrm{U}$ of $\mathrm{x}$ in $\mathrm{M}$ and a homeomorphism $\Psi:\pi^{-1}(\mathrm{U}) \rightarrow \mathrm{U}\times\mathrm{F}$ such that the chain
\begin{align}
    \mathrm{U} \overset{\pi}{\longleftarrow} \pi^{-1}(\mathrm{U}) \overset{\Psi}{\longrightarrow}\mathrm{U}\times \mathrm{F} \overset{\pi_1}{\longrightarrow} \mathrm{U}
\end{align}
commutes. Let $\Phi_t$ be a smooth local diffeomorphism on $\mathrm{U}$ such that $\Phi_0=Id$. Then, for all $\mathrm{U}_0\subset\subset\mathrm{U}$ there exists $t_0>0$ such that $\Phi_t(\mathrm{U}_0)\subset\subset\mathrm{U}$, for all $|t|<t_0$. The lift of $\Phi_t$ to $\mathrm{B}$ is locally given by $\widehat{\Phi_t} = \pi_1^{-1} \circ\Phi_t \circ \pi_1$, where $\pi_1$ is the projection onto $\mathrm{U}$ of the local trivialisation $\mathrm{U}\times \mathrm F$ of $\mathrm{B}$. Then, given any two smooth local diffeomorphisms $\Phi_t^1$ and $\Phi_t^2$ on $\mathrm{U}$ with $\Phi_0^1=\Phi_0^2=Id$, we have that $\Phi_s^2\Phi_t^1:\mathrm{U}_0\rightarrow \mathrm{U}$ and 
\begin{align}
    \widehat{\Phi_s^2 \circ \Phi_t^1} = \widehat{\Phi_s^2}\circ \widehat{\Phi_t^1}
\end{align}
for all $|t|\leq t_0$ and $|s|\leq s_0$ for some $t_0,s_0>0$. 
Let $X\in T(\mathrm{M})$ have flow $\Phi_t^X$ and let $\widetilde{\Phi_t^X}$ denote its local lift to $\mathrm{B}$. We then define the lifted vector field $\widetilde X$ on $T\mathrm{B}$ locally by
\begin{equation}
\widetilde X := \left.\frac{d}{dt}\right|_{t=0}\widehat{\Phi_t^X}.
\end{equation}
The next result shows that the commutator of lifted vector fields coincides with the lift of their commutator.

\begin{lemma}\label{lemma:tildecommutator}
    There holds $[\widetilde{X},\widetilde{Y}] = \widetilde{[X,Y]}$, for all $X,Y\in T\mathrm{M}$.
\end{lemma}

\begin{proof}
Let $X,Y\in T\mathrm{M}$, let $\Phi_t^X$ and $\Phi_s^Y$ denote their flows and $\widehat{\Phi_t^X}$, $\widehat{\Phi_s^Y}$ their lifted flows. We begin by noting that 
\begin{equation}\label{eq:commutativity_lift}
\widehat{\left(\Phi_t^X\circ \Phi_s^Y\circ \Phi_{-t}^X\right)}=\widehat{\Phi_t^X}\circ \widehat{\Phi_s^Y}\circ \widehat{\Phi_{-t}^X}
\end{equation}
for all $t,s\in\R$ small enough. Moreover, since $\Phi_t^X$ is a local diffeomorphism and $\Phi_s^Y$ is the flow of the vector field $Y\in T\mathrm{M}$, we have from Proposition \ref{prop:diffeoinvflow} that
\begin{align}
    (\widehat{\Phi^X})_*\widetilde{Y} = \frac{\d}{\d s}\Big|_{s=0} \widehat{\Phi^X_t}\circ \widehat{\Phi^Y_s} \circ \widehat{\Phi^X_{-t}}.
\end{align}
On the other hand, let $Z_t:=(\Phi_t^X)_*Y\in T\mathrm{M}$ be the push-forward of $Y$ by the local diffeomorphism $\Phi_t^X$ generated by $X$ for $|t|\leq t_0$ for some $t_0>0$ and let $\Phi_s^{Z_t}$ denote its flow, that is,
\begin{align}
    \Phi_s^{Z_t} = \Phi_t^X\circ \Phi_s^Y\circ \Phi_{-t}^X.
\end{align}
Then, we have that 
\begin{align}
    \widetilde{(\Phi_t^X)_*Y} = \widetilde{Z_t} = \frac{\d }{\d s}\Big|_{s=0} \widehat{\Phi_s^{Z_t}} =\frac{\d }{\d s}\Big|_{s=0} \left(\widehat{\Phi_t^X\circ \Phi_s^Y\circ \Phi_{-t}^X}\right),
\end{align}
and, from \eqref{eq:commutativity_lift}, we conclude that $ (\widehat{\Phi_t^X})_*\widetilde Y \;=\; \widetilde{(\Phi_t^X)_*Y}$. Next, from Lemma \ref{lemma:commutator} we have that
\begin{align}
    [X,Y] = \mathcal{L}_X Y = \frac{\d }{\d t}\Big|_{t=0} (\Phi_{-t}^X)_*Y = \frac{\d}{\d s}\Big|_{s=0} \frac{\d}{\d t}\Big|_{t=0} \Phi_s^{Z_t}
\end{align}
and thus
\begin{align}
    \widetilde{[X,Y]} = \frac{\d}{\d r}\Big|_{r=0} \widehat{\frac{\d}{\d t}\Big|_{t=0}\Phi_r^{Z_t}} = \frac{\d}{\d r}\Big|_{r=0} \frac{\d}{\d t}\Big|_{t=0}\widehat{\Phi_r^{Z_t}} = \frac{\d}{\d t}\Big|_{t=0}\frac{\d}{\d r}\Big|_{r=0}\widehat{\Phi_r^{Z_t}} = \frac{\d}{\d t}\Big|_{t=0} \widetilde{(\Phi_{-t}^X)_*Y}.
\end{align}
Likewise, since
\begin{align}
    [\widetilde{X},\widetilde{Y}] = \mathcal{L}_{\widetilde{X}}\widetilde{Y} = \frac{\d}{\d t}\Big|_{t=0} (\widehat{\Phi_{-t}^X})_*\widetilde{Y}
\end{align}
we conclude that $[\widetilde{X},\widetilde{Y}] = \widetilde{[X,Y]}$.   
\end{proof}

\subsection{Commutators on a reference frame}
In this section we expand on the computations used in Proposition \ref{prop:LARCPJP}. We set
\begin{align}
\x_0=\left(\frac1{\sqrt2},\frac1{\sqrt2},0\right),\qquad  \bfe_1=(0,0,1),\qquad  \bfe_2=\left(\frac1{\sqrt2},-\frac1{\sqrt2},0\right),
\end{align}
and we consider the orthonormal and unit-area basis $E_0=(\bfe_1,\bfe_2)$ of $T_{\x_0}\Sph$. We observe that $\lbrace   \bfe_1, \bfe_2 \rbrace  $ is an oriented orthonormal basis of $T_{\x_0}\Sph$. We note that $\widetilde{\phi_t^{u_\a}}$ and $\widetilde{\phi_t^{u_\b}}$ defined in \eqref{eq:deftildephi} are the local lifted flows of $\phi_t^{u^\a}$ and $\phi_t^{u^\b}$ on the local trivialisation $\cU\times SL(2,\R)$ of the fiber bundle $SL(\Sph)$.

We recall that for  $\a=(0,a_2,a_3)\in \Sph$ and $\b = (0,0,1)$ we have
\begin{align}
    u^\a(\x)=(a_2 y + a_3z)\left((a_2z-a_3y), a_3x, -a_2x\right), \quad  u^\b(\x)=z(-y,x,0),
\end{align}
for all $\x=(x,y,z)\in \Sph$. Moreover, for any $w\in T\Sph$ we have\footnote{We also abuse notation by denoting a smooth extension on $\R^3$ of $w$ by $w$ itself.}
\begin{align}
    \Omega_{w}(\x_0,E_0)=E_0^T(D_\x w)(\x_0)E_0\in\mathfrak{sl}(2,\R).
\end{align}
Let $w_1(\x) = u^\a(\x)$ and $w_2(\x) = u^\b(\x)$. Then, $ w_1(\x_0) =- \frac{a_2^2}{2}\bfe_1 -a_2a_3\frac{\sqrt{2}}{2}\bfe_2 $ and
\begin{align}
    (D_\x w_1)(\x_0) = \begin{pmatrix}
        -\frac{a_2a_3}{2\sqrt{2}} & - \frac{3a_2a_3}{2\sqrt{2}} & \frac{a_2^2 -2a_3^2}{2\sqrt{2}} \\
        \frac{a_2a_3}{2\sqrt{2}} &  \frac{3a_2a_3}{2\sqrt{2}} & -\frac{a_2^2 -2a_3^2}{2\sqrt{2}} \\
        -\frac{a_2^2}{\sqrt{2}} & -\frac{a_2^2}{\sqrt{2}} & -\frac{a_2a_3}{\sqrt{2}}
    \end{pmatrix}, \quad \Omega_{w_1}(\x_0) = \begin{pmatrix}
        -\frac{a_2a_3}{\sqrt{2}} & 0 \\
        \frac{a_2^2}{2} - a_3^2 & \frac{a_2a_3}{\sqrt{2}}
    \end{pmatrix}
\end{align}
and $ w_2(\x_0) = 0$, with
\begin{align}
    (D_\x w_2)(\x_0) = \begin{pmatrix}
        0 & 0 & -\frac{1}{\sqrt{2}} \\
        0 & 0 & \frac{1}{\sqrt{2}} \\
        0 & 0 & 0
    \end{pmatrix}, \quad \Omega_{w_2}(\x_0) = \begin{pmatrix}
        0 & 0 \\ -1 & 0
    \end{pmatrix}.
\end{align}
For $w_3(\x) = [w_1,w_2](\x)$ we have
\begin{align}
    w_3(\x_0) = \frac{a_2^2}{2}\bfe_2, \quad (D_\x w_3)(\x_0) = \begin{pmatrix}
        \frac{3a_2^2}{4} & \frac{3a_2^2}{4} & a_2a_3 \\
        -\frac{3a_2^2}{4} & -\frac{3a_2^2}{4} & -a_2a_3 \\
        0 & 0 & 0
    \end{pmatrix}, \quad \Omega_{w_3}(\x_0) = \begin{pmatrix}
        0 & 0 \\ \sqrt{2}a_2 a_3 & 0
    \end{pmatrix}.
\end{align}
Next, for $w_4(\x) = [w_1,w_3](\x)$ there holds $w_4(\x_0) = -\frac{3\sqrt{2}}{4}a_2^3 a_3 \bfe_2$ and 
\begin{align}
    (D_\x w_4)(\x_0) = a_2^2\begin{pmatrix}
        -\frac{11a_2a_3}{4\sqrt{2}} & -\frac{13 a_2 a_3}{4\sqrt{2}} & a_2^2\sqrt{2} - \frac{3a_3^2}{\sqrt{2}} \\
        \frac{11a_2a_3}{4\sqrt{2}} & \frac{13 a_2 a_3}{4\sqrt{2}} & -a_2^2\sqrt{2} + \frac{3a_3^2}{\sqrt{2}} \\
        -a_2^2\sqrt{2} & a_2^2\sqrt{2} & - \frac{a_2 a_3}{2\sqrt{2}}
    \end{pmatrix}, \quad \Omega_{w_4}(\x_0) = a_2^2 \begin{pmatrix}
        -\frac{a_2 a_3}{2\sqrt{2}} & -2a_2^2  \\ 2a_2^2 - 3a_3^2  &  \frac{a_2 a_3}{2\sqrt{2}}
    \end{pmatrix}.
\end{align}
Finally, for $w_5(\x)=[w_2,w_4](\x)$ we have $w_5(\x_0) = 0$ and 
\begin{align}
    (D_\x w_5)(\x_0) = a_2^2\begin{pmatrix}
        -a_2^2 & a_2^2 & - \frac{a_2 a_3}{2} \\
        a_2^2 & -a_2^2 &  \frac{a_2 a_3}{2} \\
        0 & 0 & 2 a_2^2
    \end{pmatrix}, \quad \Omega_{w_5}(\x_0) = \begin{pmatrix}
        2a_2^4 & 0 \\
        - \frac{a_2 ^3 a_3}{\sqrt{2}} & -2a_2^4
    \end{pmatrix}.
\end{align}

\section{Leading order expansion for commutator loops}\label{app:taylor-flows}
This appendix develops the small-time asymptotic expansion of commutator loop maps built from the local flows of a finite family of smooth vector fields $\cV=\lbrace V_1,\dots,V_m\rbrace  \subset\mathfrak X(\RR^d)$ treated in Section \ref{sec:LARC}.
The main goal is to justify the Lie-algebraic heuristic that concatenations of iterated flow maps, defined in \eqref{eq:loops-induction}, generate the corresponding iterated Lie brackets at leading order: for every multi-index $I=(i_1,\dots,i_n)$,
\begin{equation}\label{eq:appendix-main-expansion-preview}
\Gamma_t^I(x) = x+t^{|I|}V_I(x)+t^{|I|+1}\cE_I(t,x), \qquad V_I=[V_{i_1},[V_{i_2},\dots,[V_{i_{n-1}},V_{i_n}]\dots]],
\end{equation}
with a locally uniformly bounded remainder $\cE_I$.

From \cite{LeeManifolds}*{Ch.~9}, given a smooth vector field $V\in\mathfrak{X}(\RR^d)$ the Fundamental Theorem on flows ensures the existence of a unique smooth maximal flow $\Phi\colon \mathrm{D}\to \RR^d$ on an open flow domain $\mathrm{D}\subset\RR\times \RR^d$, with 
\begin{equation}
    \partial_t\Phi_t(x)=V(\Phi_t(x)),\quad \Phi_0(x)=x.
\end{equation}
In particular, in $\RR^d$ the chain rule yields the usual ODE identity
\begin{equation}
\partial_{tt}\Phi_t(x)=DV(\Phi_t(x))[V(\Phi_t(x))],
\end{equation}
hence at $t=0$ one has $\partial_t\Phi_0(x)=V(x)$ and $\partial_{tt}\Phi_0(x)=DV(x)[V(x)]$.
Fixing a compact $K\Subset\RR^d$, openness of $\mathrm{D}$ gives a uniform $t_0>0$ with $[-t_0,t_0]\times K\subset \mathrm{D}$, so that applying the one-dimensional Taylor theorem in time to $t\mapsto \Phi_t(x)$ (componentwise) produces the expansion
\begin{equation}\label{eq:flow-Taylor}
\Phi_t(x)=x+t V(x)+\frac{t^2}{2} DV(x)[V(x)] + t^3 \cE(t,x),
\end{equation}
with $\Vert \cE(t,x) \Vert_{C^1([-t_0,t_0]\times K)}\lesssim 1$. 
As a direct consequence we obtain on $[-t_0,t_0]\times K$ also 
\begin{equation}\label{eq:dflow-Taylor-input}
    D\Phi_t(x)=\Id + t DV(x)+t^2 \cS(t,x),
\end{equation}
where $\Vert \cS(t,x) \Vert_{L^\infty([-t_0,t_0]\times K)}\lesssim 1$.
Building on this standard setup, the following lemma establishes the leading order expansion for vector fields defined via iterated Lie brackets.

\begin{lemma}[Taylor expansion of commutator loops]\label{Alem:commutator-loop-taylor}
Let $\cV=\lbrace V_1,\dots,V_m\rbrace  \subset \mathfrak X(\RR^d)$ be smooth vector fields with (local) flows $\Phi_t^i$. For a multi-index $I=(i_1,\dots,i_n)\in\lbrace 1,\dots,m\rbrace  ^n$ define the iterated commutator vector field
\begin{equation}
V_I := [V_{i_1},[V_{i_2},\dots,[V_{i_{n-1}},V_{i_n}]\dots]],
\end{equation}
and define the associated commutator loop maps $\Gamma_t^{I}$ inductively by
\begin{equation}
\Gamma_t^{(i)}:=\Phi_t^i,\qquad  \Gamma_t^{(j,I)}:=\Gamma_t^{(j)}\circ\Gamma_t^I\circ(\Gamma_t^{(j)})^{-1}\circ(\Gamma_t^I)^{-1}.
\end{equation}
Then for every compact $K\Subset\RR^d$ there exists $t_0>0$ and maps $\cE_I:[-t_0,t_0]\times K\to\RR^d$ (continuous in $(t,x)$ and smooth in $x$) such that for all $x\in K$ and $|t|\leq t_0$ (for which all compositions are defined),
\begin{equation}\label{Aeq:Gamma-expansion}
\Gamma_t^I(x)=x+t^n V_I(x)+t^{n+1}\cE_I(t,x), \qquad n=|I|.
\end{equation}
\end{lemma}

As preliminaries, we state and prove two auxiliary lemmas which play an important role in the proof of Lemma \ref{Alem:commutator-loop-taylor}.
\begin{lemma}\label{lem:inverse}
Let $K\Subset U\subset \RR^d$, let $n\ge 1$, and suppose that
\begin{equation}
F_t(x)=x+t^nA(x)+t^{n+1}R(t,x), \qquad x\in U,
\end{equation}
where $A\in C^1(U;\RR^d)$ and
\begin{equation}
R\in C^1([-t_0,t_0]\times U;\RR^d), \qquad \sup_{|t|\le t_0}\norm{R(t,\cdot)}_{C^1(U)}<\infty.
\end{equation}
Then, after possibly shrinking $t_0>0$, the map $F_t$ is a $C^1$-diffeomorphism from a neighbourhood of $K$ onto its image. Moreover, for every $y\in F_t(K)\subset U$, its inverse satisfies
\begin{equation}\label{eq:inverse_ord_n}
F_t^{-1}(y) = y-t^nA(y)+t^{n+1}\widetilde R(t,y),
\end{equation}
where $\widetilde R$ is uniformly bounded for $|t|\le t_0$ and $y\in F_t(K)$.
\end{lemma}
\begin{proof}
Upon shrinking $U$ and picking $t_0$ sufficiently small, the $C^1$-diffeomorphism claim follows via standard arguments from 
\begin{equation}
D_xF_t(x) = I+t^n DA(x)+t^{n+1}D_xR(t,x).
\end{equation}
It remains to prove the expansion for the inverse. Let $y\in F_t(K)$, and set $ x:=F_t^{-1}(y)\in K$, so that
\begin{equation}
y=x+t^nA(x)+t^{n+1}R(t,x).
\end{equation}
In particular, there holds $|x-y|\leq C |t|^n$ for some universal $C>0$. We rewrite the above as
\begin{equation}
    x-y+t^nA(y)=-t^n\bigl(A(x)-A(y)\bigr)-t^{n+1}R(t,x), 
\end{equation}
and, using the $C^1$-regularity of $A$, the estimate $|x-y|\le C|t|^n$, $2n\geq n+1$, and the boundedness of $R$ we obtain 
\begin{equation}
x-y+t^nA(y)=O(t^{n+1})
\end{equation}
uniformly for $y\in F_t(K)$.
Equivalently, there exists a uniformly bounded function $\widetilde R(t,y)$ such that
\begin{equation}
F_t^{-1}(y)=y-t^nA(y)+t^{n+1}\widetilde R(t,y),
\end{equation}
which proves \eqref{eq:inverse_ord_n}.
\end{proof}

\begin{lemma}\label{lem:conjugation}
Fix $j$ and let $\Phi_t^j$ be the flow of $V_j$.
For any smooth vector field $W$ there exists a bounded remainder $\cR_{j,W}(t,\cdot)$ on $K$ such that for all $y\in K$,
\begin{equation}\label{eq:pushforward-expansion}
D\Phi_t^j(\Phi_{-t}^j(y)) W(\Phi_{-t}^j(y)) = W(y) + t [V_j,W](y) + t^2 \cR_{j,W}(t,y).
\end{equation}
\end{lemma}

\begin{proof}[Proof of Lemma \ref{lem:conjugation}]
Using \eqref{eq:dflow-Taylor-input} at $x=\Phi_{-t}^j(y)$ and Taylor expanding $W(\Phi_{-t}^j(y))$ in space around $y$ (noting $\Phi_{-t}^j(y)=y-tV_j(y)+O(t^2)$ from \eqref{eq:flow-Taylor}), we get
\begin{equation}
W(\Phi_{-t}^j(y)) = W(y) - t DW(y)[V_j(y)] + t^2 R_1(t,y),
\end{equation}
and
\begin{equation}
D\Phi_t^j(\Phi_{-t}^j(y)) = \Id + t DV_j(y) + t^2 R_2(t,y),
\end{equation}
with bounded $R_1,R_2$ on $[-t_0,t_0]\times K$. 
Multiplying the two expansions gives
\begin{equation}
D\Phi_t^j(\Phi_{-t}^j(y))W(\Phi_{-t}^j(y)) = W(y) + t\big(DV_j(y)[W(y)]-DW(y)[V_j(y)]\big) + t^2\cR_{j,W}(t,y),
\end{equation}
and the coefficient of $t$ is exactly $[V_j,W](y)$.
\end{proof}

\begin{proof}[Proof of Lemma \ref{Alem:commutator-loop-taylor}]
Fix a compact set $K\Subset\RR^d$. We will work for $|t|\leq t_0$ so small that all flows appearing in the induction are defined on the relevant compact neighbourhoods; since each $\Phi_t^i$ depends smoothly on $(t,x)$, such a uniform $t_0>0$ exists after possibly shrinking.
The proof is carried out via induction on $n=|I|$ in three steps: the base case $n=1$, the first commutator case $n=2$, and the induction step from $n\ge2$ to $n+1$.

\emph{Base case $n=1$.}
If $I=(i)$ then $\Gamma_t^I=\Phi_t^i$, and \eqref{Aeq:Gamma-expansion} is exactly \eqref{eq:flow-Taylor} with $V_I=V_i$ and $\cE_I(t,x):=\frac12 DV_i(x)[V_i(x)] + t \cE_i(t,x)$.
Thus \eqref{Aeq:Gamma-expansion} holds for $n=1$.

\emph{First commutator case $n=2$.}
Let $I=(j,i)$, so that $V_I=[V_j,V_i]$ and
\begin{equation}
    \Gamma_t^I = \Gamma_t^{(j)}\circ\Gamma_t^{(i)} \circ(\Gamma_t^{(j)})^{-1}\circ(\Gamma_t^{(i)})^{-1} = \Phi_t^j\circ\Phi_t^i\circ\Phi_{-t}^j\circ\Phi_{-t}^i.
\end{equation}
Fix $x\in K$ and set $y:=\Phi_{-t}^i(x)$ so that, by \eqref{eq:flow-Taylor}, $y=x-tV_i(x)+O(t^2)$.
Next,
\begin{equation}
    \Phi_{-t}^j(y) = y-tV_j(y)+\frac{t^2}{2}DV_j(y)[V_j(y)]+O(t^3).
\end{equation}
Using $y=x-tV_i(x)+O(t^2)$ and Taylor expanding $V_j(y)$ around $x$, we get
\begin{align}
    \Phi_{-t}^j\circ\Phi_{-t}^i(x) &= x-tV_i(x)-tV_j(x) +t^2DV_j(x)[V_i(x)]+\frac{t^2}{2}DV_i(x)[V_i(x)] +\frac{t^2}{2}DV_j(x)[V_j(x)] +O(t^3).
\end{align}
Applying $\Phi_t^i$ to this expression and using \eqref{eq:flow-Taylor} again gives
\begin{align}
    \Phi_t^i\circ\Phi_{-t}^j\circ\Phi_{-t}^i(x) &= x-tV_j(x) +t^2\left( DV_j(x)[V_i(x)] -DV_i(x)[V_j(x)] +\frac12 DV_j(x)[V_j(x)] \right) +O(t^3).
\end{align}
Finally, applying $\Phi_t^j$ gives
\begin{align}
    \Gamma_t^I(x) &= x+t^2\left(DV_j(x)[V_i(x)]-DV_i(x)[V_j(x)]\right)+O(t^3)\\ &= x+t^2[V_j,V_i](x)+O(t^3).
\end{align}
Therefore, since $V_I=[V_j,V_i]$,
\begin{equation}
    \Gamma_t^I(x) = x+t^2V_I(x)+t^3\cE_I(t,x),
\end{equation}
for a bounded remainder $\cE_I$ on $[-t_0,t_0]\times K$.
Thus \eqref{Aeq:Gamma-expansion} holds for all multi-indices of length $2$.

\emph{Inductive step.}
Assume \eqref{Aeq:Gamma-expansion} holds for some $I$ of length $n\ge2$:
\begin{equation}
\Gamma_t^I(x)=x+t^n V_I(x)+t^{n+1}\cE_I(t,x)\qquad (x\in K).
\end{equation}
Fix $j\in\lbrace 1,\dots,m\rbrace  $ and set $J:=(j,I)$ so that $V_J=[V_j,V_I]$ and
\begin{equation}
\Gamma_t^J=\Gamma_t^{(j)}\circ\Gamma_t^I\circ(\Gamma_t^{(j)})^{-1}\circ(\Gamma_t^I)^{-1}.
\end{equation}
Write $y:=(\Gamma_t^I)^{-1}(x)$. By Lemma \ref{lem:inverse},
\begin{equation}\label{eq:y-inverse}
y=x-t^nV_I(x)+t^{n+1}\widetilde{\cE}_I(t,x)
\end{equation}
for some bounded $\widetilde{\cE}_I$, and consider the conjugated map
\begin{equation}
C_t:=\Gamma_t^{(j)}\circ\Gamma_t^I\circ(\Gamma_t^{(j)})^{-1}.
\end{equation}
Using the inductive hypothesis for $\Gamma_t^I(x)$ at $(\Gamma_t^{(j)})^{-1}(y)$ gives
\begin{equation}
\Gamma_t^I((\Gamma_t^{(j)})^{-1}(y)) = (\Gamma_t^{(j)})^{-1}(y) + t^n V_I((\Gamma_t^{(j)})^{-1}(y)) + t^{n+1}\cE_I(t,(\Gamma_t^{(j)})^{-1}(y)).
\end{equation}
Set $\delta:=t^n V_I((\Gamma_t^{(j)})^{-1}(y)) + t^{n+1}\cE_I(t,(\Gamma_t^{(j)})^{-1}(y))$, so $\delta=O(t^n)$.
Now expand $\Gamma_t^{(j)}$ in the space variable around $(\Gamma_t^{(j)})^{-1}(y)$:
\begin{equation}
\Gamma_t^{(j)}((\Gamma_t^{(j)})^{-1}(y)+\delta) = \Gamma_t^{(j)}((\Gamma_t^{(j)})^{-1}(y)) + D\Gamma_t^{(j)}((\Gamma_t^{(j)})^{-1}(y))[\delta] + O(|\delta|^2).
\end{equation}
Since $\Gamma_t^{(j)}((\Gamma_t^{(j)})^{-1}(y))=y$ and $|\delta|^2=O(t^{2n})$, we get
\begin{equation}
C_t(y)=y + D\Gamma_t^{(j)}((\Gamma_t^{(j)})^{-1}(y))[\delta] + t^{2n}\cR(t,y).
\end{equation}
Since $n\ge2$, $t^{2n}=O(t^{n+2})$ and therefore this quadratic term is automatically beyond the order we need.
Thus,
\begin{align}
C_t(y)&= y + t^n  D\Gamma_t^{(j)}((\Gamma_t^{(j)})^{-1}(y)) V_I((\Gamma_t^{(j)})^{-1}(y)) + t^{n+1}  D\Gamma_t^{(j)}((\Gamma_t^{(j)})^{-1}(y)) \cE_I(t,(\Gamma_t^{(j)})^{-1}(y)) + t^{n+2}\cR_1(t,y).
\end{align}
Apply Lemma \ref{lem:conjugation} with $W=V_I$ to the first term and use \eqref{eq:dflow-Taylor-input} to handle the second term (it contributes only $t^{n+1}\cE_I(t,y)$ up to $O(t^{n+2})$). We obtain
\begin{equation}\label{eq:Ct-expansion}
C_t(y) = y + t^n V_I(y) + t^{n+1}\big([V_j,V_I](y)+\cE_I(t,y)\big) + t^{n+2}\cR_2(t,y).
\end{equation}
Substituting $\Gamma_t^I(y)=y+t^nV_I(y)+t^{n+1}\cE_I(t,y)$ from the inductive hypothesis into \eqref{eq:Ct-expansion} yields
\begin{equation}
C_t(y) = \Gamma_t^I(y)+t^{n+1}[V_j,V_I](y) + t^{n+2}\cR_2(t,y).
\end{equation}
Finally,
\begin{equation}
\Gamma_t^J(x)=C_t\big((\Gamma_t^I)^{-1}(x)\big)=x + t^{n+1}[V_j,V_I]\big((\Gamma_t^I)^{-1}(x)\big) + t^{n+2}\cR_3(t,x).
\end{equation}
Since $(\Gamma_t^I)^{-1}(x)=x+O(t^n)$ by \eqref{eq:y-inverse} and $[V_j,V_I]$ is smooth, we have
\begin{align}
    t^{n+1}[V_j,V_I]\big((\Gamma_t^I)^{-1}(x)\big)=t^{n+1}[V_j,V_I](x) + O(t^{2n+1})=t^{n+1}[V_j,V_I](x) +O(t^{n+2})
\end{align}
for $n\ge2$. Thus,
\begin{equation}
\Gamma_t^J(x)=x+t^{n+1}[V_j,V_I](x)+t^{n+2}\cE_J(t,x),
\end{equation}
for a bounded remainder $\cE_J$ on $[-t_0,t_0]\times K$.
This is exactly \eqref{Aeq:Gamma-expansion} for the multi-index $J$ of length $n+1$. The induction is complete.
\end{proof}

\bibliographystyle{abbrv}
\bibliography{biblio_general}

\end{document}